\documentclass[11pt]{report}
\usepackage[top=1in, includehead, bottom=1in, left=1.5in, right=1in]{geometry}
\usepackage{graphicx}
\usepackage{amsmath,amsthm,amssymb}
\usepackage{xypic}
\usepackage{mathabx}
\usepackage{bbm}
\usepackage{color}
\usepackage{enumitem}
\usepackage[hidelinks]{hyperref}
\newtheorem*{theorem*}{Theorem}
\newtheorem*{proposition*}{Proposition}
\newtheorem*{definition*}{Definition}
\newtheorem*{example*}{Example}
\newtheorem*{remark*}{Remark}
\newtheorem*{corollary*}{Corollary}
\newtheorem*{lemma*}{Lemma}
\newtheorem*{warning*}{Warning}
\newtheorem*{principle*}{Principle}
\newtheorem*{question*}{Question}
\newtheorem*{conjecture*}{Conjecture}
\newtheorem{theorem}{Theorem}[subsection]
\makeatletter
\@addtoreset{theorem}{section}
\@addtoreset{theorem}{chapter}

\makeatother
\newtheorem{proposition}[theorem]{Proposition}
\newtheorem{definition}[theorem]{Definition}
\newtheorem{example}[theorem]{Example}
\newtheorem{remark}[theorem]{Remark}
\newtheorem{corollary}[theorem]{Corollary}
\newtheorem{lemma}[theorem]{Lemma}
\newtheorem{warning}[theorem]{Warning}
\newtheorem{principle}[theorem]{Principle}
\newtheorem{question}[theorem]{Question}
\newtheorem{conjecture}[theorem]{Conjecture}
\newcommand{\oast}{\circledast}

\newcommand\frontmatter{
  \cleardoublepage
  \pagenumbering{roman}
}

\newcommand\mainmatter{
  \cleardoublepage
  \pagenumbering{arabic}
}

\usepackage{fancyhdr}
\fancypagestyle{plain}{}
\begin{document}
\frontmatter

\thispagestyle{empty}

\begin{center}

\vspace*{1in}

Categorified algebra and equivariant homotopy theory
\\[12mm]
John D. Berman
\\
Wilmington, North Carolina
\\[12mm]
Bachelor of Science, Massachusetts Institute of Technology, 2013
\\[22mm]
A Dissertation presented to the Graduate Faculty
\\
of the University of Virginia in Candidacy for the Degree of
\\
Doctor of Philosophy
\\[12mm]
Department of Mathematics
\\[8mm]
University of Virginia
\\
May, 2018

\end{center}

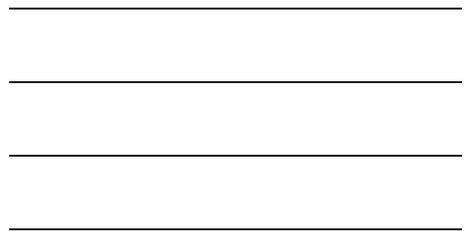
\vfill\

\begin{flushright}
\rule{6cm}{.01cm} \\[5mm]
\rule{6cm}{.01cm} \\[5mm]
\rule{6cm}{.01cm} \\[5mm]
\rule{6cm}{.01cm} \\[5mm]
\end{flushright}

\newpage
\chapter*{Abstract}
\noindent This dissertation comprises three collections of results, all united by a common theme. The theme is the study of categories via algebraic techniques, considering categories themselves as algebraic objects. This algebraic approach to category theory is central to noncommutative algebraic geometry, as realized by recent advances in the study of noncommutative motives.

We have success proving algebraic results in the general setting of symmetric monoidal and semiring $\infty$-categories, which categorify abelian groups and rings, respectively. For example, we prove that modules over the semiring category Fin of finite sets are cocartesian monoidal $\infty$-categories, and modules over Burn (the Burnside $\infty$-category) are additive $\infty$-categories.

As a consequence, we can regard Lawvere theories as cyclic $\text{Fin}^\text{op}$-modules, leading to algebraic foundations for the higher categorical study of Lawvere theories. We prove that Lawvere theories function as a home for an \emph{algebraic Yoneda lemma}.

Finally, we provide evidence for a formal duality between naive and genuine equivariant homotopy theory, in the form of a group-theoretic Eilenberg-Watts Theorem. This sets up a parallel between equivariant homotopy theory and motivic homotopy theory, where Burnside constructions are analogous to Morita theory. We conjecture that this relationship could be made precise within the context of noncommutative motives over the field with one element.

In fact, the connections equivariant homotopy theory and the field with one element recur throughout the thesis. There are promising suggestions that each of these two subjects can be advanced by further work in this area of \emph{algebraic category theory}.

\newpage\par\vspace*{.35\textheight}{\centering For my father, the original mathematician in my life.\par}

\chapter*{Acknowledgments}
I would like to thank Mike Hill for his support and all the helpful conversations, and for agreeing to advise me despite his cross-country move early in my graduate studies. I also want to thank Nick Kuhn for his unofficial role as advisor-in-residence, Clark Barwick for challenging me to learn more about higher category theory (and for his generous support from the beginning), Ben Antieau for nudging me towards derived algebraic geometry, and Jack Morava for suggesting that my results are related to the field with one element.

This thesis has benefited from conversations with countless others, including (to name just a few) Julie Bergner, Prasit Bhattacharya, Andrew Blumberg, Peter Bonventre, Saul Glasman, Rune Haugseng, Ben Knudsen, Bogdan Krstic, Lennart Meier, Denis Nardin, Christina Osborne, Jay Shah, Brian Thomas, and Dylan Wilson.

I am grateful for the kindness of my friends and family in Virginia and elsewhere. Most importantly, none of this would be possible without the example and steady love of my parents.

%\addtocontents{toc}{\unexpanded{\unexpanded{{
%  \itshape\footnotesize
%  Citations and internal cross-references are clickable links.
%  \par\bigskip
%}}}}

\tableofcontents{}

\mainmatter
\chapter*{Introduction}
\addcontentsline{toc}{chapter}{Introduction}
\noindent This thesis is comprised of three chapters of new results, each representing roughly one paper, and each with its own introduction. In this introduction, we will outline in broad strokes the unifying theme of the thesis, and its position in the body of work surrounding equivariant homotopy theory, noncommutative motives, and higher category theory.

The unifying theme is the employment of purely algebraic techniques to study categories (thinking of them as categorifications of abelian groups), rather than what we might call explicit object-morphism arguments. These techniques involve computing tensor products of symmetric monoidal categories, as well as ideas related to homological algebra, flatness, and descent. We illustrate the flavor of these results with two theorems and a conjecture. These may be taken as algebraic definitions, although it is by no means obvious that they agree with the usual categorical definitions.

\begin{theorem*}[Theorem \ref{2T4}]
An \emph{additive $\infty$-category} is a module over the semiring $\infty$-category Burn (the Burnside $\infty$-category, or virtual spans of finite sets).
\end{theorem*}

\begin{theorem*}[Theorem \ref{2T3}, Example \ref{LawvereExample}]
A \emph{Lawvere theory} is a cyclic module over the semiring category $\text{Fin}^\text{op}$ (the opposite category of finite sets).
\end{theorem*}

\begin{conjecture*}[Conjecture \ref{Conj2}]
An \emph{$\infty$-operad} is a cyclic module over $\text{Fin}^\text{op}$ which is trivial over $\text{Burn}^\text{eff}$ (the \emph{effective} Burnside 2-category, or spans of finite sets).
\end{conjecture*}

\subsubsection{Context}
\noindent The idea to treat categories as algebraic objects in their own right is not a new one; it is an idea that plays an important role in the theory of \emph{motives}, introduced by Grothendieck to clarify his 1960's work on the Weil Conjectures. More recently, \emph{algebraic category theory} (or \emph{categorified algebra}) is a central focus of work on noncommutative motives and algebraic K-theory (such as by Blumberg, Gepner, and Tabuada \cite{BGT}), and it is also foundational in Lurie's $\infty$-categorical approach to homotopy theory (HA\footnote{We will use special conventions to refer to Lurie's two books, which we reference often: HTT for \textit{Higher Topos Theory} \cite{HTT} and HA for \textit{Higher Algebra} \cite{HA}.}).

An early milestone was the definition in the 1950's of an abelian category (first by Buchsbaum \cite{Buchsbaum}, then streamlined by Grothendieck \cite{GrothAbelian}). Weil had suggested that the Weil conjectures could be solved by constructing a well-behaved cohomology theory \cite{Weil}. The notion of an abelian category, by axiomatizing the structure of a category of modules over a ring ($\text{Mod}_R$), allowed Grothendieck to distill the most important properties of cohomology theories to realize Weil's program.

Over time, the search for new algebraic cohomology theories led to a zoo of homological invariants of (noncommutative) rings, including K-theory, Hochschild homology, cyclic homology, periodic homology, and so on. There is some hope that one of these (noncommutative) cohomology theories is a Weil cohomology theory, in the sense that it can be used to reprove the Weil conjectures. For example, Blumberg and Mandell recently proved a Kunneth theorem (one of the requirements for a Weil cohomology theory) for TP \cite{BlumMand}. There is also a well-understood theory of duality for noncommutative motives \cite{BGT}, a key ingredient for Poincare duality.

However, to make sense of these results, it is necessary to work with stable $\infty$-categories, rather than rings, as a basic object of study. There were hints of this as far back as the 1950's in the guise of Morita theory. Now we know that the invariants of rings just mentioned (K-theory, Hochschild homology, etc.) all have something in common: they are \emph{Morita invariant} \cite{Tabuada}, in the sense that they may be regarded as invariants of the derived category of chain complexes of modules $\text{Mod}_R^d$ (which is a stable $\infty$-category). That is, if the $\infty$-categories $\text{Mod}_R^d$ and $\text{Mod}_S^d$ are equivalent, then $R$ and $S$ have equivalent K-theory, Hochschild homology, etc.

Therefore, we are led to consider the stable $\infty$-category $\text{Mod}_R^d$ as a basic object of study in noncommutative algebraic geometry --- rather than the ring $R$. Roughly, such a stable $\infty$-category is a \emph{noncommutative motive}. In addition to the theoretical results just mentioned related to duality and the Kunneth theorem, these ideas are also leading to new computations in algebraic K-theory (via trace methods); for recent work see Nikolaus-Scholze \cite{NikScholze}. \\

\noindent Notice the direction these developments have taken us. Categories were invented as a book-keeping tool, to keep track of large amounts of algebraic structure and therefore avoid excessively long and technical exposition. Now there are signs that they should be regarded as objects of central algebraic interest in their own right.

This is a story as old as mathematics. At each leap in abstraction, a new structure is defined to aid in book-keeping (fractions to organize work in geometry, functions for calculus, groups for algebra), but in time the structure becomes itself a basic object of study. (Consider all the work put into classifying finite simple groups.) \\

\noindent For most of the thesis, we study symmetric monoidal categories, treating them as categorifications of abelian groups. This algebraic approach is only possible due to recent developments in higher category theory. Notably, the following fact is \emph{not} true of ordinary categories, without 2-categorical technicalities.

\begin{proposition*}[HA 2.4.2.6]
A symmetric monoidal $\infty$-category is a commutative monoid in the $\infty$-category of $\infty$-categories.
\end{proposition*}

\noindent Nonetheless, our approach does lead to new results even for ordinary categories. They can only be proven, as far as we know, using the machinery of $\infty$-categories.

The author is aware of connections between this work and at least three active areas of research.
\begin{itemize}
\item noncommutative algebraic geometry (especially motives)
\item the field with one element
\item equivariant homotopy theory
\end{itemize}

\noindent The first two will be in the background throughout the thesis, but rarely discussed explicitly after this introduction. (The second in particular remains a distant goal.)

Equivariant homotopy theory will recur as a motivating example throughout the thesis, but the reader will find a more systematic discussion in the appendix.

\subsubsection{Organization}
\noindent The thesis comprises four chapters. The first chapter is an expository account of $\infty$-categories, as developed in the last fifteen years by Joyal, Lurie, Rezk, Barwick, and others. There are no new results and essentially no proofs.

The second chapter sets up the foundation of our commutative algebra approach to category theory. We show that many flavors of symmetric monoidal categories can be regarded as modules over particular semiring categories. Here we are regarding symmetric monoidal categories as categorified abelian groups and semiring categories as categorified rings.

Examples include cocartesian monoidal categories (modules over Fin), cartesian monoidal categories (modules over $\text{Fin}^\text{op}$), semiadditive categories (modules over $\text{Burn}^\text{eff}$), additive categories (modules over Burn), and connective spectra (modules over the sphere spectrum, regarded as an $\infty$-category).

The third chapter uses techniques from Chapter \ref{3} to study Lawvere theories in the guise of cyclic $\text{Fin}^\text{op}$-modules. We regard the framework of Lawvere theories as providing a sort of \emph{algebraic Yoneda lemma}, in the form of a comparison between the commutative algebra of finitely generated $\infty$-categories (of Chapter \ref{3}) and the commutative algebra of presentable $\infty$-categories (due to Lurie, HA 4.8.1).

Chapter \ref{3} and parts of Chapter \ref{4} are taken from the author's paper \cite{Berman}, but organized slightly differently. Other parts of Chapter \ref{4} are new and will appear in a separate paper.

In the fourth chapter, we introduce the following conjecture. Actually, the terms used in the statement have not yet been defined, but we rigorously state and prove a weak variant of the conjecture. We will revisit the weak variant later in this introduction.

\begin{conjecture*}[Conjecture \ref{DualConj}]
If $G$ is a finite group, the $\infty$-category $_G\text{Top}$ of genuine left $G$-spaces and $\text{Top}^G$ of naive (or Borel) right $G$-spaces are dual as noncommutative motives over the field with one element.
\end{conjecture*}

\subsubsection{Open problems}
\noindent There are also a variety of open problems and conjectures that appear throughout the thesis. Aside from the conjecture above, the open problems broadly arise in two ways:
\begin{itemize}
\item Computations in this area seem fairly difficult. For example, we have some success computing very specific tensor products of symmetric monoidal $\infty$-categories via universal properties, but we don't know how to compute tensor products explicitly. They seem to be related to the span construction. (See \S\ref{3.2.2}.)
\item There are some hints of a theory of categorified algebraic geometry (including notions of flatness, descent, and 2-vector bundles), which may help with computations, but we do not yet understand any details. (See \S\ref{4.3}.)
\end{itemize}

\noindent We take the rest of this introduction to discuss in broad strokes some of the background and goals of the thesis. For more concrete discussion, skip to the introductions to each chapter.

\section*{Noncommutative motives}
\addcontentsline{toc}{section}{Noncommutative motives}
\noindent Grothendieck proposed in the 1960's that there should be a universal Weil cohomology theory $\text{Sch}\rightarrow\text{Mo}$, where Sch is some category of schemes, and Mo is a category of motives. That is, any other Weil cohomology theory $\text{Sch}\rightarrow\mathcal{C}$ should factor through Mo. We should be able to replace `Weil cohomology theories' with other types of cohomology theories to obtain different categories of motives. Of course we are suppressing many technicalities.

Since Mo is the home of a universal Weil cohomology theory, its objects should satisfy properties that imply Poincare duality, the Kunneth theorem, etc. These properties are encapsulated in Grothendieck's \emph{standard conjectures}, one of the goals of which was to prove the Riemann Hypothesis for curves over a finite field. (This form of the Riemann Hypothesis was later proven by Deligne using a different approach.)

From the beginning, there have been suggestions of a relationship between the Weil conjectures and the usual Riemann hypothesis. That is, the Riemann zeta function is the zeta function for $\text{Spec}(\mathbb{Z})$, which we are to think of as a curve over the (nonexistent) finite field $\mathbb{F}_1$. So in some sense, the Riemann hypothesis should be a consequence of (hypothetical) Weil conjectures for curves over $\mathbb{F}_1$.

In the 1990's, Deninger outlined this program in more detail \cite{Deninger}, describing the properties that a Weil cohomology theory over $\mathbb{F}_1$ would have to satisfy. Such a cohomology theory would incorporate ideas from both algebraic geometry and functional analysis. \\

\noindent There is a major difficulty in studying $\mathbb{F}_1$: there is no convincing notion of what it should be. On the other hand, we \emph{can} say something about what an $\mathbb{F}_1$-\emph{module} should be. For example, a finitely generated free $\mathbb{F}_1$-module is a finite set.

This situation is reminiscent of noncommutative motives, as discussed earlier. A noncommutative motive (in the sense of Blumberg, Gepner, and Tabuada \cite{BGT}) is roughly a type of stable $\infty$-category. The noncommutative motive attached to a ring $R$ is its derived $\infty$-category of modules $\text{Mod}_R^d$ (or the associated $\infty$-category of perfect complexes of modules).

So we might imagine that the noncommutative motive associated to $\mathbb{F}_1$ is related to the category Fin of finite sets. Of course, Fin is not really a noncommutative motive, because it is not a stable $\infty$-category (or even additive)! However, we might hope to enlarge the $\infty$-category of noncommutative motives to an $\infty$-category of \emph{noncommutative motives over $\mathbb{F}_1$}. (Presumably this would include at least stripping away the additivity condition, and probably carefully incorporating some functional analysis as well.) If we could then construct something like a Weil cohomology theory for such objects, we might begin to fantasize about the Riemann Hypothesis. \\

\noindent Of course, this is all hypothetical, but it suggests the need for an algebraic understanding of $\infty$-categories which are \emph{not necessarily stable}, which should assemble into an $\infty$-category $\text{NMo}_{\mathbb{F}_1}$ of noncommutative motives over $\mathbb{F}_1$. We presume that $\text{NMo}_{\mathbb{F}_1}$ would come equipped with a symmetric monoidal structure for which something like Fin is the unit (in its guise as the category of finitely generated $\mathbb{F}_1$-modules).

In other words, noncommutative motives over $\mathbb{F}_1$ will be Fin-modules in some sense that we cannot yet make precise. However, along these lines we will prove:

\begin{theorem*}[Theorem \ref{2T3}]
A module over the semiring $\infty$-category Fin is precisely a cocartesian monoidal $\infty$-category.

A module over the semiring $\infty$-category Burn (of finitely generated free $H\mathbb{Z}$-modules) is an additive $\infty$-category.
\end{theorem*}

\section*{Commutative algebra of categories}
\addcontentsline{toc}{section}{Commutative algebra of categories}
\noindent We have discussed the connection between what might be called \emph{algebraic category theory} and motives. Here we discuss another important application of algebraic techniques in category theory, developed by Lurie.

Until the 1990's, the construction of a well-behaved smash product on the category of spectra was a major and long-standing open problem. Although the problem was solved a decade earlier, Lurie's solution (HA \S4.8.2) is completely formal and very simple by comparison, using the commutative algebra of presentable $\infty$-categories.

A presentable $\infty$-category is roughly an $\infty$-category which can be presented by a model category. (More specifically, it is closed under small homotopy limits and homotopy colimits and generated under homotopy colimits by a small set of objects.) The following theorem ultimately underlies most of the algebraic techniques we consider in this thesis. As far as this author knows, it is original to Lurie.

\begin{theorem*}[HA \S4.8.1]
The $\infty$-category of presentable $\infty$-categories itself has a symmetric monoidal structure $\text{Pr}^{\text{L},\otimes}$, and there is an equivalence of $\infty$-categories $$\text{CAlg}(\text{Pr}^{\text{L},\otimes})\cong\text{ClSymMonPr}^\text{L},$$ between commutative algebras in $\text{Pr}^{\text{L},\otimes}$ and presentable $\infty$-categories with closed symmetric monoidal structure.
\end{theorem*}

\noindent As an example of its application, the $\infty$-category Sp of spectra is the universal example of a stable presentable $\infty$-category, and therefore (for formal reasons) has a commutative monoid structure in $\text{Pr}^{\text{L},\otimes}$, which endows Sp with a closed symmetric monoidal structure (the smash product). \\

\noindent This sort of formal algebraic argument is a staple of higher category theory. Since it can be technically overwhelming to specify a symmetric monoidal structure on an $\infty$-category directly (which was in some sense the source of the difficulties in constructing the smash product all along), it is useful instead to give a formal argument.

Fortunately, this sort of formal algebraic argument is actually \emph{easier} in the setting of $\infty$-categories than in classical category theory --- where the analogue of the above theorem is simply not true (locally presentable categories do not assemble into a well-behaved symmetric monoidal category). \\

\noindent As another application of these ideas, Gepner, Groth, and Nikolaus \cite{GGN} have endowed the $\infty$-category SymMon of symmetric monoidal $\infty$-categories with a symmetric monoidal tensor product. (The analogue in classical category theory is extremely technical, and as far as we know has never been written up.)

This allows us to study commutative algebra in $\text{SymMon}^\otimes$. A commutative semiring $\infty$-category $\mathcal{R}$ is a commutative monoid in $\text{SymMon}^\otimes$, and we may ask for a symmetric monoidal $\infty$-category to have the further structure of an $\mathcal{R}$-module. This is the language used in the theorems stated earlier in the introduction.

We study this categorified commutative algebra in depth in Chapter \ref{3}, and compare it to Lurie's commutative algebra of $\text{Pr}^{\text{L},\otimes}$ in Chapter \ref{4}. The comparison relies heavily on the study of \emph{Lawvere theories}, which are the subject of Chapter \ref{4}.

\section*{Lawvere theories}
\addcontentsline{toc}{section}{Lawvere theories}
\noindent The study of algebraic theories arises out of the desire to understand various sorts of algebraic structure (the structures of group, ring, module, Lie algebra, etc.) in a unified way. There are a few technical frameworks available. Three of these are:
\begin{enumerate}
\item PROPs (the most general but least refined)
\item Lawvere theories (the simplest to define)
\item Operads (the least general but best behaved)
\end{enumerate}

\noindent A particular PROP, Lawvere theory, or operad is meant to encode one flavor of algebraic structure (for example, the structure of an abelian group, or of a module over a fixed ring $R$). (1)-(3) are progressively less general in the sense that any algebraic structure encoded by an operad is also encoded by a Lawvere theory, and any structure encoded by a Lawvere theory is also encoded by a PROP.

On the other hand, there are examples of structures which are encoded by a PROP but not a Lawvere theory (for example, the structure of a cogroup or Hopf algebra), and others which are encoded by a Lawvere theory but not an operad (for example, the structure of a commutative ring or $R$-module).

Between the three, Lawvere theories are particularly simple to understand conceptually.

\begin{definition*}
A \emph{Lawvere theory} is a category $\mathcal{L}$ which admits finite products, and its objects are generated under finite products by a single object.

A \emph{model} of $\mathcal{L}$ is a product-preserving functor $\mathcal{L}^\text{op}\rightarrow\text{Set}$.
\end{definition*}

\noindent For example, the Lawvere theory associated to the structure of \emph{set} is $\text{Fin}^\text{op}$, while the Lawvere theory associated to the structure of \emph{abelian group} is the Burnside category Burn. That is, a set is a product-preserving functor $\text{Fin}^\text{op}\rightarrow\text{Set}$, and an abelian group is a product-preserving functor $\text{Burn}\rightarrow\text{Set}$.

In Chapters \ref{3} and \ref{4}, we will give an algebraic description of the definition: a Lawvere theory is a cyclic module over the commutative semiring category $\text{Fin}^\text{op}$. We will also prove this for $\infty$-categorical Lawvere theories --- the story here is the same, except that a model is a product-preserving functor into the $\infty$-category Top of topological spaces.

Although Lawvere theories are less technically powerful than operads, there are at least two benefits to using them:
\begin{itemize}
\item It is typically easy to pass back and forth between a Lawvere theory and its category of models. That is, the category of models of $\mathcal{L}$ is $\text{Fun}^\times(\mathcal{L},\text{Top})$. On the other hand, if $\mathcal{C}$ is a category of models over a Lawvere theory (such as Set, Gp, Ab), $\mathcal{L}^\text{op}$ can be recovered as the full subcategory of $\mathcal{C}$ spanned by finitely generated free objects.

Note that any such $\mathcal{C}$ is a presentable category. We prove (Theorem \ref{ThmLawvAdj}) that the correspondence takes the form of an adjunction $$\text{Lawv}\rightleftarrows\text{Pr}^\text{L}$$ such that the left adjoint is fully faithful. That is, the $\infty$-category of ($\infty$-)Lawvere theories is a colocalization of the $\infty$-category of presentable $\infty$-categories.

Such a colocalization generalizes many of the fundamentals of Lawvere theories to higher category theory, and also includes algebraic analogues of the Yoneda lemma and Day convolution.%Crossreferences to the text?
\item Algebraic definitions are much more complicated in the setting of $\infty$-categories than in classical category theory, because it is less practical to give explicit point-set definitions. For example, a connective spectrum is a higher categorical analogue of an abelian group, but it is difficult to describe explicitly as `abelian group objects in Top'. Instead, we may turn to Lawvere theories: A connective spectrum is a product-preserving functor $\text{Burn}\rightarrow\text{Top}$.
\end{itemize}

\noindent Our results suggest Lawvere theories are fundamental to our algebraic perspective on category theory. However, there are many cases where we would like to study operads in this framework. We will introduce a conjectural algebraic characterization of an operad:

\begin{conjecture*}[Conjecture \ref{Conj2}]
An operad is a Lawvere theory $\mathcal{L}$ such that the structure map $\text{Fin}^\text{op}\rightarrow\mathcal{L}$ becomes an equivalence after tensoring with the effective Burnside category $\text{Burn}^\text{eff}$.
\end{conjecture*}

\subsubsection{Lawvere theories and equivariant homotopy theory}
\addcontentsline{toc}{subsection}{Lawvere theories and equivariant homotopy theory}
\noindent There is a third benefit of studying Lawvere theories, for those interested in equivariant homotopy theory. Homotopy theory comes with (at least) two distinct flavors of equivariance. This is a situation which may be unfamiliar to algebraists. On one hand, a \emph{naive}\footnote{Some people use the word Borel where we use naive, and then use naive to refer to structures intermediary between Borel and genuine equivariance. This three-way distinction becomes particularly important when working with model categories.} $G$-space is a functor $$BG\rightarrow\text{Top}$$ and a \emph{naive} $G$-spectrum is a functor $$BG\rightarrow\text{Sp}.$$ On the other hand, a \emph{genuine} $G$-space is a product-preserving functor $$\text{Fin}_G\rightarrow\text{Top},$$ and a \emph{genuine} $G$-spectrum is a product-preserving functor $$\text{Burn}_G\rightarrow\text{Sp},$$ where $\text{Fin}_G$ is the category of finite $G$-sets, and $\text{Burn}_G$ the Burnside $\infty$-category.

The last two definitions are really theorems (due to Elmendorf \cite{Elmendorf} and Guillou-May \cite{GMay2}, respectively), because work in equivariant homotopy theory predates work in $\infty$-categories. The theorems hold (to put it tersely) that the original definitions of model categories agree with these definitions of $\infty$-categories. \\

\noindent The descriptions we have given for genuine equivariant homotopy theory are Lawvere-theoretic, in the following sense:

\begin{principle*}[Principle \ref{LawvMachPr}]
Suppose $\mathcal{C}$ is an $\infty$-category with associated Lawvere theory $\mathcal{L}$, and $\mathcal{L}$ can be built by applying a combinatorial machine $\mathcal{M}$ to the category Fin of finite sets.

Then the naive equivariant analogue of $\mathcal{C}$ is $$\text{Fun}^\times(\mathcal{M}(\null^G\text{Free}),\text{Top})$$ and the genuine equivariant analogue is $$\text{Fun}^\times(\mathcal{M}(\text{Fin}_G),\text{Top}).$$ Here $^G\text{Free}$ is the category of finite \emph{free} left $G$-sets, while $\text{Fin}_G$ denotes finite right $G$-sets (not necessarily free).
\end{principle*}

\begin{example*}
If $\mathcal{C}=\text{Top}$, then $\mathcal{M}$ is the `opposite category' construction. If $\mathcal{C}$ is the $\infty$-category of connective spectra, then $\mathcal{M}$ is the `Burnside category' construction.
\end{example*}

\noindent In order to make the principle precise, we should explain what we mean by `combinatorial machine'. We won't do that in this thesis, but see the appendix for more examples.

\subsubsection{Duality}
\addcontentsline{toc}{subsection}{Duality}
\noindent We have just described the sense in which $^G\text{Free}^\text{op}$ is the Lawvere theory for naive $G$-spaces, and $\text{Fin}_G^\text{op}$ is like a Lawvere theory for genuine $G$-spaces. In Chapter \ref{5}, we will give evidence for a duality between naive and genuine equivariant homotopy theory. Our evidence for this duality comes from the following theorem:

\begin{theorem*}[Corollary \ref{PerfPairCor}]
As categories with pullbacks, coproducts, and epimorphisms (\S\ref{5.2}), $^G\text{Free}$ is dual to $\text{Fin}_G$.
\end{theorem*}

\noindent If $R$ is a ring and $\mathcal{L}$ the corresponding Lawvere theory (so $\mathcal{L}^\text{op}$ is the category of finitely generated free $R$-modules), there is a Morita equivalence between $\mathcal{L}^\text{op}$ and $\text{Mod}_R\cong\text{Fun}^\times(\mathcal{L},\text{Set})$. That is, these are equivalent as noncommutative motives.

Similarly, we might expect $^G\text{Free}$ to be Morita equivalent to naive (left) $G$-spectra, and $\text{Fin}_G$ to be Morita equivalent to genuine (right) $G$-spectra. Therefore, we would hope for a duality (in noncommutative motives over $\mathbb{F}_1$) between genuine and naive equivariant spaces. \\

\noindent Such a duality would be both familiar and unfamiliar from `ordinary' noncommutative motives. If $R$ is a ring, then Blumberg, Gepner, and Tabuada \cite{BGT} proved that $\text{Mod}_R$ and $_R\text{Mod}\cong\text{Mod}_{R^\text{op}}$ are dual as noncommutative motives.

If $G$ is instead a finite group, we have just seen evidence for a similar duality, except that in addition to transferring between a left and right group action, we also turn genuine group actions into naive ones, and vice versa!

\subsubsection{Lawvere theories and noncommutative algebra}
\addcontentsline{toc}{subsection}{Lawvere theories and noncommutative algebra}
\noindent In our discussion of duality, we suggested some sort of relationship between Lawvere theories and noncommutative motives. Further evidence for such a relationship comes in the form of the following theorem:

\begin{theorem*}[Corollary \ref{SemiaddLawv}]
There is an equivalence of categories between additive Lawvere theories and rings (not necessarily commutative).
\end{theorem*}

\noindent This is not a difficult result, at least classically; we prove it in the setting of $\infty$-categories, which is slightly more subtle.

Recall that noncommutative motives are roughly stable $\infty$-categories, and noncommutative motives over $\mathbb{F}_1$ should somehow be built by stripping away the additivity requirement for stable $\infty$-categories.

Similarly, Corollary \ref{SemiaddLawv} suggests that Lawvere theories are a generalization of rings obtained by stripping away the additivity requirement. In this sense, we can regard Lawvere theories as being a very rough first approximation to $\mathbb{F}_1$-algebras, suggesting a connection between our work on Lawvere theories and the noncommutative motive perspective.

\chapter{A Survey of Higher Category Theory}\label{1}\setcounter{subsection}{0}
\noindent In this chapter, we review the theory of $\infty$-categories. There are no new results, and essentially no proofs. Our goal is to treat the subject in a tiered approach. First we will illustrate the ways in which $\infty$-categories behave just like (and, in some cases, better than) ordinary categories. Then we will study quasicategories, a particular model for them, demonstrating that quasicategories are an effective framework for the key $\infty$-category constructions. Finally, we will discuss some aspects of the theory (fibrations) which may be less familiar from classical category theory.

For the most part, working with $\infty$-categories is very much the same as working with ordinary categories. There are two major differences:
\begin{enumerate}
\item There is a certain class of problems in category theory that are complicated by the idea of \emph{laxness}. For example, it is tempting to suggest that a monoidal category should be just a category $\mathcal{C}$, a functor $\otimes:\mathcal{C}\times\mathcal{C}\rightarrow\mathcal{C}$, and a unit object 1 of $\mathcal{C}$. This data should satisfy $X\otimes 1=1\otimes X=X$ (unitality) and $X\otimes(Y\otimes Z)=(X\otimes Y)\otimes Z$ (associativity). But such a definition is too strict to apply to any meaningful examples; unitality and associativity should really only hold up to isomorphism. But then we should pick a specific isomorphism (one for each $X$), and require those isomorphisms to satisfy properties ensuring we have chosen the \emph{right} isomorphisms. This sort of problem does not occur in the world of $\infty$-categories, because laxness is built into the theory. A monoidal category is literally in monoid in Cat (the $\infty$-category of small $\infty$-categories).
\item Arguments that require manipulating objects and morphisms explicitly are much harder in the $\infty$-categorical setting. For example, there is little hope of describing a functor $F:\mathcal{C}\rightarrow\text{Top}$ into the $\infty$-category of spaces simply by describing a space $F(X)$ for each object $X\in\mathcal{C}$, and continuous maps for each morphism. Instead, we should try to describe the associated Grothendieck fibration $\int F\rightarrow\mathcal{C}$. This crucial reliance on Grothendieck fibrations is one of the primary difficulties in working with $\infty$-categories (compared to ordinary categories). Fortunately, we can use classical categories for inspiration, by attempting to frame classical constructions and proofs in category theory in terms of Grothendieck fibrations.
\end{enumerate}

\noindent There is a give-and-take between these two differences. Writing explicit examples of $\infty$-categories and (especially) general constructions producing $\infty$-categories can be subtle. However, once we have made these constructions, a new $\infty$-categorical toolbox makes high-level arguments easier than in classical category theory.

All this being said, when comparing $\infty$-categories and classical categories, the similarities far outweigh the differences. To emphasize this, we begin in \S\ref{1.1} by collecting some of these similarities. We will also see some phenomena which are simplified by working with $\infty$-categories. Introducing these \emph{features} of $\infty$-categories before the definition of a quasicategory is intended to make the theory more digestible.

In \S\ref{1.2}, we give a rough treatment of quasicategories as a model for $\infty$-categories. Ideally, we would prove that quasicategories satisfy all the properties of \S\ref{1.1}, but instead we give only the key points and refer the reader to Lurie's \textit{Higher Topos Theory} \cite{HTT} for the rest.

Finally, we discuss fibrations in \S\ref{1.3}, addressing difference (2) above between $\infty$-categories and ordinary categories. These will play only a minimal role in the thesis, but they are always working in the background.

The standard references are Lurie's books \textit{Higher Topos Theory} \cite{HTT} and \textit{Higher Algebra} \cite{HA}, which we refer to as HTT and HA, respectively. These are encyclopedic, but the first chapter of HTT is an accessible introduction. The second chapter is also crucial but can be intimidating. Barwick and Shah \cite{BarFib} have written an excellent survey to supplement it.

For symmetric monoidal $\infty$-categories, we use HA (beginning in Chapter 2), but this is more difficult reading, largely because of the heavy reliance on $\infty$-operads. It is possible to set up the theory of symmetric monoidal $\infty$-categories without the language of $\infty$-operads, using instead the following definition: a symmetric monoidal $\infty$-category is a product-preserving functor from the Burnside category to Cat. Such a treatment is inspired by the formalism of Lawvere theories, although it is not necessary to use the language of Lawvere theories. The author hopes to write all this up in the future.

Another source we use for symmetric monoidal $\infty$-categories is Gepner, Groth, and Nikolaus's paper \cite{GGN} which constructs some common symmetric monoidal structures and is a particularly accessible example of algebraic techniques in category theory. Finally, for the span construction and Burnside categories, we use Barwick's paper on Mackey functors \cite{BarMack1}.

\section{Features of \texorpdfstring{$\infty$-}{infinity }categories}\label{1.1}
In \S\ref{1.1.2}, we collect some of the basic results on $\infty$-categories, which (in almost all cases) carry over verbatim from classical category theory. In \S\ref{1.1.3}, we repeat for symmetric monoidal $\infty$-categories, and in \S\ref{1.1.4}, for presentable $\infty$-categories.

Presentable $\infty$-categories provide a home for three of the most powerful formal toolboxes we have at our disposal: the Yoneda lemma, the adjoint functor theorem, and Bousfield localization. They also make available another set of tools just as powerful and robust as these three, but which is unfamiliar from classical category theory. The key idea is that the $\infty$-category $\text{Pr}^\text{L}$ of presentable $\infty$-categories itself admits a symmetric monoidal structure, which is important because $\text{Pr}^{\text{L},\otimes}$ (roughly) governs higher notions of rings: highly structured ring spectra, semiring $\infty$-categories (which we discuss at length in Chapter \ref{3}), and so on.

\subsection{\texorpdfstring{$\infty$-}{Infinity }categories}\label{1.1.2}
\subsubsection{The basics}
We will leave any formal definitions of $\infty$-categories to \S\ref{1.2}. For now, think of an $\infty$-category as a category such that, given morphisms $f,g:X\rightarrow Y$, we may speak of 2-morphisms $\zeta:f\rightarrow g$. Given 2-morphisms $\zeta,\eta:f\rightarrow g$, we may speak of 3-morphisms $\zeta\rightarrow\eta$, and so forth. A consequence is that we should avoid speaking of two \emph{anythings} as being the same, but only the same `up to equivalence'. (We also do not speak of isomorphisms of objects, but rather equivalences, thinking of equivalences of categories or homotopy equivalences of spaces.)

We won't worry ourselves excessively with set-theoretic issues, but we will at least try to be accurate. By a small $\infty$-category, we mean (roughly) that there is a set of objects, a set of morphisms, etc. Otherwise, we will speak of a large $\infty$-category. \\

\noindent Here are some basic examples.
\begin{enumerate}
\item Any category $\mathcal{C}$ is an $\infty$-category, with no higher morphisms except identities.
\item Any 2-category in which all 2-morphisms are invertible is an $\infty$-category. For example, the 2-category $\text{Cat}_1$ of small categories (where 2-morphisms are natural isomorphisms) is a (large) $\infty$-category.
\item Any topological space $X$ is an $\infty$-category. Roughly, the objects are points of $X$, morphisms are paths between points, 2-morphisms are homotopies, etc.
\item There are (large) $\infty$-categories Top of topological spaces and Cat of small $\infty$-categories.
\end{enumerate}

\noindent Notice we are referring to $\infty$-categories as Cat and categories as $\text{Cat}_1$. We will also typically refer to categories as 1-categories for emphasis, since we will far more often speak of $\infty$-categories.

Examples (1) and (3) each induce fully faithful functors of $\infty$-categories $\text{Top}\rightarrow\text{Cat}$ and $\text{Cat}_1\rightarrow\text{Cat}$. In other words, each has a sort of converse:
\begin{itemize}
\item A category is an $\infty$-category such that the only $n$-morphisms ($n\geq 2$) $\alpha:\zeta\rightarrow\eta$ are (equal to, not just equivalent to) identity morphisms.
\item A topological space is an $\infty$-category such that every morphism is invertible (HTT 1.2.5.1).
\end{itemize}

\noindent The second point identifies topological spaces with \emph{$\infty$-groupoids}, and is sometimes referred to as the homotopy hypothesis, a name due to Baez \cite{HomoHypo}. Here are some consequences of the homotopy hypothesis:
\begin{itemize}
\item An $\infty$-groupoid is equivalent to the trivial category (one object, one morphism) if and only if it is contractible as a space.
\item For any two objects $X,Y$ of an $\infty$-category $\mathcal{C}$, there is an $\infty$-groupoid (space) of morphisms $\text{Map}_\mathcal{C}(X,Y)$.
\end{itemize}

\noindent The homotopy hypothesis allows us to model $\infty$-categories conceptually as follows:

\begin{definition}
\label{Def1}
An $\infty$-category is a category enriched in topological spaces (topological category).
\end{definition}

\noindent This is not our preferred model, because it does not admit a good theory of fibrations. In \S\ref{1.2} we will introduce a more robust model which is equivalent (HTT 2.2.5.1) to this one.

\subsubsection{Limits and colimits}
\noindent Let $\mathcal{C}$ denote an $\infty$-category. We say that an object $X\in\mathcal{C}$ is \emph{initial} if for all $Y\in\mathcal{C}$, $\text{Map}_\mathcal{C}(X,Y)$ is contractible. Dually, $X$ is \emph{terminal} if $\text{Map}_\mathcal{C}(Y,X)$ is contractible for all $Y$.

If they exist, initial and terminal objects are \emph{unique up to contractible choice}. That is, the full subcategory spanned by initial (respectively terminal) objects is a contractible $\infty$-groupoid (HTT 1.2.12.9).

We can use initial and terminal objects to define limits and colimits. If $K$ is a small $\infty$-category, and $p:K\rightarrow\mathcal{C}$ a functor, then we may speak of a limit or colimit of $p$.

There is an $\infty$-category (overcategory) $\mathcal{C}_{/p}$, of which an object is (roughly) given by the following data: an object $T\in\mathcal{C}$, morphisms $T\rightarrow p(X)$ for all $X\in K$, 2-morphisms exhibiting that the following triangle commutes for each morphism $f:X\rightarrow Y$ in $I$, and so forth. $$\xymatrix{
&T\ar[ld]\ar[rd] & \\
p(X)\ar[rr]_{p(f)} &&p(Y)
}$$ Similarly, there is an undercategory $\mathcal{C}_{p/}$.

\begin{definition}
The limit of $p:K\rightarrow\mathcal{C}$ is (if it exists) the terminal object of $\mathcal{C}_{/p}$, and the colimit is the initial object of $\mathcal{C}_{p/}$.
\end{definition}

\noindent Note:
\begin{itemize}
\item If $K$ and $\mathcal{C}$ are both 1-categories, limits are the same as in classical category theory.
\item If $K$ is a 1-category, the limit of a diagram $p:K\rightarrow\text{Top}$ is the homotopy limit of classical homotopy theory.
\end{itemize}

\begin{example}
The following are examples of limits (colimits):
\begin{itemize}
\item If $K$ is discrete (a 1-category with no morphisms except the identities), a limit (colimit) over $K$ is called a product (coproduct).
\item If $K$ is the 1-category with three objects and two nontrivial morphisms ($X\to Z$, $Y\to Z$), a limit over $K$ (colimit over $K^\text{op}$) is called a pullback (pushout).
\item Let $\Delta$ denote the 1-category of \emph{nonempty} totally ordered finite sets and order-preserving functions between them (\emph{the simplex category}). A limit over $\Delta$ is called a totalization. A colimit over $\Delta^\text{op}$ is called a geometric realization.
\item If $K$ is the empty category, and $p:K\rightarrow\mathcal{C}$ the only possible functor, the limit of $p$ is the terminal object of $\mathcal{C}$ (if it exists) and the colimit is the initial object.
\end{itemize}
\end{example}

\noindent Totalization and geometric realization are analogous to equalizers and coequalizers in classical category theory.

Limits and colimits need not exist. If all limits (colimits) of functors $p:K\rightarrow\mathcal{C}$ exist (where $K$ is small), we say $\mathcal{C}$ is closed under small limits (colimits). This is true in many cases of interest:

\begin{example}
Each of Top, Cat, and $\text{Cat}_1$ are closed under small limits and colimits.
\end{example}

\noindent If $\mathcal{C}$ is closed under products ($K$ is discrete) and pullbacks ($K$ is the 1-category with three objects and two nontrivial maps $X\to Z$, $Y\to Z$), then $\mathcal{C}$ is closed under small limits (HTT 4.4.2.6). Dually for colimits.

\subsubsection{Functors and adjunctions}
\noindent For any $\infty$-categories $\mathcal{C}$ and $\mathcal{D}$, there is an $\infty$-category $\text{Fun}(\mathcal{C},\mathcal{D})$ of functors from $\mathcal{C}$ to $\mathcal{D}$.

If $F:\mathcal{C}\rightarrow\mathcal{D}$ is such a functor, $F$ may admit a right adjoint $G:\mathcal{D}\rightarrow\mathcal{C}$. If so, there are equivalences of spaces $\text{Map}_\mathcal{D}(FX,Y)\cong\text{Map}_\mathcal{C}(X,GY)$ for all $X\in\mathcal{C}$ and $Y\in\mathcal{D}$. Moreover, the right adjoint (if it exists) is uniquely determined up to contractible choice in $\text{Fun}(\mathcal{D},\mathcal{C})$.

Similarly, we may speak of left adjoints, which are also essentially unique. If $G$ is right adjoint to $F$, then $F$ is left adjoint to $G$.

For such an adjunction we write $F:\mathcal{C}\rightleftarrows\mathcal{D}:G$ or just $F\dashv G$ if the $\infty$-categories are understood. In either case, we take care to put the left adjoint ($F$) on the left. \\

\noindent A very important adjunction is the nerve/homotopy category adjunction $$h:\text{Cat}\rightleftarrows\text{Cat}_1:N.$$ The right adjoint sends a category to its associated $\infty$-category (its nerve). The left adjoint sends an $\infty$-category to its homotopy category.

That is, if $\mathcal{C}$ is an $\infty$-category, $h\mathcal{C}$ has the same objects as $\mathcal{C}$, and a morphism $X\rightarrow Y$ is a \emph{connected component} of the space $\text{Map}_\mathcal{C}(X,Y)$. \\

\noindent If $K$ is discrete (the nerve of a category with no nontrivial morphisms), a limit or colimit of shape $K$ is called a product or coproduct. Products and coproducts in an $\infty$-category agree with products and coproducts in the homotopy category; that is, the functor $\mathcal{C}\rightarrow h\mathcal{C}$ preserves products and coproducts (HTT 1.2.13.1).

If $\mathcal{C},\mathcal{D}$ are $\infty$-categories, and $\mathcal{D}$ admits all limits of shape $K$, then so does $\text{Fun}(\mathcal{C},\mathcal{D})$. Moreover, the limit of $p:K\rightarrow\text{Fun}(\mathcal{C},\mathcal{D})$, which is itself a functor, sends (in a precise sense) $X\in\mathcal{C}$ to $\text{lim}(p(-)(X))$. That is, limits are taken pointwise, and similarly for colimits (HTT 5.1.2.3).

\subsubsection{Subcategories and equivalences}
\noindent To specify a subcategory of an $\infty$-category $\mathcal{C}$, we need only specify which objects and morphisms are in the subcategory. This is true even though $\infty$-categories include data of 2-morphisms, 3-morphisms, etc. Precisely, we need only specify a subcategory $h\mathcal{D}$ of the homotopy category $h\mathcal{C}$, and then the subcategory of $\mathcal{C}$ is the pullback (in Cat) $$\xymatrix{
\mathcal{D}\ar[r]\ar[d] &\mathcal{C}\ar[d] \\
h\mathcal{D}\ar[r] &h\mathcal{C}.
}$$ A fully faithful functor $F:\mathcal{D}\rightarrow\mathcal{C}$ is one such that $\text{Map}_\mathcal{D}(X,Y)\rightarrow\text{Map}_\mathcal{C}(F(X),F(Y))$ is a homotopy equivalence of spaces for all $X,Y\in\mathcal{D}$.

Equivalently, a fully faithful functor is a subcategory inclusion such that $h\mathcal{D}\rightarrow h\mathcal{C}$ is fully faithful. In order to specify a fully faithful subcategory, we need only announce which objects of $\mathcal{C}$ are in the subcategory.

A functor $F:\mathcal{C}\rightarrow\mathcal{D}$ is invertible up to equivalence if and only if it is fully faithful and essentially surjective (HTT 3.1.3.2). In this case, we call $F$ an equivalence of $\infty$-categories. \\

\noindent We already mentioned that the inclusions $\text{Cat}_1\rightarrow\text{Cat}$ (the nerve) and $\text{Top}\rightarrow\text{Cat}$ are fully faithful. Each of these functors has a left adjoint. These are the homotopy category construction $h:\text{Cat}\rightarrow\text{Cat}_1$ and the classifying space construction $|-|:\text{Cat}\rightarrow\text{Top}$.

This sort of adjunction is very common. In general, we may have a subcategory inclusion $F:\mathcal{D}\rightarrow\mathcal{C}$ with a left adjoint. In such a case, we may refer to $F$ as a \emph{forgetful functor} and the left adjoint as a \emph{free functor}.

In the further event that $F$ is fully faithful (which is true in both examples above), we may refer to the left adjoint as a \emph{localization functor}. More on this in \S\ref{1.1.4}.

\subsection{Symmetric monoidal \texorpdfstring{$\infty$-}{infinity }categories}\label{1.1.3}
\subsubsection{The basics}
Formal definitions will be delayed until Chapter 2. Roughly, a symmetric monoidal $\infty$-category consists of an $\infty$-category $\mathcal{C}$, functor $\oast:\mathcal{C}\times\mathcal{C}\rightarrow\mathcal{C}$, object $1$ in $\mathcal{C}$, and equivalences $X\oast(Y\oast Z)\cong(X\oast Y)\oast Z$ (associativity), $X\oast 1\cong 1\oast X\cong X$ (unitality) and $X\oast Y\cong Y\oast X$ (commutativity). These equivalences (which are themselves choices of morphisms) should satisfy additional identities, which force us to choose more equivalences, and so forth ad infinitum. We may write $\mathcal{C}^\oast$ to emphasize that $\mathcal{C}$ has been endowed with all of this additional structure.

Symmetric monoidal $\infty$-categories are meant to encompass examples like the following:
\begin{enumerate}
\item If $\mathcal{C}^\oast$ is a symmetric monoidal category, its nerve is a symmetric monoidal $\infty$-category, also called $\mathcal{C}^\oast$.
\item If $X$ is an $\mathbb{E}_\infty$-space (homotopy coherently commutative monoid space), the associated $\infty$-groupoid is a symmetric monoidal $\infty$-category.
\item For example, if $E$ is a spectrum, $\Omega^\infty E$ is a symmetric monoidal $\infty$-category.
\item Suppose $\mathcal{C}$ is an $\infty$-category which admits finite products. (Note that the terminal object is an example of a finite product --- namely, a product indexed by the empty set.) Then there is a symmetric monoidal $\infty$-category $\mathcal{C}^\times$ with symmetric monoidal operation given by product.
\item If $\mathcal{C}$ admits finite coproducts, there is a symmetric monoidal $\infty$-category $\mathcal{C}^\amalg$, with symmetric monoidal operation given by coproduct.
\item The $\infty$-category Sp of spectra admits a symmetric monoidal structure given by smash product $\wedge$ (HA 4.8.2.19).
\end{enumerate}

\noindent Symmetric monoidal $\infty$-categories arising from example (4) are called \emph{cartesian monoidal}, and from (5) are \emph{cocartesian monoidal}.

A symmetric monoidal functor is a functor which preserves all of the structure of the symmetric monoidal $\infty$-category. That is, there are equivalences $F(X)\oast F(Y)\cong F(X\oast Y)$, etc. There is an $\infty$-category SymMon whose objects are symmetric monoidal $\infty$-categories and morphisms are symmetric monoidal functors. 

\begin{remark}[HA 2.4.1.8]
\label{RmkCartMon}
If $\mathcal{C}^\times$ and $\mathcal{D}^\times$ are cartesian monoidal $\infty$-categories, symmetric monoidal functors $\mathcal{C}^\times\rightarrow\mathcal{D}^\times$ correspond to functors $\mathcal{C}\rightarrow\mathcal{D}$ which preserve finite products. That is, if CartMon denotes the full subcategory of SymMon spanned by cartesian monoidal $\infty$-categories, the forgetful functor $$\text{CartMon}\rightarrow\text{Cat}$$ is also a subcategory inclusion --- the subcategory of $\infty$-categories with finite products, and finite-product-preserving functors between them.
\end{remark}

\noindent If $\mathcal{C}^\oast$ is a symmetric monoidal $\infty$-category, then $h\mathcal{C}^\oast$ is a symmetric monoidal 1-category.

It is not difficult to check whether a symmetric monoidal $\infty$-category is cartesian or cocartesian monoidal. In fact, $\mathcal{C}^\oast$ is cartesian (cocartesian) monoidal if and only if the homotopy category $h\mathcal{C}^\oast$ is cartesian (cocartesian) monoidal as an ordinary symmetric monoidal category.

\subsubsection{Closed symmetric monoidal $\infty$-categories}
\noindent If $\mathcal{C}^\oast$ is a symmetric monoidal $\infty$-category and $X$ an object, there is a functor $-\oast X:\mathcal{C}\rightarrow\mathcal{C}$. If this functor admits a right adjoint $\text{Hom}(X,-)$ for all $X$, then $\mathcal{C}$ is called \emph{closed symmetric monoidal}.

In this case, the mapping spaces $\text{Map}_\mathcal{C}(X,Y)$ can be endowed with the structure of objects of $\mathcal{C}$, via $\text{Hom}(X,Y)$. This makes $\mathcal{C}$ into an $\infty$-category enriched over itself in the sense of \cite{GepHaug}.

If a closed symmetric monoidal $\infty$-category $\mathcal{C}^\times$ is also cartesian monoidal, we say that $\mathcal{C}$ is cartesian closed.

Examples:
\begin{enumerate}
\item Top, Cat, and $\text{Cat}_1$ are cartesian closed.
\item $\text{Sp}^\wedge$ is closed symmetric monoidal.
\item If $\mathcal{C}^\oast$ is a closed symmetric monoidal 1-category, then its nerve is a closed symmetric monoidal $\infty$-category.
\end{enumerate}

\subsubsection{Algebras}
\noindent Given a symmetric monoidal $\infty$-category, a commutative algebra in $\mathcal{C}^\oast$ is (roughly) all of the following: an object $X\in\mathcal{C}^\oast$, morphisms $1\rightarrow X$ (the unit) and $X\oast X\rightarrow X$ (the product) satisfying commutativity, associativity, and unitality. These assemble into an $\infty$-category $\text{CAlg}(\mathcal{C}^\oast)$.

\begin{warning}
\label{Warning1}
There is a difference in terminology between category theory and higher category theory. What are called \emph{commutative algebra} objects in the higher category theory literature are called \emph{commutative monoid} objects in the classical category theory literature. This use of `commutative algebra' is a term taken from the theory of operads, which were designed more to study rings (i.e., algebras) than groups (i.e., monoids).

We follow Lurie in using the term \emph{commutative monoid} to refer to a commutative algebra in the special case that $\mathcal{C}^\times$ is cartesian monoidal.
\end{warning}

\noindent The following are examples:
\begin{enumerate}
\item If $\mathcal{C}^\oast$ is a symmetric monoidal 1-category, $\text{CAlg}(\mathcal{C}^\oast)$ is a 1-category: the category of commutative monoids in $\mathcal{C}^\oast$, in the sense of classical category theory.
\item $\text{CAlg}(\text{Top}^\times)$ is equivalent to the $\infty$-category of $\mathbb{E}_\infty$-spaces, as in homotopy theory. For example, if $E$ is a spectrum, its infinite loop space $\Omega^\infty E$ is a symmetric monoidal $\infty$-category.
\item $\text{CAlg}(\text{Cat}^\times)\cong\text{SymMon}$
\item $\text{CAlg}(\text{Cat}_1^\times)\cong\text{SymMon}_1$, the classical 2-category of symmetric monoidal categories (where 2-morphisms are monoidal natural isomorphisms).
\item $\text{CAlg}(\text{Sp}^\wedge)$ is the $\infty$-category of $\mathbb{E}_\infty$-ring spectra.
\item If $\mathcal{C}^\amalg$ is cocartesian monoidal, every object $X\in\mathcal{C}$ has a canonical and essentially unique commutative monoid structure. The unit $1\rightarrow X$ is the initial morphism, and the product $X\amalg X\rightarrow X$ is the codiagonal. Succinctly, the forgetful functor $\text{CAlg}(\mathcal{C}^\amalg)\rightarrow\mathcal{C}^\amalg$ is an equivalence.
\end{enumerate}

\noindent Examples (3) and (4) deserve special attention! It is \emph{not} true classically that symmetric monoidal categories are commutative monoids in the category of categories. They are instead commutative monoids in the 2-category of categories, which is exactly what example (4) states. But in the world of $\infty$-categories, there is no such confusion.

\subsubsection{Limits and colimits of algebras}
\noindent If $p:K\rightarrow\text{CAlg}(\mathcal{C}^\oast)$ is a functor such that the restriction to $p^\prime:K\rightarrow\mathcal{C}$ has a limit, then $p$ has a limit, and the underlying object of $\text{lim}(p)$ is the limit of $p^\prime$ (HA 3.2.2.1).

In other words, the forgetful functor $\text{CAlg}(\mathcal{C}^\oast)\rightarrow\mathcal{C}$ (which sends a commutative algebra $X$ to the object $X$) preserves all small limits that exist in $\mathcal{C}$ (HA 3.2.2.1).

Likewise, if $\mathcal{C}$ admits small colimits \emph{and $\oast$ preserves colimits} (which is true if, for example $\mathcal{C}^\oast$ is closed symmetric monoidal), then $\text{CAlg}(\mathcal{C}^\oast)$ admits small colimits (HA 3.2.3.3). However, it is not necessarily easy to describe these colimits, except in one case:

$\text{CAlg}(\mathcal{C}^\oast)$ always admits finite coproducts. The coproduct of $X$ and $Y$ has underlying object $X\oast Y$ (HA 3.2.4.7). That is, $\text{CAlg}(\mathcal{C}^\oast)$ itself has a symmetric monoidal structure given by the same $\oast$, and $\text{CAlg}(\mathcal{C}^\oast)^\oast$ is cocartesian monoidal.

\subsubsection{Modules}
\noindent If $\mathcal{C}^\oast$ is a symmetric monoidal $\infty$-category and $A\in\text{CAlg}(\mathcal{C}^\oast)$, there is a notion of module over $A$. Roughly, an $A$-module is an object $M\in\mathcal{C}$ with a product $A\oast M\rightarrow M$ which is associative and unital. These assemble into an $\infty$-category $\text{Mod}_A$. For example:
\begin{enumerate}
\item If $\mathcal{C}^\oast$ is a symmetric monoidal 1-category and $A$ a commutative algebra, $\text{Mod}_A$ is a 1-category --- the category of left $A$-modules (or equivalently right $A$-modules) in the sense of classical category theory.
\item If $A$ is an $\mathbb{E}_\infty$-ring spectrum (a commutative algebra in $\text{Sp}^\wedge$), $\text{Mod}_A$ is the $\infty$-category of $A$-module spectra, in the sense of homotopy theory.
\item If $R$ is a commutative ring and $HR$ the corresponding Eilenberg-Maclane spectrum, $\text{Mod}_{HR}$ is the $\infty$-category of chain complexes of $R$-modules (HA 7.1.2.13).
\end{enumerate}

\noindent If $p:K\rightarrow\text{Mod}_A$ is a functor such that the restriction $p^\prime:K\rightarrow\mathcal{C}$ has a limit, then $p$ has a limit whose restriction to $\mathcal{C}$ is $p^\prime$ (HA 4.2.3.3). The same is true for colimits provided $\oast$ \emph{preserves colimits} (HA 4.2.3.5). The condition that $\oast$ preserve colimits holds, for example, if $\mathcal{C}^\oast$ is closed symmetric monoidal.

That is, if $\oast$ preserves colimits, then the forgetful functor $\text{Mod}_A\rightarrow\mathcal{C}$ preserves all limits and colimits that exist in $\mathcal{C}$.

\subsubsection{Relative tensor products}
\noindent If $\mathcal{C}$ admits small colimits and $\oast$ preserves small colimits (for example, $\mathcal{C}^\oast$ is closed symmetric monoidal), then for every commutative algebra object $A$, $\text{Mod}_A$ has a symmetric monoidal structure given by relative tensor product $\oast_A$ (HA 4.5.2.1).

For example, suppose $\mathcal{C}^\oast=\text{Ab}^\otimes$ is the 1-category of abelian groups. Then a commutative algebra is a commutative ring $R$, and $\text{Mod}_R^{\otimes_R}$ is the ordinary category of $R$-modules under relative tensor product.

There is an equivalence of $\infty$-categories between commutative algebras in $\text{Mod}_A^{\oast_A}$ and commutative algebras in $\mathcal{C}^\oast$ with a map from $A$: $$\text{CAlg}(\text{Mod}_A^{\oast_A})\cong\text{CAlg}(\mathcal{C}^\oast)_{A/}.$$ An object of one of these equivalent $\infty$-categories is called a commutative $A$-algebra, which assemble into an $\infty$-category $\text{CAlg}_A$. \\

\noindent If $\mathcal{C}^\oast$ is closed symmetric monoidal and $A$ a commutative algebra, then in most cases $\text{Mod}_A^{\oast_A}$ is also closed symmetric monoidal. (For example, this is true if $\mathcal{C}$ is presentable by a careful application of the adjoint functor theorem.) We may denote the internal Hom by $\text{Hom}_A(-,-)$.

\subsubsection{Tensor up and Hom up}
\begin{proposition}
\label{TupHup}
Suppose that $\mathcal{C}^\oast$ is symmetric monoidal. If $R\rightarrow A$ is a morphism of commutative algebras, there is a restriction functor $\text{Fgt}:\text{Mod}_A\rightarrow\text{Mod}_R$. If $\mathcal{C}$ is closed under small colimits, and $\oast$ preserves colimits (for example, $\mathcal{C}^\oast$ is closed symmetric monoidal), then Fgt has a left adjoint `\emph{tensor up}' functor. The left adjoint is even a symmetric monoidal functor, and it takes the form $$A\oast_R -:\text{Mod}_R^{\oast_R}\rightarrow\text{Mod}_A^{\oast_A}.$$ In the further event that $\mathcal{C}^\oast$ is closed symmetric monoidal and presentable, then Fgt has a right adjoint `\emph{Hom up}' functor which takes the form $$\text{Hom}_R(A,-):\text{Mod}_R\rightarrow\text{Mod}_A.$$ To see this, note that $\text{Mod}_\mathcal{A}\rightarrow\text{Mod}_\mathcal{R}$ preserves colimits (HA 3.4.4), so by the adjoint functor theorem, it has a right adjoint. The restriction of the right adjoint along $\text{Mod}_\mathcal{A}\rightarrow\text{Mod}_\mathcal{R}$ is itself right adjoint to $\mathcal{A}\otimes_\mathcal{R} -$, so indeed it takes the form $\text{Hom}_\mathcal{R}(\mathcal{A},-)$.
\end{proposition}

\subsubsection{Lax symmetric monoidal functors}
\noindent Sometimes a functor between symmetric monoidal functors may still be well-behaved in terms of algebras and modules, even though it is not itself symmetric monoidal.

Such a functor $F:\mathcal{C}^\oast\rightarrow\mathcal{D}^\oast$ is called \emph{lax symmetric monoidal} if, roughly, there are morphisms $1_\mathcal{D}\rightarrow F(1_\mathcal{C})$ and $F(X)\oast F(Y)\rightarrow F(X\oast Y)$ which are \emph{not necessarily} equivalences, but otherwise preserve the symmetric monoidal structure.

\begin{warning}
In HA, Lurie refers to lax symmetric monoidal functors as `$\mathcal{C}^\oast$-algebras in $\mathcal{D}^\oast$,' due to a connection with operads which we will not discuss.
\end{warning}

\noindent A commutative algebra in $\mathcal{C}^\oast$ is precisely a lax symmetric monoidal functor $\ast\rightarrow\mathcal{C}^\oast$, where $\ast$ is trivial (one object, one morphism) symmetric monoidal category. Therefore, if $F:\mathcal{C}^\oast\rightarrow\mathcal{D}^\oast$ is lax symmetric monoidal, it induces $$F_\ast:\text{CAlg}(\mathcal{C}^\oast)\rightarrow\text{CAlg}(\mathcal{D}^\oast).$$ Moreover, for each $A\in\text{CAlg}(\mathcal{C}^\oast)$ and $A$-module $M$, $F(M)$ inherits an $F(A)$-module structure, inducing another functor $$F_\ast:\text{Mod}_A\rightarrow\text{Mod}_{F(A)}.$$ Lax symmetric monoidal functors arise frequently from adjunctions. If $F$ is a symmetric monoidal functor, then its right adjoint (if there is one) is always lax symmetric monoidal, but usually not symmetric monoidal \cite{GepHaug} A.5.11.

\subsection{Presentable \texorpdfstring{$\infty$-}{infinity }categories}\label{1.1.4}
\noindent Here we will discuss some results that are deep even in classical category theory. For a classical treatment of locally presentable categories, see \cite{LocalPresent}. As a historical note, work on presentable $\infty$-categories predates all other work on $\infty$-categories, via the formalism of model categories. There is a rough correspondence between presentable $\infty$-categories and model categories, which we will not discuss. See HTT A.2.

\subsubsection{The basics}
\noindent Certain set-theoretic issues are unavoidable in any flavor of category theory. For example, there are a variety of large $\infty$-categories that we study regularly: Top, Cat, $\text{Cat}_1$, Set, Ab (abelian groups), Sp (spectra), and so forth. These are large in the sense that their objects do not constitute a set. However, in each of these cases, the large $\infty$-categories are completely controlled by a small $\infty$-category of \emph{compact objects}. Take Set for example. Every set can be written as a small coproduct of copies of the singleton set.

\begin{definition}
A presentable $\infty$-category is a large $\infty$-category $\mathcal{C}$ which is closed under small limits and colimits and generated under colimits by a (small) set of compact objects.
\end{definition}

\noindent For the purposes of this thesis, we will not need to worry about what a compact object is.

\begin{warning}
Here there is a difference in terminology between higher category theory and category theory. Presentable $\infty$-categories are higher analogues of what are classically called \emph{locally presentable} categories.
\end{warning}

\noindent Here are some examples.
\begin{itemize}
\item Any locally presentable 1-category is a presentable $\infty$-category.
\item For example, the following 1-categories are presentable $\infty$-categories: Set (sets), Grp (groups), Ab (abelian groups), Ring (rings), CRing (commutative rings), Sch (schemes), etc.
\item Top, Cat, $\text{Cat}_1$, and $\text{Gpoid}_1$ (groupoids) are presentable $\infty$-categories.
\item If $\mathcal{D}$ is a presentable $\infty$-category and $\mathcal{C}$ any $\infty$-category, then $\text{Fun}(\mathcal{C},\mathcal{D})$ is presentable (HTT 5.5.3.6).
\item If $\mathcal{C}^\oast$ is a closed symmetric monoidal presentable $\infty$-category, then $\text{CAlg}(\mathcal{C}^\oast)$ is presentable (HA 3.2.3.5). For each commutative algebra $A$, $\text{Mod}_A$ and $\text{Alg}_A$ are presentable (HA 4.2.3.7).
\end{itemize}

\noindent The benefit of presentable $\infty$-categories is that they are large enough to admit all small colimits, but small enough to be handled without set-theoretic contradictions. As a result, there are a variety of extremely useful high-level theorems involving presentable $\infty$-categories. We will discuss four; the first three are familiar from classical category theory. The fourth is not, but is similarly powerful.

\subsubsection{The adjoint functor theorem}
\begin{theorem}[adjoint functor theorem HTT 5.5.2.9]
If $\mathcal{C}$ and $\mathcal{D}$ are presentable $\infty$-categories, a functor $F:\mathcal{C}\rightarrow\mathcal{D}$ has a right adjoint if and only if it preserves all small colimits.

On the other hand, $F$ has a left adjoint if and only if it preserves small limits and filtered colimits.
\end{theorem}

\noindent Let $\widehat{\text{Cat}}$ denote the (very large) $\infty$-category of large $\infty$-categories. Let $\text{Pr}^\text{L}$ (respectively $\text{Pr}^\text{R}$) denote the subcategory whose objects are presentable $\infty$-categories and morphisms are left adjoint functors (respectively right adjoint functors).

By the adjoint functor theorem, morphisms in $\text{Pr}^\text{L}$ are (equivalently) functors that preserve small colimits. Because adjoint functors are uniquely determined, there is an equivalence of $\infty$-categories (HTT 5.5.3.4) $$\text{Pr}^\text{L}\cong(\text{Pr}^\text{R})^\text{op}.$$ This equivalence identifies an object in $\text{Pr}^\text{L}$ with the same object in $\text{Pr}^\text{R}$ and a functor in $\text{Pr}^\text{L}$ with its right adjoint in $\text{Pr}^\text{R}$.

$\text{Pr}^\text{L}$ is closed under small limits and colimits, and they are both calculated as limits in $\widehat{\text{Cat}}$ (HTT 5.5.3.13, 5.5.3.18). That is, the forgetful functor $\text{Pr}^\text{L}\rightarrow\widehat{\text{Cat}}$ preserves small limits, and $$\text{Pr}^\text{L}\cong(\text{Pr}^\text{R})^\text{op}\rightarrow\widehat{\text{Cat}}^\text{op}$$ preserves small colimits. (Colimits in $\widehat{\text{Cat}}^\text{op}$ are limits in $\widehat{\text{Cat}}$).

\begin{warning}
Although $\text{Pr}^\text{L}$ admits small limits and colimits, it is not itself a presentable $\infty$-category, because it is not a large $\infty$-category. Instead, it is a very large $\infty$-category. It is \emph{not} closed under \emph{large} limits and colimits.
\end{warning}

\subsubsection{The Yoneda lemma}
\noindent There are a few ideas which together form a sort of `Yoneda lemma package.' Typically, the first two together are called the Yoneda lemma.

\begin{theorem}[Yoneda lemma HTT 5.1.3.1]
If $\mathcal{C}$ is an $\infty$-category, then $\text{Fun}(\mathcal{C}^\text{op},\text{Top})$ is presentable and there is a functor $$Y:\mathcal{C}\rightarrow\text{Fun}(\mathcal{C}^\text{op},\text{Top})$$ which sends $X$ to the representable functor $\text{Map}_\mathcal{C}(-,X)$. The functor $Y$ is fully faithful.
\end{theorem}

\noindent We will sometimes write $\mathcal{P}(\mathcal{C})=\text{Fun}(\mathcal{C}^\text{op},\text{Top})$ and refer to this as the \emph{$\infty$-category of presheaves on $\mathcal{C}$}.

\begin{theorem}[Yoneda lemma continued]
Suppose $X\in\mathcal{C}$ and $\underline{X}$ is the corresponding representable object in $\mathcal{P}(\mathcal{C})$. If $Y\in\mathcal{P}(\mathcal{C})$, then $\text{Map}_{\mathcal{P}(\mathcal{C})}(\underline{X},Y)\cong Y(X)$.
\end{theorem}

\noindent We use $\text{Fun}^\text{L}(\mathcal{C},\mathcal{D})$ to refer to the full subcategory of $\text{Fun}(\mathcal{C},\mathcal{D})$ spanned by functors that are left adjoints (have right adjoints). When $\mathcal{C}$ and $\mathcal{D}$ are presentable, these may also be described as colimit-preserving functors.

\begin{theorem}[HTT 5.1.5.6]
\label{Yoneda3}
For any $\infty$-category $\mathcal{C}$ and presentable $\infty$-category $\mathcal{D}$, composition with the Yoneda embedding induces an equivalence of $\infty$-categories $$\text{Fun}^\text{L}(\mathcal{P}(\mathcal{C}),\mathcal{D})\rightarrow\text{Fun}(\mathcal{C},\mathcal{D}).$$ That is, $\mathcal{P}(\mathcal{C})$ is the `presentable $\infty$-category freely generated by $\mathcal{C}$'
\end{theorem}

\noindent The last theorem is suggestive of an adjunction $\mathcal{P}:\text{Cat}\rightleftarrows\text{Pr}^\text{L}:\text{Fgt}$. In reality, there is no such adjunction, because the objects of $\text{Pr}^\text{L}$ are \emph{large} $\infty$-categories, which are not in $\text{Cat}$. But it is true that there is a `free' functor $$\mathcal{P}:\text{Cat}\rightarrow\text{Pr}^\text{L}.$$ It sends an $\infty$-category to its $\infty$-category of presheaves, and we can also describe how it acts on morphisms of $\text{Cat}$:

For any functor $F:\mathcal{C}\rightarrow\mathcal{D}$ between small $\infty$-categories, Theorem \ref{Yoneda3} guarantees the existence of an (essentially unique) left adjoint $\mathcal{P}(F)$ completing the commutative diagram $$\xymatrix{
\mathcal{C}\ar[r]^{\subseteq}\ar[d]_F &\mathcal{P}(\mathcal{C})\ar@{-->}[d]^{\mathcal{P}(F)} \\
\mathcal{D}\ar[r]_{\subseteq} &\mathcal{P}(\mathcal{D}).
}$$ The right adjoint to $\mathcal{P}(F)$ is given by precomposition by $F$ (HTT 5.2.6.3).

\subsubsection{Bousfield Localization}
\noindent Suppose $\mathcal{C}$ and $\mathcal{D}$ are presentable, and $L:\mathcal{C}\rightarrow\mathcal{D}$ a left adjoint (i.e., colimit-preserving) functor. If the right adjoint to $L$ is fully faithful, $L$ is called an \emph{accessible localization functor} and $\mathcal{D}$ an \emph{accessible localization} of $\mathcal{C}$.

Objects of $\mathcal{C}$ in the full subcategory $\mathcal{D}\subseteq\mathcal{C}$ are called \emph{local objects}. Morphisms $f$ of $\mathcal{C}$ such that $Lf$ is an equivalence are called \emph{local morphisms}. \\

\noindent Local objects and morphisms determine each other as follows: (HTT 5.5.4.2)
\begin{itemize}
\item a morphism $f:X\rightarrow Y$ is local if and only if the precomposition by $f$ induces an equivalence $-\circ f:\text{Map}_\mathcal{C}(Y,Z)\rightarrow\text{Map}_\mathcal{C}(X,Z)$ for all local $Z$.
\item an object $Z$ is local if and only if $-\circ f:\text{Map}_\mathcal{C}(Y,Z)\rightarrow\text{Map}_\mathcal{C}(X,Z)$ is an equivalence for all local morphisms $f:X\rightarrow Y$.
\end{itemize}

\noindent The localization functor $L$ satisfies the following universal property: it is initial among functors out of $\mathcal{C}$ which invert all the local morphisms (HTT 5.2.7.12, 5.5.4.20).

Related to this observation, if $L$ is a localization, then $L$ is idempotent (HTT 5.2.7.4). That is, $LX\cong X$ if $X\in\mathcal{D}$, or the composition $\mathcal{D}\subseteq\mathcal{C}\xrightarrow{L}\mathcal{D}$ is equivalent to the identity functor. \\

\noindent The universal property suggests that we may specify a localization of $\mathcal{C}$ by specifying the local morphisms. This approach originates from Bousfield \cite{BousLocal} who studied such localizations of model categories.

\begin{theorem}[HTT 5.5.4.15]
Let $S$ be a collection of morphisms of $\mathcal{C}$ (which is presentable). Suppose that small colimits of morphisms in $S$ are in $S$, morphisms in $S$ are stable under pushout (by \emph{any} morphism), and morphisms in $S$ satisfy the two-out-of-three property (that is, if $f\cong g\circ h$ and two of $f,g,h$ are in $S$, so is the third).

If $S$ is also generated (under these three properties) by a small set of morphisms, then there is a universal localization functor $L:\mathcal{C}\rightarrow\mathcal{D}$ for which morphisms are local if and only if they are in $S$, and $\mathcal{D}$ is presentable.
\end{theorem}

\noindent In fact, localizations play a fundamental role in the theory of presentable $\infty$-categories, as proven by Simpson \cite{Simpson}:

\begin{theorem}[HTT 5.5.1.1]
A (large) $\infty$-category is presentable if and only if it is an accessible localization of $\mathcal{P}(\mathcal{C})$ for some small $\infty$-category $\mathcal{C}$.
\end{theorem}

\begin{example}
\label{CompSegSp}
The presentable $\infty$-category Cat is an accessible localization of the $\infty$-category $\mathcal{P}(\Delta)=\text{Fun}(\Delta^\text{op},\text{Top})$ of \emph{simplicial spaces}. The local maps are generated by the Segal maps $\text{Sp}^n\rightarrow\Delta^n$ (where $\text{Sp}^n$ is the \emph{spine} of the $n$-simplex) and the Rezk map $\text{Eq}\rightarrow\ast$, where $\text{Eq}$ is the contractible $\infty$-category on two objects (the nerve of the 1-category $\bullet\cong\bullet$). The local objects are \emph{complete Segal spaces}.

Thus $\infty$-categories can be modeled as complete Segal spaces. See \cite{CSegalSp}.
\end{example}

\begin{remark}
If $\mathcal{C}$ is a presentable $\infty$-category, given by an accessible localization of $\mathcal{P}(S)$, then for any presentable $\infty$-category $\mathcal{D}$, there is a full subcategory inclusion $$\text{Fun}^\text{L}(\mathcal{C},\mathcal{D})\subseteq\text{Fun}^\text{L}(\mathcal{P}(S),\mathcal{D})\cong\text{Fun}(S,\mathcal{D})$$ spanned by those functors $S\rightarrow\mathcal{D}$ such that the induced left adjoint functor from $\mathcal{P}(S)$ inverts all local morphisms.

For example, taking the complete Segal space perspective as in the last example, a left adjoint functor $\text{Cat}\rightarrow\mathcal{D}$ can be described as a cosimplicial object in $\mathcal{D}$, $F_\bullet:\Delta\rightarrow\mathcal{D}$, such that the Segal maps $$F_1\amalg_{F_0}\cdots\amalg_{F_0}F_1\rightarrow F_n$$ and Rezk map $$\text{colim}_{[n]\in(\Delta/\text{Eq})^\text{op}}F_n\rightarrow F_0$$ are all equivalences.
\end{remark}

\subsubsection{Tensor products of presentable $\infty$-categories}
\noindent The $\infty$-category of presentable $\infty$-categories admits a symmetric monoidal structure $\text{Pr}^{\text{L},\otimes}$ for which $$\mathcal{C}\otimes\mathcal{D}=\text{Fun}^R(\mathcal{C}^\text{op},\mathcal{D}),$$ where $\text{Fun}^R$ denotes right adjoint functors (HA 4.8.1.15).

As far as this author is aware, this observation is due to Lurie and has no earlier analogue in classical category theory.

\begin{warning}
Since $\mathcal{C}^\text{op}$ is the opposite of a presentable $\infty$-category while $\mathcal{D}$ is itself presentable, the adjoint functor theorem does not apply: right adjoint functors $\mathcal{C}^\text{op}\rightarrow\mathcal{D}$ are \emph{not} necessarily just those that preserve small limits and filtered colimits.

However, in many cases we can get a handle on $\text{Fun}^R(\mathcal{C}^\text{op},\mathcal{D})$ as follows: $\mathcal{D}$ is an accessible localization of some presheaf $\infty$-category $\mathcal{P}(\mathcal{A})$ at some collection $S$ of local morphisms. Therefore, to specify such a right adjoint functor we must specify a left adjoint functor $\mathcal{P}(\mathcal{A})\rightarrow\mathcal{C}^\text{op}$ which inverts morphisms in $S$.

But left adjoint functors out of $\mathcal{P}(\mathcal{A})$ correspond to ordinary functors $\mathcal{A}\rightarrow\mathcal{C}^\text{op}$ (by a strengthening of Theorem \ref{Yoneda3} also at HTT 5.1.5.6).
\end{warning}

\noindent Since $\text{Pr}^\text{L}$ is a subcategory of $\widehat{\text{Cat}}$, a commutative algebra in $\text{Pr}^{\text{L},\otimes}$ is a symmetric monoidal $\infty$-category $\mathcal{C}^\oast$ (commutative algebra in $\widehat{\text{Cat}}^\times$) such that $\mathcal{C}$ is presentable, and all the maps $-\oast X:\mathcal{C}\rightarrow\mathcal{C}$ are left adjoints.

That is, $\text{CAlg}(\text{Pr}^{\text{L},\otimes})$ is the $\infty$-category of presentable $\infty$-categories with closed symmetric monoidal structure. Examples include:
\begin{itemize}
\item $\text{Set}^\times$, $\text{Top}^\times$, $\text{Cat}_1^\times$, $\text{Cat}^\times$, $\text{Gpoid}^\times$ (which are cartesian closed)
\item $\text{Ab}^\otimes$ (abelian groups)
\item $\text{CMon}_\infty^\wedge$ ($\mathbb{E}_\infty$-spaces)
\item $\text{Ab}_\infty^\wedge$ (grouplike $\mathbb{E}_\infty$-spaces, also known as infinite loop spaces or connective spectra)
\item $\text{Sp}^\wedge$ (spectra)
\end{itemize}

\noindent For a construction of (3)-(5), see HA 4.8.2.19 and \cite{GGN}.

The examples $\text{CMon}_\infty$, $\text{Ab}_\infty$, and $\text{Sp}$ are of particular interest, due to the following result of Gepner, Groth, and Nikolaus:

\begin{theorem}[\cite{GGN} 4.6]
\label{ThmGGN}
If $\mathcal{C}$ is a presentable $\infty$-category, then $$\text{CMon}(\mathcal{C})\cong\text{CMon}_\infty\otimes\mathcal{C}$$ $$\text{Ab}(\mathcal{C})\cong\text{Ab}_\infty\otimes\mathcal{C}$$ $$\text{Sp}(\mathcal{C})\cong\text{Sp}\otimes\mathcal{C}$$ as presentable $\infty$-categories. Here CMon, Ab, and Sp refer to commutative monoid objects, abelian group objects, and spectrum objects in $\mathcal{C}$.
\end{theorem}

\begin{corollary}[\cite{GGN} 5.1]
\label{CorGGN}
If $\mathcal{C}^\otimes$ is a \emph{closed symmetric monoidal} presentable $\infty$-category, then $\text{CMon}(\mathcal{C})$, $\text{Ab}(\mathcal{C})$, and $\text{Sp}(\mathcal{C})$ inherit closed symmetric monoidal structures. That is, they are commutative algebras in $\text{Pr}^{\text{L},\otimes}$, since they are tensor products of commutative algebras. These symmetric monoidal structures are uniquely determined by the condition that the free functor is symmetric monoidal: $$\text{Free}:\mathcal{C}^\otimes\rightarrow\text{CMon}(\mathcal{C})^\otimes.$$
\end{corollary}

\begin{warning}
Recall from Warning \ref{Warning1}: When we write (for example) $\text{CMon}(\mathcal{C})$, we are referring to commutative monoids in $\mathcal{C}^\times$, \emph{not} commutative algebras with respect to the closed symmetric monoidal structure $\mathcal{C}^\otimes$.
\end{warning}

\begin{example}
Since Cat is cartesian closed, and $\text{SymMon}\cong\text{CMon}(\text{Cat})$, symmetric monoidal $\infty$-categories admit (externally) a closed symmetric monoidal structure $\text{SymMon}^\otimes$. We will use this extensively in Chapter \ref{3} to study the commutative algebra of symmetric monoidal $\infty$-categories.
\end{example}

\noindent We may also be able to make good sense of modules over commutative algebras: It is conjectured that if $\mathcal{V}^\otimes\in\text{CAlg}(\text{Pr}^{\text{L},\otimes})$, then $\mathcal{V}^\otimes$-modules play the role of `presentable $\mathcal{V}^\otimes$-enriched $\infty$-categories.' This conjecture seems to be known to experts, but is not accessible to current techniques in enriched higher category theory.

One aspect of this story that is already established is that any $\mathcal{V}^\otimes$-module is naturally $\mathcal{V}^\otimes$-enriched \cite{GepHaug} 7.4.13.

\section{Quasicategories}\label{1.2}
\noindent In Definition \ref{Def1}, we defined $\infty$-categories to be topological categories (categories enriched in topological spaces). However, this definition is poorly suited to proving many of the properties listed in \S\ref{1.1}. In classical category theory, a category is really a combinatorial object, and most of the properties of \S\ref{1.1} are combinatorial properties. But a topological category mixes that combinatorics with a large dose of homotopy theory.

However, it is possible to model topological spaces in a purely combinatorial way. Simplicial sets were developed, starting in the 1960s to do exactly this. We will presuppose familiarity with simplicial sets; a classical reference is \cite{QuillenHA}.

Let $\Delta$ denote the category of finite, nonempty, totally ordered sets. That is, $$\langle n\rangle=\{0<1<\cdots<n\}$$ is an object of $\Delta$ for each nonnegative integer $n$, and every object of $\Delta$ is isomorphic to some $\langle n\rangle$. A morphism of $\Delta$ is a function $f:\langle n\rangle\rightarrow\langle m\rangle$ satisfying $f(i)\leq f(j)$ any time $i\leq j$.

A simplicial set is a functor $X_\cdot:\Delta^\text{op}\rightarrow\text{Set}$. Simplicial sets are a model for topological spaces (in the sense that there is a Quillen equivalence of model categories); more properly, a topological space is a simplicial set \emph{satisfying one condition}. Therefore, we have another (more combinatorial) model for $\infty$-categories:

\begin{definition}
An $\infty$-category is a simplicial category (category enriched in simplicial sets).
\end{definition}

\noindent Simplicial categories are a bit easier to work with than topological categories. They are discussed at some length in HTT Appendix 3. However, there is a still simpler model: $\infty$-categories can themselves be modeled by simplicial sets satisfying a condition (slightly weaker than the topological space condition). Such a simplicial set is a quasicategory. Being a purely combinatorial model for $\infty$-categories, divorced of most homotopy theory, quasicategories are well-suited to the study of higher category theory.

In \S\ref{1.2.1}, we introduce quasicategories and some basic constructions involving them. In \S\ref{1.2.2}, we demonstrate how some of the more involved constructions work. The point is not to be thorough, but to emphasize that the theory is relatively user-friendly. For details, consult HTT.

Although the notion of quasicategory is due to Boardman and Vogt \cite{BoardmanVogt}, most of the important results in \S\ref{1.2.1} and \S\ref{1.2.2} are due to Joyal \cite{Joyal}, including especially the homotopy hypothesis (identification of $\infty$-groupoids with Kan complexes) and the construction of limits and colimits.

In classical category theory, categories of spans, also known as correspondences or Burnside categories, play an important role in the background. Those who do not delve deep into category theory can usually avoid working with these objects, typically by using ad hoc constructions in their place. But in representation theory, group cohomology, and equivariant homotopy theory (that is, wherever there are group actions; see the appendix), they are harder to avoid.

In higher category theory, ad hoc constructions tend to be harder to handle, and the span construction is of more central importance. Thus, in the final part of this section (\S\ref{1.2.3}), we introduce the span construction. In the setting of $\infty$-categories, these ideas are due to Barwick \cite{BarMack1}.

\subsection{Quasicategories}\label{1.2.1}
\subsubsection{Simplicial sets}
\noindent A simplicial set is a functor $\Delta^\text{op}\rightarrow\text{Set}$, where $\Delta$ is the simplex category: the category of finite, nonempty, totally ordered sets. We write $$\text{sSet}=\text{Fun}(\Delta^\text{op},\text{Set}),$$ and, given a simplicial set $X_\cdot$, we write $X_n$ for $X_\cdot$ evaluated at $\langle n\rangle$. We call $X_n$ the set of \emph{$n$-simplices} in $X_\cdot$. There are two families of examples that are especially critical:

\begin{example}
Let $X$ be a topological space, and let $\Delta^n$ denote the solid $n$-simplex (as a topological space). We may label the vertices of $\Delta^n$ by $0,1,\ldots,n$, and choosing a $k$-dimensional face of $\Delta^n$ consists of choosing $k+1$ vertices on that face.

The \emph{singular simplicial complex} $S_\cdot(X)$ is a simplicial set whose $n$-simplices are continuous maps $S_n(X)=\text{Map}(\Delta^n,X)$. For any order-preserving function $f:\langle n\rangle\rightarrow \langle m\rangle$, there is a corresponding continuous map $\Delta^n\rightarrow\Delta^m$ which sends vertex $i$ of $\Delta^n$ to vertex $f(i)$ of $\Delta^m$. Precomposition with this continuous map determines a function $S_m(X)\rightarrow S_n(X)$, endowing $S_\cdot(X)$ with the structure of a simplicial set.
\end{example}

\begin{example}
Let $\mathcal{C}$ be a (small) category, and $\vec{\Delta}^n$ the category $0\to 1\to\cdots\to n$; that is, if $i\leq j$, then there is exactly one morphism $i\to j$, and otherwise there are none.

The \emph{nerve} $N_\cdot(\mathcal{C})$ is a simplicial set with $n$-simplices $N_n(X)=\text{Fun}(\vec{\Delta}^n,\mathcal{C})$. For any $f:\langle n\rangle\rightarrow\langle m\rangle$, there is a corresponding functor $\vec{\Delta}^n\rightarrow\vec{\Delta}^m$. As above, this makes $N_\cdot(\mathcal{C})$ into a simplicial set.
\end{example}

\noindent We can recover a (small) category $\mathcal{C}$ from its nerve, as follows:

The set of objects is $N_0(\mathcal{C})$, and the set of morphisms is $N_1(\mathcal{C})$. The two order-preserving functions $\langle 0\rangle\rightarrow\langle 1\rangle$ induce functions $N_1(\mathcal{C})\rightarrow N_0(\mathcal{C})$ that identify the source and target of a morphism.

Given 1-simplices (morphisms) $f,g\in N_1(\mathcal{C})$, if the source of $f$ is the same as the target of $g$, there is a (unique) 2-simplex $\sigma\in N_2(\mathcal{C})$ such that the two shorter faces $\sigma$ are $f$ and $g$. That is, the map $\langle 1\rangle\rightarrow\langle 2\rangle$ which sends $\{0,1\}$ to $\{0,1\}$ (respectively $\{1,2\}$) sends $\sigma$ to $g$ (respectively $f$). The third face of $\sigma$ (restriction along the third map $\langle 1\rangle\rightarrow\langle 2\rangle$) is the composition $f\circ g$.

On the other hand, if $X$ is a space, $S_\cdot(X)$ retains enough information to recover any homotopical invariant of $X$. For example, the singular homology is defined in terms of the data of $S_\cdot(X)$: that is, simplices in $X$ and their faces.

$S_\cdot(X)$ also includes data about homotopies. Suppose $\ell,\ell^\prime\in S_1(X)$ are two 1-simplices (paths in $X$) which have the same endpoints (that is, are the same upon either restriction $\langle 0\rangle\rightarrow\langle 1\rangle$). So they are paths from $x$ to $y$. Let $\bar{y}$ be the constant path at $y$ (obtained from $y$ by restricting along $\langle 1\rangle\rightarrow\langle 0\rangle$). A homotopy from $x$ to $y$ is a 2-simplex whose faces are $\ell$, $\ell^\prime$, and $\bar{y}$: 
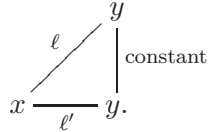
$$\xymatrix{
&y \\
x\ar@{-}[ru]^{\ell}\ar@{-}[r]_{\ell^\prime} &y\ar@{-}[u]_{\text{constant}}.
}$$ Using such notions, we can recover the homotopy groups of $X$.

\subsubsection{Horn filling}
\noindent Conversely, we can determine exactly when a simplicial set is the nerve of a category, or the singular complex of a space. We must first make a definition.

\begin{definition}
For $n\geq 0$, let $\Delta^n$ denote the nerve of the category $\vec{\Delta}^n$. The simplicial set morphisms $\text{Map}_\text{sSet}(\vec{\Delta}^n, X_\cdot)$ are in bijection with $X_n$ ($n$-simplices), so we think of $\Delta^n$ as a solid $n$-simplex.
\end{definition}

\begin{definition}
When $0\leq i\leq n$, there is a simplicial set $\Lambda^n_i$, called the $i$-horn in the $n$-simplex. The $k$-simplices are functors $F:\vec{\Delta}^k\rightarrow\vec{\Delta}^n$ such that there is at least one $j\in\vec{\Delta}^n$ which is not equal to $i$ and not hit by $F$.

We think of $\Lambda^n_i$ as a hollow $n$-simplex with one face removed (the face opposite the vertex $i$). 

A horn in the simplicial set $X_\cdot$ is a simplicial set morphism $\Lambda^n_i\rightarrow X_\cdot$.
\end{definition}

\begin{proposition}
A simplicial set $X_\cdot$ is the nerve of a category if and only if, for each $0<i<n$, each $i$-horn in $X_\cdot$ lifts uniquely to an $n$-simplex. That is, for any diagram of the following form, there is a unique dotted line lifting it: $$
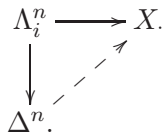
\xymatrix{
\Lambda_i^n\ar[r]\ar[d] &X_\cdot \\
\Delta^n.\ar@{-->}[ru] &
}$$ We say that \emph{all inner horns have a unique lift}.
\end{proposition}

\begin{proposition}
\label{DefKanCx}
A simplicial set $X_\cdot$ is the singular complex of a topological space if and only if, for each $0\leq i\leq n$, each $i$-horn in $X_\cdot$ lifts to an $n$-simplex. So we may choose a dotted line in the above diagram, but not uniquely.

We say that \emph{all horns have a lift}, and call $X_\cdot$ a \emph{Kan complex}.
\end{proposition}

\noindent In fact, there is an equivalence of model categories between simplicial sets and topological spaces \cite{QuillenHA}, which allows us to study homotopy theory using Kan complexes instead of spaces.

Notice two differences between the propositions. In a Kan complex, all horns lift, but not uniquely. In the nerve of a category, only inner horns lift ($0<i<n$), but they do so uniquely.

\subsubsection{Simplicial sets and $\infty$-categories}
Similarly, given an $\infty$-category $\mathcal{C}$, there should be an associated simplicial set $\mathcal{C}_\cdot$. The construction is just as above: the $n$-simplices are functors $\vec{\Delta}^n\rightarrow\mathcal{C}$ (remember $\vec{\Delta}^n$ is a 1-category).

As before, the associated simplicial set should retain all information about the $\infty$-category. For example, $\mathcal{C}_0$ is the set of objects, and $\mathcal{C}_1$ the set of morphisms. A 2-simplex $\sigma$ with faces $\sigma_{01}$, $\sigma_{12}$, and $\sigma_{02}$ (which are thought of as morphisms) recovers an equivalence of morphisms $\sigma_{02}\cong\sigma_{12}\circ\sigma_{01}$ which tells us how to compose morphisms \emph{up to equivalence}. In doing so, the 2-simplices retain all the information about 2-morphisms. The 3-simplices retain information about the 3-morphisms, and so on.

The simplicial set $\mathcal{C}_\cdot$ has some special properties. For example, any horn $\sigma:\Lambda_1^2\rightarrow\mathcal{C}_\cdot$ records the data of two morphisms $\sigma_{01}$ and $\sigma_{12}$ which can be composed. Thus there will be some lift to $\Delta^2$ recording the composition $(\sigma_{12}\circ\sigma_{01})=\sigma_{12}\circ\sigma_{01}$.

However, there will also be a lift of $\sigma$ to $\Delta^2$ for every morphism equivalent to $\sigma_{12}\circ\sigma_{01}$. We are led to the following definition, generalizing Proposition \ref{DefKanCx}:

\begin{definition}
A \emph{quasicategory} is a simplicial set such that all inner horns have lifts (not necessarily unique).
\end{definition}

\begin{theorem}
\label{InftyModels}
There are equivalences of model categories between topological categories, simplicial categories, and simplicial sets. Quasicategories are the fibrant objects in the model structure on simplicial sets. Therefore, quasicategories are a model for $\infty$-categories.
\end{theorem}

\noindent The equivalence between quasicategories and simplicial categories is due to unpublished work of Joyal (HTT 2.2.5.1). The equivalence of these to topological categories is due to Ilias \cite{Ilias}.

Henceforth, we will use the terms `quasicategory' and `$\infty$-category' interchangeably. If we need to use another model for $\infty$-categories, we will say so explicitly.

\begin{warning}
A quasicategory does not retain information about how to compose two morphisms, but only how to compose them \emph{up to equivalence}, via a choice of lift of the inner horn $\Lambda_1^2$. It turns out this is all we will ever need.
\end{warning}

\noindent On the other hand, given a quasicategory $\mathcal{C}_\cdot$, the map $\langle 1\rangle\rightarrow\langle 0\rangle$ describes a function $\mathcal{C}_0\rightarrow\mathcal{C}_1$ which sends an object $X\in\mathcal{C}$ to the identity morphism $\text{Id}:X\rightarrow X$. So there is a good notion of identity morphisms.

As indicated in \S\ref{1.1.2}, any category is an $\infty$-category (via the nerve) and any topological space is an $\infty$-category (via the singular complex).

\subsubsection{Constructions}
If $X$ and $Y$ are objects in an $\infty$-category $\mathcal{C}$, the simplicial set of morphisms from $X$ to $Y$ is defined as follows: the $n$-simplices are $(n+1)$-simplices in $\mathcal{C}$ such that the $(n+1)^\text{st}$ vertex is $Y$ and the face opposite that vertex is the constant face at $X$. This object is a Kan complex, or space (HTT 1.2.2.3), and is denoted $\text{Map}_\mathcal{C}(X,Y)$. For example, a morphism in $\text{Map}_\mathcal{C}(X,Y)$ is any 2-simplex of the following form: $$\xymatrix{
&X\ar[d] \\
X\ar[r]\ar[ru]^{\text{Id}} &Y.
}$$ The opposite of an $\infty$-category $\mathcal{C}$ is the simplicial set $$\Delta^\text{op}\xrightarrow{\text{op}}\Delta^\text{op}\xrightarrow{\mathcal{C}}\text{Set},$$ where the first functor gives a finite totally ordered set the opposite total order. This simplicial set is a quasicategory $\mathcal{C}^\text{op}$ (HTT 1.2.1).

The homotopy category of $\mathcal{C}$ is the category whose objects are the objects of $\mathcal{C}$. A morphism $X\rightarrow Y$ is an element of $\text{Map}_\mathcal{C}(X,Y)$, with two elements declared the same if there is a path between them (HTT 1.2.3.8).

We call a morphism $f:X\rightarrow Y$ of $\mathcal{C}$ an \emph{equivalence} if the following horn (which is not an inner horn) has a lift: $$\xymatrix{
&Y\ar@{-->}[d]^{?} \\
X\ar[ru]^f\ar[r]_{\text{Id}} &X.
}$$ A morphism of $\mathcal{C}$ is an equivalence if and only if the corresponding morphism in $h\mathcal{C}$ is an isomorphism. This is a nontrivial result attributed to Joyal (see HTT 1.2.4.3). As a consequence,

\begin{proposition}[HTT 1.2.5.1]
For a quasicategory $\mathcal{C}$, the following are equivalent:
\begin{itemize}
\item $\mathcal{C}$ is a Kan complex;
\item every morphism has an inverse up to equivalence;
\item $h\mathcal{C}$ is a groupoid.
\end{itemize}
\end{proposition}

\noindent Correspondingly, we use the terms \emph{topological space}, \emph{Kan complex}, and \emph{$\infty$-groupoid} interchangeably. \\

\noindent A functor of quasicategories $F:\mathcal{C}\rightarrow\mathcal{D}$ is just a morphism of simplicial sets (a natural transformation of functors $\Delta^\text{op}\rightarrow\text{Set}$), and we say that $F$ is an equivalence of quasicategories if the induced functor $hF:h\mathcal{C}\rightarrow h\mathcal{D}$ is an equivalence of categories.

\begin{definition}
\label{HomSSet}
If $X_\cdot$, $Y_\cdot$ are simplicial sets, there is a simplicial set $\text{Hom}_\text{sSet}(X_\cdot,Y_\cdot)$ whose $n$-simplices are maps of simplicial sets $\Delta^n\times X_\cdot\rightarrow Y_\cdot$.
\end{definition}

\noindent If $\mathcal{C}_\cdot$ and $\mathcal{D}_\cdot$ are quasicategories, then $\text{Hom}_\text{sSet}(\mathcal{C}_\cdot,\mathcal{D}_\cdot)$ is a quasicategory (HTT 1.2.7.3), called the $\infty$-category of functors $\text{Fun}(\mathcal{C},\mathcal{D})$. Its objects are functors, and its morphisms are called \emph{natural transformations}.

If $\mathcal{C}_\cdot$ and $\mathcal{D}_\cdot$ are moreover Kan complexes, then $\text{Fun}(\mathcal{C},\mathcal{D})$ is a Kan complex.\\

\noindent Finally, given an $\infty$-category $\mathcal{C}$, let $\mathcal{C}^\text{iso}$ denote the subcategory of all objects and just those morphisms which are equivalences. Note $\mathcal{C}^\text{iso}$ is the maximal $\infty$-groupoid contained in $\mathcal{C}$.

\subsection{Limits and colimits}\label{1.2.2}
\subsubsection{Preliminary constructions}
\noindent In \S\ref{1.1.2}, we defined initial and terminal objects, subcategories, and full subcategories. These definitions were model-independent and we will say no more about them.

We also defined the limit of a functor $p:K\rightarrow\mathcal{C}$ to be the terminal object of an overcategory $\mathcal{C}_{/p}$. We will now describe $\mathcal{C}_{/p}$ as a quasicategory.

First, notice that $\mathcal{C}_{/p}$ should satisfy a universal property. A functor $F:\mathcal{D}\rightarrow\mathcal{C}_{/p}$ ought to consist of the following information:
\begin{itemize}
\item for every object $D\in\mathcal{D}$ \emph{or} $k\in K$, an object $F(D)$ or $p(k)$ in $\mathcal{C}$;
\item for every morphism in $\mathcal{D}$ \emph{or} in $K$, a morphism in $\mathcal{C}$;
\item for every object $D\in\mathcal{D}$ \emph{and} $k\in K$, a morphism $F(D)\rightarrow p(k)$ in $\mathcal{C}$; etc.
\end{itemize}
\noindent This data is equivalent to the data of a functor into $\mathcal{C}$ from some $\infty$-category we shall call $\mathcal{D}\star K$, which agrees with $p$ when restricted along an inclusion $K\rightarrow\mathcal{D}\star K$.

Intuitively, $\mathcal{D}\star K$ is a sort of disjoint union of the categories $\mathcal{D}$ and $K$, except that there are also (unique) morphisms from each object of $\mathcal{D}$ to each object of $K$. (There are no morphisms from objects of $K$ to objects of $\mathcal{D}$.)

We don't wish to dwell on the definition (see HTT 1.2.8.1), but we give a summary to emphasize that the simplicial set $\mathcal{D}\star K$ is not technical:

If $\mathcal{D}$ and $K$ are quasicategories, an $n$-simplex in $\mathcal{D}\star K$ consists of the following: a partition of $\langle n\rangle$ into two disjoint subsets $T_0,T_1$, with every element of $T_0$ less than every element of $T_1$; and a $T_0$-simplex in $\mathcal{D}$ and $T_1$-simplex in $K$. This defines a quasicategory (HTT 1.2.8.3).

\begin{definition}
If $p:K\rightarrow\mathcal{C}$ is a functor of $\infty$-categories, then $\mathcal{C}_{/p}$ is the simplicial set whose $n$-simplices are given by dotted line lifts in the following diagram of simplicial sets $$\xymatrix{
K\ar[d]\ar[r]^p &\mathcal{C} \\
\Delta^n\star K.\ar@{-->}[ru] &
}$$
\end{definition}

\begin{proposition}[HTT 1.2.9.2, 1.2.9.3]
$\mathcal{C}_{/p}$ is a quasicategory, and satisfies the following universal property: for all $\infty$-categories $\mathcal{D}$, $\text{Fun}(\mathcal{D},\mathcal{C}_{/p})$ is \emph{equal} (as a simplicial set) to the $\infty$-category of lifts $$\xymatrix{
K\ar[d]\ar[r]^p &\mathcal{C} \\
\mathcal{D}\star K.\ar@{-->}[ru] &
}$$
\end{proposition}

\noindent The undercategory $\mathcal{C}_{p/}$ is defined dually, or as $((\mathcal{C}^\text{op})_{/p^\prime})^\text{op}$, where $$p^\prime:K^\text{op}\rightarrow\mathcal{C}^\text{op}.$$

\subsubsection{Limits and colimits}
\noindent If $p:K\rightarrow\mathcal{C}$ is a functor, the \emph{limit} of $p$ is the terminal object of $\mathcal{C}_{/p}$ (if it exists), and the \emph{colimit} is the initial object of $\mathcal{C}_{p/}$.

That is, a limit $\text{lim}(p)$ is properly a functor $\Delta^0\star K\rightarrow\mathcal{C}$ which lifts $p:K\rightarrow\mathcal{C}$. The restriction to $\Delta^0\rightarrow\mathcal{C}$ picks out an object of $\mathcal{C}$ which we will also abusively call the \emph{limit of $p$}.

Given any functor $\bar{p}:\Delta^0\star K\rightarrow\mathcal{C}$, we will say that $\bar{p}$ is a \emph{limit diagram} if it is terminal in $\mathcal{C}_{/p}$, where $p$ is the restriction of $\bar{p}$ to $K$.

\begin{definition}
\label{FunK}
A functor $F:\mathcal{C}\rightarrow\mathcal{D}$ preserves limits of shape $K$ if the composition functor $$F\circ -:\text{Fun}(\Delta^0\star K,\mathcal{C})\rightarrow\text{Fun}(\Delta^0\star K,\mathcal{D})$$ sends limit diagrams to limit diagrams. We write $\text{Fun}^K(\mathcal{C},\mathcal{D})$ for the full subcategory of $\text{Fun}(\mathcal{C},\mathcal{D})$ spanned by functors preserving limits of shape $K$.
\end{definition}

\noindent A colimit is a functor $K\star\Delta^0\rightarrow\mathcal{C}$ lifting $p:K\rightarrow\mathcal{C}$, and we define colimit-preserving functors similarly.

\subsubsection{Example: Products}
\noindent A simplicial set is discrete if the functor $\Delta^\text{op}\rightarrow\text{Set}$ is constant (all morphisms are sent to the identity). Equivalently, it is discrete if it is the nerve of a discrete category, or the singular complex of a discrete space.

If $K$ is discrete, we call limits over $K$ \emph{products} and colimits over $K$ \emph{coproducts}. For example:
\begin{itemize}
\item If $K=\emptyset$, there is only one functor $p:K\rightarrow\mathcal{C}$, and $\mathcal{C}_{/p}=\mathcal{C}_{p/}=\mathcal{C}$. The limit (colimit) over $p$ is the terminal (initial) object of $\mathcal{C}$, if it exists.
\item If $K$ is the discrete $\infty$-category on one object, a functor $p:K\rightarrow\mathcal{C}$ corresponds to a choice of object $X\in\mathcal{C}$. We write $\mathcal{C}_{/X}$ and $\mathcal{C}_{X/}$. The overcategory $\mathcal{C}_{/X}$ always has a terminal object (namely $X$), so $\text{lim}(p)$ always exists and is equivalent to $X$. Similarly, $\text{colim}(p)$ is also $X$.
\item All products and coproducts can be computed in the homotopy category. That is, $\mathcal{C}\rightarrow h\mathcal{C}$ preserves products and coproducts (provided they exist).
\end{itemize}

\subsubsection{Example: Pullbacks}
\noindent A limit over the following diagram category $K$ is a \emph{pullback}: $$
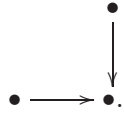
\xymatrix{
&\bullet\ar[d] \\
\bullet\ar[r] &\bullet.
}$$ A colimit over $K^\text{op}$ is a \emph{pushout}.

Note that $\Delta^0\star K$ is the 1-category given by a single commutative square, so a limit over $K$ is given by a pullback square $\Delta^0\star K\rightarrow\mathcal{C}$.

\subsection{The span construction}\label{1.2.3}
\subsubsection{Spans of 1-categories}
\noindent Let $\mathcal{L}_\text{CMon}$ denote the category of finitely generated free commutative monoids; that is, $\mathbb{N}^n$, for $n\geq 0$. $\mathcal{L}_\text{CMon}$ admits an alternate description which is purely combinatorial, as a category of spans of finite sets.

If $\mathcal{C}$ is a category with pullbacks, there is a category $\text{Span}(\mathcal{C})$ of spans of $\mathcal{C}$. An object is an object of $\mathcal{C}$. A morphism from $X$ to $Y$ is a \emph{span diagram} $$
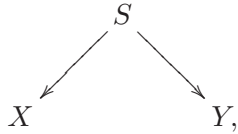
\xymatrix{
&S\ar[ld]\ar[rd] &\\
X &&Y,
}$$ where $S$ is any other object of $\mathcal{C}$. A composition is given by a pullback as follows: $$\xymatrix{
&&P\ar[ld]\ar[rd] &&\\
&S\ar[ld]\ar[rd] &&T\ar[ld]\ar[rd] &\\
X &&Y &&Z.
}$$

\begin{proposition}
\label{BurnCMon}
There is an equivalence of categories between $\mathcal{L}_\text{CMon}$ and $\text{Span}(\text{Fin})$, where Fin denotes the category of finite sets.
\end{proposition}

\begin{warning}
\label{WarningSpan}
As defined, $\text{Span}(\mathcal{C})$ is not really a category, because pullbacks are only well-defined \emph{up to isomorphism}. There are two solutions:
\begin{itemize}
\item We may define a morphism to be an isomorphism class of span diagrams. Then composition is well-defined and produces a category $\text{Span}_1(\mathcal{C})$. This is what we mean in Proposition \ref{BurnCMon}.
\item We may define a bicategory (weak 2-category) $\text{Span}_2(\mathcal{C})$. A 2-morphism between spans $X\leftarrow S\rightarrow Y$ and $X\leftarrow S^\prime\rightarrow Y$ is an isomorphism $S\xrightarrow{\sim}S^\prime$ compatible with the maps to $X$ and $Y$, and each choice of pullback corresponds to a choice of composition (which are, in any case, all the same up to equivalence).
\end{itemize}
\end{warning}

\noindent The warning illustrates a problem with \emph{laxness}. When we describe the span construction in the setting of $\infty$-categories, there will be no such concern. In the event that $\mathcal{C}$ is a 1-category, our $\infty$-category $\text{Span}_\infty(\mathcal{C})$ is in fact the nerve of a bicategory ($\text{Span}_2(\mathcal{C})$ of the warning). And the homotopy category is $h\text{Span}_\infty(\mathcal{C})\cong\text{Span}_1(\mathcal{C})$ of the warning. \\

\noindent We should give some justification for Proposition \ref{BurnCMon}. Consider that a commutative monoid homomorphism $f:\mathbb{N}^m\rightarrow\mathbb{N}^n$ is determined by where it sends each of $m$ generators of $\mathbb{N}^m$; in each case, the generator is sent to an $n$-tuple of natural numbers. That is, morphisms in $\mathcal{L}_\text{CMon}$ can be represented by $m\times n$-matrices of natural numbers.

This observation produces a functor $\mathcal{L}_\text{CMon}\rightarrow\text{Span}_1(\text{Fin})$ as follows: $\mathbb{N}^n$ is sent to $\{1,2,\ldots,n\}$. A morphism $f:\mathbb{N}^m\rightarrow\mathbb{N}^n$ corresponding to a matrix $M_{ij}$ is sent to a span (determined up to isomorphism) $$\xymatrix{
&S\ar[ld]_{\alpha}\ar[rd]^{\beta} & \\
\{1,2,\ldots,m\} &&\{1,2,\ldots,n\},
}$$ where $S$ has precisely $M_{ij}$ elements $x$ such that $\alpha(x)=i$ and $\beta(x)=j$. Composition of spans by pullback corresponds to matrix multiplication, so we have the desired equivalence of categories $$\mathcal{L}_\text{CMon}\xrightarrow{\sim}\text{Span}_1(\text{Fin}).$$ $\text{Span}_1(\text{Fin})$ also satisfies a remarkable universal property:

Recall that a semiadditive category $\mathcal{C}$ is one which has finite products and coproducts, and the two notions agree (so $X\times Y\cong X\amalg Y$ and the initial and terminal objects are the same). For example, $\text{Mod}_R$ is semiadditive for any commutative ring (or semiring) $R$. The product and coproduct are both given by direct sum $\oplus$; the initial and terminal objects are both given by $0$.

CMon is semiadditive, and therefore so is $\text{Span}(\text{Fin})$, since it is equivalent to a full subcategory of CMon. Products and coproducts in $\text{Span}_1(\text{Fin})$ are given by disjoint union of finite sets.

\begin{proposition}
$\text{Span}(\text{Fin})$ is the free semiadditive category on one generator. That is, if $\mathcal{C}$ is semiadditive and $\text{Fun}^\times$ denotes product-preserving functors, the forgetful functor $$\text{Fun}^\times(\text{Span}(\text{Fin}),\mathcal{C})\rightarrow\mathcal{C}$$ which sends $F$ to $F(\{0\})$, is an equivalence of categories.
\end{proposition}

\begin{corollary}
\label{BurnCMon2}
If $\mathcal{C}$ admits finite products, then $\text{Fun}^\times(\text{Span}(\text{Fin}),\mathcal{C})$ is equivalent to the category of commutative monoids in $\mathcal{C}^\times$.
\end{corollary}

\noindent We omit the proofs (which are not hard) since we will later repeat them for $\infty$-categories.

\begin{remark}
Using Corollary \ref{BurnCMon2} as inspiration, we may define symmetric monoidal $\infty$-categories to be product-preserving functors $\text{Span}(\text{Fin})\rightarrow\text{Cat}$. This works because (as in \S\ref{1.1.3}), symmetric monoidal $\infty$-categories are commutative monoids in Cat.
\end{remark}

\subsubsection{Spans of $\infty$-categories}
\noindent This construction is due to Barwick \cite{BarMack1}.

We will denote by $\text{Tw}(\Delta^n)$ the following \emph{twisted $n$-simplex category}.

Objects are pairs $0\leq i\leq j\leq n$. There is a morphism $ij\rightarrow i^\prime j^\prime$ (which is unique if it exists) if and only if $i\leq i^\prime$ and $j\geq j^\prime$. For example, $\text{Tw}(\Delta^3)$ is the diagram category $$
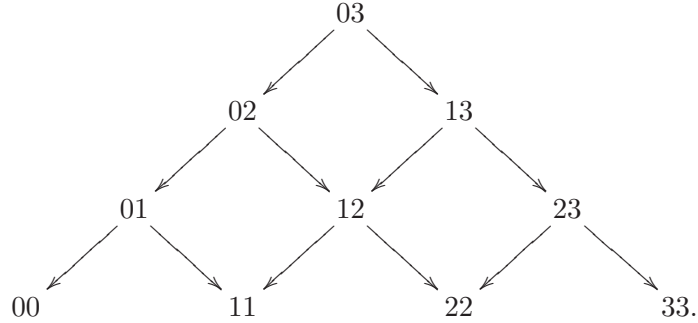
\xymatrix{
&&&03\ar[ld]\ar[rd] &&&\\
&&02\ar[ld]\ar[rd] &&13\ar[ld]\ar[rd] &&\\
&01\ar[ld]\ar[rd] &&12\ar[ld]\ar[rd] &&23\ar[ld]\ar[rd] &\\
00 &&11 &&22 &&33.
}$$ We will call the evident squares \emph{distinguished squares}; in this case, when $n=3$, there are three distinguished squares.

As usual, we think of $\text{Tw}(\Delta^n)$ as an $\infty$-category via the nerve.

\begin{definition}
Let $\mathcal{C}$ be an $\infty$-category which admits pullbacks. Denote by $\text{Span}(\mathcal{C})$ the simplicial set whose $n$-simplices are functors $\text{Tw}(\Delta^n)\rightarrow\mathcal{C}$ taking distinguished squares to pullbacks.
\end{definition}

\noindent We have the following results:
\begin{itemize}
\item $\text{Span}(\mathcal{C})$ is a quasicategory \cite{BarMack1} (3.4).
\item If $\mathcal{C}$ is a 1-category, $\text{Span}(\mathcal{C})$ is (the nerve of) the bicategory $\text{Span}_2(\mathcal{C})$ of Warning \ref{WarningSpan}.
\item If $\mathcal{C}$ is disjunctive (for example, Fin), $\text{Span}(\mathcal{C})$ is semiadditive \cite{BarMack1} 4.3.
\item $\text{Span}(\text{Fin})$ is the free semiadditive $\infty$-category on one object \cite{Glasman} A.1.
\end{itemize}

\noindent We will prove two additional properties in Chapter \ref{3}:
\begin{itemize}
\item If $\mathcal{D}$ admits finite products, $\text{Fun}^\times(\text{Span}(\text{Fin}),\mathcal{D})\cong\text{CMon}(\mathcal{D})$.
\item $\text{Span}(\text{Fin})$ is the $\infty$-category of finitely generated free $\mathbb{E}_\infty$-spaces.
\end{itemize}

\section{Fibrations}\label{1.3}
\noindent In the last section, we introduced quasicategories and saw that they are a fairly robust model for $\infty$-categories: we can use them without tremendous difficulty to work with most of the basic $\infty$-category constructions. For comparison, working with limits, colimits, and the span construction is much more difficult in the setting of topological or simplicial categories.

However, there is one very important obstacle: The two $\infty$-categories Top and Cat play a critical role in the theory, but they do not admit convenient descriptions as quasicategories.

They \emph{do} admit very simple descriptions as simplicial categories, but this is typically not much help to us. In \S\ref{1.1.4} for example, we frequently spoke of functors $\mathcal{C}\rightarrow\text{Top}$ or even $\mathcal{C}\rightarrow\text{Cat}$. But if $\mathcal{C}$ is a quasicategory while Top and Cat are simplicial categories, just writing down such a functor can be very difficult. There is a type mismatch, so we have to pass back and forth along the Quillen equivalence between quasicategories and simplicial categories.

Consider the following related situation: if $M$ is a manifold and $G$ is a group, we often have occasion to consider continuous functions $M\rightarrow BG$. But working with such maps is once again complicated by a type mismatch: $M$ has an explicit description as a topological space, while $BG$ has a construction via some algebraic universal property. In cases where we want to write down explicit topological maps, it is easier to produce a principal $G$-bundle over $M$.

Similarly, instead of studying the functor $F:\mathcal{C}\rightarrow\text{Top}$, we will study the Grothendieck construction $p:\int F\rightarrow\mathcal{C}$, which can be regarded as a kind of fibration with classifying object Top.

We may think of the $\infty$-category $\int F$ as follows: an object is a pair $(X,x)$ where $X\in\mathcal{C}$ and $x\in F(X)$. A morphism $(X,x)\rightarrow(Y,y)$ consists of a morphism $\phi:X\rightarrow Y$ in $\mathcal{C}$, and a path $F(\phi)(x)\rightarrow y$ in the space $F(Y)$.

The forgetful functor $p:\int F\rightarrow\mathcal{C}$ retains all of the information about the functor $F$. For example, the fiber of $p$ over $X$ is $F(X)$. \\

\noindent Grothendieck constructions are well-suited to descriptions in terms of quasicategories. Suppose, for example, that we are given an object $X$ of a quasicategory $\mathcal{C}$, and we would like to write down the corepresentable functor $$\text{Map}_\mathcal{C}(X,-):\mathcal{C}\rightarrow\text{Top}.$$ Ordinarily, to write a functor, we would need only to write down a morphism of simplicial sets. However, in this case, we do not have a good model for Top as a simplicial set, so this is not a feasible option.

However, the associated Grothendieck construction $p:\int F\rightarrow\mathcal{C}$ is as follows: an object of $\int F$ is a pair $(Y,f)$ with $Y\in\mathcal{C}$ and $f\in\text{Map}_\mathcal{C}(X,Y)$. That is, $\int F\cong\mathcal{C}_{X/}$. Therefore, the Grothendieck construction is modeled by the forgetful functor $$\mathcal{C}_{X/}\rightarrow\mathcal{C}.$$ As in \S\ref{1.2.2}, this functor \emph{can} be described in terms of quasicategories. \\

\noindent In this section, we will show that we can effectively study Grothendieck constructions instead of functors into Top and Cat, first for Top in \S\ref{1.3.1}, then for Cat in \S\ref{1.3.2}. We will also use the Grothendieck construction to give a description of limits and colimits in Top and Cat.

When first learning this material, there is a tendency to think, ``Surely, these ideas are only technical points which can be overlooked!'' This is not so. To illustrate this point, we use the Grothendieck construction in \S\ref{1.3.3} to study adjunctions of $\infty$-categories. \\

\noindent The ideas of this section originate in classical category theory. They were initiated by Grothendieck's work on descent, in SGA I \cite{SGA1}.

\subsection{Functors to Top}\label{1.3.1}
\noindent The category sSet is naturally enriched in itself (Definition \ref{HomSSet}) and therefore a simplicial category. Let Top be the full subcategory of sSet spanned by Kan complexes. Then Top is also a simplicial category, therefore a (large) $\infty$-category.

\subsubsection{The Grothendieck construction}
\noindent If $\ast$ denotes the space with one point, call the undercategory $\text{Top}_{\ast/}$ the $\infty$-category of \emph{pointed spaces}. We will usually write just $\text{Top}_\ast$.

\begin{definition}
\label{GrothTop}
If $\mathcal{C}$ is an $\infty$-category and $F:\mathcal{C}\rightarrow\text{Top}$ a functor, define the Grothendieck construction $p$ via pullback (in sSet) along the forgetful functor (on the right): $$\xymatrix{
\int F\ar[r]\ar[d]_p &\text{Top}_\ast\ar[d]^{\text{Fgt}} \\
\mathcal{C}\ar[r]_-{F} &\text{Top}.
}$$
\end{definition}

\noindent Intuitively, an object of $\int F$ is an object $X\in\mathcal{C}$ and a point of $F(X)$, but we should not expect to use the definition to make serious computations of Grothendieck constructions. Instead, we will typically describe the Grothendieck construction first, and deduce the existence of the functor $F$.

\subsubsection{Left fibrations}
\noindent We may think of a functor $p:\int F\rightarrow\mathcal{C}$ (as in Definition \ref{GrothTop}) as an \emph{$\infty$-category cofibered in $\infty$-groupoids}, because each fiber $p^{-1}(X)$ is an $\infty$-groupoid and these fibers vary functorially in $X$.

Such a functor $p$ satisfies an important lifting property. If $f:X\rightarrow Y$ is a morphism in $\mathcal{C}$, and $(X,x)$ an object of $\int F$ lifting $X$, there is a morphism $(X,x)\rightarrow(Y,F(f)(x))$ in $\int F$ lifting $f$. More generally, $p$ is a \emph{left fibration}:

\begin{definition}
\label{DefLeftFib}
A map of simplicial sets $p:S\rightarrow\mathcal{C}$ is called a \emph{left fibration} if it has the right lifting property with respect to all horns $\Lambda_i^n$ for $i\neq n$. That is, for every diagram of simplicial sets of this form, there is a dotted line making a commutative square: $$\xymatrix{
\Lambda_i^n\ar[d]\ar[r] &S\ar[d]^p \\
\Delta^n\ar@{-->}[ru]\ar[r] &\mathcal{C}.
}$$
\end{definition}

\noindent If $\mathcal{C}$ is a quasicategory and $p:S\rightarrow\mathcal{C}$ is a left fibration, then $S$ is certainly also a quasicategory.

\begin{example}[HTT 2.1.1.3]
A functor of 1-categories $p:S\rightarrow\mathcal{C}$ exhibits $S$ as a category cofibered in groupoids if and only if the nerve of $p$ is a left fibration of simplicial sets.
\end{example}

\begin{example}[HTT 2.1.2.2]
If $p:K\rightarrow\mathcal{C}$ is a functor of $\infty$-categories, then the forgetful functor $\mathcal{C}_{p/}\rightarrow\mathcal{C}$ is a left fibration.

For example, $\text{Top}_\ast\rightarrow\text{Top}$ is a left fibration.
\end{example}

\noindent This last example is nontrivial; in fact, Joyal proves $\mathcal{C}_{p/}\rightarrow\mathcal{C}$ is a left fibration first, and only then concludes that $\mathcal{C}_{p/}$ is an $\infty$-category!

Left fibrations are stable under pullback. That is, if $S\rightarrow\mathcal{C}$ is a left fibration and $\mathcal{D}\rightarrow\mathcal{C}$ any map of simplicial sets, then $S\times_\mathcal{C}\mathcal{D}\rightarrow\mathcal{D}$ is a left fibration.

\begin{corollary}
If $F:\mathcal{C}\rightarrow\text{Top}$ is a functor, the Grothendieck construction $p:\int F\rightarrow\mathcal{C}$ is a left fibration.
\end{corollary}

\subsubsection{Joyal's correspondence}
\noindent Let $\text{LFib}(\mathcal{C})$ denote the full subcategory of $\text{sSet}_{/\mathcal{C}}$ spanned by left fibrations over $\mathcal{C}$. This inherits the structure of a simplicial category, thus an $\infty$-category. The following theorem is due to unpublished work of Joyal, but see \cite{BarFib} 1.4.

\begin{theorem}
\label{GrJoyal}
The Grothendieck construction describes an equivalence of $\infty$-categories $$\textstyle{\int}:\text{Fun}(\mathcal{C},\text{Top})\rightarrow\text{LFib}(\mathcal{C}).$$
\end{theorem}

\noindent Thus, we may think of functors into Top as left fibrations of quasicategories.

\begin{example}
The corepresentable functor $\text{Map}_\mathcal{C}(X,-):\mathcal{C}\rightarrow\text{Top}$ corresponds to the left fibration $\mathcal{C}_{X/}\rightarrow\mathcal{C}$.
\end{example}

\begin{example}[twisted arrow construction \cite{BarMack1}]
If $\mathcal{C}$ is an $\infty$-category, there is a functor $\text{Map}_\mathcal{C}(-,-):\mathcal{C}\rightarrow\text{Top}$. The corresponding left fibration is given by Barwick's \emph{twisted arrow} $\infty$-category $\text{Tw}(\mathcal{C})$.

An $n$-simplex in $\text{Tw}(\mathcal{C})$ is a map $(\Delta^n)^\text{op}\star\Delta^n\rightarrow\mathcal{C}$.
\end{example}

\noindent One example of the twisted arrow construction appeared in the span construction of \S\ref{1.2.3}: the \emph{twisted simplex category} $\text{Tw}(\Delta^n)$ is the twisted arrow category of $\Delta^n$.

\subsubsection{Right fibrations}
\noindent We have seen that $\mathcal{C}_{X/}\rightarrow\mathcal{C}$ is a left fibration classifying the corepresentable functor $\text{Map}_\mathcal{C}(X,-)$. Dually, the \emph{representable} functor $$\text{Map}_\mathcal{C}(-,X):\mathcal{C}^\text{op}\rightarrow\text{Top}$$ has an associated left fibration $(\mathcal{C}^\text{op})_{X/}\rightarrow\mathcal{C}^\text{op}$. However, it will be easier to speak of the opposite \emph{right fibration} $\mathcal{C}_{/X}\rightarrow\mathcal{C}$.

\begin{definition}
A map of simplicial sets $p:S\rightarrow\mathcal{C}$ is called a \emph{right fibration} if it has the right lifting property with respect to all horns $\Lambda_i^n$ for $i\neq 0$. (Compare Definition \ref{DefLeftFib}.)
\end{definition}

\noindent By definition, $S\rightarrow\mathcal{C}$ is a right fibration if and only if $S^\text{op}\rightarrow\mathcal{C}^\text{op}$ is a left fibration. Therefore, Theorem \ref{GrJoyal} can be reformulated as follows: there is an equivalence $$\text{Fun}(\mathcal{C}^\text{op},\text{Top})\cong\text{RFib}(\mathcal{C}).$$ For example, the right fibration $\mathcal{C}_{/X}\rightarrow\mathcal{C}$ corresponds to the representable functor $\text{Map}_\mathcal{C}(-,X)$.

\subsubsection{Limits and colimits in Top}
\noindent We asserted in \S\ref{1.1.2} that the $\infty$-category Top admits all small limits and colimits. In fact, we can describe them using left fibrations.

For the next two theorems, let $K$ be a small $\infty$-category, $p:K\rightarrow\text{Top}$ a functor, and $q:\int p\rightarrow K$ the corresponding left fibration.

\begin{theorem}[HTT 3.3.4.6]
The colimit of $p$ is equivalent to the Grothendieck construction $\int p$.
\end{theorem}

\begin{theorem}[HTT 3.3.3.4]
Let $\text{Fun}_K(K,\int p)$ denote the \emph{$\infty$-category of sections} of $q$; that is, the full subcategory of $\text{Fun}(K,\int p)$ spanned by those functors which are (equal to, not just equivalent to) the identity when composed $$K\rightarrow\textstyle{\int}\,p\xrightarrow{q} K.$$ The limit of $p$ is equivalent to $\text{Fun}_K(K,\int p)^\text{iso}$, the \emph{space of sections} of $q$.
\end{theorem}

\subsection{Functors to Cat}\label{1.3.2}
\noindent Let Cat denote the following (large) simplicial category (as in HTT 3.0.0.1): the objects are quasicategories, and for any two quasicategories $\mathcal{C},\mathcal{D}$, the simplicial set of morphisms from $\mathcal{C}$ to $\mathcal{D}$ is the Kan complex $\text{Fun}(\mathcal{C},\mathcal{D})^\text{iso}$.

We will parallel our description of left fibrations (\S\ref{1.3.1}) as much as possible.

\subsubsection{The Grothendieck construction}
\noindent As in \S\ref{1.3.1}, we will want to define the Grothendieck construction of a functor $F:\mathcal{C}\rightarrow\text{Cat}$ via pullback along a functor $\text{Cat}_\ast^\text{lax}\rightarrow\text{Cat}$.

However, by $\text{Cat}_\ast^\text{lax}$ we do \emph{not} mean the undercategory $\text{Cat}_{\ast/}$. Recall that an object of $\int F$ is a pair $(X,x)$ with $X\in\mathcal{C}$ and $x\in F(X)$. A morphism $(X,x)\rightarrow(Y,y)$ consists of a morphism $\phi:X\rightarrow Y$ in $\mathcal{C}$ and a morphism $F(\phi)(x)\rightarrow y$ in $F(Y)$.

If we think about how to define $\int F$ via a pullback, we will find that $\text{Cat}_\ast^\text{lax}$ should be as follows:
\begin{itemize}
\item An object is a pair $(\mathcal{A},A)$, where $\mathcal{A}$ is an $\infty$-category and $A\in\mathcal{A}$.
\item A morphism $(\mathcal{A},A)\rightarrow(\mathcal{B},B)$ is a pair $(F,f)$ where $F:\mathcal{A}\rightarrow\mathcal{B}$ is a functor and $f:F(A)\rightarrow B$ a morphism in $\mathcal{B}$.
\end{itemize}

\noindent We might call $\text{Cat}^\text{lax}_\ast$ a \emph{lax undercategory}. In contrast, in the undercategory $\text{Cat}_{\ast/}$, morphisms $(\mathcal{A},A)\rightarrow(\mathcal{B},B)$ are functors which actually send $A$ to $B$. \\

To construct $\text{Cat}_\ast^\text{lax}$, we will use an alternative description of its objects. An object will be an $\infty$-category $\mathcal{C}$ and a functor $\mathcal{C}^\text{op}\rightarrow\text{Top}$ \emph{which is representable} --- or, equivalently, a \emph{right} fibration $p:S\rightarrow\mathcal{C}$ corresponding to a representable functor to Top.

A morphism from $p$ to $p^\prime$ will correspond to a commutative square $$\xymatrix{
S\ar[d]_p\ar[r] &S^\prime\ar[d]^{p^\prime} \\
\mathcal{C}\ar[r] &\mathcal{C}^\prime.
}$$ The reader may verify with the Yoneda lemma that this encapsulates what we are calling $\text{Cat}_\ast^\text{lax}$.

\begin{definition}
Let RFib denote the full subcategory of $\text{Fun}(\Delta^1,\text{Cat})$ spanned by the right fibrations. An object of RFib will be called \emph{representable} if the corresponding functor $\mathcal{C}^\text{op}\rightarrow\text{Top}$ is equivalent to one which is representable; that is, $\text{Map}_\mathcal{C}(-,X)$ for some $X\in\mathcal{C}$.

$\text{Cat}_\ast^\text{lax}$ is the full subcategory of RFib spanned by the representable right fibrations. It comes equipped with a forgetful functor $\text{Cat}_\ast^\text{lax}\rightarrow\text{Cat}$ which sends a representable right fibration $p:S\rightarrow\mathcal{C}$ to $\mathcal{C}$.
\end{definition}

\begin{definition}
\label{DefGrCat}
If $\mathcal{C}$ is an $\infty$-category and $F:\mathcal{C}\rightarrow\text{Cat}$ a functor, the Grothendieck construction $p$ is the pullback (in sSet) $$\xymatrix{
\int F\ar[r]\ar[d]_p &\text{Cat}_\ast^\text{lax}\ar[d]^{\text{Fgt}} \\
\mathcal{C}\ar[r]_F &\text{Cat}.
}$$
\end{definition}

\subsubsection{Cocartesian fibrations}
\noindent We may think of a functor $p:\int F\rightarrow\mathcal{C}$ (as in Definition \ref{DefGrCat}) as an \emph{$\infty$-category cofibered in $\infty$-categories}. In \S\ref{1.3.1}, we saw that $\infty$-categories cofibered in \emph{$\infty$-groupoids} admit a description as left fibrations of simplicial sets. There is a similar fibrational description of $\infty$-categories cofibered in $\infty$-categories.

The difficulty we had in defining $\text{Cat}_\ast^\text{lax}$ is suggestive that this fibrational description will be more complex than that in \S\ref{1.3.1}. We would like to encode the following property of $p$: if $f:X\rightarrow Y$ is a morphism in $\mathcal{C}$, and $X^\prime$ a lift of $X$ to $\int F$, there is a \emph{distinguished} morphism $(X,x)\rightarrow (Y,F(f)(x))$ in $\int F$ lifting $\mathcal{C}$. We will call such a distinguished morphism \emph{$p$-cocartesian}.

\begin{definition}
A map of simplicial sets is an \emph{inner fibration} if it has the right lifting property with respect to all \emph{inner} horns $\Lambda_i^n\rightarrow\Delta^n$.
\end{definition}

\noindent Compare to the definitions of left and right fibrations, which require lifting not only with respect to inner horns but also the leftmost (respectively, rightmost) horn. In particular, any left or right fibration is an inner fibration.

Additionally, any time $S\rightarrow\mathcal{C}$ is an inner fibration to a quasicategory, then $S$ itself has lifts of inner horns, and is therefore also a quasicategory.

\begin{definition}[\cite{BarFib} 3.1]
Let $p:S\rightarrow\mathcal{C}$ be an inner fibration of quasicategories. Let $f:X\rightarrow Y$ be a morphism of $S$. We say that $f$ is \emph{$p$-cocartesian} if, for all $n\geq 2$ and simplicial set maps $\Lambda_0^n\rightarrow S$ which send the vertices $0$ and $1$ to $X$ and $Y$ (respectively), there is a lift $$\xymatrix{
\Lambda_0^n\ar[d]\ar[r] &S\ar[d]^p \\
\Delta^n\ar[r]\ar@{-->}[ru] &\mathcal{C}.
}$$ In words, supplying a map out of $Y$ (in $S$) amounts to supplying a map out of $p(Y)$ in $\mathcal{C}$, along with a lift of the restriction $p(X)\rightarrow p(Y)\rightarrow\bullet$ to a map out of $X$ (in $S$).
\end{definition}

\begin{definition}
A functor $p:S\rightarrow\mathcal{C}$ is a \emph{cocartesian fibration} if it is an inner fibration, and for every morphism $f:X\rightarrow Y$ in $\mathcal{C}$ and lift of $X$ to $X^\prime\in S$, there exists a $p$-cocartesian morphism $f^\prime:X^\prime\rightarrow Y^\prime$ lifting $f$.
\end{definition}

\noindent Notice that any left fibration is a cocartesian fibration, because all morphisms are $p$-cocartesian.

\begin{example}[HTT 2.4.2.2]
A functor of 1-categories $p:S\rightarrow\mathcal{C}$ exhibits $S$ as a category cofibered in categories if and only if the nerve of $p$ is a cocartesian fibration.
\end{example}

\begin{example}[HTT 2.4.2.3]
If $S\rightarrow\mathcal{C}$ is a cocartesian fibration of $\infty$-categories and $\mathcal{D}\rightarrow\mathcal{C}$ any functor, the fiber product of simplicial sets, $S\times_\mathcal{C}\mathcal{D}\rightarrow\mathcal{D}$, is a cocartesian fibration.
\end{example}

\begin{example}[HTT 2.4.7.12]
The functor $\text{Cat}_\ast^\text{lax}\rightarrow\text{Cat}$ is a cocartesian fibration of $\infty$-categories.
\end{example}

\noindent Combining the last two examples, for any functor $F:\mathcal{C}\rightarrow\text{Cat}$, the Grothendieck construction produces a cocartesian fibration $p:\int F\rightarrow\mathcal{C}$.

\subsubsection{The correspondence}
\noindent In Theorem \ref{GrJoyal}, we learned that $\text{Fun}(\mathcal{C},\text{Top})\cong\text{LFib}(\mathcal{C})$. Thus functors $F:\mathcal{C}\rightarrow\text{Top}$ correspond to left fibrations $p:\int F\rightarrow\mathcal{C}$, and natural transformations $\eta:F\rightarrow F^\prime$ correspond to functors over $\mathcal{C}$ $$\xymatrix{
\int F\ar[rd]_p\ar[rr]^{\int\eta} &&\int F^\prime\ar[ld]^{p^\prime} \\
&\mathcal{C}. &
}$$ There will be a similar correspondence between functors $F:\mathcal{C}\rightarrow\text{Cat}$ and cocartesian fibrations over $\mathcal{C}$. However, in this case natural transformations correspond to functors $\int\eta:\int F\rightarrow\int F^\prime$ over $\mathcal{C}$ \emph{which send $p$-cocartesian morphisms to $p^\prime$-cocartesian morphisms}.

\begin{theorem}[\cite{BarFib} 3.7]
If $\mathcal{C}$ is an $\infty$-category, let $\text{Cocart}(\mathcal{C})$ denote the subcategory of $\text{Cat}_{/\mathcal{C}}$ spanned by cocartesian fibrations and functors between them which send cocartesian morphisms to cocartesian morphisms.

The Grothendieck construction induces an equivalence of $\infty$-categories $$\textstyle{\int}:\text{Fun}(\mathcal{C},\text{Cat})\rightarrow\text{Cocart}(\mathcal{C}).$$
\end{theorem}

\begin{example}[\cite{BarFib} 3.9]
The functor $\text{Fun}(\Delta^1,\mathcal{C})\rightarrow\mathcal{C}$ (which sends a morphism $X\rightarrow Y$ to $Y$) is a cocartesian fibration. The associated functor $\mathcal{C}\rightarrow\text{Cat}$ sends $X$ to $\mathcal{C}_{/X}$.
\end{example}

\subsubsection{Cartesian fibrations}
\noindent In \S\ref{1.3.1}, we saw that functors $\mathcal{C}\rightarrow\text{Top}$ can be described in one of two equivalent ways: as left fibrations to $\mathcal{C}$, or right fibrations to $\mathcal{C}^\text{op}$. It is easy to translate between these two languages: the left fibration $p:\mathcal{D}\rightarrow\mathcal{C}$ corresponds to the right fibration $p^\text{op}:\mathcal{D}^\text{op}\rightarrow\mathcal{C}^\text{op}$.

Similarly, a functor $\mathcal{C}\rightarrow\text{Cat}$ can be described either as a cocartesian fibration to $\mathcal{C}$ or as a cartesian fibration to $\mathcal{C}^\text{op}$.

\begin{definition}
Let $p:S\rightarrow\mathcal{C}$ be an inner fibration of quasicategories. Let $f:X\rightarrow Y$ be a morphism of $S$. We say that $f$ is \emph{$p$-cartesian} if, for all $n\geq 2$ and simplicial set maps $\Lambda_n^n\rightarrow S$ which send the vertices $n-1$ and $n$ to $X$ and $Y$ (respectively), there is a lift $$\xymatrix{
\Lambda_n^n\ar[d]\ar[r] &S\ar[d]^p \\
\Delta^n\ar[r]\ar@{-->}[ru] &\mathcal{C}.
}$$
\end{definition}

\begin{definition}
A functor $p:S\rightarrow\mathcal{C}$ is a \emph{cartesian fibration} if it is an inner fibration, and for every morphism $f:X\rightarrow Y$ in $\mathcal{C}$ and lift of $Y$ to $Y^\prime\in S$, there exists a $p$-cartesian morphism $f^\prime:X^\prime\rightarrow Y^\prime$ lifting $f$.
\end{definition}

\noindent Let $\text{Cat}_\ast^\text{oplax}$ denote the full subcategory of $\text{Fun}(\Delta^1,\text{Cat})$ spanned by left fibrations $S\rightarrow\mathcal{C}$ which correspond to corepresentable functors $\mathcal{C}\rightarrow\text{Top}$.

\begin{theorem}[HTT 3.2.0.1 or \cite{BarFib} 3.7]
\label{GrJoyal2}
There is an equivalence of $\infty$-categories $$\text{Fun}(\mathcal{C}^\text{op},\text{Cat})\rightarrow\text{Cart}(\mathcal{C})$$ which sends $F:\mathcal{C}^\text{op}\rightarrow\text{Cat}$ to $p:S\rightarrow\mathcal{C}$, where $p^\text{op}$ is given by the pullback $$\xymatrix{
S^\text{op}\ar[d]_{p^\text{op}}\ar[r] &\text{Cat}_\ast^\text{oplax}\ar[d] \\
\mathcal{C}^\text{op}\ar[r]_F &\text{Cat}.
}$$
\end{theorem}

\begin{warning}
To any functor $F:\mathcal{C}\rightarrow\text{Cat}$ are associated two fibrations: a cocartesian fibration $p:\int F\rightarrow\mathcal{C}$ and a cartesian fibration $p^\vee:\int^\vee F\rightarrow\mathcal{C}^\text{op}$.

However, because we used $\text{Cat}_\ast^\text{lax}$ to define $p$ and $\text{Cat}_\ast^\text{oplax}$ to define $p^\vee$, the two fibrations do not have a straightforward relation to each other. Unlike in \S\ref{1.3.1}, they are not opposites of each other.
\end{warning}

\subsubsection{Dualizing cartesian and cocartesian fibrations}
\noindent Continuing the warning, we will now describe how to transfer between the cocartesian and cartesian fibrations classifying the same functor. In the process, we will learn a bit about the structure of (co)cartesian fibrations.

If $F:\mathcal{C}\rightarrow\text{Cat}$, $\int F$ and $\int^\vee F$ are each built by `gluing together' $F(X)$ over all objects $X\in\mathcal{C}$. For each morphism $f:X\rightarrow Y$ in $\mathcal{C}$ and each object $x\in F(X)$, we `glue' $x$ to $F(f)(x)\in F(Y)$. However, the way we do this is different in the two cases.

In $\int F$, we throw in extra $p$-cocartesian morphisms $x\rightarrow F(f)(x)$. In $\int^\vee F$, we throw in $p$-cartesian morphisms in the other direction $x\leftarrow F(f)(x)$.

This description suggests there are two important types of morphisms in $\int F$ and $\int^\vee F$: the morphisms that were in each $\infty$-category $F(X)$ to begin with, and the $p$-cocartesian (or $p$-cartesian) morphisms that were added.

\begin{definition}
If $p:S\rightarrow\mathcal{C}$ is either a cocartesian or cartesian fibration, a morphism $f$ of $S$ is called \emph{vertical} if $p(f)$ is an equivalence.
\end{definition}

\begin{proposition}[HTT 5.2.8.15]
If $p:S\rightarrow\mathcal{C}$ is a cocartesian fibration, the $p$-cocartesian morphisms and vertical morphisms form a factorization system, in the sense that every morphism factors essentially uniquely as a $p$-cocartesian morphism followed by a vertical morphism.

Dually for cartesian fibrations, morphisms factor as a vertical morphism followed by $p$-cartesian morphism.
\end{proposition}

\noindent Given a cocartesian fibration $p:S\rightarrow\mathcal{C}$, it is possible to write down the cartesian fibration $p^\vee:S^\vee\rightarrow\mathcal{C}^\text{op}$ classifying the same functor $\mathcal{C}\rightarrow\text{Cat}$. This involves a slight modification of the span construction, due to Barwick, Glasman, and Nardin \cite{BarGN}. We summarize the construction as an example of the utility of the span construction.

Recall from \S\ref{1.2.3} that $\text{Tw}(\Delta^n)$ is a category of pairs $ij$ ($0\leq i\leq j\leq n$). We call a morphism $ij\rightarrow i^\prime j^\prime$ \emph{egressive} if $i=i^\prime$ and \emph{ingressive} if $j=j^\prime$.

An $n$-simplex of the quasicategory $S^\vee$ is a functor $\text{Tw}(\Delta^n)\rightarrow S$ which sends distinguished squares to pullbacks, egressive morphisms to $p$-cocartesian morphisms, and ingressive morphisms to vertical morphisms. For example, a morphism from $X$ to $Y$ is a span $$\xymatrix{
&T\ar[ld]_{\text{cocart}}\ar[rd]^{\text{vert}} &\\
X &&Y.
}$$ The effect is that $S^\vee$ retains the vertical morphisms of $S$ but reverses the direction of the $p$-cocartesian morphisms (which become $p$-cartesian).

\subsubsection{Limits and colimits in Cat}
\noindent Just as limits and colimits in Top can be described in terms of left fibrations, limits and colimits in Cat can be described in terms of cocartesian fibrations. However, the descriptions are less concrete.

Let $K$ be a small $\infty$-category, $p:K\rightarrow\text{Cat}$ a functor, and $q:\int p\rightarrow K$ the associated cocartesian fibration.

\begin{theorem}[HTT 3.3.4.3]
The colimit of $p$ is the localization of $\int p$ at the $p$-cocartesian morphisms. That is, there is a functor $\int p\rightarrow\text{colim}(p)$ which is initial among functors that send $p$-cocartesian morphisms to equivalences.
\end{theorem}

\begin{theorem}[HTT 3.3.3.2]
Let $\text{Fun}_K^{\,\flat}(K,\int p)$ denote the $\infty$-category of sections of $q$ which send all morphisms of $K$ to $p$-cocartesian morphisms in $\int p$. The limit of $p$ is equivalent to $\text{Fun}_K^{\,\flat}(K,\int p)$.
\end{theorem}

\subsection{Adjunctions}\label{1.3.3}
\noindent To any functor $L:\mathcal{C}\rightarrow\mathcal{D}$ is associated a functor $\Delta^1\rightarrow\text{Cat}$ and therefore a cocartesian fibration $p:S\rightarrow\Delta^1$. An object of $S$ is \emph{either} an object of $\mathcal{C}$ or $\mathcal{D}$. A morphism between objects of $\mathcal{C}$ (respectively $\mathcal{D}$) is just a morphism in $\mathcal{C}$ (or $\mathcal{D}$). There are no morphisms from objects of $\mathcal{D}$ to objects of $\mathcal{C}$.

But all the information of the functor is encapsulated in the morphisms from objects $X\in\mathcal{C}$ to objects $Y\in\mathcal{D}$. In particular, such a morphism in $S$ is just a morphism $L(X)\rightarrow Y$ in $\mathcal{D}$.

Now suppose that $L$ has a right adjoint $F$. Then $$\text{Map}_\mathcal{D}(L(X),Y)\cong\text{Map}_\mathcal{C}(X,R(Y)).$$ Therefore, we can describe maps $X\rightarrow Y$ in $S$ another way: as morphisms $X\rightarrow R(Y)$ in $\mathcal{D}$.

Unpacking, this means that $S\rightarrow\Delta^1$ is both the \emph{cocartesian fibration} associated to $L$ and the \emph{cartesian fibration} associated to $R$. This leads us to our first definition of adjoint $\infty$-functors:

\begin{definition}[HTT 5.2.2.1]
An \emph{adjunction} of $\infty$-categories is a functor $p:S\rightarrow\Delta^1$ which is both a cocartesian fibration and a cartesian fibration.

As a cocartesian fibration, $p$ corresponds to the functor $\Delta^1\rightarrow\text{Cat}$ which picks out $L:\mathcal{C}\rightarrow\mathcal{D}$. As a cartesian fibration, $p$ corresponds to a functor $(\Delta^1)^\text{op}\rightarrow\text{Cat}$ which picks out $R:\mathcal{D}\rightarrow\mathcal{C}$.

We say $L$ is left adjoint to $R$ ($R$ is right adjoint to $L$) and write $L\dashv R$.
\end{definition}

\noindent We can immediately see that left and right adjoints are unique (if they exist). For example, $L$ has a right adjoint precisely if the associated cocartesian fibration $p:S\rightarrow\Delta^1$ is also a cartesian fibration, in which case $R$ can be recovered (up to equivalence) by Theorem \ref{GrJoyal2}.

\subsubsection{Bicartesian fibrations}
\begin{definition}
A bicartesian fibration is a functor which is both a cartesian fibration and a cocartesian fibration.
\end{definition}

\noindent An adjunction is a bicartesian fibration to $\Delta^1$. In general, bicartesian fibrations describe \emph{families} of adjunctions.

Suppose that $F:\mathcal{C}\rightarrow\text{Cat}$ is a functor such that $F(f)$ has a right adjoint for each morphism $f$ in $\mathcal{C}$. Then $F$ factors through $\text{Cat}^\text{L}$, the subcategory of Cat spanned by $\infty$-categories and left adjoint functors between them. Via the equivalence (given by taking right adjoints) $$\text{Cat}^\text{L}\cong(\text{Cat}^\text{R})^\text{op},$$ $F$ determines a functor $$F^R:\mathcal{C}^\text{op}\rightarrow\text{Cat}.$$ If $X$ is an object of $\mathcal{C}$ and $f$ a morphism, then $F^R(X)=F(X)$ and $F^R(f)$ is right adjoint to $F(f)$.

In this situation, the cocartesian fibration specified by $F$ and the cartesian fibration specified by $F^R$ (both of which are functors into $\mathcal{C}$) are equivalent!

\begin{theorem}[HTT 5.2.2.5]
The cocartesian (respectively cartesian) fibration associated to $F:\mathcal{C}\rightarrow\text{Cat}$ is a bifibration if and only if $F(f)$ has a right (left) adjoint for all morphisms $f$ in $\mathcal{C}$.
\end{theorem}

\noindent We end this chapter with an example. Suppose $f:X\rightarrow Y$ is a morphism in an $\infty$-category $\mathcal{C}$. The functor $f_\ast:\mathcal{C}_{/X}\rightarrow\mathcal{C}_{/Y}$ has a right adjoint given by pullback $X\times_Y -:\mathcal{C}_{/Y}\rightarrow\mathcal{C}_{/X}$, provided all pullbacks along $f$ exist.

The overcategories and functors $f_\ast$ all assemble into a functor $\mathcal{C}\rightarrow\text{Cat}$ that sends $X$ to $\mathcal{C}_{/X}$. Recall that the corresponding cocartesian fibration is $t:\text{Fun}(\Delta^1,\mathcal{C})\rightarrow\mathcal{C}$, which sends a morphism to its target.

\begin{example}
The `target' functor $t:\text{Fun}(\Delta^1,\mathcal{C})\rightarrow\mathcal{C}$ is a cartesian fibration (therefore a bifibration) if and only if $\mathcal{C}$ admits all pullbacks.
\end{example}

\chapter{Commutative Algebra of Categories}\label{3}\setcounter{subsection}{0}
\noindent Our goal in this chapter is to develop a categorification of commutative algebra, as a new toolbox for studying symmetric monoidal $\infty$-categories (which we regard as categorified abelian groups). The results of this chapter and some of the next chapter appear in the author's paper \cite{Berman}, but the exposition is changed and hopefully clarified. Where possible, theorems are cross-referenced with the corresponding results in \cite{Berman}.

Consider the classical situation of the ring $R=\mathbb{Z}[\frac{1}{2}]$. An abelian group $A$ admits the structure of an $R$-module if and only if the `multiplication by 2' homomorphism $2:A\rightarrow A$ is invertible. In this case, admitting an $R$-module structure is a \emph{property} of an abelian group, and not extra data. Equivalently, the unique ring homomorphism $\mathbb{Z}\rightarrow R$ induces a fully faithful functor $$\text{CRing}_{R/}\rightarrow\text{CRing}_{\mathbb{Z}/},$$ which is to say that $\mathbb{Z}\rightarrow\mathbb{Z}[\frac{1}{2}]$ is an \emph{epimorphism} of commutative rings.

Here we come upon an odd phenomenon in the category of commutative rings: epimorphisms are \emph{not} the same as surjections. What is true, however, is that epimorphisms of commutative rings coincide with injections of affine schemes (provided the homomorphism is of finite type \cite{EGA} 17.2.6). Thus we can identify
\begin{itemize}
\item properties of abelian groups which are classified by actions of commutative rings;
\item affine scheme `injections' into $\text{Spec}(\mathbb{Z})$.
\end{itemize}

\noindent Bousfield and Kan call such rings \emph{solid}, and they have classified all of them \cite{Core}. Because $\text{Spec}(\mathbb{Z})$ has just one point for each prime $p$, and a generic point for the prime $(0)$, solid rings are all built out of quotients and localizations of $\mathbb{Z}$ in a suitable way (quotients or localizations depending on whether the generic point of $\text{Spec}(\mathbb{Z})$ is included in the subset).

In this chapter, we prove that many \emph{properties} of symmetric monoidal $\infty$-categories are classified by the actions of solid semiring $\infty$-categories. We will begin by explaining the term \emph{semiring $\infty$-category}.

Most of the categories and $\infty$-categories that arise frequently can be made symmetric monoidal in two different ways. One symmetric monoidal structure (usually the categorical coproduct) behaves additively, the other (usually closed symmetric monoidal) behaves multiplicatively, and the multiplicative structure distributes over the additive structure. For example, Top has a cartesian product which distributes over disjoint union. The $\infty$-category $\text{Ab}_\infty$ (grouplike $\mathbb{E}_\infty$-spaces, or connective spectra) has a smash product which distributes over wedge product. We can make this language precise utilizing results of Gepner, Groth, and Nikolaus \cite{GGN}. See \S\ref{3.1.1} for details.

In particular, each of Top and $\text{Ab}_\infty$ have the structure of a semiring $\infty$-category (and in a natural way).

\begin{principle}
\label{Princ}
Top and $\text{Ab}_\infty$ classify the \emph{properties} of symmetric monoidal $\infty$-categories being cocartesian monoidal (respectively additive).
\end{principle}

\noindent What we have said is not literally true. The problem is that Top and $\text{Ab}_\infty$ are \emph{not finitely generated as semiring $\infty$-categories}, and will therefore have poor algebraic properties. However, there are two ways to resolve this problem, and either way the principle becomes true.

\begin{enumerate}
\item We may observe that Top and $\text{Ab}_\infty$ are finitely generated under colimits; that is, they are \emph{presentable}. Since they are also closed symmetric monoidal, they are commutative algebras in the $\infty$-category $\text{Pr}^{\text{L},\otimes}$ of presentable $\infty$-categories, and in this setting Principle \ref{Princ} holds true (HA 4.8).

There are many benefits to this first approach. Presentable $\infty$-categories have excellent formal properties (see \S\ref{1.1.4}), which make them an ideal setting for proving universal properties. For example, Lurie uses a variant of Principle \ref{Princ} to define a well-behaved symmetric monoidal smash product of spectra (HA 4.8.2). This is in answer to a long-standing open problem of late twentieth century homotopy theory; the first solution \cite{EKMM} predates Lurie's solution by a decade, but his is the first from an $\infty$-categorical perspective and is surprisingly slick compared with its predecessors.
\item We may insist on working just with finitely generated semiring $\infty$-categories, and take the smallest subsemirings of Top and $\text{Ab}_\infty$; that is, the subcategories generated by sums and products of the additive and multiplicative units. In the case of Top, we recover in this way the semiring category Fin of finite sets.

In the case of $\text{Ab}_\infty$, we recover the subcategory of finite wedge powers of the sphere spectrum. This is equivalent to the Burnside $\infty$-category $\text{Burn}$ (see Remark \ref{RmkBurnGpCp}), which can also be described by applying a group-completion to the Hom-objects of $\text{Span}(\text{Fin})$, the 2-category of spans considered in \S\ref{1.2.3}. We sometimes refer to $\text{Span}(\text{Fin})$ as the \emph{effective} Burnside category.
\end{enumerate}

\begin{theorem*}[cf. Theorems \ref{2T1}, \ref{2T2}, \ref{2T3}, \ref{2T4}, \ref{2T5}, \ref{2T7b}]
Each of the following is a semiring $\infty$-category: Fin, $\text{Fin}^\text{op}$, $\text{Fin}^\text{iso}$, $\text{Fin}^\text{inj}$, $\text{Fin}^{\text{inj},\text{op}}$, $\text{Fin}_\ast$, $\text{Fin}_\ast^\text{op}$, $\text{Span}(\text{Fin})$, and Burn. Superscripts denote that we are only allowing injections (inj) or bijections (iso). $\text{Fin}_\ast$ denotes pointed finite sets.

Moreover, for each semiring $\mathcal{R}$ on this list, the forgetful functor $\text{Mod}_R\rightarrow\text{SymMon}$ is fully faithful; that is, being an $\mathcal{R}$-module is a property of a symmetric monoidal $\infty$-category, rather than containing extra structure. These properties are as follows: \\
\begin{tabular}{l r}
$\mathcal{R}$ &$\text{Mod}_\mathcal{R}$ \\ \hline
$\text{Fin}^\text{iso}$ &arbitrary symmetric monoidal \\
$\text{Fin}$ &cocartesian monoidal \\
$\text{Fin}^\text{op}$ &cartesian monoidal \\
$\text{Fin}^\text{inj}$ &semi-cocartesian monoidal \\
$\text{Fin}^{\text{inj},\text{op}}$ &semi-cartesian monoidal \\
$\text{Fin}_\ast$ &cocartesian monoidal with zero object \\
$\text{Fin}_\ast^\text{op}$ &cartesian monoidal with zero object \\
$\text{Span}(\text{Fin})$ &semiadditive (both cocartesian and cartesian monoidal) \\
$\text{Burn}$ &additive \\
\end{tabular}
\end{theorem*}

\noindent We will delay the last result (that Burn-modules are additive categories) until Chapter \ref{4} (Corollary \ref{2T7b}). The others are proven in \S\ref{3.2.1}.

\begin{remark}
All of the results of the last theorem are true for 1-categories as well, although the author is not aware of a proof (even of the 1-categorical results) that does not use the language of $\infty$-categories. Thus this chapter may be of interest to some classical category theorists who are not otherwise interested in $\infty$-categories.
\end{remark}

\noindent We also prove an additional result, which does not have a natural 1-categorical analogue.

\begin{theorem*}[Theorem \ref{MainS}]
There is a (solid) semiring $\infty$-category $\vec{\mathbb{S}}$ for which $\text{Mod}_{\vec{\mathbb{S}}}$ is equivalent to the $\infty$-category of connective spectra.
\end{theorem*}

\noindent The idea of identifying spectra with $\mathbb{S}$-modules in a larger category should be familiar to stable homotopy theorists; in fact, standard constructions of spectra (like Elmendorf-Kriz-Mandell-May \cite{EKMM}) take this approach.\footnote{Theorem \ref{MainS} should not itself be taken as a construction or definition of spectra. We cannot work seriously with symmetric monoidal $\infty$-categories without knowing in advance something about connective spectra, so such a definition would be circular.}

A benefit of such an approach is that it provides for us a way to compare symmetric monoidal $\infty$-categories to spectra, internal to our categorified commutative algebra. For example, we have a free functor from symmetric monoidal $\infty$-categories to spectra, given by tensoring with $\vec{\mathbb{S}}$. This operation can be regarded as a relative of algebraic K-theory.

Roughly, if $\mathcal{C}$ is a symmetric monoidal $\infty$-category, then $\mathcal{C}\otimes\vec{\mathbb{S}}$ is obtained from $\mathcal{C}$ first by formally inverting all morphisms (taking the classifying space) and then group-completing. In contrast, most constructions of algebraic K-theory (such as \cite{Mandell}) operate by throwing out all non-invertible morphisms and then group-completing.

Various notions of categorified rings have appeared before, including to study iterated K-theory \cite{BDRR}, Tannaka duality \cite{FunProGpoid}, and smash products of spectra (HA 4.8). The framework in this chapter is general enough to permit comparisons to many of the other notions of categorified rings; the comparison to what may be called \emph{presentable} categorified rings (as in HA 4.8 and \cite{FunProGpoid}) is the topic of Chapter \ref{4}.

\subsubsection{Organization}
\noindent In \S\ref{3.1}, we will define semiring $\infty$-categories and solid semiring $\infty$-categories. In particular, we will explain what we mean when we say (of a solid semiring $\infty$-category) that being an $\mathcal{R}$-module is a \emph{property} of a symmetric monoidal $\infty$-category, rather than extra structure.

Actually, all of \S\ref{3.1} is presented in more generality, beginning with a presentable and cartesian closed $\infty$-category $\mathcal{V}$, and considering semiring (or solid semiring) objects in $\mathcal{V}$. This introduction has been written with the example $\mathcal{V}=\text{Cat}$ in mind, but we can also study semiring 1-categories by choosing $\mathcal{V}=\text{Cat}_1$. Choosing $\mathcal{V}=\text{Set}$ (respectively $\text{Top}$) recovers the notion of ordinary commutative semirings (respectively $\mathbb{E}_\infty$-semiring spaces).

In \S\ref{3.2} and \S\ref{3.3}, we specialize to $\mathcal{V}=\text{Cat}$, proving the results discussed in the introduction above. \S\ref{3.2} contains all the results except about $\vec{\mathbb{S}}$-modules, which is the content of \S\ref{3.3}.

\section{Semiring objects}\label{3.1}
\subsection{Semirings}\label{3.1.1}
\noindent Suppose that $\mathcal{V}^\otimes$ is a presentable $\infty$-category with a closed symmetric monoidal structure. That is, for each object $X\in\mathcal{V}$, the functor $$-\otimes X:\mathcal{V}\rightarrow\mathcal{V}$$ has a right adjoint $\text{Hom}_\mathcal{V}(X,-)$.

As in \S\ref{1.1.4} (under the heading `Tensor products of presentable $\infty$-categories'), $\mathcal{V}^\otimes$ is then a commutative algebra in $\text{Pr}^{\text{L},\otimes}$. By Theorem \ref{ThmGGN} and its corollary, commutative monoids in $\mathcal{V}$ are endowed with an external tensor product, such that the free functor $$\text{Free}:\mathcal{V}^\otimes\rightarrow\text{CMon}(\mathcal{V})^\otimes$$ is symmetric monoidal.

\begin{definition}
\label{DefCSRing}
A \emph{commutative semiring} in $\mathcal{V}$ is a commutative algebra in $\text{CMon}(\mathcal{V})^\otimes$. For the $\infty$-category thereof, we write $$\text{CSRing}(\mathcal{V})=\text{CAlg}(\text{CMon}(\mathcal{V})^\otimes).$$
\end{definition}

\begin{example}
If $\mathcal{V}=\text{Set}$, a commutative monoid in Set is just a commutative monoid. The tensor product on CMon is the usual one (which agrees with the tensor product of abelian groups, in the case that both commutative monoids are groups). A commutative semiring in Set is just a commutative semiring\footnote{Semirings are also sometimes called rigs, emphasizing `ring with \emph{inverses}'.}.
\end{example}

\begin{example}
\label{ExCSRingCat}
If $\mathcal{V}=\text{Cat}$, a commutative monoid in Cat is a symmetric monoidal $\infty$-category. We will refer to a commutative semiring in Cat as a \emph{commutative semiring $\infty$-category}.

Roughly, a commutative semiring $\infty$-category is an $\infty$-category with two symmetric monoidal operations ($\oplus$ and $\otimes$), where $\otimes$ `distributes' over $\oplus$.

Similarly, a commutative semiring in $\text{Cat}_1$ is a commutative semiring 1-category.
\end{example}

\begin{example}
If $K$ is a collection of small simplicial sets, let $\text{Cat}^K$ denote the $\infty$-category of small $\infty$-categories which admit colimits from $K$ (alternatively limits from $K$). Lurie has constructed a closed symmetric monoidal structure on $\text{Cat}^K$ (HA 4.8.1). For example, $\text{Cat}_\text{pb}$ (the $\infty$-category of small $\infty$-categories which admit pullbacks) has a closed symmetric monoidal structure. Commutative semirings in $\text{Cat}_\text{pb}$ will play a central role in our study of motives for group cohomology (Chapter \ref{5}).
\end{example}

\noindent In this thesis, we are especially concerned with the last two examples: $\mathcal{V}=\text{Cat}$ in \S\ref{3.2}, and $\mathcal{V}=\text{Cat}_\text{pb}$ in \S\ref{5.1}. But using Definition \ref{DefCSRing}, it is not immediately clear how to even construct examples of commutative semirings in Cat and $\text{Cat}_\text{pb}$. To produce examples, we will typically use either the following lemma, or the theory of idempotent monoids as in Lemma \ref{MainLemma}.

\begin{lemma}[\cite{Berman} 2.4]
\label{EasyLemma}
Suppose $\mathcal{V}$ and $\mathcal{W}$ are presentable and cartesian closed $\infty$-categories, and $L:\mathcal{V}\rightleftarrows\mathcal{W}:R$ is an adjunction such that $L$ preserves finite products. Then there is an induced adjunction $$L_\ast:\text{CMon}(\mathcal{V})\rightleftarrows\text{CMon}(\mathcal{W}):R_\ast$$ such that:
\begin{enumerate}
\item the adjunction lifts to a symmetric monoidal adjunction (that is, $L_\ast$ lifts to a symmetric monoidal functor and $R_\ast$ to a lax symmetric monoidal functor);
\item $L_\ast$ (respectively $R_\ast$) agrees with $L$ (respectively $R$) after forgetting commutative monoid structures.
\end{enumerate}
\end{lemma}

\begin{remark}
Since $R_\ast$ and $L_\ast$ are both (at least) lax symmetric monoidal, they induce functors $\text{CSRing}(\mathcal{V})\rightleftarrows\text{CSRing}(\mathcal{W})$ which agree with $L$ and $R$ after forgetting commutative semiring structures. Therefore, if $A$ is a commutative semiring in $\mathcal{V}$ (respectively $\mathcal{W}$), $L(A)$ (respectively $R(A)$) inherits a commutative semiring structure. We will give examples at the beginning of \S\ref{3.2}.
\end{remark}

\begin{proof}
Given an adjoint pair $L:\mathcal{V}\rightleftarrows\mathcal{W}:R$ such that $L$ is product-preserving, $L$ lifts to a functor $\mathcal{V}^\times\rightarrow\mathcal{W}^\times$ in $\text{CAlg}(\text{Pr}^{\text{L},\otimes})$; that is, a symmetric monoidal left-adjoint functor. Tensoring with $\text{CMon}_\infty$, we obtain another symmetric monoidal left-adjoint functor \\ $L_\ast:\text{CMon}(\mathcal{V})^\otimes\rightarrow\text{CMon}(\mathcal{W})^\otimes$. Since tensoring with $\text{CMon}_\infty$ is naturally equivalent to taking commutative monoid objects, $L_\ast(X)$ is given by composition of $\infty$-operad maps $$\text{Comm}\xrightarrow{X}\mathcal{V}\xrightarrow{L}\mathcal{W},$$ where Comm is the commutative operad. Thus $L_\ast$ is compatible with the forgetful functor down to $\mathcal{V}$ and $\mathcal{W}$.

Since $L_\ast$ is a symmetric monoidal functor with a right adjoint, by \cite{GepHaug} A.5.11, we have a symmetric monoidal adjunction between $L_\ast$ and $R_\ast$.

By construction there is a commutative diagram of left adjoint functors $$\xymatrix{
\mathcal{V}\ar[r]^L\ar[d]_{\text{Free}} &\mathcal{W}\ar[d]^{\text{Free}} \\
\text{CMon}(\mathcal{V})\ar[r]_{L_\ast} &\text{CMon}(\mathcal{W}).
}$$ Taking right adjoints, we see that $R_\ast$ is compatible with $R$ after forgetting the commutative monoid structures.
\end{proof}

\subsection{Solid semirings}\label{3.1.2}
\noindent Recall the following phenomenon of commutative algebra: For some commutative rings $R$, being an $R$-module is a \emph{property} of an abelian group, and not extra structure. For example, $\mathbb{Z}/p$-modules correspond to abelian groups which are annihilated by $p$, and $\mathbb{Z}[\frac{1}{p}]$-modules correspond to abelian groups for which multiplication by $p$ is invertible.

Bousfield and Kan call such rings \emph{solid}, and they have classified all of them \cite{Core}. The (finitely generated) solid rings are just quotients and localizations of $\mathbb{Z}$, as well as products $\mathbb{Z}[S^{-1}]\times\mathbb{Z}/n$ (where each prime divisor of $n$ is in $S$).

Note that the solid rings $R$ are exactly those such that the homomorphism $\mathbb{Z}\rightarrow R$ is an epimorphism of commutative rings. Equivalently, they are subobjects of the affine scheme $\text{Spec}(\mathbb{Z})$. Thus, each solid ring records some of the algebraic geometry of $\mathbb{Z}$.

In this section, we will generalize solid rings, replacing the category of abelian groups by some symmetric monoidal $\infty$-category $\mathcal{C}^\oast$, and commutative rings by $\text{CAlg}(\mathcal{C}^\oast)$. Usually, we will want to take $\mathcal{C}^\oast=\text{CMon}(\mathcal{V})^\otimes$, for some $\mathcal{V}$ which is presentable and closed symmetric monoidal.

\begin{definition}
\label{DefSolid}
Let $\mathcal{C}^\oast$ be symmetric monoidal. Write 1 for its unit. If $R\in\text{CAlg}(\mathcal{C}^\oast)$, the following are equivalent, in which case we call $R$ \emph{solid}.
\begin{enumerate}
\item the forgetful functor $\text{Mod}_R\rightarrow\mathcal{C}$ is fully faithful;
\item the functor $-\oast R:\mathcal{C}\rightarrow\mathcal{C}$ is a localization (called a \emph{smashing localization});
\item the multiplication map $R\oast R\rightarrow R$ is an equivalence in $\mathcal{C}$;
\item either of the maps $R\rightarrow R\oast R$ induced by tensoring the unit $1\rightarrow R$ with $R$ is an equivalence in $\mathcal{C}$;
\item the map $1\rightarrow R$ is an epimorphism in $\text{CAlg}(\mathcal{C}^\oast)$.
\end{enumerate}
\end{definition}

\noindent Although we are not aware of a specific work which includes all of these conditions under the name `solid', none of the conditions are new. So the reader who objects that this `definition' requires proof (that all the conditions are equivalent) may consult \cite{GGN} for all but the last condition.

As for (5), it is directly equivalent to (3), because $X\rightarrow Y$ is an epimorphism if and only if (by definition) the codiagonal $Y\amalg_X Y\rightarrow Y$ is an equivalence. The unit 1 is initial in $\text{CAlg}(\mathcal{C}^\oast)$ and so $Y\amalg_1 Y\cong Y\amalg Y$. Moreover, the coproduct $\amalg$ in $\text{CAlg}(\mathcal{C}^\oast)$ is just $\oast$, so the equivalence follows directly.

\begin{remark}
\label{georemark}
In commutative rings, we should not expect epimorphisms to look anything like literal surjections. For example, as in the last remark, localizations $R\rightarrow R[S^{-1}]$ are epimorphisms. On the other hand, it is true that any surjection of commutative rings is an epimorphism.

Instead, we might think of epimorphisms as having geometric meaning. Provided that $f:R\rightarrow A$ is finitely generated, $f$ is an epimorphism if and only if $\text{Spec}(A)\rightarrow\text{Spec}(R)$ is injective in the sense that every fiber is either an isomorphism or empty (\cite{EGA} 17.2.6). So in general, we may think of solid objects as capturing some of the `algebraic geometry' of the symmetric monoidal unit of $\mathcal{C}^\oast$.
\end{remark}

\begin{remark}
\label{SolidRingSp}
If $\mathcal{C}^\oast=\text{Sp}^\wedge$, commutative algebras in $\mathcal{C}^\oast$ are $\mathbb{E}_\infty$-ring spectra. Here, epimorphisms have even less in common with surjections. Consider the map of Eilenberg-Maclane ring spectra $\phi:HR\rightarrow HA$ induced by a ring homomorphism $R\rightarrow A$. For $\phi$ to be an epimorphism, we would need $HA\wedge_{HR} HA\rightarrow HA$ to be an equivalence; that is, not only is $R\rightarrow A$ a ring epimorphism ($A\otimes_R A\rightarrow A$ is an isomorphism), but also $\text{Tor}_\ast^R(A,A)\cong 0$ for all $\ast>0$. This is in general \emph{not} true when $R\rightarrow A$ is surjective.

For example, $H\mathbb{Z}\rightarrow HR$ is an epimorphism of ring spectra only when $R$ is a subring of $\mathbb{Q}$ (a localization of $\mathbb{Z}$). While Bousfield and Kan do not prove this, it is a straightforward corollary of their results \cite{Core}, and details can be found in the MathOverflow answer \cite{MathOverflow}.

Note that this is a classification of solid $H\mathbb{Z}$-algebras, not solid ring spectra. But the classification of solid ring spectra is very similar. In particular, a commutative ring spectrum $E$ is solid if and only if it is a Moore spectrum for which $\pi_0 E$ is isomorphic to a subring of $\mathbb{Q}$ \cite{SolidSp}.
\end{remark}

\noindent In the following two lemmas, we will
\begin{itemize}
\item learn that any idempotent object of $\mathcal{C}^\oast$ admits a canonical commutative algebra structure, which is solid;
\item (when $\mathcal{C}=\text{CMon}(\mathcal{V})$ for some presentable cartesian closed $\mathcal{V}$) learn how to recognize solid semirings by identifying the corresponding \emph{property} encoded by their modules (that is, identifying the full subcategory $\text{Mod}_A\subseteq\mathcal{C}$).
\end{itemize}

\noindent These lemmas lie at the heart of most of the main results of this chapter. Typically, we will be given some commutative monoid $A$ in $\mathcal{V}$, and we will know something about the behavior of $\text{Hom}(A,-)$. Using Lemma \ref{MainLemma}, we will conclude that $A$ is a solid commutative semiring, and that the full subcategory $\text{Mod}_A\subseteq\text{CMon}(\mathcal{V})$ has a convenient description.

\begin{definition}
Suppose $1$ denotes the symmetric monoidal unit of $\mathcal{C}^\oast$, and let $1\rightarrow A$ be a morphism in $\mathcal{C}$. We say $\mathcal{A}$ is an \emph{idempotent object} of $\mathcal{C}$ if the induced morphism $1\otimes A\rightarrow A\otimes A$ is an equivalence.
\end{definition}

\begin{lemma}[HA 4.8.2.9]
Given $A\in\text{CAlg}(\mathcal{C}^\oast)$, the unit of $A$ produces a map $1\rightarrow A$. By Definition \ref{DefSolid}, $A$ is solid if and only if $1\rightarrow A$ is idempotent. The induced functor from solid commutative algebras to idempotent objects is an equivalence of $\infty$-categories: $$\text{CAlg}(\mathcal{C}^\oast)^\text{solid}\rightarrow\mathcal{C}_{1/}^\text{idem}.$$
\end{lemma}

\noindent Suppose $\mathcal{V}^\otimes$ is presentable and closed symmetric monoidal. Denote by $\mathbb{I}$ the free commutative monoid on the terminal object of $\mathcal{V}$, so that $\text{Hom}_{\text{CMon}(\mathcal{V})}(\mathbb{I},X)\cong X$. We will usually write Hom instead of $\text{Hom}_{\text{CMon}(\mathcal{V})}$.

\begin{lemma}
\label{MainLemma}
Let $A$ be a commutative monoid in $\mathcal{V}$, $1\rightarrow A$ a morphism in $\mathcal{V}$ from the terminal object (inducing a commutative monoid morphism $\mathbb{I}\rightarrow A$), and $\mathcal{P}$ a full subcategory of $\mathcal{C}$. Precomposition by $\mathbb{I}\rightarrow A$ induces for each $X\in\text{CMon}(\mathcal{V})$ a commutative monoid morphism $$U_X:\text{Hom}(A,X)\rightarrow X.$$ The following are equivalent:
\begin{enumerate}
\item $\text{Hom}(A,X)\in\mathcal{P}$ for all $X\in\text{CMon}(\mathcal{V})$, and $U_X$ is an equivalence for all $X\in\mathcal{P}$;
\item $A$ is idempotent (therefore a solid semiring), and $\text{Mod}_A\cong\mathcal{P}$ as full subcategories of $\mathcal{C}$.
\end{enumerate}
\end{lemma}

\begin{proof}
First, assume (2). By Proposition \ref{TupHup}, the inclusion $\text{Mod}_A\subseteq\text{CMon}(\mathcal{V})$ has a right adjoint given by $\text{Hom}(A,-)$. Thus $\text{Hom}(A,X)\in\mathcal{P}$ for all $X$. Moreover, if $X\in\text{Mod}_A$, then $U_X:\text{Hom}(A,X)\rightarrow X$ is the counit of the adjunction, which is an equivalence since the left adjoint is fully faithful.

This proves $(2)\to(1)$. Conversely, assume (1). Then $\mathbb{I}\otimes A\rightarrow A\otimes A$ induces, for each $X\in\text{CMon}(\mathcal{V})$, a map $\text{Hom}(A\otimes A,X)\rightarrow\text{Hom}(A,X)$ which, by tensor-Hom adjunction, is equivalent to the map $$U_{\text{Hom}(A,X)}:\text{Hom}(A,\text{Hom}(A,X))\rightarrow\text{Hom}(A,X).$$ Since $\text{Hom}(A,X)\in\mathcal{P}$, this map is an equivalence, so we have equivalences of \emph{spaces} $$\text{Map}_{\text{CMon}(\mathcal{V})}(A\otimes A,X)\rightarrow\text{Map}_{\text{CMon}(\mathcal{V})}(A,X)$$ for all $X$. By the Yoneda lemma, it follows that $A\rightarrow A\otimes A$ is an equivalence, so $A$ is idempotent and therefore has a canonical solid semiring structure.

Now we know that both $\mathcal{P}$ and $\text{Mod}_A$ are full subcategories of $\text{CMon}(\mathcal{V})$. We need only show that a commutative monoid has an $A$-module structure if and only if it is in $\mathcal{P}$. This is basically the same as the implication $(1)\to(2)$.

Specifically, if $X\in\mathcal{P}$, then $X\cong\text{Hom}(A,X)$, which has an $A$-module structure. Conversely, if $X$ is an $A$-module, then $X\cong\text{Hom}(A,X)$, so $X\in\mathcal{P}$.
\end{proof}

\section{Commutative semiring \texorpdfstring{$\infty$-}{infinity }categories}\label{3.2}
\noindent In this section, we will study commutative semiring $\infty$-categories; that is, commutative semiring objects in $\mathcal{V}=\text{Cat}$. Recall that commutative monoid objects in Cat are symmetric monoidal $\infty$-categories. As in Example \ref{ExCSRingCat}, the $\infty$-category SymMon is itself endowed with a symmetric monoidal structure $\otimes$, and commutative algebras in SymMon are commutative semiring $\infty$-categories.

We may notate a commutative semiring $\infty$-category $\mathcal{R}^{\oplus,\otimes}$, to suggest that $\mathcal{R}$ is an $\infty$-category with two symmetric monoidal operations, satisfying some sort of distributive law.

Commutative semiring $\infty$-categories are familiar from a host of examples, such as $\text{Set}^{\amalg,\times}$ and $\text{Ab}^{\oplus,\otimes}$. In general:

\begin{proposition}
\label{PropPrRing}
If $\mathcal{C}^\oast$ is a presentable and closed symmetric monoidal $\infty$-category, then $\mathcal{C}^{\amalg,\oast}$ is a commutative semiring $\infty$-category.
\end{proposition}

\noindent We will prove this proposition in Chapter \ref{4} (see Lemma \ref{LemPrRing}), and in the meantime we will not need to use it.

In general, we write $\text{Fun}(-,-)$ for the internal Hom in Cat, and $\text{Hom}(-,-)$ for the internal Hom in SymMon (being sure to indicate the symmetric monoidal structure on each of the arguments). The internal $\text{Hom}(\mathcal{C}^\oast,\mathcal{D}^\odot)$ is itself a symmetric monoidal $\infty$-category; the symmetric monoidal structure is induced by objectwise $\odot$ (inherited from $\mathcal{D}^\odot$).

\subsubsection{Examples}
\noindent Here we describe a few ways to construct commutative semiring $\infty$-categories from other commutative semiring objects. In \S\ref{3.2.1} we will give specific examples.

Consider the following adjunctions (the first of which is an equivalence): $$(-)^\text{op}:\text{Cat}\rightleftarrows\text{Cat}:(-)^\text{op}$$ $$j:\text{Top}\rightleftarrows\text{Cat}:(-)^\text{iso}$$ $$h:\text{Cat}_\infty\rightleftarrows\text{Cat}:N.$$ The functor $j$ is the inclusion of $\infty$-groupoids into $\infty$-categories, $N$ is the nerve, and $h$ is the formation of homotopy categories. For an $\infty$-category $\mathcal{C}$, $\mathcal{C}^\text{iso}$ is the subcategory spanned by all objects, but only those morphisms which are equivalences.

In each adjunction, the left adjoint is product-preserving, so Lemma \ref{EasyLemma} applies. Thus we have:

\begin{example}
If $\mathcal{C}$ is a commutative semiring 1-category, its nerve (which we also denote $\mathcal{C}$) is a commutative semiring $\infty$-category.
\end{example}

\begin{example}
\label{CSRingExSp}
If $X$ is an $\mathbb{E}_\infty$-space, then $X$ is also a symmetric monoidal $\infty$-category. For example, if $E$ is a spectrum, then $\Omega^\infty E$ is a symmetric monoidal $\infty$-category.

If $E$ is an $\mathbb{E}_\infty$-ring spectrum, then $\Omega^\infty E$ is a commutative semiring $\infty$-category.

When $E$ is connective, we will write $\vec{E}$ for this symmetric monoidal $\infty$-category. The arrow indicates that we are thinking of $\vec{E}$ as an $\infty$-category, not an $\infty$-groupoid (so the morphisms are, in principle, directed).
\end{example}

\begin{example}
\label{hCSRingEx}
If $\mathcal{C}$ is a commutative semiring $\infty$-category, so are $\mathcal{C}^\text{op}$ and $\mathcal{C}^\text{iso}$. The homotopy category $h\mathcal{C}$ is a commutative semiring 1-category.
\end{example}

\subsection{Cartesian monoidal \texorpdfstring{$\infty$-}{infinity }categories}\label{3.2.1}
\noindent The examples of this section arise out of the study of \emph{algebraic theories}, such as $\infty$-operads. However, knowledge of $\infty$-operads is unnecessary to read this section; we use this language in the next two examples for two reasons only:
\begin{itemize}
\item to establish the three formulas of Example \ref{ExFormulas};
\item because in Chapter \ref{4}, we will use the material in this section to study algebraic theories in more depth.
\end{itemize}

\noindent The reader should feel free to verify the three formulas by alternative methods, or take them as black boxes.

\begin{example}
\label{ExOp}
Following HA, an $\infty$-operad $\mathcal{O}$ is a type of fibration $p:[\mathcal{O}^\otimes]\xrightarrow{p}\text{Fin}_\ast$. We are using square brackets to emphasize that $[\mathcal{O}^\otimes]$ is a different $\infty$-category from $\mathcal{O}$ (as opposed to a symmetric monoidal structure on $\mathcal{O}$). Really, these correspond to what are traditionally called \emph{colored operads}. If we wish to insist that $\mathcal{O}$ be single-colored, we must ask that the underlying $\infty$-category $\mathcal{O}$ have just one object up to equivalence.

By HA 2.2.4, to an $\infty$-operad is associated a \emph{symmetric monoidal envelope} $\text{Env}(\mathcal{O})^\oast$. This is a symmetric monoidal $\infty$-category satisfying the universal property $$\text{Hom}(\text{Env}(\mathcal{O})^\oast,\mathcal{C}^\oast)\cong\text{Alg}_{\mathcal{O}}(\mathcal{C}^\oast).$$ As an $\infty$-category, $\text{Env}(\mathcal{O})$ is the subcategory of $[\mathcal{O}^\otimes]$ spanned by all objects and active morphims\footnote{An active morphism in $\text{Fin}_\ast$ is $f:X\rightarrow Y$ such that $f^{-1}(\ast)=\ast$. An active morphism in $[\mathcal{O}^\otimes]$ is a morphism $f$ for which $p(f)$ is active.} between them. That is, it is given by the pullback: $$\xymatrix{
\text{Env}(\mathcal{O})\ar[r]\ar[d] &\text{Fin}\ar[d]^{(-)_{+}} \\
[\mathcal{O}^\otimes]\ar[r]_p &\text{Fin}_\ast.
}$$
\end{example}

\begin{example}[HA 2.1.1.18-20]
\label{ExFormulas}
The symmetric monoidal envelope of the commutative operad is Fin, of the $\mathbb{E}_0$-operad is $\text{Fin}^\text{inj}$, and of the trivial operad is $\text{Fin}^\text{iso}$. Note that Fin denotes finite sets, and the superscript indicates that we are only allowing injections (inj) or bijections (iso) of finite sets. In each case, the symmetric monoidal operation is disjoint union. Thus: $$\text{Hom}(\text{Fin}^{\text{iso},\amalg},\mathcal{C}^\oast)\cong\mathcal{C}^\oast;$$ $$\text{Hom}(\text{Fin}^{\text{inj},\amalg},\mathcal{C}^\oast)\cong\mathcal{C}_{1/}^\oast;$$ $$\text{Hom}(\text{Fin}^\amalg,\mathcal{C}^\oast)\cong\text{CAlg}(\mathcal{C}^\oast)^\oast.$$ In each case, the equivalence is compatible with the forgetful functors to $\mathcal{C}$. (On the left hand sides, the forgetful functors are given by evaluation at a singleton set.)
\end{example}

\noindent We will use these three formulas in conjunction with Lemma \ref{MainLemma} to prove that each of $\text{Fin}^\text{iso}$, $\text{Fin}^\text{inj}$, and Fin are commutative semiring $\infty$-categories, indeed solid commutative semiring $\infty$-categories. We will then treat two additional examples of solid semirings: the effective Burnside $\infty$-category $\text{Span}(\text{Fin})$, and the category $\text{Fin}_\ast$ of finite \emph{pointed} sets.

\subsubsection{The semiring $\text{Fin}^\text{iso}$}
\begin{theorem}[\cite{Berman} 2.9]
\label{2T1}
Under $\amalg$, $\text{Fin}^\text{iso}$ is the unit of $\text{SymMon}^\otimes$. Under $\amalg$ and $\times$, it is the initial object in CSRingCat (commutative semiring $\infty$-categories).
\end{theorem}

\noindent In particular $\text{Fin}^\text{iso}$ is a solid semiring $\infty$-category, and \emph{all} symmetric monoidal $\infty$-categories are $\text{Fin}^\text{iso}$-modules.

\begin{proof}
Let $\mathbb{I}$ denote the unit of $\text{SymMon}^\otimes$; since $\text{Free}:\text{Cat}^\times\rightarrow\text{SymMon}^\otimes$ is a symmetric monoidal functor, $\mathbb{I}$ is the free symmetric monoidal $\infty$-category on one object. There is a symmetric monoidal functor $i:\mathbb{I}\rightarrow\text{Fin}^{\text{iso},\amalg}$ obtained via adjunction from the functor $1\rightarrow\text{Fin}^\text{iso}$ (which picks out the singleton set). Precomposition with $i$ induces the forgetful functor $$\text{Hom}(\text{Fin}^{\text{iso},\amalg},\mathcal{C}^\oast)\rightarrow\mathcal{C}^\oast,$$ which is an equivalence by Example \ref{ExFormulas}. By the Yoneda lemma, $i$ is an equivalence of symmetric monoidal $\infty$-categories, completing the proof.
\end{proof}

\subsubsection{The semiring $\text{Fin}^\text{inj}$}
\begin{definition}
Let $\mathcal{C}^\oast$ be a symmetric monoidal $\infty$-category with unit $1$. We say that $\mathcal{C}^\oast$ is \emph{semi-cartesian monoidal} if $1$ is a terminal object of $\mathcal{C}$, and \emph{semi-cocartesian monoidal} if it is an initial object.
\end{definition}

\begin{theorem}[\cite{Berman} 3.17]
\label{2T2}
$\text{Fin}^\text{inj}$ is a solid semiring $\infty$-category, and a symmetric monoidal $\infty$-category is a $\text{Fin}^\text{inj}$-module if and only if it is semi-cocartesian monoidal.
\end{theorem}

\begin{corollary}
Dually, $\text{Fin}^{\text{inj},\text{op}}$ is a solid semiring $\infty$-category whose modules are semi-cartesian monoidal $\infty$-categories.
\end{corollary}

\begin{proof}
Let the functor $i:1\rightarrow\text{Fin}^\text{inj}$ pick out a singleton set. Precomposition by $i$ induces the forgetful functors $$U_\mathcal{C}:\text{Hom}(\text{Fin}^{\text{inj},\amalg},\mathcal{C}^\oast)\rightarrow\mathcal{C}.$$ By Example \ref{ExFormulas}, these are equivalent to the forgetful functors $$\mathcal{C}_{1/}\rightarrow\mathcal{C}.$$ Therefore, $U_\mathcal{C}$ is an equivalence if $\mathcal{C}$ is semi-cocartesian monoidal. On the other hand, $\text{Hom}(\text{Fin}^{\text{inj},\amalg},\mathcal{C}^\oast)\cong\mathcal{C}_{1/}^\oast$ is always semi-cocartesian monoidal. Lemma \ref{MainLemma} completes the proof.
\end{proof}

\subsubsection{The semiring $\text{Fin}$}
\begin{theorem}[\cite{Berman} 3.1]
\label{2T3}
Fin is a solid semiring $\infty$-category, and a symmetric monoidal $\infty$-category is a Fin-module if and only if it is cocartesian monoidal.
\end{theorem}

\begin{corollary}
\label{2T3b}
Dually, $\text{Fin}^\text{op}$ is a solid semiring $\infty$-category whose modules are cartesian monoidal $\infty$-categories.
\end{corollary}

\begin{proof}
Let the functor $i:1\rightarrow\text{Fin}$ pick out a singleton set. Precomposition by $i$ induces the forgetful functors $$U_\mathcal{C}:\text{Hom}(\text{Fin}^\amalg,\mathcal{C}^\oast)\rightarrow\mathcal{C}.$$ By Example \ref{ExFormulas}, these are equivalent to the forgetful functors $$\text{CAlg}(\mathcal{C}^\oast)\rightarrow\mathcal{C}.$$ Thus $U_\mathcal{C}$ is an equivalence if $\mathcal{C}$ is cocartesian monoidal (HA 2.4.3.9). We also know that $\text{Hom}(\text{Fin}^\amalg,\mathcal{C}^\oast)\cong\text{CAlg}(\mathcal{C}^\oast)^\oast$ is always cocartesian monoidal (HA 3.2.4). Lemma \ref{MainLemma} completes the proof.
\end{proof}

\subsubsection{The semiring $\text{Span}(\text{Fin})$}
\noindent Consult \S\ref{1.2.3} for the span construction. Recall that $\text{Span}(\text{Fin})$, the $\infty$-category of spans of finite sets, is really a 2-category (the \emph{effective Burnside 2-category}).

\begin{definition}
Suppose that $\mathcal{C}^\oast$ is a symmetric monoidal $\infty$-category. We will say that $\mathcal{C}^\oast$ is \emph{semiadditive} if it is both cartesian monoidal and cocartesian monoidal.
\end{definition}

\begin{warning}
We are departing from standard terminology in regarding a semiadditive $\infty$-category as a type of symmetric monoidal $\infty$-category rather than a type of $\infty$-category. The latter structure is called preadditive by Gepner, Groth, and Nikolaus (\cite{GGN} Definition 2.1). However, the two definitions are equivalent: an $\infty$-category $\mathcal{C}$ is preadditive if and only if its homotopy category $h\mathcal{C}$ is preadditive/semiadditive (\cite{GGN} Example 2.2), which is true if and only if $h\mathcal{C}$ has cocartesian monoidal and cartesian monoidal structures which agree.
\end{warning}

\begin{theorem}[\cite{Berman} 3.4]
\label{2T4}
There is an equivalence of symmetric monoidal $\infty$-categories, $\text{Span}(\text{Fin})^\amalg\cong\text{Fin}^\amalg\otimes\text{Fin}^{\text{op},\amalg}$.

Therefore, $\text{Span}(\text{Fin})$ is a solid semiring $\infty$-category, and a symmetric monoidal $\infty$-category is a $\text{Span}(\text{Fin})$-module if and only if it is semiadditive.
\end{theorem}

\begin{proof}
Since Fin and $\text{Fin}^\text{op}$ are idempotent in $\text{SymMon}^\otimes$, so is their tensor product. Therefore $\text{Fin}\otimes\text{Fin}^\text{op}$ is a solid semiring $\infty$-category, and a symmetric monoidal $\infty$-category is a $(\text{Fin}\otimes\text{Fin}^\text{op})$-module if and only if it is both a $\text{Fin}$-module and a $\text{Fin}^\text{op}$-module; that is, if and only if it is semiadditive.

Now we need only show that $\text{Span}(\text{Fin})\cong\text{Fin}\otimes\text{Fin}^\text{op}$. However, we have just established that $\text{Fin}\otimes\text{Fin}^\text{op}$ is the free semiadditive $\infty$-category on one object. (We are referring to the free functor $\text{Cat}\rightarrow\text{SemiaddCat}$.) Glasman has already done the hard work here by showing that $\text{Span}(\text{Fin})$ is \emph{also} the free semiadditive $\infty$-category on one generator \cite{Glasman} (Theorem A.1). This completes the proof.
\end{proof}

\begin{remark}
This is one of the only examples we know of an explicit computation of a tensor product in $\text{SymMon}^\otimes$. Notice that the tensor product of two 1-categories is not necessarily still a 1-category. In this case, it is a 2-category. The author does not know whether tensor products of symmetric monoidal 1-categories are always 2-categories.
\end{remark}

\begin{corollary}
\label{CofreeSemiadd}
If $\mathcal{C}^\oast$ is symmetric monoidal, $\text{Hom}(\text{Span}(\text{Fin})^\amalg,\mathcal{C}^\oast)$ is the $\infty$-category of \emph{commutative-cocommutative bialgebras} in $\mathcal{C}^\oast$; that is, $\text{CCoalg}(\text{CAlg}(\mathcal{C}^\oast)^\oast)$.
\end{corollary}

\subsubsection{The semiring $\text{Fin}_\ast$}
\noindent We will let $\text{Fin}_\ast$ denote the category of finite \emph{pointed} sets. Recall that the coproduct in $\text{Fin}_\ast$ is the wedge product $\vee$, obtained by taking a disjoint union, then identifying the two basepoints.

\begin{definition}
A cartesian (respectively cocartesian) monoidal $\infty$-category \emph{with zero} is a cartesian (cocartesian) monoidal $\infty$-category such that the terminal object is also initial (or vice versa). In this case, we call the initial and terminal object a \emph{zero object}.
\end{definition}

\begin{theorem}[\cite{Berman} 3.18]
\label{2T5}
There is an equivalence of symmetric monoidal $\infty$-categories, $\text{Fin}_\ast^\vee\cong\text{Fin}^\amalg\otimes\text{Fin}^{\text{inj},\text{op},\amalg}$.

Therefore, $\text{Fin}_\ast$ is a solid semiring $\infty$-category, whose modules are cocartesian monoidal $\infty$-categories with zero.
\end{theorem}

\begin{corollary}
Dually, $\text{Fin}_\ast^\text{op}\cong\text{Fin}^\text{op}\otimes\text{Fin}^\text{inj}$ is a solid semiring $\infty$-category whose modules are cartesian monoidal $\infty$-categories with zero.
\end{corollary}

\noindent There are two approaches to prove the theorem. We may use Lemma \ref{MainLemma}, noting that $\text{Hom}(\text{Fin}_\ast^\vee,\mathcal{C}^\oast)$ is Segal's $\Gamma$-object construction \cite{SegalGamma}.

Alternatively, we may use a variant of the Yoneda lemma. This is not difficult, but for clarity of exposition we are delaying it to Theorem \ref{SMYonedaT}. This is the approach taken in \cite{Berman} 3.18; see there for a proof.

\subsubsection{Semiring 1-categories}
\noindent All of the results of \S\ref{3.2.1} above hold for 1-categories as well. That is, the following are solid semiring categories with modules as indicated:
\begin{itemize}
\item $\text{Fin}^\text{iso}$ is the initial object in $\text{CSRingCat}_1$ (so all symmetric monoidal categories are modules);
\item $\text{Fin}^\text{inj}$ and $\text{Fin}^{\text{inj},\text{op}}$ (modules are semi-cocartesian or semi-cartesian monoidal categories);
\item Fin and $\text{Fin}^\text{op}$ (modules are cocartesian or cartesian monoidal categories);
\item $\text{Fin}_\ast$ and $\text{Fin}^\text{op}_\ast$ (modules are cocartesian or cartesian monoidal categories with zero);
\item the effective Burnside 1-category $h\text{Span}(\text{Fin})$ (modules are semiadditive categories).
\end{itemize}

\noindent These results all follow from the following two observations combined with Lemma \ref{MainLemma}:
\begin{itemize}
\item formation of the homotopy category $h:\text{SymMon}^\otimes\rightarrow\text{SymMon}_1^\otimes$ is symmetric monoidal;
\item a symmetric monoidal $\infty$-category is cartesian monoidal (cocartesian monoidal, semi-cartesian monoidal, etc.) if and only if its homotopy category is.
\end{itemize}

\subsection{Problems}\label{3.2.2}
\noindent We have just seen eight examples of solid semiring $\infty$-categories $\mathcal{R}$. In each case, $\mathcal{R}$-modules encode some property en route to semiadditivity: (co)cartesian monoidal, semi-(co)cartesian monoidal, etc. These solid semirings are all closely related, as follows:

A morphism in $\text{Span}(\text{Fin})$ from $X$ to $Y$ (finite sets) takes the form $X\leftarrow T\rightarrow Y$, for some other finite set $T$. We will say that such a morphism is, for example, (injective, arbitrary) if $T\rightarrow X$ is injective and $T\rightarrow Y$ is arbitrary. By considering functions which are injective, bijective, or arbitrary, we obtain in this way nine types of structured morphisms in $\text{Span}(\text{Fin})$. For each type of structured morphism, we may consider the subcategory of $\text{Span}(\text{Fin})$ spanned by all objects, but only morphisms of the specified type. In this way, we recover each of the 8 solid semiring $\infty$-categories we have considered, and one that we haven't:
\begin{itemize}
\item $\text{Fin}^\text{inj}$ -- (bijective,injective) morphisms;
\item $\text{Fin}^{\text{inj},\text{op}}$ -- (injective,bijective) morphisms;
\item $\text{Fin}$ -- (bijective,arbitrary) morphisms;
\item $\text{Fin}^\text{op}$ -- (arbitrary,bijective) morphisms;
\item $\text{Fin}_\ast\cong\text{Fin}\otimes\text{Fin}^{\text{inj},\text{op}}$ -- (injective,arbitrary) morphisms;
\item $\text{Fin}_\ast^\text{op}\cong\text{Fin}^\text{op}\otimes\text{Fin}^\text{inj}$ -- (arbitrary,injective) morphisms;
\item $\text{Fin}^\text{iso}$ -- (bijective,bijective) morphisms;
\item $\text{Fin}_\ast^\text{inj}$ -- (injective,injective) morphisms;
\item $\text{Burn}^\text{eff}\cong\text{Fin}\otimes\text{Fin}^\text{op}$ -- (arbitrary,arbitrary) morphisms.
\end{itemize}

\noindent The new $\infty$-category, $\text{Fin}_\ast^\text{inj}$, can also be described as follows: objects are finite pointed sets. A morphism is a function $f:X\rightarrow Y$ such that $f^{-1}(y)$ is either empty or a singleton for all $y$ \emph{except possibly the basepoint}. In \cite{Berman} 3.30, we provide some evidence for an equivalence $\text{Fin}_\ast^\text{inj}\cong\text{Fin}^\text{inj}\otimes\text{Fin}^{\text{inj},\text{op}}$ of symmetric monoidal $\infty$-categories.

These examples suggest a close relationship between the tensor product of symmetric monoidal $\infty$-categories and the span construction:

\begin{remark}
If $P$ and $Q$ are properties chosen from the list \{bijective, injective, arbitrary\}, and $\text{Span}^{P,Q}(\text{Fin})$ denotes the subcategory of $\text{Span}(\text{Fin})$ spanned by all objects and $(P,Q)$ morphisms, then $$\text{Span}^{P,Q}(\text{Fin})\cong\text{Fin}^{P,\text{op}}\otimes\text{Fin}^Q,$$ except possibly when $P$ and $Q$ are both `injective' (in which case we conjecture the identity still holds).
\end{remark}

\begin{question}
Is there a general class of tensor products of symmetric monoidal $\infty$-categories which can be computed via span constructions?
\end{question}

\noindent An affirmative answer to this question may provide us with many new examples of tensor products, which a priori seem difficult to compute.

On the other hand, span constructions are readily computable but are not known to satisfy convenient universal properties. For example, $\text{Span}(\text{Fin}_G)$, also called the effective Burnside category for a finite group $G$, governs equivariant stable homotopy theory. By a theorem of Guillou-May (\cite{GMay2} 0.1), equivariant spectra can be described as product-preserving functors $\text{Span}(\text{Fin}_G)\rightarrow\text{Sp}$. (See \cite{BarMack1} for this language). A universal property for $\text{Span}(\text{Fin}_G)$ would be convenient for equivariant stable homotopy theory.

\begin{question}
Let $G$ be a finite group and $\text{Fin}_G$ the category of finite $G$-sets. Can $\text{Burn}_G^\text{eff}=\text{Span}(\text{Fin}_G)$ be decomposed as a tensor product? Is it true that $$\text{Burn}_G^\text{eff}\cong\text{Fin}_G\otimes_{\text{Fin}_G^\text{iso}}\text{Fin}_G^\text{op}?$$
\end{question}

\noindent This last question is equivalent to asking whether $$\text{Hom}(\text{Burn}_G^{\text{eff},\amalg},\text{Top}^\times)\cong\text{Hom}_{\text{Fin}_G^\text{iso}}(\text{Fin}_G^\amalg,\text{Hom}(\text{Fin}_G^{\text{op},\amalg},\text{Top}^\times));$$ or whether equivariant $\mathbb{E}_\infty$-spaces coincide with `$G$-commutative monoids' in equivariant spaces. Here we mean $G$-commutative monoids in the sense of Mazur \cite{Mazur} and Kaledin \cite{Kaledin}. The more refined notions of Hill and Hopkins \cite{HillHopkins} or Barwick et al. \cite{Barwick} require that we work with \emph{equivariant} symmetric monoidal $\infty$-categories (see \S\ref{A.4}). In that fully equivariant setting, we very much expect to have a statement of the form $\text{Burn}_G^\text{eff}\cong\text{Fin}_G\otimes_{\text{Fin}_G^\text{iso}}\text{Fin}_G^\text{op}$, but the successes of \cite{Mazur} and \cite{Kaledin} suggest that there is some hope of proving such a statement even for ordinary symmetric monoidal $\infty$-categories.

\section{Group-completion of categories}\label{3.3}
\noindent In the last section, we saw that various properties of symmetric monoidal $\infty$-categories (en route to full-fledged additive $\infty$-categories) correspond to the structure of modules over a solid semiring $\infty$-category.

In this section, we consider another example of a solid semiring $\infty$-category with dramatically different properties.

Suppose that $E$ is a spectrum, so that $\Omega^\infty E$ is an $\mathbb{E}_\infty$-space (or symmetric monoidal $\infty$-groupoid). As in Example \ref{CSRingExSp}, $\Omega^\infty E$ can be regarded as a symmetric monoidal $\infty$-category. In the event that $E$ is \emph{connective}, we can recover $E$ from $\Omega^\infty E$; that is, there is an equivalence of $\infty$-categories $$\text{Ab}_\infty\cong\text{Sp}_{\geq 0}$$ between grouplike $\mathbb{E}_\infty$-spaces and connective spectra.

If $E$ is a connective spectrum, we use $\vec{E}$ (rather than $\Omega^\infty E$) to denote the corresponding symmetric monoidal $\infty$-category. If $E$ is an $\mathbb{E}_\infty$-ring spectrum, then $\vec{E}$ is a commutative semiring $\infty$-category. See Example \ref{CSRingExSp}.

We have described a full subcategory inclusion $$\text{Ab}_{\infty}\rightarrow\text{SymMon}.$$ We will say (somewhat abusively) that a symmetric monoidal $\infty$-category \emph{is} a connective spectrum if it is in this full subcategory.

\begin{remark}
A symmetric monoidal $\infty$-category $\mathcal{C}^\oast$ is a connective spectrum if and only if both: all objects are invertible up to equivalence ($\mathcal{C}^\oast$ is grouplike); \emph{and} all morphisms are invertible up to equivalence ($\mathcal{C}$ is an $\infty$-groupoid).

But both of these conditions are simultaneously necessary: $\text{Fin}^\text{iso}$ is a symmetric monoidal $\infty$-category which is an $\infty$-groupoid but not grouplike. On the other hand, for an example of a symmetric monoidal category which is grouplike but not an $\infty$-groupoid, take the full subcategory of $\text{Mod}_R$ spanned by invertible modules (under $\otimes$).

If, however, $\mathcal{C}$ is a commutative semiring $\infty$-category which is grouplike (a \emph{ring $\infty$-category}), then $\mathcal{C}$ is an $\infty$-groupoid by Corollary \ref{CorGpoid} below.
\end{remark}

\noindent Our goal in this section is to show that the sphere spectrum $\vec{\mathbb{S}}$ is a solid semiring $\infty$-category, and its modules are precisely connective spectra. We expect this result to have connections with algebraic K-theory as follows:
\begin{itemize}
\item the localization functor $-\otimes\vec{\mathbb{S}}:\text{SymMon}\rightarrow\text{Ab}_\infty$ is a variant of algebraic K-theory, sending a symmetric monoidal $\infty$-category to the group-completion of its classifying space;
\item the colocalization functor $\text{Hom}(\vec{\mathbb{S}},-):\text{SymMon}\rightarrow\text{Ab}_\infty$ assigns to a symmetric monoidal $\infty$-category its \emph{Picard spectrum}.
\end{itemize}

\subsection{Group-completion of \texorpdfstring{$\infty$-}{infinity }categories}\label{3.3.1}
\noindent We should say a word about the group completion of commutative monoids, since we are discussing a categorified version of that story.

Commutative monoids are endowed with a tensor product; this satisfies the property that if $A$ and $B$ are abelian groups, $A\otimes B$ is the ordinary tensor product of abelian groups. The unit of $\text{CMon}^\otimes$ is the additive monoid of natural numbers, $\mathbb{N}$. There is also a commutative monoid $\mathbb{Z}$ of all integers, which deserves to be called the \emph{group completion} of $\mathbb{N}$.

As it turns out, $\mathbb{Z}$ is idempotent under $\otimes$, and therefore a solid semiring. A module over $\mathbb{Z}$ is precisely an \emph{abelian group}, and the forgetful functor $\text{Ab}\rightarrow\text{CMon}$ has left adjoint given by $$-\otimes\mathbb{Z}:\text{CMon}\rightarrow\text{Ab}.$$ Therefore, the free abelian group on a commutative monoid $M$ (also called the \emph{group completion} of $M$) is $M\otimes\mathbb{Z}$. \\

\noindent This is the story we hope to generalize. In SymMon, the unit is $\text{Fin}^\text{iso}$, which plays the role of $\mathbb{N}$. By the Barratt-Priddy Theorem \cite{BPQ}, $\vec{\mathbb{S}}$ is what we might call the `group-completion' of $\text{Fin}^\text{iso}$; it plays the role of $\mathbb{Z}$.

Since the inclusions $\text{Ab}_\infty^\wedge\rightarrow\text{CMon}_\infty^\wedge\rightarrow\text{SymMon}^\otimes$ are symmetric monoidal (Lemma \ref{EasyLemma} and Example \ref{CSRingExSp}), they preserve idempotent objects. Therefore $\vec{\mathbb{S}}$ is a solid semiring $\infty$-category (because $\mathbb{S}$ is an idempotent spectrum). So being an $\vec{\mathbb{S}}$-module is a \emph{property} of a symmetric monoidal $\infty$-category.

\begin{definition}
Call an $\vec{\mathbb{S}}$-module a \emph{group-complete} symmetric monoidal $\infty$-category, and the \emph{group-completion} of a symmetric monoidal $\infty$-category is given by the free (smashing localization) functor $-\otimes\vec{\mathbb{S}}:\text{SymMon}\rightarrow\text{Mod}_{\vec{\mathbb{S}}}$.
\end{definition}

\noindent Naively, we might expect a symmetric monoidal $\infty$-category to be group-complete precisely when all its objects are invertible up to equivalence (when it is \emph{grouplike}). However, this is not enough! In fact, a group-complete spectrum is one which is both grouplike and an $\infty$-groupoid; that is, a connective spectrum.

\begin{lemma}[\cite{Berman} 4.7]
Every group-complete symmetric monoidal $\infty$-category (that is, $\vec{\mathbb{S}}$-module) is an $\infty$-groupoid.
\end{lemma}

\begin{theorem}[\cite{Berman} 4.5]
\label{MainS}
A symmetric monoidal $\infty$-category is group-complete if and only if it is a connective spectrum. That is, $$\text{Mod}_{\vec{\mathbb{S}}}\cong\text{Ab}_\infty\cong\text{Sp}_{\geq 0}$$ are equivalent as full subcategories of SymMon.
\end{theorem}

\begin{proof}[Proof of lemma.]
Suppose $\mathcal{M}^{+}$ is an $\vec{\mathbb{S}}^{+,\times}$-module. Recall that the homotopy category construction $h:\text{SymMon}_\infty^\otimes\rightarrow\text{SymMon}^\otimes$ is symmetric monoidal (Example \ref{hCSRingEx}), so $h\mathcal{M}^{+}$ is an $h\vec{\mathbb{S}}$-module (as 1-categories). It suffices to show that $h\mathcal{M}$ is a groupoid.

Let 0 and 1 denote the additive and multiplicative units of $h\vec{\mathbb{S}}$, and $-1$ an additive inverse of 1, with a chosen isomorphism $\alpha:(-1)+1\xrightarrow{\sim}0$. An integer $n$ denotes (as an object of $h\vec{\mathbb{S}}$) $1^{+n}$ if $n$ is positive or $(-1)^{+|n|}$ if $n$ is negative.

The module structure on $h\mathcal{M}$ induces symmetric monoidal functors $h\vec{\mathbb{S}}^{+}\otimes h\mathcal{M}^{+}\rightarrow h\mathcal{M}^{+}$ and therefore $$m_{-}:h\vec{\mathbb{S}}^{+}\rightarrow\text{Hom}(h\mathcal{M}^{+},h\mathcal{M}^{+})$$ such that $m_1$ is the identity functor and $m_0$ is the constant functor sending all of $h\mathcal{M}$ to the unit $0\in h\mathcal{M}$. Denote $m_{-1}X$ by $-X$, and note that $\alpha:(-1)+1\xrightarrow{\sim}0$ induces a natural isomorphism $\alpha_X:(-X)+X\xrightarrow{\sim}0$.

Suppose $f:X\rightarrow Y$ in $h\mathcal{M}$. The inverse to $f$ will be $-f:-Y\rightarrow -X$ after an appropriate shift. Specifically, the inverse to $f+0:X+0\rightarrow Y+0$ is the composition $$Y+0\xrightarrow{\alpha_X^{-1}+Y}X+(-X)+Y\xrightarrow{X+(-f)+Y}X+(-Y)+Y\xrightarrow{X+\alpha_Y}X+0.$$ To prove it is inverse amounts to a diagram chase; the diagram proving that $f^{-1}\circ f=\text{id}$ is as follows (the right square commutes by naturality of $\alpha_X$): $$\xymatrix{
X+0\ar[rrr]^-{\alpha_X^{-1}+X}\ar[d]_-{f} &&&X+(-X)+X\ar[r]^-{X+\alpha_X}\ar[d]_-{X+(-f)+f} &X+0\ar[d]^-{\text{id}} \\
Y+0\ar[r]_-{\alpha_X^{-1}+Y} &X+(-X)+Y\ar[rr]_-{X+(-f)+Y} &&X+(-Y)+Y\ar[r]_-{X+\alpha_Y} &X+0.
}$$
\end{proof}

\begin{proof}[Proof of theorem.]
Any connective spectrum is an $\mathbb{S}$-module in $\text{Ab}_\infty^\wedge$. Since $\text{Ab}_\infty^\wedge\subseteq\text{SymMon}^\otimes$ is symmetric monoidal, any connective spectrum in $\text{SymMon}^\otimes$ is an $\vec{\mathbb{S}}$-module (thus group-complete).

Conversely, suppose $\mathcal{C}^\oast$ is group-complete. By the lemma, it is an $\infty$-groupoid. Since $\pi_0\mathcal{C}$ is a module over $\pi_0\vec{\mathbb{S}}\cong\mathbb{Z}$, $\mathcal{C}^\oast$ is also grouplike, completing the proof.
\end{proof}

\begin{corollary}
\label{CorGpoid}
Suppose $\mathcal{R}$ is a commutative ring $\infty$-category -- that is, a commutative semiring $\infty$-category which is grouplike (every object has an additive inverse up to equivalence). Then $\mathcal{R}$ is an $\infty$-groupoid.
\end{corollary}

\begin{proof}
If $\mathcal{R}$ is a commutative ring $\infty$-category, $\mathcal{R}^\text{iso}$ is a commutative ring $\infty$-groupoid (connective commutative ring spectrum), and by Lemma \ref{EasyLemma} the inclusion $\mathcal{R}^\text{iso}\rightarrow\mathcal{R}$ is a semiring functor\footnote{Specifically, the semiring functor $\mathcal{R}^\text{iso}\rightarrow\mathcal{R}$ is the counit of the symmetric monoidal adjunction \\$i:\text{CMon}_\infty^{\wedge}\rightleftarrows\text{SymMon}_\infty^\otimes:(-)^\text{iso}$ evaluated at $\mathcal{R}$.}. The composite semiring functor $$\vec{\mathbb{S}}\rightarrow\mathcal{R}^\text{iso}\rightarrow\mathcal{R},$$ exhibits $\mathcal{R}$ as an $\vec{\mathbb{S}}$-algebra, thus an $\vec{\mathbb{S}}$-module. So $\mathcal{R}$ is an $\infty$-groupoid.
\end{proof}

\subsection{Comparison to K-theory}\label{3.3.2}
\noindent By Proposition \ref{TupHup}, the fully faithful inclusion $\text{Ab}_\infty\subseteq\text{SymMon}$ has both a left adjoint (localization) $$-\otimes\vec{\mathbb{S}}:\text{SymMon}\rightarrow\text{Ab}_\infty$$ and a right adjoint (colocalization) $$\text{Hom}(\vec{\mathbb{S}},-):\text{SymMon}\rightarrow\text{Ab}_\infty.$$ These can be described in terms of algebraic K-theory and the Picard spectrum (or Picard $\infty$-groupoid), respectively:

\begin{proposition}
\label{cofree2}
If $\mathcal{C}^\oast$ is a symmetric monoidal $\infty$-category, $\text{Hom}(\vec{\mathbb{S}}^{+},\mathcal{C}^\oast)$ is the symmetric monoidal subcategory of $\mathcal{C}^\oast$ spanned by those objects which have inverses (up to equivalence) and morphisms which have inverses (also up to equivalence). This is sometimes known as the \emph{Picard spectrum} $\text{Pic}(\mathcal{C}^\oast)$.
\end{proposition}

\begin{proof}
Let $\mathcal{D}$ be the full subcategory of $\mathcal{C}$ spanned by invertible objects (invertible up to equivalence). Then $\text{Hom}(\vec{\mathbb{S}}^{+},\mathcal{D}^\oast)\rightarrow\text{Hom}(\vec{\mathbb{S}}^{+},\mathcal{C}^\oast)$ is an equivalence, since every object of $\vec{\mathbb{S}}$ is invertible (and therefore is sent by any symmetric monoidal functor to an invertible object of $\mathcal{C}$).

Now we know there is a subcategory inclusion $$i:\text{Hom}(\vec{\mathbb{S}}^{+},\mathcal{D}^{\text{iso},\oast})\subseteq\text{Hom}(\vec{\mathbb{S}}^{+},\mathcal{D}^\oast)$$ because $\mathcal{D}^\text{iso}$ is a subcategory of $\mathcal{D}$. (This follows from the definition of symmetric monoidal functors, HA 2.1.3.7.) Any object of $\text{Hom}(\vec{\mathbb{S}}^{+},\mathcal{D}^\oast)$ factors through $\mathcal{D}^\text{iso}$ (is in the image of $i$, up to equivalence) since every morphism in $\vec{\mathbb{S}}$ is invertible. Moreover, because $\text{Hom}(\vec{\mathbb{S}}^{+},\mathcal{D}^\oast)$ is an $\vec{\mathbb{S}}$-module (Proposition \ref{TupHup}), it is an $\infty$-groupoid. Thus any morphism $\phi:F\rightarrow G$ in $\text{Hom}(\vec{\mathbb{S}}^{+},\mathcal{D}^\oast)$ is in the image of $i$ (up to equivalence), since for any object $x\in\vec{\mathbb{S}}$, the map $\phi_x:F(X)\rightarrow G(X)$ must be an equivalence. Therefore $i$ is an equivalence.

But $\mathcal{D}^{\text{iso},\oast}$ is group-complete, and therefore an $\vec{\mathbb{S}}$-module. Since moreover $\vec{\mathbb{S}}$ is solid, $\text{Hom}(\vec{\mathbb{S}}^{+},\mathcal{D}^{\text{iso},\oast})\cong\mathcal{D}^\text{iso}$. Therefore $$\text{Hom}(\vec{\mathbb{S}}^{+},\mathcal{C}^\oast)\cong\vec{E}\cong\mathcal{D}^\text{iso},$$ as desired.
\end{proof}

\noindent The localization $\mathcal{C}^\oast\otimes\vec{\mathbb{S}}$ is not as straightforward to compute. The inclusion $\text{Top}\rightarrow\text{Cat}$ has (in addition to the right adjoint $-^\text{iso}$) a \emph{left} adjoint which sends $\mathcal{C}\in\text{Cat}$ to its \emph{classifying space} $|\mathcal{C}|$. Roughly, $|\mathcal{C}|$ is obtained from $\mathcal{C}$ by formally adjoining inverses to all morphisms.

Moreover, the inclusion $\text{Ab}_\infty\rightarrow\text{CMon}_\infty$ also has a left adjoint, given by $\infty$-group completion \cite{QuillenGpCp}. We will denote this group completion $K(-)$, for \emph{K-theory}.

Recall that the algebraic K-theory of a symmetric monoidal $\infty$-category $\mathcal{C}^\oast$ is $K(\mathcal{C}^{\text{iso},\oast})$.

\begin{proposition}
If $\mathcal{C}^\oast$ is symmetric monoidal, then $\vec{\mathbb{S}}\otimes\mathcal{C}^\oast\cong K(|\mathcal{C}^\oast|)$.
\end{proposition}

\begin{proof}
The subcategory inclusion $\text{Ab}_\infty\rightarrow\text{SymMon}_\infty$ factors $$\text{Ab}_\infty\rightarrow\text{CMon}_\infty\rightarrow\text{SymMon}_\infty,$$ so its left adjoint $\vec{\mathbb{S}}\otimes -$ factors as a composition of left adjoints $$\xymatrix{
\text{SymMon}_\infty\ar[r]^{|-|} &\text{CMon}_\infty\ar[r]^{K(-)} &\text{Ab}_\infty.
}$$
\end{proof}

\noindent For example, if $\mathcal{C}$ has an initial or terminal object, its classifying space is contractible. This classical fact can be restated algebraically as follows:

\begin{corollary}
\label{vanish}
$\vec{\mathbb{S}}\otimes\text{Fin}^{\text{inj}}\cong 0\cong\vec{\mathbb{S}}\otimes\text{Fin}^{\text{inj},\text{op}}$. As a result, $\vec{\mathbb{S}}\otimes\mathcal{C}^\oast=0$ for any $\mathcal{C}^\oast$ whose unit is initial or terminal, and in particular for $\mathcal{C}^\oast$ any of the solid semiring $\infty$-categories discussed before: Fin, $\text{Fin}_\ast$, $\text{Fin}^\text{inj}$, $\text{Span}(\text{Fin})$, or their opposites.
\end{corollary}

\noindent Here is a slightly stronger result:

\begin{corollary}
\label{VanishCor}
If $\mathcal{C}^\oast$ is a $\text{Fin}^\text{inj}$-module, then the only object which is invertible (up to equivalence) is the unit.
\end{corollary}

\begin{proof}
By Proposition \ref{cofree2}, $\text{Hom}(\vec{\mathbb{S}}^{+}\otimes\text{Fin}^{\text{inj},\amalg},\mathcal{C}^\oast)\cong\text{Hom}(\vec{\mathbb{S}}^{+},\mathcal{C}^\oast)$ is the maximal subgroupoid spanned by invertible objects. But it is also a module over $\vec{\mathbb{S}}\otimes\text{Fin}^\text{inj}\cong 0$, and therefore contractible. So every invertible object is equivalent to the unit.
\end{proof}

\subsection{Group completion of 1-categories}\label{3.3.3}
\noindent We end this section by discussing the group-completion of symmetric monoidal 1-categories, where vestiges of the sphere spectrum remain to wreak havoc.

Just as connective spectra were examples of symmetric monoidal $\infty$-categories, we have:

\begin{example}
If $M$ is a commutative monoid (for example, an abelian group), then $M$ is a (discrete) symmetric monoidal category.

If $R$ is a commutative ring (or a commutative semiring), then $R$ is a (discrete) semiring category.

If $\mathcal{C}\in\text{SymMon}$ is in the full subcategory $\text{Ab}\subseteq\text{SymMon}$, we might say abusively that $\mathcal{C}$ \emph{is} an abelian group.
\end{example}

\begin{warning}
Solid (connective) ring spectra are all solid as semiring $\infty$-categories. However, solid rings are not necessarily solid as semiring categories.

For example, $\mathbb{Z}/2$ is a solid ring, but the map of semiring categories $\mathbb{Z}/2\otimes\mathbb{Z}/2\rightarrow\mathbb{Z}/2$ is \emph{not} an equivalence ($\pi_1(H\mathbb{Z}/2\wedge H\mathbb{Z}/2)\cong\mathbb{Z}/2$, not $0$), and therefore $\mathbb{Z}/2$ is not solid as a semiring category.

The difference is that the inclusion $\text{Top}\rightarrow\text{Cat}_\infty$ has a right adjoint, so Lemma \ref{EasyLemma} applies to show that $\text{Ab}_\infty\rightarrow\text{SymMon}_\infty$ is a symmetric monoidal functor. But $\text{Set}\rightarrow\text{Cat}$ does \emph{not} have a right adjoint.

Thus Ab \emph{cannot} be recovered as a category of modules in SymMon, in contrast to Theorem \ref{MainS}.
\end{warning}

\noindent Very naively, we might expect the following analogue to Theorem \ref{MainS}: the full subcategory $\text{Ab}\subseteq\text{SymMon}$ can be recovered as modules over $\mathbb{Z}$. But the warning tells us this is not true. (More on this in Example \ref{Z}.) Instead, we should expect to recover \emph{2-abelian groups} as modules over the homotopy category $h\vec{\mathbb{S}}$.

\subsubsection{The semiring category $h\vec{\mathbb{S}}$}
\begin{definition}
A \emph{2-abelian group} is a symmetric monoidal category which is a groupoid and every object has an inverse (up to isomorphism).
\end{definition}

\begin{example}
The homotopy category $h\vec{\mathbb{S}}$ is a 2-abelian group which is not an abelian group. As a groupoid, $h\vec{\mathbb{S}}$ has objects labeled by $\mathbb{Z}\cong\pi_0\mathbb{S}$, and all automorphism groups are isomorphic to $\mathbb{Z}/2\cong\pi_1\mathbb{S}$. As a symmetric monoidal category, the symmetry automorphism $\sigma:1+1\xrightarrow{\sim} 1+1$ is the (unique) nontrivial automorphism of 2.
\end{example}

\begin{theorem}
\label{hSThm}
$h\vec{\mathbb{S}}$ is a solid semiring category, and a symmetric monoidal category is an $h\vec{\mathbb{S}}$-module if and only if it is a 2-abelian group.
\end{theorem}

\begin{proof}
Since $h:\text{SymMon}\rightarrow\text{SymMon}_1$ is symmetric monoidal, $h\vec{\mathbb{S}}$ is idempotent, therefore a solid semiring category. Moreover, by Theorem \ref{MainS}, $\mathcal{C}^\oast$ is an $\vec{\mathbb{S}}$-module if and only if $h\mathcal{C}^\oast$ is a 2-abelian group. The result follows by Lemma \ref{MainLemma}.
\end{proof}

\subsubsection{The semiring category $\vec{\mathbb{Z}}$}
\begin{example}
\label{Z}
We think of the ring $\mathbb{Z}$ as a commutative semiring $\infty$-category in one of two (equivalent) ways: either as the nerve of the discrete semiring 1-category $\mathbb{Z}$, or as the connective Eilenberg-Maclane spectrum $\vec{H\mathbb{Z}}$. To avoid confusion, we will use the notation $\vec{\mathbb{Z}}$.

Since $\vec{\mathbb{Z}}$ is an $\vec{\mathbb{S}}$-module, any module over $\vec{\mathbb{Z}}$ is a connective spectrum. That is, $$\text{Mod}_{\vec{\mathbb{Z}}}(\text{SymMon}_\infty^\otimes)\cong\text{Mod}_{H\mathbb{Z}}(\text{Ab}_\infty^\wedge),$$ which is the $\infty$-category of chain complexes of abelian groups, concentrated in nonnegative degrees \cite{Shipley}.

But $\vec{\mathbb{Z}}$ is not solid, since $H\mathbb{Z}\wedge H\mathbb{Z}$ is the integral dual Steenrod algebra, and \emph{not} $H\mathbb{Z}$. On the other hand, the map $\pi_\ast(H\mathbb{Z}\wedge H\mathbb{Z})\rightarrow\pi_\ast H\mathbb{Z}$ is an isomorphism for $\ast=0,1$ and only stops being an isomorphism at $\ast=2$ (where $\pi_2(H\mathbb{Z}\wedge H\mathbb{Z})\cong\mathbb{Z}$).

The upshot is this: that $\mathbb{Z}$ \emph{is} solid as a semiring 1-category, but not as a semiring $n$-category for any $n>1$.
\end{example}

\begin{theorem}
\label{Z2}
The semiring (1-)category $\mathbb{Z}$ is solid, and a symmetric monoidal (1-)category $\mathcal{C}^\oast$ admits the structure of a $\mathbb{Z}$-module if and only if $\mathcal{C}^\oast$ is a 2-abelian group and the symmetry isomorphisms $\sigma:X\oast Y\rightarrow Y\oast X$ in $\mathcal{C}$ are all identity morphisms when $X=Y$.
\end{theorem}

\begin{example}
If $A$ is any abelian group, the groupoid $BA$ (with one object and morphisms labeled by $A$) is a symmetric monoidal category and a $\mathbb{Z}$-module.
\end{example}

\begin{lemma}
Let $\mathbb{N}^{+}$ be the discrete symmetric monoidal (1-)category corresponding to the commutative monoid of nonnegative integers under addition. As symmetric monoidal (1-)categories, $\mathbb{Z}^{+}\cong\mathbb{N}^{+}\otimes h\vec{\mathbb{S}}^{+}$.
\end{lemma}

\begin{proof}[Proof of lemma.]
Denote $\mathcal{T}=\mathbb{N}\otimes h\vec{\mathbb{S}}$. There are semiring functors from $\mathbb{N}$ and $h\vec{\mathbb{S}}$ to $\mathcal{T}$ given by (for example) $$h\vec{\mathbb{S}}\cong\text{Fin}^{\text{iso}}\otimes h\vec{\mathbb{S}}\rightarrow\mathbb{N}\otimes h\vec{\mathbb{S}}\cong\mathcal{T}.$$ First, we will analyze the image of $h\mathbb{S}^{+}$ in $\mathcal{T}$.

Since the `isomorphism classes' functor $\pi_0:\text{Cat}\rightarrow\text{Set}$ is product-preserving, likewise $\pi_0:\text{SymMon}^\otimes\rightarrow\text{CMon}^\otimes$ is symmetric monoidal (by Lemma \ref{EasyLemma}). So $$\pi_0(\mathcal{T})\cong\mathbb{N}\otimes\mathbb{Z}\cong\mathbb{Z}.$$ Recall that every object of $h\vec{\mathbb{S}}$ has exactly one nontrivial automorphism. For $2\in h\vec{\mathbb{S}}$, this is the symmetry isomorphism of the symmetric monoidal structure $\sigma:1\oplus 1\rightarrow 1\oplus 1$. For any other $n\in h\vec{\mathbb{S}}$, it is $\sigma\oplus(n-2)$. However, since the symmetric monoidal functor $\mathbb{N}^{+}\rightarrow\mathcal{T}$ preserves the symmetry automorphism, all symmetry automorphisms in $\mathcal{T}$ are identities. Therefore, every morphism of $h\vec{\mathbb{S}}$ is sent to the identity in $\mathcal{T}$.

Let $\mathcal{T}^\prime$ be the (semiring) subcategory of $\mathcal{T}$ spanned by the image of $h\vec{\mathbb{S}}$. We have just seen that $\mathcal{T}^\prime\cong\mathbb{Z}$. In particular, there are semiring functors $h\vec{\mathbb{S}}\rightarrow\mathcal{T}^\prime$ and $\mathbb{N}\rightarrow\mathcal{T}^\prime$, and therefore an induced semiring functor from the coproduct $\mathcal{T}\rightarrow\mathcal{T}^\prime$, such that $\mathcal{T}\rightarrow\mathcal{T}^\prime\subseteq\mathcal{T}$ is equivalent to the identity on $\mathcal{T}$. So $\mathcal{T}\cong\mathcal{T}^\prime\cong\mathbb{Z}$, as desired.
\end{proof}

\begin{proof}[Proof of theorem.]
We already saw in Example \ref{Z} that $\mathbb{Z}$ is solid. Since it is itself a 2-abelian group (an $h\vec{\mathbb{S}}$-module), all $\mathbb{Z}$-modules are 2-abelian groups.

Say that a 2-abelian group is \emph{strictly commutative} if $\sigma:X\oast X\rightarrow X\oast X$ is the identity for all $X$. We will proceed using Lemma \ref{MainLemma}. So we need to show that $\text{Hom}(\mathbb{Z}^{+},\mathcal{C}^\oast)$ is always strictly commutative, and moreover $\text{Hom}(\mathbb{Z}^{+},\mathcal{C}^\oast)\rightarrow\mathcal{C}^\oast$ is an equivalence if $\mathcal{C}^\oast$ is strictly commutative.

If $\mathcal{C}^\oast$ is any symmetric monoidal category, and $F:\mathbb{Z}^{+}\rightarrow\mathcal{C}^\oast$ a symmetric monoidal functor, then by the definition of a symmetric monoidal functor, the following diagram commutes $$\xymatrix{
F(n)\oast F(n)\ar[r]^\sigma\ar[d]_{\epsilon} &F(n)\oast F(n)\ar[d]^{\epsilon} \\
F(n+n)\ar[r]_{=} &F(n+n).
}$$ Here the bottom morphism is the identity and $\epsilon$ is an isomorphism, so the symmetry isomorphism $\sigma$ in $\mathcal{C}^\oast$ is also the identity. Since this holds for every $n$, the symmetry map $\sigma:F\oast F\rightarrow F\oast F$ in the symmetric monoidal category $\text{Hom}(\mathbb{Z}^{+},\mathcal{C}^\oast)$ is also the identity. Therefore, $\text{Hom}(\mathbb{Z}^{+},\mathcal{C}^\oast)$ is strictly commutative.

On the other hand, suppose that $\mathcal{C}^\oast$ is a strictly commutative 2-abelian group. Note that being a strictly commutative 2-abelian group is invariant under equivalence of symmetric monoidal categories, so we may as well assume $\mathcal{C}^\oast$ is also permutative (or if not, replace it by a permutative category). Then $$\text{Hom}(\mathbb{Z}^{+},\mathcal{C}^\oast)\cong\text{Hom}(\mathbb{N}^{+},\text{Hom}(h\vec{\mathbb{S}}^{+},\mathcal{C}^\oast))\cong\text{Hom}(\mathbb{N}^{+},\mathcal{C}^\oast)$$ by the lemma and the fact that $\mathcal{C}^\oast$ is an $h\vec{\mathbb{S}}$-module (a 2-abelian group). So we need only show that the evaluation at 1 functor $\text{ev}_1:\text{Hom}(\mathbb{N}^{+},\mathcal{C}^\oast)\rightarrow\mathcal{C}$ is an equivalence of categories.

Consider the functor $F:\mathcal{C}\rightarrow\text{Hom}(\mathbb{N}^{+},\mathcal{C}^\oast)$ given by $$F(X)(n)=X^{\oast n}$$ $$F(f)(n)=f^{\oast n}:X^{\oast n}\rightarrow Y^{\oast n}.$$ Then $\text{ev}_1F$ is the identity, so $\text{ev}_1$ is full and essentially surjective. Moreover, $\text{ev}_1$ is faithful, and therefore an equivalence of categories, as desired.
\end{proof}

\chapter{Lawvere Theories}\label{4}\setcounter{subsection}{0}
\noindent This chapter is a systematic treatment of Lawvere theories from an $\infty$-categorical and algebraic point of view. There are two perspectives we could take: on one hand, Lawvere theories encode a general notion of \emph{algebraic theory}.

On the other hand, recall that in our discussion of Principle \ref{Princ}, we described two variants on the commutative algebra of categories: the presentable version (of HA and \cite{GGN}) and the finitely generated version (of Chapter \ref{3}). The second perspective on Lawvere theories is that they provide a context for comparing these two frameworks, via algebraic analogues of the Yoneda lemma and Day convolution.%crossreferences

We will use these ideas to tie up two loose ends from Chapter \ref{3}:
\begin{itemize}
\item Every presentable and closed symmetric monoidal $\infty$-category ($\mathcal{C}^\otimes$) is an example of a commutative semiring $\infty$-category ($\mathcal{C}^{\amalg,\otimes}$).
\item The Burnside category is a solid commutative semiring $\infty$-category whose modules are precisely additive $\infty$-categories.
\end{itemize}

\subsubsection{Classical Lawvere theories}
\noindent Lawvere theories arise originally out of an observation from algebra: we can understand everything about abelian groups just by understanding \emph{finitely generated free} abelian groups. That is, let $\text{Ab}_\text{fgf}$ denote the category of finitely generated free abelian groups (with group homomorphisms between them). Any abelian group $A$ induces a functor $$\text{Hom}(-,A):\text{Ab}_\text{fgf}^\text{op}\rightarrow\text{Set}$$ which preserves finite products. Conversely, any such functor gives rise to an abelian group; that is, there is an equivalence of categories $$\text{Ab}\cong\text{Fun}^\times(\text{Ab}_\text{fgf}^\text{op},\text{Set}).$$ The key here is that $\text{Ab}_\text{fgf}$ is generated under finite coproducts by a single object $\mathbb{Z}$. Therefore, in order to specify a product-preserving functor $F:\text{Ab}_\text{fgf}^\text{op}\rightarrow\text{Set}$, we need only specify a single object $X=F(\mathbb{Z})$, along with functions $X^{\times m}\rightarrow X^{\times n}$ corresponding to unitary operations, binary operations, etc. This process recovers an abelian group structure on $X$.

In principle, this process could be used to define the category Ab. The reader might object this is a circular definition, since we first need to define finitely generated free abelian groups. However, it is possible to define $\text{Ab}_\text{fgf}^\text{op}$ a priori via a purely combinatorial construction. Namely, it is the \emph{virtual Burnside category} Burn, obtained from the \emph{effective Burnside category} $\text{Burn}^\text{eff}=\text{Span}(\text{Fin})$ by group-completing the Hom sets (allowing formal differences of spans).

So we may define $\text{Ab}=\text{Fun}^\times(\text{Burn}^\text{op},\text{Set})$. The reader may make another reasonable objection: that the usual definition of an abelian group is much easier to handle. However, in the setting of $\infty$-categories, where point-set descriptions are unavailable, definitions via the Burnside category are much more appealing. They also interact very well with equivariant structures. For example:
\begin{itemize}
\item a connective spectrum is a product-preserving functor $\text{Burn}\rightarrow\text{Top}$;
\item a symmetric monoidal $\infty$-category is a product-preserving functor \\$\text{Burn}^\text{eff}\rightarrow\text{Cat}$;
\item a $G$-equivariant connective spectrum is a product-preserving functor $\text{Burn}_G\rightarrow\text{Top}$ (see Appendix \ref{A}).
\end{itemize}

\noindent It is no coincidence that we can rebuild Ab from $\text{Ab}_\text{fgf}$. This situation occurs in a wide range of algebraic examples. We call $\text{Ab}_\text{fgf}^\text{op}$ a \emph{Lawvere theory}, and $\text{Ab}$ its \emph{category of models}. See \cite{Lawvere} for the original reference on Lawvere theories.

\begin{warning}
Notice the variance change in $\text{Ab}_\text{fgf}^\text{op}$. This is an unfortunate historical convention which may cause some confusion if we are not careful. Throughout this chapter, we will consistently use $\mathcal{L}$ to refer to a Lawvere theory (such as $\text{Ab}_\text{fgf}^\text{op}$) and $\mathcal{F}$ to refer to its opposite (such as $\text{Ab}_\text{fgf}$). The letter $\mathcal{F}$ stands for `finitely generated free'.
\end{warning}

\noindent We will be generalizing Lawvere theories to the setting of $\infty$-categories. We are aware of two previous treatments of higher Lawvere theories: Cranch's thesis \cite{Cranch} and Gepner-Groth-Nikolaus \cite{GGN}. Our approach is more algebraic, but we rely heavily on HA and \cite{GGN}. \\

\noindent The general setup is as follows: Suppose that $\mathcal{C}$ is a presentable $\infty$-category. We hope to recover $\mathcal{C}$ from its finitely generated free objects. However, in order to do this, we should be able to identify which objects are `finitely generated free'. This requires any of the following equivalent information:
\begin{itemize}
\item a single object $\text{Free}(\ast)\in\mathcal{C}$, which makes $\mathcal{C}$ \emph{pointed};
\item a free (that is, left adjoint) functor $\text{Free}:\text{Top}\rightarrow\mathcal{C}$, which allows us to speak of \emph{free objects} in $\mathcal{C}$;
\item a forgetful (that is, right adjoint) functor $\text{Fgt}:\mathcal{C}\rightarrow\text{Top}$, which makes $\mathcal{C}$ \emph{concrete} (in the sense that objects have underlying spaces).
\end{itemize}
\noindent We call $\mathcal{C}$ with such data a \emph{pointed presentable $\infty$-category}. The $\infty$-category thereof is the undercategory $$\text{Pr}^\text{L}_\ast=(\text{Pr}^\text{L})_{\text{Top}/}.$$ For $\mathcal{C}\in\text{Pr}^\text{L}_\ast$, we say that an object is \emph{finitely generated free} if it is of the form $\ast^{\amalg n}$ for $n\geq 0$, or equivalently if it is in the image of the free functor applied to $\text{Fin}\subseteq\text{Top}$. Write $\mathcal{C}^\circ$ for the $\infty$-category of finitely generated free objects in $\mathcal{C}$.

Since the free functor preserves colimits, $\mathcal{C}^\circ$ is closed under finite coproducts, and moreover its objects are all generated by finite coproducts on a single object. Adopting the language of Chapter \ref{3}, $\mathcal{C}^\circ$ is a \emph{cyclic Fin-module}. Alternatively, $\mathcal{L}_\mathcal{C}=(\mathcal{C}^\circ)^\text{op}$ is a cyclic $\text{Fin}^\text{op}$-module, or a Lawvere theory\footnote{Cyclic $\text{Fin}^\text{op}$-modules are $\infty$-categories generated under finite products by a single object (see Theorem \ref{2T3}).}.

\begin{definition}
Given a cyclic Fin-module $\mathcal{F}^\amalg$, we say that an \emph{algebraic presheaf} on $\mathcal{F}$ is a product-preserving functor $\mathcal{F}^\text{op}\rightarrow\text{Top}$, writing $$\mathcal{P}_\amalg(\mathcal{F})=\text{Hom}(\mathcal{F}^{\text{op},\amalg},\text{Top}^\times).$$ Given a Lawvere theory $\mathcal{L}^\times$, we say that a \emph{model} of $\mathcal{L}$ is a product-preserving functor $\mathcal{L}\rightarrow\text{Top}$, writing $$\text{Mdl}(\mathcal{L})=\text{Hom}(\mathcal{L}^\times,\text{Top}^\times).$$ That is, models of $\mathcal{L}$ are the same as algebraic presheaves of $\mathcal{L}^\text{op}$.
\end{definition}

\noindent For many $\mathcal{C}$ (including Top, $\text{CMon}_\infty$, $\text{Ab}_\infty$, $\text{CRing}_\infty$, or modules over a commutative ring space), there is an equivalence $$\mathcal{C}\cong\mathcal{P}_\amalg(\mathcal{C}^\circ)\cong\text{Mdl}(\mathcal{L}_\mathcal{C}).$$ Our main result will be to formalize this equivalence as follows:

\begin{theorem*}[Theorem \ref{ThmLawvAdj}]
There is an adjunction $$\text{Mdl}:\text{Lawv}\rightleftarrows\text{Pr}^\text{L}_\ast:\mathcal{L}_{(-)},$$ whose left adjoint is fully faithful.
\end{theorem*}

\noindent In other words, Lawv is a \emph{colocalization} of $\text{Pr}^\text{L}_\ast$. In fact, Gepner, Groth, and Nikolaus have effectively described the full subcategory Lawv of $\text{Pr}^\text{L}_\ast$:

\begin{theorem*}[\cite{GGN} B.7]
A pointed presentable $\infty$-category $\mathcal{C}$ is representable by a Lawvere theory if and only if the forgetful functor $\mathcal{C}\rightarrow\text{Top}$ is conservative and preserves geometric realizations.
\end{theorem*}

\noindent Finally, the colocalization is compatible with a symmetric monoidal structure (tensor product) on Lawv. The tensor product of ordinary Lawvere theories is classical. The analogue for higher Lawvere theories is new, and there are at least three ways to construct it:
\begin{enumerate}
\item By Theorem \ref{2T4}, Lawvere theories are cyclic modules over $\text{Fin}^\text{op}$. They therefore inherit a relative tensor product $\otimes_{\text{Fin}^\text{op}}$ from $\text{SymMon}^\otimes$. In fact, since $\text{Fin}^\text{op}$ is solid, $\otimes_{\text{Fin}^\text{op}}=\otimes$.
\item For $K$ a set of simplicial sets, Lurie constructs a tensor product of categories which admit limits in $K$ (HA 4.8.1). If we choose $K$ to be finite discrete simplicial sets, then we have a tensor product of categories which admit finite products.
\item Since Lawv is a full subcategory of $\text{Pr}^\text{L}_\ast$, we may simply define its tensor product to be the tensor product in $\text{Pr}^\text{L}_\ast$.
\end{enumerate}

\noindent We will see that these three approaches are equivalent.

\subsubsection{Colocalizations}
\noindent We may view the colocalization above as an \emph{algebraic Yoneda lemma}, and in this sense Lawvere theories play a central role in the theory of symmetric monoidal $\infty$-categories. It will usually be easier to study cyclic Fin-modules $\mathcal{F}$ (which are much more natural for stating our results) as opposed to Lawvere theories $\mathcal{L}$. But keep in mind there is an equivalence of $\infty$-categories $$(-)^\text{op}:\text{CycMod}_\text{Fin}\xrightarrow{\sim}\text{Lawv}.$$ The ordinary Yoneda lemma very nearly establishes an adjunction $$\mathcal{P}:\text{Cat}\rightleftarrows\text{Pr}^\text{L}:\text{Fgt},$$ in the sense that left adjoint functors from an $\infty$-category of presheaves $\mathcal{P}(\mathcal{C})\rightarrow\mathcal{D}$ can be identified with ordinary functors $\mathcal{C}\rightarrow\mathcal{D}$; the identification is via composition with the Yoneda embedding $\mathcal{C}\subseteq\mathcal{P}(\mathcal{C})$.

The obstacle to producing a genuine adjunction is set-theoretic. The `right adjoint' really lands in \emph{large} $\infty$-categories, while the `left adjoint' takes \emph{small} $\infty$-categories as arguments. See the discussion after Theorem \ref{Yoneda3}.

However, these set-theoretic issues can be bypassed by introducing some size constraints. Let $\text{Pr}^\text{L}_\text{mk}$ denote the $\infty$-category of \emph{marked} presentable $\infty$-categories. An object is a presentable $\infty$-category with a \emph{small} set of marked objects. A morphism is a left adjoint functor which sends marked objects to marked objects. Now there is a \emph{literal} adjunction $$\mathcal{P}^\natural:\text{Cat}\rightleftarrows\text{Pr}^\text{L}_\text{mk}:\text{Mk},$$ where the right adjoint records the full subcategory spanned by the marked objects, and the left adjoint sends $\mathcal{C}$ to its $\infty$-category of presheaves $\mathcal{P}(\mathcal{C})$, where the marked objects are the representable presheaves. The Yoneda lemma can be reformulated as follows:

\begin{lemma*}[Yoneda lemma]
There is an adjunction $\mathcal{P}^\natural:\text{Cat}\rightleftarrows\text{Pr}^\text{L}_\text{mk}:\text{Mk}$, and the left adjoint is fully faithful.
\end{lemma*}

\noindent In this sense, the Yoneda lemma describes Cat as a colocalization of $\text{Pr}^\text{L}_\text{mk}$. This is useful because $\text{Pr}^\text{L}_\text{mk}$ is in many ways easier to work with than Cat. In fact, in a sense presentable $\infty$-categories were studied prior to $\infty$-categories themselves, in the guise of model categories.

Generally speaking, we might attempt to formulate different flavors of higher category theory as colocalizations of $\text{Pr}^\text{L}$ or its relatives; these colocalizations encode versions of the Yoneda lemma. For example, the following conjecture seems to be known to experts in enriched higher category theory, although this author is not sure whether it has been formulated before (and certainly not in this language). It is hinted at in \cite{GepHaug} (see especially 7.4.13 for a description of the right adjoint in the conjectured colocalization).

\begin{conjecture}[Enriched Yoneda lemma]
\label{Conj0}
If $\mathcal{V}^\otimes$ is a closed symmetric monoidal presentable $\infty$-category (that is, a commutative algebra in $\text{Pr}^{\text{L},\otimes}$), $\text{Cat}^\mathcal{V}$ is the $\infty$-category of $\mathcal{V}$-enriched $\infty$-categories, and $\text{Mod}^\text{L}_{\mathcal{V},\text{mk}}$ is the $\infty$-category of $\mathcal{V}$-modules (in $\text{Pr}^\text{L}$) with some objects marked, then there is an adjunction $$\mathcal{P}^\mathcal{V}:\text{Cat}^\mathcal{V}\rightleftarrows\text{Mod}^\text{L}_{\mathcal{V},\text{mk}}:\text{Mk}.$$ The left adjoint is fully faithful and takes the form $\mathcal{P}^\mathcal{V}(\mathcal{C})=\text{Fun}^\mathcal{V}(\mathcal{C}^\text{op},\mathcal{V})$.
\end{conjecture}

\noindent Such a conjecture would provide a framework to study enriched higher category theory without prolific book-keeping and technicalities. Instead, we would work in the familiar setting of the commutative algebra of $\text{Pr}^{\text{L},\otimes}$.

Our main theorem can be viewed as an algebraic (rather than enriched) analogue of the Yoneda lemma. We repeat it here in the language of cyclic Fin-modules, although it has already been stated above in terms of Lawvere theories.

\begin{theorem*}[Algebraic Yoneda lemma, Theorem \ref{ThmLawvAdj}]
There is an adjunction $$\mathcal{P}_\amalg:\text{CycMod}_\text{Fin}\rightleftarrows\text{Pr}^\text{L}_\ast:(-)^\circ,$$ and the left adjoint is fully faithful.
\end{theorem*}

\noindent Most of the ingredients of the theorem are known already in Lurie's books HTT and HA; we will review those in \S\ref{4.1.1} (which will also clarify the sense in which this is a \emph{Yoneda lemma}).

\begin{remark}
There should be a similar colocalization replacing `cyclic Fin-modules' (Lawvere theories) by `Fin-modules with marked generators' (colored Lawvere theories), and `pointed presentable $\infty$-categories' with `marked presentable $\infty$-categories'. This would more closely resemble the forms of the regular and enriched Yoneda lemmas described above, but we won't work at this level of generality.
\end{remark}

\subsubsection{Symmetric monoidal colocalizations}
\noindent In fact, the colocalization described above is even \emph{symmetric monoidal}; that is, the left adjoint $\mathcal{P}_\amalg:\text{CycMod}_\text{Fin}\rightarrow\text{Pr}^\text{L}_\ast$ lifts to a symmetric monoidal functor.

If a colocalization of $\text{Pr}^\text{L}_\ast$ (or one of its relatives) encodes a sort of Yoneda lemma, a symmetric monoidal colocalization encodes in addition a theory of Day convolution.

That is, for some Lawvere theories $\mathcal{L}$ (or cyclic Fin-modules $\mathcal{F}=\mathcal{L}^\text{op}$), we can place a symmetric monoidal structure (called Day convolution) on $\text{Mdl}(\mathcal{L})=\mathcal{P}_\amalg(\mathcal{F})$. We will use this to prove that Burn is the solid semiring $\infty$-category classifying additive $\infty$-categories, and more generally to provide comparisons with enriched higher category theory.

\begin{question}
In general, a symmetric monoidal colocalization of $\text{Pr}^\text{L}_\ast$ (or $\text{Pr}^\text{L}_\text{mk}$) which is also presentable is a full subcategory closed under small colimits and tensor products, and generated by a small set of objects. We have said that these encode Yoneda lemmas and Day convolution, and that they include Lawv and Cat. Is it possible to classify all such colocalizations?
\end{question}

\subsubsection{Organization}
In \S\ref{4.1}, we review the comparison between cocartesian monoidal and presentable $\infty$-categories. The results summarized in \S\ref{4.1.1} are due to Lurie (HA), sometimes in new language. The new observation is that his tensor product of cocartesian monoidal $\infty$-categories can be extended to the tensor product of symmetric monoidal $\infty$-categories (which we considered in Chapter \ref{3}).

Due to this comparison between his symmetric monoidal structure and ours, we can use his results (in \S\ref{4.1.2}) to describe $\mathcal{C}^\oast\otimes\text{Fin}^\amalg$, where $\mathcal{C}^\oast$ is an arbitrary symmetric monoidal $\infty$-category. We also show in \S\ref{4.1.3} that any closed symmetric monoidal presentable $\infty$-category is an example of a commutative semiring $\infty$-category.

In \S\ref{4.2}, we discuss the theory of cyclic modules over a commutative semiring $\infty$-category (in \S\ref{4.2.1}), which we then use to study Lawvere theories (in \S\ref{4.2.2}). Using this language, we establish our main result, that Lawv is a symmetric monoidal colocalization of $\text{Pr}^\text{L}_\ast$.

In \S\ref{4.2.3} we show that cyclic modules over the effective Burnside 2-category $\text{Burn}^\text{eff}$ are in bijection with $\mathbb{E}_1$-semiring spaces, and as a consequence we prove that Burn-modules are additive $\infty$-categories; this is the last part of the main theorem from Chapter \ref{3}.

Finally, in \S\ref{4.3}, we discuss some open problems, particularly related to the following two questions. In each case, we present some evidence in favor of our conjectures.
\begin{enumerate}
\item (\S\ref{4.3.1}) Associated to an operad is a Lawvere theory which is trivializable over the Burnside category. Is the converse also true? If so, operads would have a compelling description in terms of categorified algebraic geometry.
\item (\S\ref{4.3.2}) Can a symmetric monoidal $\infty$-category $\mathcal{C}^\oast$ (maybe sufficiently nice) be recovered via descent from $\mathcal{C}\otimes\text{Fin}$, $\mathcal{C}\otimes\text{Fin}^\text{op}$, and $\mathcal{C}\otimes\vec{\mathbb{S}}$? The first two objects are often related to span constructions, while the third is related to algebraic K-theory (as in Chapter \ref{3}). This is really a question about the algebraic geometry of the commutative semiring $\infty$-category $\text{Fin}^\text{iso}$, and it is related to question (1).
\end{enumerate}

\noindent Taken together, these are questions about the foundations of the nascent subject of \emph{categorified algebraic geometry}.

\section{Presentable semiring \texorpdfstring{$\infty$-}{infinity }categories}\label{4.1}
\noindent If $K$ is a set of simplicial sets, let $\text{Cat}(K)$ denote the subcategory of Cat spanned by those small $\infty$-categories which admit $K$-indexed colimits, and functors which preserve $K$-indexed colimits. Let $\text{Fun}^K(\mathcal{C},\mathcal{D})$ denote the full subcategory of $\text{Fun}(\mathcal{C},\mathcal{D})$ on functors which preserve $K$-indexed colimits.

Lurie (HA 4.8.1.6) has constructed a closed symmetric monoidal structure on $\text{Cat}(K)$ (or $\widehat{\text{Cat}}(K)$ for a large set $K$), such that the right adjoint to $-\otimes\mathcal{C}$ is $\text{Fun}^K(\mathcal{C},-)$.

\begin{example}
\label{ExSMPrL}
Choose $K$ to be the (large) set of all small simplicial sets. Then $\text{Pr}^\text{L}$ is a symmetric monoidal full subcategory of $\widehat{\text{Cat}}(K)$. Thus, there is a symmetric monoidal operation on $\text{Pr}^\text{L}$ (HA 4.8.1.15), the same as that considered in \S\ref{1.1.4}.
\end{example}

\begin{example}
\label{ExCfFin}
By Theorem \ref{2T3}, $\text{CocartMon}\cong\text{Mod}_\text{Fin}$, and therefore inherits a closed symmetric monoidal structure $\otimes_\text{Fin}$ from SymMon. Since Fin is solid, $\otimes_\text{Fin}$ agrees with $\otimes$ (the usual tensor product of symmetric monoidal $\infty$-categories).

Choose $K$ to be the set of finite, discrete simplicial sets, so that $\text{Cat}(K)\cong\text{CocartMon}$. This equivalence is compatible with the symmetric monoidal structures $$\text{Cat}(K)^{\otimes_\text{Lurie}}\cong\text{CocartMon}^{\otimes_\text{Fin}},$$ because in either case $-\otimes\mathcal{C}$ is left adjoint to $$\text{Fun}^\amalg(\mathcal{C},-)\cong\text{Hom}(\mathcal{C}^\amalg,-^\amalg).$$
\end{example}

\noindent In summary, if $\mathcal{C}$ and $\mathcal{D}$ are cocartesian monoidal $\infty$-categories, then their tensor product (as symmetric monoidal $\infty$-categories) coincides with their tensor product in the sense of HA 4.8.1.

In \S\ref{4.1.1}, we will exploit this observation along with Example \ref{ExSMPrL} to provide a comparison between symmetric monoidal $\infty$-categories and presentable $\infty$-categories. This is a sort of \emph{(co)cartesian monoidal Yoneda lemma}, and is mostly a reformulation of results in HTT and HA.

Then in \S\ref{4.1.2}, we describe $\text{Fin}^\amalg\otimes\mathcal{C}^\oast$ (where $\mathcal{C}^\oast$ is an arbitrary symmetric monoidal $\infty$-category), which is the \emph{free cocartesian monoidal $\infty$-category generated by $\mathcal{C}^\oast$} (by Proposition \ref{TupHup}).

Finally, in \S\ref{4.1.3}, we prove Proposition \ref{PropPrRing}. That is, if $\mathcal{C}^\otimes$ is presentable and closed symmetric monoidal, then $\mathcal{C}^{\amalg,\otimes}$ is a commutative semiring $\infty$-category. In Chapter \ref{3}, we were only able to construct examples of commutative semiring $\infty$-categories which were solid. This result provides a new source of examples.

\subsection{The cartesian monoidal Yoneda lemma}\label{4.1.1}
\noindent The following propositions might together be called an \emph{algebraic Yoneda lemma}.

\begin{proposition}
\label{SMYoneda1}
The product-preserving presheaf functor is symmetric monoidal\footnote{By $\mathcal{C}^{\text{op},\amalg}$, we mean the coproduct in $\mathcal{C}$, which is a product in $\mathcal{C}^\text{op}$.} $$\text{Hom}(-^{\text{op},\amalg},\text{Top}^\times):\text{CocartMon}^\otimes\rightarrow\text{Pr}^{\text{L},\otimes}.$$ We write $\mathcal{P}_\amalg(\mathcal{C})=\text{Hom}(\mathcal{C}^{\text{op},\amalg},\text{Top}^\times)$.
\end{proposition}

\begin{remark}[cf. HTT 5.2.6.3]
If $F:\mathcal{C}^\amalg\rightarrow\mathcal{D}^\amalg$ is a cocartesian monoidal functor, then $\mathcal{P}_\amalg(F)$ has right adjoint given by precomposition with $F$, $$\text{Hom}(\mathcal{D}^{\text{op},\amalg},\text{Top}^\times)\xrightarrow{-\circ F}\text{Hom}(\mathcal{C}^{\text{op},\amalg},\text{Top}^\times).$$
\end{remark}

\begin{proof}
See HA 4.8.1.8 and Remark \ref{ExCfFin}.
\end{proof}

\begin{example}[Day convolution]
If $\mathcal{R}^{\amalg,\otimes}$ is a commutative semiring $\infty$-category with cocartesian monoidal additive structure, then $\mathcal{P}(\mathcal{R})$ inherits a closed symmetric monoidal structure $\oast$, called \emph{Day convolution}, because $\mathcal{P}_\amalg(-)$ sends commutative algebras to commutative algebras. See HA 4.8.1.13.
\end{example}

\begin{proposition}
\label{SMYoneda2}
If $\mathcal{C}^\amalg$ is a cocartesian monoidal $\infty$-category, then there are full subcategory inclusions $$\mathcal{C}\subseteq\mathcal{P}_\amalg(\mathcal{C})\subseteq\mathcal{P}(\mathcal{C}),$$ and the composition is the usual Yoneda embedding.
\end{proposition}

\begin{proof}
By HTT 5.1.3.2, representable functors $\mathcal{C}^\text{op}\rightarrow\text{Top}$ preserve limits, and in particular finite products. Thus if $\mathcal{C}^\amalg$ is cocartesian monoidal, the Yoneda embedding factors through $$\mathcal{C}\rightarrow\text{Fun}^\times(\mathcal{C}^\text{op},\text{Top}).$$
\end{proof}

\begin{example}
\label{ExSpanBurn}
Note that $\text{Span}(\text{Fin})^\amalg$ is cocartesian monoidal and equivalent to its opposite. By corollary \ref{CofreeSemiadd}, $$\text{Hom}(\text{Span}(\text{Fin})^\amalg,\text{Top}^\times)\cong\text{CAlg}(\text{CocCoalg}(\text{Top}^\times)^\times)\cong\text{CAlg}(\text{Top}^\times)\cong\text{CMon}_\infty.$$ Therefore, $\text{Span}(\text{Fin})$ is a full subcategory of $\text{CMon}_\infty$; namely, the subcategory of finitely generated free $\mathbb{E}_\infty$-spaces, or the \emph{effective Burnside 2-category} $\text{Burn}^\text{eff}$.
\end{example}

\begin{example}
Similarly, the 1-category $h\text{Span}(\text{Fin})$ is the full subcategory of CMon spanned by finitely generated free commutative monoids; that is, commutative monoids $\mathbb{N}^k$, as $k\geq 0$ varies. This is the \emph{effective Burnside 1-category} $h\text{Burn}^\text{eff}$.
\end{example}

\begin{corollary}[Enrichment of modules]
\label{CorEnrMod}
Suppose $\mathcal{R}^{\amalg,\otimes}$ is a commutative semiring $\infty$-category whose additive structure is cocartesian monoidal (that is, $\mathcal{R}$ is a commutative Fin-algebra). Then any $\mathcal{R}$-module is naturally enriched in $\mathcal{P}_\amalg(\mathcal{R})$.
\end{corollary}

\begin{proof}
Suppose $\mathcal{M}$ is an $\mathcal{R}$-module. Then $\mathcal{M}$ is a full subcategory of $\mathcal{P}_\amalg(\mathcal{M})$, which is a $\mathcal{P}_\amalg(\mathcal{R})$-module, and therefore enriched in $\mathcal{P}_\amalg(\mathcal{R})$ by \cite{GepHaug} 7.4.13. The full subcategory $\mathcal{M}$ inherits the enrichment.
\end{proof}

\begin{example}
\label{EnrAddEx}
Any semiadditive $\infty$-category is a $\text{Burn}^\text{eff}$-module, and therefore enriched in $\mathcal{P}_\amalg(\text{Burn}^\text{eff})\cong\text{CMon}_\infty$. This is the higher categorical analogue of the classical observation that a semiadditive category is enriched in commutative monoids.
\end{example}

\begin{example}
If $G$ is a finite group, $\mathcal{P}_\amalg(\text{Fin}_G)$ is the $\infty$-category of genuine equivariant right $G$-spaces (Definition \ref{ElmenDef}). Therefore, any $\text{Fin}_G$-module is naturally enriched in right $G$-spaces.

On the other hand, $\mathcal{P}_\amalg(\text{Burn}_G)$ is the $\infty$-category of equivariant connective right $G$-spectra (Definition \ref{GMDef}). Therefore, any $\text{Burn}_G$-module is naturally enriched in right $G$-spectra.
\end{example}

\begin{proposition}[HTT 5.3.6.10]
\label{SMYoneda3}
If $\mathcal{C}^\amalg$ is cocartesian monoidal, and $\mathcal{D}$ is presentable, then precomposition along the Yoneda embedding of Proposition \ref{SMYoneda2} induces an equivalence $$\text{Fun}^\text{L}(\mathcal{P}_\amalg(\mathcal{C}),\mathcal{D})\xrightarrow{\sim}\text{Hom}(\mathcal{C}^\amalg,\mathcal{D}^\amalg).$$
\end{proposition}

\subsection{Free cocartesian monoidal \texorpdfstring{$\infty$-}{infinity }categories}\label{4.1.2}
\noindent By Proposition \ref{TupHup} and since $\text{CocartMon}\cong\text{Mod}_\text{Fin}$, the forgetful functor $$\text{CocartMon}\rightarrow\text{SymMon}$$ has both a left adjoint (free functor), given by $-\otimes\text{Fin}^\amalg$, and a right adjoint (cofree functor), given by $\text{Hom}(\text{Fin}^\amalg,-)$.

Each of these adjoints can be described more concretely. The cofree construction amounts to passing to commutative algebra objects, $$\text{Hom}(\text{Fin}^\amalg,\mathcal{C}^\oast)\cong\text{CAlg}(\mathcal{C}^\oast),$$ as in Example \ref{ExFormulas}.

The free construction can be analyzed by means of Proposition \ref{SMYoneda2}: $$\mathcal{C}^\oast\otimes\text{Fin}^\amalg\subseteq\text{Hom}(\mathcal{C}^{\text{op},\oast}\otimes\text{Fin}^{\text{op},\amalg},\text{Top}^\times).$$ Since $\text{Top}^\times$ is cartesian monoidal, $\text{Hom}(\text{Fin}^{\text{op},\amalg},\text{Top}^\times)\cong\text{Top}^\times$. Thus, by the tensor-Hom adjunction, there is a full subcategory inclusion $$\mathcal{C}^\oast\otimes\text{Fin}^\amalg\subseteq\text{Hom}(\mathcal{C}^{\text{op},\oast},\text{Top}^\times).$$ In fact, we can describe the full subcategory $\mathcal{C}\otimes\text{Fin}$ as follows:

\begin{theorem}[\cite{Berman} 3.10]
\label{SMYonedaT}
If $\mathcal{C}^\oast$ is a symmetric monoidal $\infty$-category, the forgetful functor $\text{Fgt}_\mathcal{C}:\text{Hom}(\mathcal{C}^{\text{op},\oast},\text{Top}^\times)\rightarrow\text{Fun}(\mathcal{C}^\text{op},\text{Top})$ has a left adjoint $\text{Free}_\mathcal{C}$. If $X\in\mathcal{C}$, let $\underline{X}$ denote the representable functor $\text{Map}_\mathcal{C}(-,X)$.

Then the free cocartesian monoidal $\infty$-category $\mathcal{C}^\oast\otimes\text{Fin}^\amalg$ is equivalent to the full subcategory of $\text{Hom}(\mathcal{C}^{\text{op},\oast},\text{Top}^\times)$ spanned by the objects $\text{Free}_\mathcal{C}(\underline{X})$, $X\in\mathcal{C}$, along with its cocartesian monoidal structure.
\end{theorem}

\begin{remark}
Dually (by considering $\mathcal{C}^\text{op}$ instead of $\mathcal{C}$), $\mathcal{C}^\oast\otimes\text{Fin}^{\text{op},\amalg}$ is the opposite of the corresponding full subcategory of $\text{Hom}(\mathcal{C}^\oast,\text{Top}^\times)$.
\end{remark}

\begin{proof}
The proof is somewhat technical and requires Lemma \ref{LemmaES} which, while straightforward, we would rather delay to the next section for reasons of exposition. Therefore, see \cite{Berman} Theorem 3.10 for a proof.
\end{proof}

\noindent By Theorem \ref{SMYonedaT}, computation of the free functor $-\otimes\text{Fin}^\amalg:\text{SymMon}\rightarrow\text{CocartMon}$ reduces to computation of another free functor $\text{Fun}(\mathcal{C}^\text{op},\text{Top})\rightarrow\text{Hom}(\mathcal{C}^{\text{op},\oast},\text{Top}^\times)$. In general, this is still a difficult problem, but in the event that $\mathcal{C}^\oast$ is \emph{cyclic}, we can say something more concrete (Proposition \ref{FreeLawv}), and in the further event that $\mathcal{C}^\oast$ is the symmetric monoidal envelope of an operad, we can be very concrete indeed. We will return to this in \S\ref{4.3.1}.

\subsection{Presentable semiring \texorpdfstring{$\infty$-}{infinity }categories}\label{4.1.3}
\noindent Taken together, the three propositions of \S\ref{4.1.1} are suggestive of a symmetric monoidal adjunction $$\mathcal{P}_\amalg:\text{CocartMon}^\otimes\rightleftarrows\text{Pr}^{\text{L},\otimes}:\text{Fgt},$$ such that the unit of the adjunction is fully faithful. However, such an adjunction does not quite exist, due to set-theoretic reasons (see the introduction).

Nonetheless, $\mathcal{P}_\amalg$ and Fgt behave much like an adjunction. For example, any functor with a symmetric monoidal left adjoint is lax symmetric monoidal --- and indeed, Fgt is lax symmetric monoidal:

\begin{lemma}
\label{LemPrRing}
The forgetful functor $\text{Pr}^\text{L}\rightarrow\widehat{\text{SymMon}}$ (which sends a presentable $\infty$-category $\mathcal{C}$ to the cocartesian monoidal $\infty$-category $\mathcal{C}^\amalg$) lifts to a lax symmetric monoidal functor, from $\text{Pr}^{\text{L},\otimes}$ to $\widehat{\text{SymMon}}^\otimes$.\footnote{Recall $\widehat{\text{SymMon}}$ refers to \emph{large} symmetric monoidal $\infty$-categories.}
\end{lemma}

\begin{proof}
Let $K$ denote the set of all small simplicial sets, so that an object of $\text{Cat}(K)$ is an $\infty$-category which admits all small colimits. By HTT 5.3.6.10, there is an adjunction $$\mathcal{P}_\amalg:\widehat{\text{CocartMon}}\rightleftarrows\widehat{\text{Cat}}(K):\text{Fgt},$$ which is even symmetric monoidal by HA 4.8.1.8. Moreover, the inclusion $\text{Pr}^\text{L}\subseteq\widehat{\text{Cat}}(K)$ is symmetric monoidal by HA 4.8.1.15. Thus the composite is lax symmetric monoidal: $\text{Pr}^{\text{L},\otimes}\rightarrow\widehat{\text{SymMon}}^\otimes$.
\end{proof}

\noindent We may now conclude Proposition \ref{PropPrRing} from Chapter \ref{3}.

\begin{corollary*}[Proposition \ref{PropPrRing}]
If $\mathcal{C}^\oast$ is a presentable and closed symmetric monoidal $\infty$-category, then $\mathcal{C}^{\amalg,\oast}$ is a (large) commutative semiring $\infty$-category.
\end{corollary*}

\begin{example}
The following are commutative semiring $\infty$-categories: Set, CMon, Ab, $\text{Cat}_1$, Top, $\text{CMon}_\infty$, $\text{Ab}_\infty$, Sp, and Cat.
\end{example}

\section{Cyclic modules}\label{4.2}
\noindent In this section, we will set up the theory of cyclic modules over a semiring $\infty$-category. A cyclic $\mathcal{R}$-module is an $\mathcal{R}$-module $\mathcal{M}$ equipped with an essentially surjective morphism of $\mathcal{R}$-modules $\mathcal{R}\rightarrow\mathcal{M}$. This amounts to a choice of object $X\in\mathcal{M}$ such that every object of $\mathcal{M}$ is of the form $R\otimes X$ for some $R\in\mathcal{R}$.

\begin{warning}
Classically, if $R$ is a commutative ring, cyclic $R$-modules are all quotients of $R$, and therefore correspond to ideals of $R$. In this way, every cyclic $R$-module even has a ring structure!

Analogous statements are \emph{not} true of cyclic $\mathcal{R}$-modules when $\mathcal{R}$ is a commutative semiring $\infty$-category. For example, cyclic $\mathcal{R}$-modules need not have any semiring structure. Consider the following motivating example:

If $\mathcal{R}=\text{Burn}$ is the Burnside $\infty$-category, then cyclic Burn-modules correspond to associative (or $\mathbb{E}_1$) ring spaces. We will prove this in \S\ref{4.2.3}.

Thus we may view cyclic $\mathcal{R}$-modules as being a sort of generalization of noncommutative algebra.
\end{warning}

\noindent Cyclic $\mathcal{R}$-modules also encode information related to \emph{algebraic theories} relative to $\mathcal{R}$. For example, cyclic $\text{Fin}^\text{iso}$-modules are PROPs, and cyclic $\text{Fin}^\text{op}$-modules are Lawvere theories. If $G$ is a finite group, then cyclic $\text{Fin}_G^\text{op}$-modules can be regarded as \emph{equivariant Lawvere theories}, and appear to play a central role in equivariant homotopy theory (see Appendix \ref{A}).

\subsubsection{Organization}
\noindent In \S\ref{4.2.1}, we set up the general theory of cyclic modules. Then in \S\ref{4.2.2}, we introduce Lawvere theories as cyclic $\text{Fin}^\text{op}$-modules, proving our main result that Lawv is a symmetric monoidal colocalization of $\text{Pr}^\text{L}_\ast$. Results of Gepner, Groth, and Nikolaus \cite{GGN} allow us to recognize and construct many examples of Lawvere theories.

For example, if $R$ is a semiring space, there is a Lawvere theory $\text{Burn}_R$ that models $R$-modules. In \S\ref{4.2.3}, we show that these are the only semiadditive Lawvere theories, providing an equivalence $$\text{CycMod}_{\text{Burn}^\text{eff}}\cong\mathbb{E}_1\text{SRing}(\text{Top}).$$ We conclude that additive $\infty$-categories are modules over the Burnside $\infty$-category, which completes the proof of the main result of Chapter \ref{3}.

\subsection{Cyclic modules}\label{4.2.1}
\begin{definition}
Let $\mathcal{R}$ be a commutative semiring $\infty$-category. A \emph{pointed $\mathcal{R}$-module} is an $\mathcal{R}$-module $\mathcal{M}$ along with a choice of distinguished object $X\in\mathcal{R}$; or, equivalently, an $\mathcal{R}$-module $\mathcal{M}$ together with a distinguished map of $\mathcal{R}$-modules $\mathcal{R}\rightarrow\mathcal{M}$. We denote by $$\text{Mod}_{\mathcal{R},\ast}=(\text{Mod}_\mathcal{R})_{\mathcal{R}/}$$ the $\infty$-category thereof.
\end{definition}

\begin{remark}
$\text{Mod}_{\mathcal{R},\ast}$ can be identified with $\mathbb{E}_0$-algebras in $\text{Mod}_\mathcal{R}$ (HA 2.1.3.10) and therefore inherits a symmetric monoidal structure $\otimes_\mathcal{R}$ from $\text{Mod}_\mathcal{R}$.
\end{remark}

\begin{definition}
A \emph{cyclic $\mathcal{R}$-module} is a pointed $\mathcal{R}$-module such that the distinguished map $\mathcal{R}\rightarrow\mathcal{M}$ is essentially surjective. We denote by $\text{CycMod}_\mathcal{R}$ the full subcategory of $\text{Mod}_{\mathcal{R},\ast}$ spanned by cyclic modules.
\end{definition}

\begin{example}
\label{PROPExample}
A cyclic $\text{Fin}^\text{iso}$-module (or just cyclic symmetric monoidal $\infty$-category) is a symmetric monoidal $\infty$-category $\mathcal{T}^\oast$ with a distinguished object $X$, such that every object of $\mathcal{T}$ is equivalent to $X^{\oast n}$ for some nonnegative integer $n$. Classically, these are called \emph{PROPs} (product and permutation categories) \cite{PROP}.

Given a PROP $\mathcal{T}^\oast$, a \emph{model} of $\mathcal{T}^\oast$ valued in $\mathcal{C}^\oast$ (often $\mathcal{C}^\oast=\text{Set}^\times$, $\text{Ab}^\otimes$, $\text{Top}^\times$, or $\text{Sp}^\wedge$) is a symmetric monoidal functor $\mathcal{T}^\oast\rightarrow\mathcal{C}^\oast$, and the $\infty$-category thereof is $$\text{Mdl}(\mathcal{T}^\oast,\mathcal{C}^\oast)=\text{Hom}(\mathcal{T}^\oast,\mathcal{C}^\oast).$$

When $\mathcal{C}=\text{Top}^\times$, we write just $\text{Mdl}(\mathcal{T}^\oast)=\text{Hom}(\mathcal{T}^\oast,\text{Top}^\times)$.

We can think of the model of $\mathcal{T}^\oast$ as picking out an object of $\mathcal{C}$ (the image of $X$) along with various maps $X^{\oast m}\rightarrow X^{\oast n}$ corresponding to the maps of $\mathcal{T}$. In this way, PROPs model different kinds of algebraic structure (groups, cogroups, abelian groups, rings, commutative rings, Hopf algebras, Lie algebras, etc.).
\end{example}

\begin{example}
Many of the solid semiring $\infty$-categories of Chapter \ref{3} are PROPs:
\begin{itemize}
\item Fin models commutative algebras;
\item $\text{Fin}^\text{op}$ models cocommutative coalgebras;
\item $\text{Burn}^\text{eff}=\text{Span}(\text{Fin})$ models commutative-cocommutative bialgebras (Corollary \ref{CofreeSemiadd});
\item $\text{Fin}^\text{inj}$ models pointed objects;
\item $\text{Fin}_\ast$ models augmented commutative algebras, etc.
\end{itemize}
\end{example}

\noindent With these examples in mind, we often think of cyclic $\mathcal{R}$-modules as `$\mathcal{R}$-indexed algebraic theories'. This perspective is meaningful even when $\mathcal{R}$ is more complex:

\begin{example}
\label{EqEx}
Fix a finite group $G$, and consider the semiring category of finite $G$-sets with isomorphisms $\text{Fin}_G^\text{iso}$. We may think of cyclic $\text{Fin}_G^\text{iso}$-modules as \emph{equivariant PROPs}.

For example, the subcategory inclusions $\text{Fin}_G^\text{iso}\rightarrow\text{Fin}_G^\text{op}$, $\text{Fin}_G^\text{iso}\rightarrow\text{Burn}_G$ are semiring functors, where $\text{Burn}_G$ is the classical Burnside category of spans of finite $G$-sets. These exhibit $\text{Fin}_G^{\text{op},\amalg}$ and $\text{Burn}_G^\amalg$ as $\text{Fin}_G^\text{iso}$-modules, which are certainly cyclic, and therefore equivariant PROPs.

By Elmendorf's theorem (\cite{Elmendorf} Theorem 1), $\text{Fin}_G^\text{op}$-models in spaces are genuine equivariant $G$-spaces: $$\text{Hom}(\text{Fin}_G^{\text{op},\amalg},\text{Top}^\times)\cong\text{Top}_G.$$ And by Guillou-May's Theorem (\cite{GMay2} Theorem 0.1), $\text{Burn}_G$-models in spectra are genuine equivariant $G$-spectra: $$\text{Hom}(\text{Burn}_G^\amalg,\text{Sp}^\wedge)\cong\text{Sp}_G.$$ In general, we can think of $\text{Fin}_G^{\text{op},\amalg}$ as the equivariant PROP modeling coefficient system objects, and $\text{Burn}_G$ as the equivariant PROP modeling Mackey functor objects.
\end{example}

\begin{example}
\label{LawvereExample}
If a PROP $\mathcal{T}^\oast$ is cartesian monoidal, it is called a Lawvere theory \cite{Lawvere}. By Theorem \ref{2T3}, Lawvere theories are cyclic $\text{Fin}^\text{op}$-modules.
\end{example}

\noindent Although 1-categorical PROPs and Lawvere theories have been studied extensively, higher categorical PROPs and Lawvere theories have only begun to be studied in the past few years. The only sources we are aware of are Cranch's thesis \cite{Cranch} and the appendix of \cite{GGN}.

\subsubsection{Cyclic modules are closed under tensor products}
\noindent We have defined $\text{CycMod}_\mathcal{R}$ as a full subcategory of pointed $\mathcal{R}$-modules; now we will show that it inherits good algebraic properties.

\begin{lemma}[\cite{Berman} 2.25]
\label{LemmaES}
Let $\mathcal{R}$ be a commutative semiring $\infty$-category. If $F:\mathcal{A}\rightarrow\mathcal{B}$ is an essentially surjective $\mathcal{R}$-module functor and $\mathcal{C}$ is another $\mathcal{R}$-module, then the functor $F_\ast:\mathcal{C}\otimes_\mathcal{R}\mathcal{A}\rightarrow\mathcal{C}\otimes_\mathcal{R}\mathcal{B}$ is also essentially surjective.
\end{lemma}

\begin{proof}
Let $\mathcal{S}$ be the full subcategory of $\mathcal{C}\otimes_\mathcal{R}\mathcal{B}$ spanned by the image of $F_\ast$. Then the functor $$j:\mathcal{C}\rightarrow\text{Hom}_\mathcal{R}(\mathcal{B},\mathcal{C}\otimes_\mathcal{R}\mathcal{B})$$ (obtained via adjunction from the identity $\mathcal{C}\otimes_\mathcal{R}\mathcal{B}\rightarrow\mathcal{C}\otimes_\mathcal{R}\mathcal{B}$) factors $$\mathcal{C}\rightarrow\text{Hom}_\mathcal{R}(\mathcal{B},\mathcal{S})\subseteq\text{Hom}_\mathcal{R}(\mathcal{B},\mathcal{C}\otimes_\mathcal{R}\mathcal{B}),$$ since the image of $j(X)$ is the same (up to equivalence) as the image of $j(X)\circ F$, but $j(X)\circ F:\mathcal{A}\rightarrow\mathcal{C}\otimes_\mathcal{R}\mathcal{B}$ factors through $\mathcal{S}$.

Undoing the tensor-Hom adjunction, the identity on $\mathcal{C}\otimes_\mathcal{R}\mathcal{B}$ factors $$\mathcal{C}\otimes_\mathcal{R}\mathcal{B}\rightarrow\mathcal{S}\subseteq\mathcal{C}\otimes_\mathcal{R}\mathcal{B}$$ (up to equivalence). Therefore, $\mathcal{S}\subseteq\mathcal{C}\otimes_\mathcal{R}\mathcal{B}$ is essentially surjective. Since $\mathcal{S}$ is the image of $F_\ast$, this completes the proof.
\end{proof}

\noindent Both of the following propositions are direct corollaries of the lemma.

\begin{proposition}
\label{CyclicTensor}
If $\mathcal{M}$ and $\mathcal{N}$ are two cyclic $\mathcal{R}$-modules, $\mathcal{M}\otimes_\mathcal{R}\mathcal{N}$ is also cyclic. The structure map from $\mathcal{R}$ is given by $$\mathcal{R}\cong\mathcal{R}\otimes_\mathcal{R}\mathcal{R}\rightarrow\mathcal{M}\otimes_\mathcal{R}\mathcal{N},$$ the tensor product of the structure maps $\mathcal{R}\rightarrow\mathcal{M},\mathcal{N}$. That is, $\text{CycMod}_\mathcal{R}^{\otimes_\mathcal{R}}$ is symmetric monoidal, and the subcategory inclusion $\text{CycMod}_\mathcal{R}^{\otimes_\mathcal{R}}\subseteq\text{Mod}_{\mathcal{R},\ast}^{\otimes_\mathcal{R}}$ is also symmetric monoidal.
\end{proposition}

\begin{proposition}
For $\mathcal{A}\in\text{CAlg}_\mathcal{R}$, the functor $$\mathcal{A}\otimes_\mathcal{R}-\colon\text{Mod}_\mathcal{R}\rightarrow\text{Mod}_\mathcal{A}$$ restricts to $$\mathcal{A}\otimes_\mathcal{R}-\colon\text{CycMod}_\mathcal{R}\rightarrow\text{CycMod}_\mathcal{A}.$$
\end{proposition}

\begin{corollary}
Higher categorical PROPs and Lawvere theories form symmetric monoidal $\infty$-categories $$\text{PROP}^\otimes=\text{CycMod}_{\text{Fin}^\text{iso}}^\otimes$$ $$\text{Lawv}^\otimes\cong\text{CycMod}_{\text{Fin}^\text{op}}^\otimes$$ under the ordinary tensor product of symmetric monoidal $\infty$-categories. This tensor product has the following universal property for PROPs $$\text{Mdl}(\mathcal{T}\otimes\mathcal{T}^\prime,\mathcal{C}^\oast)\cong\text{Mdl}(\mathcal{T},\text{Mdl}(\mathcal{T}^\prime,\mathcal{C}^\oast))$$ and the slightly simpler universal property for Lawvere theories $$\text{Mdl}(\mathcal{L}\otimes\mathcal{L}^\prime)\cong\text{Mdl}(\mathcal{L},\text{Mdl}(\mathcal{L}^\prime)).$$
\end{corollary}

\subsection{Lawvere theories}\label{4.2.2}
\noindent We will now turn to Lawvere theories, or cyclic $\text{Fin}^\text{op}$-modules $\mathcal{L}$. However, to maintain the usual notation of the Yoneda lemma, it will be more convenient to study cyclic Fin-modules $\mathcal{F}=\mathcal{L}^\text{op}$. We will try to be careful in using $\mathcal{L}$ vs. $\mathcal{F}$ to make it clear which object we are speaking of, but in any case there is an equivalence of $\infty$-categories $$(-)^\text{op}:\text{CycMod}_\text{Fin}\xrightarrow{\sim}\text{CycMod}_{\text{Fin}^\text{op}}\cong\text{Lawv}.$$ If $\mathcal{F}=\mathcal{L}^\text{op}$, we use either $\mathcal{P}_\amalg(\mathcal{F})$ or $\text{Mdl}(\mathcal{L})$ to refer to $\text{Hom}(\mathcal{L}^\times,\text{Top}^\times)$, depending on context.

Let $\text{Pr}^\text{L}_\ast$ denote the $\infty$-category of \emph{pointed} presentable $\infty$-categories. Since left adjoint functors out of Top are compltely determined by the image of the contractible space, specifying a pointing on a presentable $\infty$-category $\mathcal{C}$ amounts to any of the following equivalent data:
\begin{itemize}
\item a distinguished object $X\in\mathcal{C}$;
\item a left adjoint functor $\text{Free}:\text{Top}\rightarrow\mathcal{C}$;
\item a right adjoint functor $\text{Fgt}:\mathcal{C}\rightarrow\text{Top}$.
\end{itemize}

\noindent By the second point, pointed presentable $\infty$-categories are $\mathbb{E}_0$-algebras in $\text{Pr}^{\text{L},\otimes}$, so $\text{Pr}^\text{L}_\ast$ inherits a symmetric monoidal structure $\otimes$ from $\text{Pr}^\text{L}$.

The algebraic presheaf functor induces $\mathcal{P}_\amalg:\text{CycMod}_\text{Fin}\rightarrow\text{Pr}^\text{L}_\ast$, and the symmetric monoidal functor $\text{Fin}\rightarrow\mathcal{F}$ which exhibits $\mathcal{F}$ as cyclic induces a left adjoint functor $$\text{Top}\cong\mathcal{P}_\amalg(\text{Fin})\rightarrow\mathcal{P}_\amalg(\mathcal{F})$$ which exhibits $\mathcal{P}_\amalg(\mathcal{F})$ as pointed. (The right adjoint to this map is given by precomposition along $\text{Fin}\rightarrow\mathcal{F}$).

There is also a functor $(-)^\circ:\text{Pr}^\text{L}_\ast\rightarrow\text{CycMod}_\text{Fin}$ which sends a presentable $\infty$-category $\mathcal{C}$ (with distinguished object $X$) to the full subcategory spanned by $X^{\amalg n}$ ($n\geq 0$). We call $\mathcal{C}^\circ$ the \emph{skeleton} of $\mathcal{C}$. Our main result is as follows:

\begin{theorem}
\label{ThmLawvAdj}
There is a symmetric monoidal adjunction $$\mathcal{P}_\amalg:\text{CycMod}^\otimes_\text{Fin}\rightleftarrows\text{Pr}^{\text{L},\otimes}_\ast:(-)^\circ,$$ with fully faithful left adjoint. This adjunction takes the equivalent form $$\text{Mdl}:\text{Lawv}^\otimes\rightleftarrows\text{Pr}^{\text{L},\otimes}_\ast:\mathcal{L}_{(-)},$$ which exhibits Lawv as a colocalization of $\text{Pr}^\text{L}_\ast$.
\end{theorem}

\begin{proof}
The composition $(\mathcal{P}_\amalg(-))^\circ$ is the identity by Proposition \ref{SMYoneda2}. Thus we have a functor $$\text{Fun}^\text{L}(\mathcal{P}_\amalg(\mathcal{F}),\mathcal{V})\xrightarrow{(-)^\circ}\text{Hom}_\ast(\mathcal{F},\mathcal{V}^\circ),$$ which is an equivalence by Proposition \ref{SMYoneda3}. This establishes the adjunction, and the left adjoint is fully faithful by taking $\mathcal{V}=\mathcal{P}_\amalg(\mathcal{F}^\prime)$ for another cyclic Fin-module $\mathcal{F}^\prime$. The left adjoint is symmetric monoidal by Proposition \ref{SMYoneda1}.
\end{proof}

\begin{definition}
A \emph{multiplicative Lawvere theory} is a commutative algebra in $\text{Lawv}^\otimes$, or equivalently a commutative semiring $\infty$-category equipped with an essentially surjective commutative semiring functor from $\text{Fin}^\text{op}$.
\end{definition}

\begin{remark}
\label{RmkLawvEnr}
As in Corollary \ref{CorEnrMod}, if $\mathcal{L}$ is a multiplicative Lawvere theory, $\text{Mdl}(\mathcal{L})$ inherits a closed symmetric monoidal Day convolution, and any $\mathcal{L}$-module (or any module over $\mathcal{F}=\mathcal{L}^\text{op}$) is naturally enriched in $\text{Mdl}(\mathcal{L})$.
\end{remark}

\subsubsection{Detecting Lawvere theories}
\noindent A result of Gepner, Groth, and Nikolaus characterizes the image of $\mathcal{P}_\amalg$ in $\text{Pr}^\text{L}_\ast$ (that is, characterizes which pointed presentable $\infty$-categories arise as models over a Lawvere theory).

\begin{theorem}[\cite{GGN} B.7]
\label{ThmLawvDet}
Suppose $\mathcal{C}$ is a presentable $\infty$-category with distinguished object $X$, with corresponding forgetful functor $\text{Fgt}:\mathcal{C}\rightarrow\text{Top}$. Then $\mathcal{C}\cong\text{Mdl}(\mathcal{L})$ for some Lawvere theory $\mathcal{L}$ if and only if Fgt is conservative and preserves sifted colimits.
\end{theorem}

\begin{remark}
A functor $R:\mathcal{C}\rightarrow\mathcal{D}$ is called conservative if, any time $f$ is a morphism in $\mathcal{C}$ such that $R(f)$ is an equivalence, then $f$ is also an equivalence.

In the setup of the theorem, $R$ is a right adjoint, so it necessarily preserves filtered colimits. It additionally preserves sifted colimits if and only if it preserves geometric realizations (colimits indexed by $\Delta^\text{op}$).
\end{remark}

\begin{example}
All of the following arise as $\infty$-categories of models over Lawvere theories: Top, $\text{CMon}_\infty$, $\text{Ab}_\infty$, $\text{Alg}_\mathcal{O}(\text{Top}^\times)$ and $\text{Alg}_\mathcal{O}(\text{Ab}_\infty^\otimes)$ for any operad $\mathcal{O}$, and $\text{Mod}_R$ for any commutative semiring space $R$.
\end{example}

\subsubsection{Lawvere theories are self-dual}
\noindent If $\mathcal{C}$ is an $\infty$-category of models over a Lawvere theory, we can directly compute presentable tensor products involving $\mathcal{C}$, also by work of Gepner, Groth, and Nikolaus.

\begin{theorem}[\cite{GGN} B.3]
Let $\mathcal{L}$ be a Lawvere theory, and $\mathcal{C}=\text{Mdl}(\mathcal{L})$. For any $\mathcal{D}\in\text{Pr}^\text{L}$, $$\mathcal{C}\otimes\mathcal{D}\cong\text{Mdl}(\mathcal{L},\mathcal{D})=\text{Hom}(\mathcal{L}^\times,\mathcal{D}^\times).$$
\end{theorem}

\begin{example}
If $\mathcal{L}$ is a multiplicative Lawvere theory, and $\mathcal{V}^\otimes$ is a closed symmetric monoidal presentable $\infty$-category, then $$\text{Mdl}(\mathcal{L},\mathcal{V})\cong\text{Mdl}(\mathcal{L})\otimes\mathcal{V}$$ inherits a closed symmetric monoidal structure from $\otimes$ and Day convolution on $\text{Mdl}(\mathcal{L})$. We might refer to this closed symmetric monoidal structure on $\text{Mdl}(\mathcal{L},\mathcal{V})$ itself as \emph{Day convolution}.

When $\mathcal{L}=\text{Burn}^\text{eff}$ and $\mathcal{V}^\otimes=\text{Cat}^\times$, this is the source of the tensor product structure on SymMon (discussed in Chapter \ref{3}).
\end{example}

\subsubsection{The Lawvere theory associated to a PROP}
\noindent Suppose $\mathcal{T}^\oast$ is a PROP (cyclic $\text{Fin}^\text{iso}$-module). Then the free cartesian monoidal $\infty$-category on $\mathcal{T}^\oast$ is $\mathcal{T}^\oast\otimes\text{Fin}^\text{op}$ (Proposition \ref{TupHup}), and it is a cyclic $\text{Fin}^\text{op}$-module (therefore a Lawvere theory) by Proposition \ref{CyclicTensor}. It is the \emph{Lawvere theory associated to $\mathcal{T}^\oast$}, in the sense that it is the unique Lawvere theory $\mathcal{L}$ with $$\text{Mdl}(\mathcal{L})\cong\text{Mdl}(\mathcal{T}).$$

\begin{proposition}
\label{FreeLawv}
If $\mathcal{T}^\oast$ is a PROP, the associated Lawvere theory $\mathcal{T}^\oast\otimes\text{Fin}^\text{op}$ is equivalent to $\text{Mdl}(\mathcal{T}^\oast)^\circ$, the $\infty$-category of finitely generated free models of $\mathcal{T}$.
\end{proposition}

\begin{proof}
By Theorem \ref{ThmLawvAdj}; see also Theorem \ref{SMYonedaT} of which this is a special case.
\end{proof}

\begin{example}
Consider the PROP $\text{Fin}^\amalg$, which is an algebraic theory for \emph{commutative algebras}. The associated Lawvere theory $\text{Fin}\otimes\text{Fin}^\text{op}\cong\text{Span}(\text{Fin})$ is the $\infty$-category of finitely generated free $\mathbb{E}_\infty$-spaces. Yet again, we see that $\text{Span}(\text{Fin})\cong\text{Burn}^\text{eff}$.

This generalizes as follows: The symmetric monoidal envelope of an operad is a PROP, and its associated Lawvere theory can be described via a span construction. We will discuss this further in \S\ref{4.3.1}.
\end{example}

\subsection{Semiadditive Lawvere theories}\label{4.2.3}
\noindent If $R$ is an associative (that is, $\mathbb{A}_\infty$ or $\mathbb{E}_1$) semiring space, let $\text{Mod}_R$ be the $\infty$-category of $R$-modules, which is made into a pointed presentable $\infty$-category via the forgetful functor $\text{Mod}_R\rightarrow\text{Top}$. Since the forgetful functor is conservative (HA 4.2.3.2) and preserves sifted colimits (HA 4.2.3.5), there is an associated Lawvere theory $\text{Burn}_R=\text{Mod}_R^\circ$.

\begin{remark}
\label{RmkBurnGpCp}
We use the notation $\text{Burn}_R$ because this Lawvere theory can be obtained from $\text{Burn}^\text{eff}$ (roughly) by tensoring all Hom-$\mathbb{E}_\infty$-spaces with $R$.

Let $R=\Omega^\infty\mathbb{S}$ (the group-completion of $\text{Fin}^\text{iso}$ \cite{BPQ}). The \emph{virtual Burnside $\infty$-category} (or just \emph{the Burnside $\infty$-category}) is $\text{Burn}=\text{Burn}_{\Omega^\infty\mathbb{S}}$, which is roughly obtained from $\text{Burn}^\text{eff}$ by group-completing Hom-objects. Since $\Omega^\infty\mathbb{S}$-modules are connective spectra, Burn is a full subcategory of spectra: $$\text{Burn}\cong\text{Sp}^\circ.$$
\end{remark}

\begin{theorem}
Every cyclic $\text{Burn}^\text{eff}$-module (that is, every semiadditive Lawvere theory) is equivalent to $\text{Burn}_R$ for some associative semiring space $R$.
\end{theorem}

\begin{proof}
If $\mathcal{L}$ is a semiadditive Lawvere theory with distinguished object $X$, let $R=\text{End}(X)$. Since $\mathcal{L}$ is $\text{CMon}_\infty$-enriched, $R$ is a semiring space. If $\mathcal{L}^\prime$ is another semiadditive Lawvere theory with the same associated semiring space, then any semiadditive functor $\mathcal{L}\rightarrow\mathcal{L}^\prime$ which is an equivalence on $\text{End}(X)$ will be an equivalence of Lawvere theories (by semiadditivity). The difficulty is in constructing such a functor.

To avoid this technicality, we use a more opaque proof strategy. Note that a semiadditive Lawvere theory can be identified with a right adjoint functor $\text{Mdl}(\mathcal{L})\rightarrow\text{CMon}_\infty$ from a presentable $\text{CMon}_\infty$-module $\text{Mdl}(\mathcal{L})$ which is conservative and preserves sifted colimits (Theorem \ref{ThmLawvDet}). By HA 4.8.5.8, $\text{Mdl}(\mathcal{L})$ takes the form $\text{Mod}_R$ for some $R$, and therefore $$\mathcal{L}\cong\text{Mdl}(\mathcal{L})^\circ\cong\text{Burn}_R.$$
\end{proof}

\begin{corollary}
\label{SemiaddLawv}
As symmetric monoidal $\infty$-categories, $$\text{CycMod}_{\text{Burn}^\text{eff}}^\otimes\cong\text{SRing}(\text{Top})^\otimes.$$
\end{corollary}

\begin{proof}
Combine the theorem with HA 4.8.5.11 to see that these are equivalent as $\infty$-categories. To obtain an equivalence of \emph{symmetric monoidal} $\infty$-categories, we must show that $\text{Mod}_R\otimes\text{Mod}_S\cong\text{Mod}_{R\otimes S}$, where the first $\otimes$ is of presentable $\infty$-categories. By HA 4.8.4.6, $\text{Mod}_R\otimes\text{Mod}_S\cong\text{Mod}_R(\text{Mod}_S)$, which is $\text{Mod}_{R\otimes S}$ by definition.
\end{proof}

\subsubsection{Additive \texorpdfstring{$\infty$-}{infinity }categories}
\noindent We are finally ready to prove that additive $\infty$-categories are modules over the Burnside category, as advertised in Chapter \ref{3}.

A semiadditive category is naturally enriched in commutative monoids. Recall from classical category theory that a semiadditive category is called \emph{additive} if all the Hom-commutative monoids are in fact abelian groups. We can make the same definition for $\infty$-categories, and we may replace `abelian groups' by $R$-modules for any solid semiring space $R$.

\begin{definition}
If $R$ is a solid commutative semiring space, and $\mathcal{C}$ is a semiadditive $\infty$-category, we say that $\mathcal{C}$ is \emph{$R$-additive} if each Hom-object of $\mathcal{C}$, which is a priori an $\mathbb{E}_\infty$-space (by Example \ref{EnrAddEx}), is in fact an $R$-module.
\end{definition}

\noindent For example, a $\text{Fin}^\text{iso}$-additive $\infty$-category is just a semiadditive $\infty$-category, and a $\Omega^\infty\mathbb{S}$-additive $\infty$-category is just an additive $\infty$-category (a semiadditive $\infty$-category all of whose Hom-$\mathbb{E}_\infty$-spaces are grouplike).

\begin{theorem}
\label{2T7}
If $R$ is a solid commutative semiring space, then $\text{Burn}_R$ is a solid semiring $\infty$-category, and a symmetric monoidal $\infty$-category $\mathcal{C}^\oast$ is a $\text{Burn}_R$-module if and only if $\mathcal{C}$ is $R$-additive (with $\oast=\oplus$ the direct sum).
\end{theorem}

\begin{proof}
First, $\text{Burn}_R$ is idempotent in Lawv by Corollary \ref{SemiaddLawv} (since $R$ itself is idempotent). Now it suffices to show that $\mathcal{C}^\oast$ has a $\text{Burn}_R$-module structure if and only if it is $R$-additive.

Any $\text{Burn}_R$-module is enriched in $\text{Mdl}(\text{Burn}_R)\cong\text{Mod}_R$ by Remark \ref{RmkLawvEnr}, and therefore $R$-additive.

Conversely, suppose $\mathcal{C}$ is $R$-additive, and let $F:\mathcal{C}^\text{op}\rightarrow\text{CMon}_\infty$ be a semiadditive functor. For any object $X\in\mathcal{C}$, $\text{End}(X)$ is an $R$-algebra, so there are semiring maps $$R\rightarrow\text{End}(X)\rightarrow\text{End}(F(X))$$ which exhibit $F(X)$ as an $R$-module. Since $R$ is solid, this implies that $F$ factors through the full subcategory $\text{Mod}_R\subseteq\text{CMon}_\infty$. Therefore $$\text{Hom}(\mathcal{C}^\text{op},\text{Top})\cong\text{Hom}(\mathcal{C}^\text{op},\text{CMon}_\infty)\cong\text{Hom}(\mathcal{C}^\text{op},\text{Mod}_R),$$ where the first equivalence is because $\mathcal{C}$ is semiadditive. However, $\text{Hom}(\mathcal{C}^\text{op},\text{Mod}_R)$ is a presentable $\text{Mod}_R$-module. Since there is a functor of commutative semiring $\infty$-categories $\text{Burn}_R\xrightarrow{\subseteq}\text{Mod}_R$, it follows that $\text{Hom}(\mathcal{C}^\text{op},\text{Top})$ is a $\text{Burn}_R$-module.

Finally, $\mathcal{C}$ is a full subcategory of $\text{Hom}(\mathcal{C}^\text{op},\text{Top})$ which is closed under direct sum, so it too is a $\text{Burn}_R$-module. This completes the proof.
\end{proof}

\begin{corollary}
\label{2T7b}
Let $\text{Burn}=\text{Burn}_{\Omega^\infty\mathbb{S}}$ be the full subcategory of Sp spanned by finite wedge powers of the sphere spectrum $\mathbb{S}^{\vee n}$. (Roughly, Burn is obtained from $\text{Burn}^\text{eff}$ by group-completing the Hom-objects.) Then Burn is a solid semiring $\infty$-category, and Burn-modules are additive $\infty$-categories.
\end{corollary}

\section{Problems}\label{4.3}
\noindent In this section, we will introduce a system of conjectures on higher algebra, which would necessitate a more serious understanding of some of the `categorified algebraic geometry' alluded to in Chapter \ref{3}. In \S\ref{4.3.1}, we propose a reinterpretation of $\infty$-operads as particularly well-behaved Lawvere theories. That is, we conjecture that an $\infty$-operad is a Lawvere theory which is trivial over $\text{Burn}^\text{eff}$; as such, there is a fibration $\mathcal{L}\xrightarrow{p}\text{Burn}^\text{eff}$ which encodes some of the disjunctive behavior of $\infty$-operads a la Lurie.

To prove this conjecture, we would hope to see that $\infty$-operads have good behavior computationally in the commutative algebra of $\infty$-categories. This can be viewed as a statement about the `flatness' of $\infty$-operads. In \S\ref{4.3.2}, we discuss some conjectures surrounding the algebraic geometry of the initial commutative semiring $\infty$-category $\text{Fin}^\text{iso}$, and speculate that a version of flat descent would imply some of the conjectures from \S\ref{4.3.1}.

\subsection{Operads}\label{4.3.1}
\subsubsection{Operadic Lawvere theories via span constructions}
\noindent Recall from Example \ref{ExOp} and HA that an $\infty$-operad is a type of fibration $[\mathcal{O}^\otimes]\rightarrow\text{Fin}_\ast$. The square brackets emphasize that $[\mathcal{O}^\otimes]$ is itself an $\infty$-category, rather than a symmetric monoidal structure on an $\infty$-category $\mathcal{O}$.

Associated to any $\infty$-operad is a symmetric monoidal $\infty$-category $\text{Env}(\mathcal{O})^\oast$, the \emph{symmetric monoidal envelope} of $\mathcal{O}^\otimes$, satisfying the universal property that $$\text{Hom}(\text{Env}(\mathcal{O})^\oast,\mathcal{C}^\oast)\cong\text{Alg}_\mathcal{O}(\mathcal{C}^\oast)$$ for any $\mathcal{C}^\oast\in\text{SymMon}$.

We will assume all our $\infty$-operads are single-colored (\emph{reduced}, in the language of HA; see also \cite{Berman} 3.23), so the underlying $\infty$-category of colors $\mathcal{O}$ is contractible, and $[\mathcal{O}^\otimes]$ has objects indexed by the nonnegative integers. The symmetric monoidal envelope is defined via the pullback $$\xymatrix{
\text{Env}(\mathcal{O})\ar[d]\ar[r] &\text{Fin}\ar[d] \\
[\mathcal{O}^\otimes]\ar[r] &\text{Fin}_\ast.
}$$ In particular, $\text{Env}(\mathcal{O})^\oast$ is the PROP modeling $\mathcal{O}$-algebras. Therefore, the Lawvere theory modeling $\mathcal{O}$-algebras is $$\mathcal{L}_\mathcal{O}\cong\text{Env}(\mathcal{O})^\oast\otimes\text{Fin}^\text{op}.$$ According to Proposition \ref{FreeLawv}, $\mathcal{L}_\mathcal{O}$ is the $\infty$-category of finitely generated free $\mathcal{O}$-algebras (in Top).

Classically, if $X$ is a finite set, the free $\mathcal{O}$-algebra generated by $X$ admits an explicit description; the underlying set is $$\coprod_{T\in\text{Fin}}\mathcal{O}(T)\times_{\Sigma_T}X^T,$$ where $\Sigma_T$ denotes the symmetric group $\text{Aut}(T)$ and $\mathcal{O}(T)$ is the set of ways to multiply the objects of $T$ via the operad. In the event that $\mathcal{O}$ is a classical (1-categorical) operad, we can use this description of the free algebra to describe the Lawvere theory $\mathcal{L}_\mathcal{O}$ as a span construction.

\begin{definition}
If $\mathcal{O}$ is an operad, let $\text{Span}_\mathcal{O}$ be the category whose objects are finite sets, and a morphism from $X$ to $Y$ consists of the following data (up to isomorphism): a finite set $T$, functions $X\leftarrow T\xrightarrow{f} Y$, and for each $y\in Y$, an object of $\mathcal{O}(T_y)$, where $T_y=f^{-1}(y)$. We say that this last structure makes $f$ \emph{operadic}.

A composition of two morphisms $X\rightarrow Y\rightarrow Z$ is given by a pullback $$\xymatrix{
&&S\times_Y T\ar[ld]\ar[rd] &&\\
&S\ar[ld]\ar[rd] &&T\ar[ld]\ar[rd] &\\
X &&Y &&Z.
}$$ Each fiber of $S\times_Y T\rightarrow T$ is canonically in bijection with a fiber of $S\rightarrow Y$; thus $S\times_Y T\rightarrow T$ inherits an operadic structure from $S\rightarrow Y$, which makes the total span $X\leftarrow S\times_Y T\rightarrow Z$ a morphism in $\text{Span}_\mathcal{O}$. 
\end{definition}

\noindent Notice that the set $\text{Span}_\mathcal{O}(1,X)$ is in bijection with $$\coprod_{T\in\text{Fin}}\mathcal{O}(T)\times_{\Sigma_T}X^T,$$ where here we are summing over possible choices of the spanning object $T$.

\begin{proposition}
If $\mathcal{O}$ is an operad, $\text{Span}_\mathcal{O}^\text{op}$ is equivalent to each of the following:
\begin{enumerate}
\item the full subcategory of $\text{Alg}_\mathcal{O}$ on the finitely generated free algebras;
\item the Lawvere theory for $\mathcal{O}$-algebras;
\item $\text{Env}(\mathcal{O})\otimes\text{Fin}^\text{op}$.
\end{enumerate}
\end{proposition}

\begin{proof}
There is a functor $I:\text{Span}_\mathcal{O}^\text{op}\rightarrow\text{Alg}_\mathcal{O}$ which sends $X$ to the free $\mathcal{O}$-algebra $\text{Free}(X)$, spans $X\xleftarrow{g} Y=Y$ to $\text{Free}(g)$, and spans $X=X\xrightarrow{f} Y$ to the morphism $\text{Free}(Y)\rightarrow\text{Free}(X)$ adjoint to $f^\ast:Y\rightarrow\text{Free}(X)$. Here $f^\ast(y)$ corresponds to the object $$(o,f)\in\mathcal{O}(X_y)\times_{\Sigma_{X_y}}Y^{X_y}\subseteq\text{Free}(X),$$ with $o\in\mathcal{O}(X_y)$ determined by the operadic structure on $f$. We may check $I$ is a functor by verifying that pullback squares correspond to composition.

Moreover, $I:\text{Span}_\mathcal{O}(Y,X)\rightarrow\text{Free}(X)^Y$ is an equivalence, as discussed before the proposition statement. Therefore $I$ is fully faithful, proving (1). (2) and (3) follow.
\end{proof}

\begin{conjecture}
\label{Conj1}
If $\mathcal{O}$ is an $\infty$-operad (even colored), the associated Lawvere theory $\text{Env}(\mathcal{O})\otimes\text{Fin}^\text{op}$ can be described via a span construction $\text{Span}_\mathcal{O}^\text{op}$ as above.
\end{conjecture}

\noindent The difficulty in proving such a conjecture is in constructing $\text{Span}_\mathcal{O}^\text{op}$ and constructing a functor to $\text{Alg}_\mathcal{O}$. In contrast, it is not hard to see that $\text{Span}_\mathcal{O}^\text{op}$ should have Hom-spaces which recover finitely generated free $\mathcal{O}$-algebras, just as above. (For the $\infty$-categorical analogue of the formula for free $\mathcal{O}$-algebras, see HA 3.1.)

\subsubsection{Which Lawvere theories are operadic?}
\noindent Letting $\text{Op}^\otimes$ denote the $\infty$-category of (reduced) $\infty$-operads endowed with the Boardman-Vogt tensor product (HA 2.2.5), we expect fully faithful symmetric monoidal functors $$\text{Op}^\otimes\rightarrow\text{Lawv}^\otimes\xrightarrow{\text{Mdl}}\text{Pr}^{\text{L},\otimes}_\ast.$$ Conjecture \ref{Conj1} addresses the question: What is the Lawvere theory associated to an operad? Conversely, it is natural to ask: Which Lawvere theories arise from operads?

Assuming Conjecture \ref{Conj1}, to any $\infty$-operad (with structural fibration $[\mathcal{O}^\otimes]\xrightarrow{p}\text{Fin}_\ast$), there is a tower of pullbacks $$\xymatrix{
\text{Env}(\mathcal{O})\ar[r]\ar[d] &\text{Fin}\ar[d] \\
[\mathcal{O}^\otimes]\ar[d]\ar[r]^p &\text{Fin}_\ast\ar[d] \\
\mathcal{L}_\mathcal{O}\ar[r] &\text{Burn}^\text{eff}.
}$$

\begin{proposition}
\label{PropOpL}
Assuming Conjecture \ref{Conj1}, if $\mathcal{L}$ is a Lawvere theory arising from an operad, the operad is given by the pullback $\mathcal{L}\times_{\text{Burn}^\text{eff}}\text{Fin}_\ast$.
\end{proposition}

\begin{remark}
\label{RmkOdot}
The tower of pullbacks in Proposition \ref{PropOpL} endows both $\text{Env}(\mathcal{O})$ and $[\mathcal{O}^\otimes]$ with symmetric monoidal structures. The symmetric monoidal structure on $\text{Env}(\mathcal{O})$ is the usual one, and the symmetric monoidal structure on $[\mathcal{O}^\otimes]$ is expected to be the one mentioned by Lurie in HA 2.2.4.7.
\end{remark}

\noindent Proposition \ref{PropOpL} describes how to reconstruct an operad from the associated Lawvere theory, but it says nothing about how to determine \emph{whether} a Lawvere theory arises from an operad.

\subsubsection{A recognition principle}
\begin{definition}
Suppose $\mathcal{L}$ is a Lawvere theory. Then $\mathcal{L}\otimes\text{Burn}^\text{eff}$ is a cyclic $\text{Burn}^\text{eff}$-module by Lemma \ref{LemmaES}. Therefore we may identify $\mathcal{L}\otimes\text{Burn}^\text{eff}$ with a semiring space $R$ by Corollary \ref{SemiaddLawv}. We say $R$ is the \emph{characteristic semiring} of $\mathcal{L}$.
\end{definition}

\begin{example}
The characteristic semiring of $\text{Burn}_R$ is $R$ itself.
\end{example}

\begin{example}
If $\mathcal{L}_\text{CRing}$ denotes the Lawvere theory for connective $\mathbb{E}_\infty$-ring spectra, then $$\text{Mdl}(\mathcal{L}_\text{CRing}\otimes\text{Burn}^\text{eff})\cong\text{Hom}(\text{Burn}^\text{eff},\text{CRingSp}_{\geq 0}^\times),$$ which is in fact contractible, since there are no maps to a nontrivial ring from the zero ring (which is the unit of $\text{CRingSp}_{\geq 0}^\times$).

Therefore, the characteristic semiring of $\mathcal{L}_\text{CRing}$ is $0$.
\end{example}

\begin{example}
Suppose $\mathcal{L}_\mathcal{O}$ is the Lawvere theory associated to a (reduced) $\infty$-operad $\mathcal{O}^\otimes$. Then the essentially surjective maps $$\text{Fin}^\text{iso}\rightarrow\text{Env}(\mathcal{O})$$ $$\text{Fin}^\text{op}\rightarrow\mathcal{L}_\mathcal{O}$$ become equivalences after tensoring with $\text{Fin}$. (See \cite{Berman} 3.24.) Thus, $\mathcal{L}_\mathcal{O}\otimes\text{Burn}^\text{eff}\cong\text{Burn}^\text{eff}$, so the characteristic semiring of $\mathcal{L}_\mathcal{O}$ is $\text{Fin}^\text{iso}$.
\end{example}

\begin{conjecture}
\label{Conj2}
A Lawvere theory $\mathcal{L}$ has an associated (reduced) $\infty$-operad if and only if $\mathcal{L}$ is \emph{trivial over $\text{Burn}^\text{eff}$} in the sense that the functor $$\text{Burn}^\text{eff}\cong\text{Fin}^\text{op}\otimes\text{Burn}^\text{eff}\rightarrow\mathcal{L}\otimes\text{Burn}^\text{eff}$$ is an equivalence of $\infty$-categories.
\end{conjecture}

\begin{remark}
The conjecture suggests that $\infty$-operads can be regarded as a flavor of 2-vector bundle (in the sense that they are locally free). We should mention that Baas, Dundas, Richter, and Rognes have also studied 2-vector bundles in the context of iterated K-theory \cite{BDRR}.
\end{remark}

\noindent Assuming Conjecture \ref{Conj1} as well, Proposition \ref{PropOpL} would apply so that we could reconstruct the $\infty$-operad via the pullback of $$\mathcal{L}\rightarrow\mathcal{L}\otimes\text{Burn}^\text{eff}\cong\text{Burn}^\text{eff}$$ along $\text{Fin}_\ast\subseteq\text{Burn}^\text{eff}$.

We end this section with idle speculation about a proof strategy for Conjecture \ref{Conj2}. Regard the horizontal functors in Proposition \ref{PropOpL}'s tower of pullbacks as \emph{disjunctive fibrations}, in the sense that $[\mathcal{O}^\otimes]\xrightarrow{p}\text{Fin}_\ast$:
\begin{itemize}
\item is a symmetric monoidal functor (from the symmetric monoidal structure $\odot$ of HA 2.2.4.7 to $\amalg$);
\item is a fibration (HA 2.1.1.13);
\item and records some information about how to map into $X\odot Y$ (HA 2.1.1.14).
\end{itemize}

\noindent Suppose $\mathcal{L}$ is a Lawvere theory trivial over $\text{Burn}^\text{eff}$. Since $\mathcal{L}$ is cartesian monoidal, we know how to map into $X\times Y$ for $X,Y\in\mathcal{L}$. Therefore we have some hope of proving that $\mathcal{L}\rightarrow\text{Burn}^\text{eff}$ is a `disjunctive fibration' (whatever this means). If disjunctive fibrations are stable under pullback, then we might show the pullback along $\text{Fin}_\ast\subseteq\text{Burn}^\text{eff}$ is an $\infty$-operad is the sense of HA 2.1.1.10.

\subsection{Descent}\label{4.3.2}
\noindent A major obstacle in categorified commutative algebra (as discussed in this chapter and the last) is the difficulty in making any computations, such as the computation of tensor products of symmetric monoidal $\infty$-categories. We do have a few techniques for computing tensor products: $\mathcal{C}^\oast\otimes\vec{\mathbb{S}}$ is related to algebraic K-theory (see \S\ref{3.3.2}), which however is famously difficult to compute; and $\mathcal{C}^\oast\otimes\text{Fin}$ and $\mathcal{C}^\oast\otimes\text{Fin}^\text{op}$ are related to $\text{Hom}(\mathcal{C}^\oast,\text{Top}^\times)$ by Theorem \ref{SMYonedaT}. We also know that tensor products are related in many examples to span constructions (see \S\ref{3.2.2} and Conjecture \ref{Conj1}). However, many specific computations remain out of reach.

In classical commutative algebra, many tensor products (over a ring $R$) can be computed by utilizing algebrao-geometric techniques including \emph{flatness} and the geometry of $\text{Spec}(R)$. Here is an example of this sort of technique:

\begin{definition}
A square $$\xymatrix{
A\ar[r]\ar[d] &B\ar[d] \\
C\ar[r] &D
}$$ in a category (or $\infty$-category) $\mathcal{C}$ is called \emph{rigid bicartesian} if the square is both a pullback and a pushout, and each morphism is both a monomorphism and an epimorphism.
\end{definition}

\begin{example}
A rigid bicartesian square of schemes corresponds roughly to a cover of $D$ by two dense open subschemes $B$ and $C$. In such a case, we may try to study a quasicoherent sheaf over $D$ by understanding its restrictions to $B$ and $C$.

For example, the following rigid bicartesian square of commutative rings induces a rigid bicartesian square of affine schemes: $$\xymatrix{
\mathbb{Z}\ar[r]\ar[d] &\mathbb{Z}[\frac{1}{2}]\ar[d] \\
\mathbb{Z}[\frac{1}{3}]\ar[r] &\mathbb{Z}[\frac{1}{6}].
}$$ For any $\mathbb{Z}$-module $A$, applying $-\otimes A$ to the square yields another pullback square, since $\mathbb{Z}[\frac{1}{6}]$ is flat: $$\xymatrix{
A\ar[r]\ar[d] &A[\frac{1}{2}]\ar[d] \\
A[\frac{1}{3}]\ar[r] &A[\frac{1}{6}].
}$$ In this way, we recover $A$ from its restrictions to $\text{Spec}(\mathbb{Z}[\frac{1}{2}])$ and $\text{Spec}(\mathbb{Z}[\frac{1}{3}])$.
\end{example}

\begin{example}
The following is a rigid bicartesian square of commutative semiring $\infty$-categories: $$\xymatrix{
\text{Fin}^\text{iso}\ar[r]\ar[d] &\text{Fin}\ar[d] \\
\text{Fin}^\text{op}\ar[r] &\text{Burn}^\text{eff}.
}$$ By inspection, every map is a monomorphism and the square overall is a pullback. Every map is moreover an epimorphism since each of these semiring $\infty$-categories is solid (Theorems \ref{2T3} and \ref{2T4}), and the square is a pushout by Theorem \ref{2T4}. By analogy with the last example, we might ask:
\end{example}

\begin{question}
For which symmetric monoidal $\infty$-categories $\mathcal{C}^\oast$ is the square $$\xymatrix{
\mathcal{C}\ar[r]\ar[d] &\mathcal{C}\otimes\text{Fin}\ar[d] \\
\mathcal{C}\otimes\text{Fin}^\text{op}\ar[r] &\mathcal{C}\otimes\text{Burn}^\text{eff}
}$$ a pullback of symmetric monoidal $\infty$-categories?
\end{question}

\noindent The question amounts to asking: to what extent do Fin and $\text{Fin}^\text{op}$ `cover' $\text{Fin}^\text{iso}$ in a geometric sense? If the last square is a pullback, we will say that $\mathcal{C}^\oast$ \emph{admits cartesian monoidal descent}.

Notice that any symmetric monoidal $\infty$-category which admits cartesian monoidal descent can be described explicitly by a pullback of $\infty$-categories, rather than resorting to a (potentially unwieldy) description of the associated $\infty$-operad a la Lurie. That is, in order to specify a cartesian (or cocartesian) monoidal $\infty$-category, we need only specify the underlying $\infty$-category, but specifying a generic symmetric monoidal $\infty$-category is a priori much more difficult. It would be useful to have a technique for building any symmetric monoidal $\infty$-category from (co)cartesian monoidal pieces.

However, unlike in ordinary commutative algebra, not all $\mathcal{C}^\oast$ admit cartesian monoidal descent.

\begin{example}
The square $$\xymatrix{
\vec{\mathbb{S}}\ar[r]\ar[d] &\vec{\mathbb{S}}\otimes\text{Fin}\ar[d] \\
\vec{\mathbb{S}}\otimes\text{Fin}^\text{op}\ar[r] &\vec{\mathbb{S}}\otimes\text{Burn}^\text{eff}
}$$ is not a pullback; indeed, all terms but $\vec{\mathbb{S}}$ are contractible by Corollary \ref{vanish}, so the pullback is $0$.
\end{example}

\begin{question}
Can any $\mathcal{C}^\otimes$ be recovered from $\mathcal{C}\otimes\text{Fin}$ (a cocartesian monoidal $\infty$-category), $\mathcal{C}\otimes\text{Fin}^\text{op}$ (a cartesian monoidal $\infty$-category) and $\mathcal{C}\otimes\vec{\mathbb{S}}$ (a spectrum)?

If $\mathcal{C}\otimes\vec{\mathbb{S}}\cong 0$, does $\mathcal{C}$ necessarily admit cartesian monoidal descent?
\end{question}

\subsubsection{Descent for operads}
\noindent We end this chapter by mentioning the relationship between cartesian monoidal descent and the conjectures of \S\ref{4.3.1}. If $\mathcal{O}^\otimes$ is a (reduced) $\infty$-operad, then $\text{Env}(\mathcal{O})$ is a $\text{Fin}^\text{inj}$-module (\cite{Berman} 3.23). Therefore, $\text{Env}(\mathcal{O})\otimes\vec{\mathbb{S}}\cong 0$, so it seems possible that $\text{Env}(\mathcal{O})$ may admit cartesian monoidal descent.

\begin{proposition}
Assume Conjecture \ref{Conj1}, and also assume that $\text{Env}(\mathcal{O})\otimes\text{Fin}^{\text{inj},\text{op}}$ admits cartesian monoidal descent for some operad $\mathcal{O}^\otimes$. Then $[\mathcal{O}^\otimes]^\odot\cong\text{Env}(\mathcal{O})\otimes\text{Fin}^{\text{inj},\text{op}}$, where $\odot$ is the symmetric monoidal structure on $[\mathcal{O}^\otimes]$ as in Remark \ref{RmkOdot}. That is, $[\mathcal{O}^\otimes]$ satisfies the universal property $$\text{Hom}([\mathcal{O}^\otimes]^\odot,\mathcal{C}^\oast)\cong\text{Alg}_\mathcal{O}(\mathcal{C}_{/1}^\oast).$$ Moreover, the tower of pullbacks of Proposition \ref{PropOpL} takes the form $$\xymatrix{
\text{Env}(\mathcal{O})\ar[r]\ar[d] &\text{Env}(\mathcal{O})\otimes\text{Fin}\ar[d] &\text{Fin}\ar[l]_-{\sim}\ar[d] \\
\text{Env}(\mathcal{O})\otimes\text{Fin}^{\text{inj},\text{op}}\ar[r]\ar[d] &\text{Env}(\mathcal{O})\otimes\text{Fin}_\ast\ar[d] &\text{Fin}_\ast\ar[l]_-{\sim}\ar[d] \\
\text{Env}(\mathcal{O})\otimes\text{Fin}^\text{op}\ar[r] &\text{Env}(\mathcal{O})\otimes\text{Burn}^\text{eff} &\text{Burn}^\text{eff}.\ar[l]_-{\sim}
}$$
\end{proposition}

\begin{proof}
Since $\text{Env}(\mathcal{O})\otimes\text{Fin}^{\text{inj},\text{op}}$ admits cartesian monoidal descent, the bottom left square in the proposition is a pullback. But by Proposition \ref{PropOpL}, the same pullback is equivalent to $[\mathcal{O}^\otimes]$. Therefore the two are equivalent, and everything else follows.
\end{proof}

\chapter{Duality in Equivariant Homotopy Theory}\label{5}\setcounter{subsection}{0}
\noindent If $G$ is a (discrete) group, let $\text{Fin}_G$ denote the category of finite right $G$-sets, and $^G\text{Fin}$ the category of \emph{free} left $G$-sets with finitely many orbits. We may also use, for example, $_G\text{Fin}$ for finite left $G$-sets, keeping in mind that $_G\text{Fin}\cong\text{Fin}_{G^\text{op}}$.

In general, we will use superscripts to denote free group actions, and subscripts to denote arbitrary actions. For example, $^G\text{Fin}_H$ denotes the category of $(G,H)$-bisets such that the left $G$-action is free.

We will always assume that $G$ is finite. In this case, $\text{Fin}^G$ is a full subcategory of $\text{Fin}_G$.

If $X$ is a right $G$-set and Y is a left $G$-set, then we write $X\times_G Y=X\times Y/\sim$, where the equivalence relation is defined by $(xg,y)\sim(x,gy)$. We will prove two analogues of the Eilenberg-Watts Theorem for groups. First recall the classical Eilenberg-Watts Theorem:

\begin{theorem}[Eilenberg-Watts \cite{EWThm}]
If $R$ and $S$ are rings and $F:\text{Mod}_R\rightarrow\text{Mod}_S$ a functor, the following are equivalent:
\begin{itemize}
\item $F$ is right exact and preserves finite coproducts (direct sums);
\item $F=-\otimes_R M$ for some $M\in\null_R\text{Mod}_S$.
\end{itemize}
\end{theorem}

\begin{theorem}[Lemma \ref{PropGlo}]
If $G$ and $H$ are finite groups and $F:\text{Fin}_G\rightarrow\text{Fin}_H$ a functor, the following are equivalent:
\begin{itemize}
\item $F$ preserves pullbacks, epimorphisms, and finite coproducts (disjoint unions);
\item $F=-\times_G X$ for some $X\in\null^G\text{Fin}_H$.
\end{itemize}
\end{theorem}

\begin{theorem}[Proposition \ref{PropCoglo}]
If $G$ and $H$ are finite groups and $F:\text{Fin}^G\rightarrow\text{Fin}^H$ a functor, the following are equivalent:
\begin{itemize}
\item $F$ preserves pullbacks, epimorphisms, and finite coproducts (disjoint unions);
\item $F=-\times_G X$ for some $X\in\null_G\text{Fin}^H$.
\end{itemize}
\end{theorem}

\noindent In fact, we will prove something more general, which subsumes both results. Say that a $(G,H)$-biset is \emph{separable} if it has the form $G^\text{op}\times T$, where $T$ is a right $H$-set, and the actions are the canonical ones, $g(a,x)h=(ga,xh)$. Let $^G\text{Fin}_H^\text{sep}$ be the full subcategory of $^G\text{Fin}_H$ spanned by the separable bisets.

\begin{theorem}[Theorem \ref{PerfPair}]
There is a perfect pairing $$-\times_{(H,G)}-:\null^H\text{Fin}_G\times\null^G\text{Fin}_H^\text{sep}\rightarrow\text{Fin}$$ in the sense that the induced functors $$^H\text{Fin}_G\rightarrow\text{Fun}_\text{pce}(\null^G\text{Fin}_H^\text{sep},\text{Fin})$$ $$^G\text{Fin}_H^\text{sep}\rightarrow\text{Fun}_\text{pce}(\null^H\text{Fin}_G,\text{Fin})$$ are both equivalences. Here $\text{Fun}_\text{pce}$ denotes functors that preserve pullbacks, (finite) coproducts, and epimorphisms.
\end{theorem}

\noindent The key to the proof is a study of torsors internal to the category $^H\text{Fin}_G$ (\S\ref{5.1.1}). We find that $^H\text{Fin}_G$ is (roughly) \emph{generated by a single object $H^\text{op}\times G$ via disjoint unions and torsors}. \\

\noindent The Eilenberg-Watts Theorem is the tip of the iceberg for a rich body of ideas in noncommutative algebra, including Morita theory and noncommutative motives. In this chapter, we prove the above duality for finite groups, then discuss how it fits into a picture that unites \emph{noncommutative motivic homotopy theory} (arising from the Eilenberg-Watts theorem) with \emph{global equivariant homotopy theory}. The following table summarizes the analogy: \\
\begin{tabular}{|c|c|} \hline
Ring $R$ &Group $G$ \\ \hline
$\text{Mod}_R$ &$\text{Fin}_G$ \\ \hline
$_R\text{Mod}$ &$^G\text{Fin}$ \\ \hline
Morita category &Burnside category \\ \hline
Noncommutative motivic &Global equivariant \\ 
homotopy theory &homotopy theory \\ \hline
Algebraic K-theory &Equivariant sphere spectrum \\ \hline
\end{tabular} \\

\noindent We anticipate that this analogy could be made rigorous by an understanding of \emph{noncommutative motives over the field with one element}.

\begin{remark}
Of course, there is no such thing as a field with one element. However, we can speak of certain constructions over $\mathbb{F}_1$; for example, a finitely generated free $\mathbb{F}_1$-module is a finite set, suggesting that (if there is a Lawvere theory for $\mathbb{F}_1$-modules) $\text{Fin}^\text{op}$ is that Lawvere theory.

Similarly, a finitely generated free module over the hypothetical group algebra $\mathbb{F}_1[G]$ is a finite free $G$-set, and $^G\text{Fin}$ plays the role of the Lawvere theory for $\text{Mod}_{\mathbb{F}_1[G]}$.

In contrast, the \emph{noncommutative motive} associated to a ring $R$ is (roughly speaking) equivalent either to the $\infty$-category $\text{Mod}_R^\text{fgf}$ of finitely generated free $R$-modules or to $\text{Mod}_R$ itself. (The two are Morita equivalent.) Therefore, we may regard $^G\text{Fin}$ as a noncommutative motive over $\mathbb{F}_1$ corresponding to $\mathbb{F}_1[G]$.
\end{remark}

\noindent Let $\text{NMo}_{\mathbb{F}_1}$ denote a hypothetical $\infty$-category of noncommutative motives over $\mathbb{F}_1$. Summarizing the last remark, we are interested in comparing two types of objects of $\text{NMo}_{\mathbb{F}_1}$:
\begin{enumerate}
\item Any ring $R$ should give rise to a noncommutative motive over $\mathbb{F}_1$ via the composite $$\text{Ring}\rightarrow\text{NMo}_{\mathbb{Z}}\subseteq\text{NMo}_{\mathbb{F}_1}.$$ The functor $\text{Ring}\rightarrow\text{NMo}_{\mathbb{Z}}$ is understood by work of Blumberg, Gepner, and Tabuada \cite{BGT}. The full subcategory of $\text{NMo}_{\mathbb{Z}}$ spanned by rings is the \emph{Morita 2-category} \cite{MoritaCat}.
\item Any group $G$ should correspond to a `group algebra' $\mathbb{F}_1[G]$, and therefore a noncommutative motive over $\mathbb{F}_1$ via a functor $$\text{Gp}\rightarrow\text{NMo}_{\mathbb{F}_1}.$$ We present evidence that the full subcategory of $\text{NMo}_{\mathbb{F}_1}$ spanned by groups is the \emph{global Burnside 2-category}, in the form of the following result:
\end{enumerate}

\begin{theorem*}[Theorems \ref{CogloCat}, \ref{GloCat}]
Let $\text{Cat}_\text{pce}$ denote the 2-category of categories which admit pullbacks and finite coproducts, and functors between them which preserve pullbacks, finite coproducts, and epimorphisms. The full subcategory of $\text{Cat}_\text{pce}$ spanned by $^\mathcal{G}\text{Fin}$ (as $G$ varies over finite groupoids) is equivalent to the global effective Burnside 2-category $\text{Burn}_\text{glo}^\text{eff}$ (\S\ref{A.3}).

The full subcategory spanned by $\text{Fin}_\mathcal{G}$ (as $\mathcal{G}$ varies over finite groupoids) is equivalent to $(\text{Burn}_\text{glo}^\text{eff})^\text{op}$.
\end{theorem*}

\noindent Therefore, we anticipate that (noncommutative) motivic homotopy theory reflects the part of $\text{NMo}_{\mathbb{F}_1}$ which includes rings $R$, while global equivariant homotopy theory reflects the part of $\text{NMo}_{\mathbb{F}_1}$ which includes group algebras $\mathbb{F}_1[G]$.

Here we are regarding $\text{Cat}_\text{pce}$ as a coarse approximation for $\text{NMo}_{\mathbb{F}_1}$.

\begin{remark}
Actually, the global Burnside 2-category, which we hypothesize is the part of $\text{NMo}_{\mathbb{F}_1}$ corresponding to group algebras $\mathbb{F}_1[G]$, governs \emph{stable} global equivariant homotopy theory.

In contrast, \emph{unstable} global equivariant homotopy theory is governed by a 2-category $\text{Gpoid}_\text{f}$ of finite groupoids (\S\ref{A.3}), which is expected to be roughly the full subcategory of $\text{Alg}_{\mathbb{F}_1}$ corresponding to group algebras $\mathbb{F}_1[G]$. This is summarized in the following table. ($\text{Gpoid}_\text{f}$ refers to finite groupoids; see Appendix \ref{A} for details.)
\end{remark}

\begin{tabular}{|c|c|c|c|} \hline
&$G$-equivariant &Global equivariant &$\mathbb{F}_1$-motivic? \\ \hline
Unstable &$\text{Fin}_G$ &$\text{Gpoid}_\text{f}$ &$\text{Alg}_{\mathbb{F}_1}$ \\ \hline
Stable &$\text{Burn}_G$ &$\text{Burn}_\text{glo}$ &$\text{NMo}_{\mathbb{F}_1}$ \\ \hline
\end{tabular}

\subsubsection{Equivariant homotopy theory}
\noindent There is one major structural difference between Theorem \ref{PerfPair} and the Eilenberg-Watts Theorem. While the Eilenberg-Watts Theorem provides a duality between $\text{Mod}_R$ and $_R\text{Mod}\cong\text{Mod}_{R^\text{op}}$, Theorem \ref{PerfPair} provides a duality between $\text{Fin}_G$ and $^G\text{Fin}\cong\text{Fin}^{G^\text{op}}$. That is, there is a loss of symmetry, reflecting the distinction between finite $G$-sets and finite \emph{free} $G$-sets.

The difference between $\text{Fin}_G$ and $\text{Fin}^G$ is also the difference between genuine and naive equivariant homotopy theory. In particular, the $\infty$-category of \emph{naive} equivariant right $G$-spaces is $$\text{Top}^G=\text{Fun}^\times((\text{Fin}^G)^\text{op},\text{Top}),$$ while that of \emph{genuine} equivariant right $G$-spaces is $$\text{Top}_G=\text{Fun}^\times(\text{Fin}_G^\text{op},\text{Top}).$$ In the language of Chapter \ref{4}, these are Lawvere-theoretic descriptions of equivariant spaces. The Lawvere theory for $\text{Top}^G$ is $(\text{Fin}^G)^\text{op}$. On the other hand, $\text{Top}_G$ does not arise from a Lawvere theory, but from the \emph{equivariant} Lawvere theory $\text{Fin}_G^\text{op}$. This is all discussed further in \S\ref{A.1}.

Compare to the situation for $R$-modules ($R$ a ring): the Lawvere theory for $\text{Mod}_R$ is $\text{Mod}_R^\text{fgf}$, the category of finitely generated free $R$-modules. In the language of \cite{Tabuada}, there are Morita equivalences of dg-categories $$BR\subseteq\text{Mod}_R^\text{fgf}\subseteq\text{Mod}_R,$$ which means that $\text{Mod}_R^\text{fgf}$ and $\text{Mod}_R$ are equivalent as noncommutative motives (the noncommutative motive associated to $R$).

Blumberg, Gepner, and Tabuada \cite{BGT} show that $\text{Mod}_R$ and $_R\text{Mod}$ are dual as noncommutative motives; this is essentially a dramatized version of the Eilenberg-Watts Theorem.

Since $^G\text{Fin}$ is the Lawvere theory for naive left $G$-spaces, and $\text{Fin}_G$ the equivariant Lawvere theory for genuine right $G$-spaces, we hypothesize that Theorem \ref{PerfPair} can be extended as follows:

\begin{conjecture}
\label{DualConj}
The $\infty$-categories $^G\text{Top}$ (of naive left $G$-spaces) and $\text{Top}_G$ (of genuine right $G$-spaces) are dual as noncommutative motives over $\mathbb{F}_1$.

That is, genuine equivariant $G$-spaces are modules over a formal dual $\mathbb{F}_1[G]^\vee$.
\end{conjecture}

\noindent Unfortunately, this conjecture is not even rigorously stated, pending a definition of $\text{NMo}_{\mathbb{F}_1}$.

\subsubsection{Organization}
\noindent In \S\ref{5.1}, we prove that $\text{Fin}_G$ and $^G\text{Fin}$ are dual, in the sense of a perfect pairing in $\text{Cat}_\text{pce}$.

Then in \S\ref{5.2}, we prove that the global effective Burnside category is a full subcategory of $\text{Cat}_\text{pce}$ in \emph{two ways}; one inclusion is covariant, and the other contravariant. As a result, product-preserving functors $\text{Cat}_\text{pce}\rightarrow\text{Sp}$ give rise to two different families of $G$-spectra. One of these families assembles into a global equivariant spectrum, and the other into a coglobal equivariant spectrum.

We will also see how some familiar families of equivariant spectra arise this way.

\section{Weak duality}\label{5.1}
\noindent Our goal in this section is to show that $\text{Fin}_G$ and $^G\text{Fin}$ exhibit a duality, as categories with pullbacks and finite coproducts.

In fact, we will prove something more general. Let $^G\text{Fin}_H$ denote the category of finite $(G,H)$-bisets whose left $G$-action is free. In other words, $$^G\text{Fin}_H\cong\text{Fun}(BH,\null^G\text{Fin}).$$

\begin{definition}
Say that $X\in\null^G\text{Fin}_H$ is \emph{separable} if it is of the form $G\times T$, where $T\in\text{Fin}_H$ and the $(G,H)$-biset structure is $g(g^\prime,t)h=(gg^\prime,th)$. Let $^G\text{Fin}_H^\text{sep}$ denote the full subcategory of $^G\text{Fin}_H$ spanned by separable bisets.
\end{definition}

\noindent We will show that $^H\text{Fin}_G$ is dual to $^G\text{Fin}_H^\text{sep}$. First we discuss what we mean by duality.

\begin{definition}
Let $\text{Cat}_\text{pce}$ denote the 2-category of categories which admit pullbacks and finite coproducts, and those functors between them which preserve pullbacks, finite coproducts, \emph{and epimorphisms}.

If $\mathcal{C},\mathcal{D}\in\text{Cat}_\text{pce}$, let $\text{Fun}_\text{pce}(\mathcal{C},\mathcal{D})$ denote the full subcategory of $\text{Fun}(\mathcal{C},\mathcal{D})$ spanned by such functors.
\end{definition}

\begin{definition}
Suppose $\mathcal{C},\mathcal{D}\in\text{Cat}_\text{pce}$, and $F:\mathcal{C}\times\mathcal{D}\rightarrow\text{Fin}$ is a functor which preserves pullbacks, finite coproducts, and epimorphisms in each variable separately.

We say that $F$ is a \emph{perfect pairing} if the induced functors $\mathcal{C}\rightarrow\text{Fun}_\text{pce}(\mathcal{D},\text{Fin})$ and $\mathcal{D}\rightarrow\text{Fun}_\text{pce}(\mathcal{C},\text{Fin})$ are both equivalences of categories.
\end{definition}

\noindent Now we will describe a perfect pairing $$-\times_{(H,G)}-:\null^H\text{Fin}_G\times\null^G\text{Fin}_H^\text{sep}\rightarrow\text{Fin}$$ for finite groups $G,H$. If $X\in\null^H\text{Fin}_G$, and $Y\in\null^G\text{Fin}_H$, we write $X\times_{(H,G)}Y=X\times Y/\sim$, where the equivalence relation $\sim$ is generated by $(xg,y)\sim(x,gy)$ and $(hx,y)\sim(x,yh)$.

\begin{example}
There is a $(G,H)$-biset $H^\text{op}\times G\in\null^H\text{Fin}_G$. The actions are given by $h(a,b)g=(ha,bg)$, and $(H^\text{op}\times G)\times_{(H,G)}X\cong X$.
\end{example}

\begin{example}
\label{ExGBi}
Any finite group $G$ is itself endowed with the structure of a canonical $(G,G)$-biset $G^\text{bi}$ with the actions given by left multiplication on the left and right multiplication on the right. Then $G^\text{bi}\times_{(G,G)}G^\text{bi}$ is the set of conjugacy classes of $G$.

More generally, for any group homomorphism $\phi:G\rightarrow H$, there is a biset $H_\phi\in\null^H\text{Fin}_G$ with underlying set $H$, left action given by left multiplication, and right action given by $ag=a\phi(g)$. If $X\in\null^G\text{Fin}_H$, we say that $x,y\in X$ are $\phi$-conjugate if $x=g^{-1}y\phi(g)$ for some $g\in G$. Then $H_\phi\times_{(H,G)}X$ is the set of $\phi$-conjugacy classes of $X$.
\end{example}

\begin{theorem}
\label{PerfPair}
For any finite groups $G$ and $H$, the functor $$-\times_{(H,G)}-:\null^H\text{Fin}_G\times\null^G\text{Fin}_H^\text{sep}\rightarrow\text{Fin}$$ is a perfect pairing.
\end{theorem}

\begin{corollary}
\label{PerfPairCor}
Choosing $H$ to be the trivial group, we see that $$-\times_G -:\text{Fin}_G\times\null^G\text{Fin}\rightarrow\text{Fin}$$ is a perfect pairing.
\end{corollary}

\begin{remark}
Notice there is an asymmetry in the theorem statement. For \emph{every} $X\in\null^G\text{Fin}_H$, there is a functor $-\times_{(H,G)}X:\null^H\text{Fin}_G\rightarrow\text{Fin}$. The theorem asserts two things:
\begin{itemize}
\item this functor preserves pullbacks, finite coproducts, and epimorphisms if and only if $X$ is separable;
\item all such functors arise from a separable biset $X$.
\end{itemize}
That is, in situations like $X=G_\phi$ (Example \ref{ExGBi}), $-\times_{(H,G)}X$ does not preserve pullbacks.
\end{remark}

\noindent The rest of the section constitutes a proof of Theorem \ref{PerfPair}. The key is that any functor which preserves pullbacks, finite coproducts, and epimorphisms also preserves torsors. The category $^H\text{Fin}_G$ is \emph{generated by torsors} in a particular sense, which ensures that any functor $F\in\text{Fun}_\text{pce}(\null^H\text{Fin}_G,\text{Fin})$ is completely determined by $F(H\times G)\in\null^G\text{Fin}_H^\text{sep}$.

In \S\ref{5.1.1}, we introduce torsors in a general context. In \S\ref{5.1.2}, we complete the proof of Theorem \ref{PerfPair}.

\subsection{Torsors}\label{5.1.1}
\begin{definition}
\label{KTors}
Suppose $\mathcal{C}$ is a category which admits pullbacks and finite coproducts, and $K$ is a finite group. A \emph{$K$-torsor} in $\mathcal{C}$ is an object $X\in\mathcal{C}$ with a $K$-action, and an epimorphism $X\xrightarrow{f}Y$ which is invariant under the $K$-action on $X$, such that the following square is a pullback $$\xymatrix{
X^{\amalg|K|}\ar[r]^{\pi}\ar[d]_{\tilde{\pi}} &X\ar[d]^f \\
X\ar[r]_f &Y.
}$$ The maps $\pi$ and $\tilde{\pi}$ are defined as follows: $\pi$ acts on each copy of $X$ by the identity, while $\tilde{\pi}$ acts on the copy of $X$ corresponding to $k\in K$ via multiplication by $k$.
\end{definition}

\begin{remark}
If $F:\mathcal{C}\rightarrow\mathcal{D}$ is a functor which preserves pullbacks, finite coproducts, and epimorphisms, then $F$ sends $K$-torsors to $K$-torsors.
\end{remark}

\begin{example}
\label{ExFinTors}
The $K$-torsors in Fin are precisely the maps $X\rightarrow X/K$, where $X$ is a free right $K$-set. (This is the reason for the term `torsor'; each fiber of the map is a $K$-torsor.)

Indeed, if $X\xrightarrow{f}Y$ is a $K$-torsor in Fin, since $X^{\amalg|K|}\cong X\times_Y X$, the fiber $X_y$ over $y\in Y$ must satisfy $|X_y|^2=|K||X_y|$. By assumption, $f$ is an epimorphism, so $|X_y|$ is nonempty. Therefore, $X_y$ has cardinality $K$. Moreover, the action of $K$ on $X_y$ must be free, because otherwise if there were $a\in K$ and $x\in X$ with $ax=x$, then $X^{\amalg|K|}\xrightarrow{\pi,\tilde{\pi}}X\times_Y X$ would identify $(x,1)$ with $(x,a)$ and therefore fail to be injective.
\end{example}

\begin{example}
In the category of topological spaces, if we replace the term `epimorphism' by `fiber bundle', then a $K$-torsor in Top is exactly a principal $K$-bundle.
\end{example}

\begin{example}
Suppose $K/F$ is a field extension. Then $\text{Spec}(K)\rightarrow\text{Spec}(F)$ is a torsor in the category of schemes if and only if $K/F$ is a Galois extension.
\end{example}

\begin{lemma}
\label{EpiTors}
Every epimorphism $f:H^\text{op}\times G\rightarrow O$ in $^H\text{Fin}_G$ is a $K$-torsor for some $K$.
\end{lemma}

\noindent This lemma is the key to Theorem \ref{PerfPair}. It means that $^H\text{Fin}_G$ is \emph{generated by torsors} in the sense that every biset is a disjoint union of orbits, and every map of orbits $O_1\rightarrow O_2$ extends to a diagram $O\rightarrow O_1\rightarrow O_2$ where both $O\rightarrow O_1$ and $O\rightarrow O_2$ are torsors.

\begin{proof}
Denote $\ast=f(1,1)$, and let $K=f^{-1}(\ast)$. Notice that $H^\text{op}\times G$ has the structure of a group, and $K$ is a subgroup (the subgroup of all $(h,g)$ with $h\ast g=\ast$). Also notice that $f(x,y)=f(x^\prime,y^\prime)$ if and only if $(x^{-1}x^\prime,y^\prime y^{-1})\in K$. Therefore, the fibers of $f$ are exactly right cosets $K(x,y)$, and $O=(H^\text{op}\times G)/K$\footnote{Notice $K$ and $O$ determine each other, but not all $K$ arise from an orbit $O$.}.

This means that $f$ is invariant under the action of $K$ on $H^\text{op}\times G$ by left multiplication (as a subgroup). All that remains to prove that $f$ is a $K$-torsor is to show that the square of Definition \ref{KTors} is a pullback, and we can check this just on the level of finite sets. For any $(x,y)\in H^\text{op}\times G$, we check that the part of the square lying over $f(x,y)$ is a fiber product. But this takes the form $$\xymatrix{
K(x,y)^{\amalg|K|}\ar[r]\ar[d] &K(x,y)\ar[d] \\
K(x,y)\ar[r] &\ast,
}$$ which is a pullback because $K(x,y)$ is a $K$-torsor in the usual sense.
\end{proof}

\subsection{Proof of weak duality}\label{5.1.2}
\begin{lemma}
The functor $-\times_{(H,G)}-:\null^H\text{Fin}_G\times\null^G\text{Fin}_H^\text{sep}\rightarrow\text{Fin}$ preserves pullbacks, finite coproducts, and epimorphisms separately in each variable.
\end{lemma}

\begin{proof}
It is immediate that $-\times_{(H,G)}-$ preserves finite coproducts and epimorphisms. Since we are only considering separable bisets in the second variable, the functor takes the form $$X\times_{(H,G)}(G\times T)\cong T\times_H X.$$ But $T\times_H X$ preserves pullbacks in both $T$ and $X$, when $X$ is restricted to \emph{free} $H$-sets.
\end{proof}

\begin{lemma}
\label{GFinHsep}
If $\mathcal{C}$ admits pullbacks and finite coproducts, define the functor $$\Theta:\text{Fun}_\text{pce}(\null^G\text{Fin}_H^\text{sep},\mathcal{C})\rightarrow\text{Fun}(BG,\text{Fun}_\text{pce}(\text{Fin}_H,\mathcal{C})),$$ by restriction along $G\times -:\text{Fin}_H\rightarrow\null^G\text{Fin}_H^\text{sep}$, with $G$-action on the restricted functor given by the action on the first coordinate of $G\times -$. Then $\Theta$ is an equivalence of categories.
\end{lemma}

\begin{example}
\label{GFinBG}
Choosing $H=\ast$, $\text{Fun}_\text{pce}(\null^G\text{Fin},\mathcal{C})\cong\text{Fun}(BG,\mathcal{C})$.
\end{example}

\begin{proof}
First, we show that $\Theta$ is essentially surjective. Suppose $F:\text{Fin}_H\rightarrow\mathcal{C}$ is a functor which preserves pullbacks, finite coproducts, and epimorphisms, and is acted on (via natural transformations) by $G$. Define $F^\prime:\null^G\text{Fin}_H^\text{sep}\rightarrow\mathcal{C}$ as follows:
\begin{itemize}
\item $F^\prime(G\times T)=F(T)$.
\item If $f:G\times S\rightarrow G\times T$ and $T$ is an $H$-orbit, then $f$ takes the form $g\times h$ where $g:G\rightarrow G$ is right multiplication by an element $g\in G$ and $h:S\rightarrow T$ is a map of $H$-sets. In this case $F^\prime(f)$ is the composition of $F(h)$ with the action of $g$ on $F$. (We can compose in either order since the action of $g$ on $F$ is a natural transformation.)
\item If $f:G\times S\rightarrow G\times T$ and $T$ is not an $H$-orbit, then $f$ decomposes as a coproduct of maps into orbits, so $F^\prime$ extends uniquely to all of $^G\text{Fin}_H^\text{sep}$ subject to the constraint that it preserve finite coproducts.
\end{itemize}

\noindent Note that $F^\prime$ preserves finite coproducts, epimorphisms, and pullbacks (which we can check over one orbit at a time), and $\Theta(F^\prime)\cong F$ by construction. Therefore $\Theta$ is essentially surjective.

Now we prove $\Theta$ is fully faithful. Suppose $F_1,F_2:\null^G\text{Fin}_H^\text{sep}\rightarrow\mathcal{C}$ preserve pullbacks, finite coproducts, and epimorphisms. Suppose $\zeta,\eta:F_1\rightarrow F_2$ are two natural transformations with $\Theta(\zeta)=\Theta(\eta)$. The latter two are natural transformations of functors $BG\rightarrow\text{Fun}_\text{pce}(\text{Fin}_H,\text{Fin})$, so they are determined by their value on the single object $\ast\in BG$, and this value is a natural transformation of functors $\text{Fin}_H\rightarrow\text{Fin}$. But these two natural transformations $\Theta(\zeta)_\ast$ and $\Theta(\eta)_\ast$ are the same, so they have the same value at each $T\in\text{Fin}_H$, which means $\zeta,\eta$ have the same value at each $G\times T\in\null^G\text{Fin}_H^\text{sep}$. Thus $\zeta=\eta$, so $\Theta$ is faithful.

On the other hand, suppose $\eta:\Theta(F_1)\rightarrow\Theta(F_2)$ is a natural transformation. Define $\eta^\prime:F_1\rightarrow F_2$ by $\eta^\prime_{G\times T}=(\eta_\ast)_T$, noting again that $\eta_\ast$ is a natural transformation of functors $\text{Fin}_H\rightarrow\text{Fin}$. Then $\eta^\prime$ is a natural transformation by the same argument that we used to show $\Theta$ is essentially surjective, and $\Theta(\eta^\prime)=\eta$ by construction, so $\Theta$ is full. This completes the proof.
\end{proof}

\begin{lemma}
\label{LemmaPhi}
The functor $-\times_{(H,G)}-$ induces an equivalence of categories $$\Phi:\null^G\text{Fin}_H^\text{sep}\rightarrow\text{Fun}_\text{pce}(\null^H\text{Fin}_G,\text{Fin}).$$
\end{lemma}

\begin{proof}
First we prove that $\Phi$ is fully faithful. If $f:X\rightarrow Y$ is a morphism in $^G\text{Fin}_H^\text{sep}$, then $\Phi(f)_{H^\text{op}\times G}=f$. Therefore, if $\Phi(f)=\Phi(g)$, necessarily $f=g$, so $\Phi$ is faithful.

On the other hand, if $\eta:\Phi(X)\rightarrow\Phi(Y)$ is a natural transformation, let $$f=\eta_{H^\text{op}\times G}:X\rightarrow Y.$$ For any orbit $O\in\null^H\text{Fin}_G$, the following square commutes by naturality $$\xymatrix{
X\ar[r]\ar[d]_f &X\times_{{G,H}}O\ar[d]^{\eta_O} \\
Y\ar[r] &Y\times_{(G,H)}O.
}$$ Since the horizontal maps are epimorphisms, $\eta_O$ is the unique morphism which fills in the square, and therefore $\eta_O=f\times_{(G,H)}O$ for all orbits $O$. It follows that $\eta=\Phi(f)$, so $\Phi$ is full.

It only remains to show that $\Phi$ is essentially surjective. Suppose $F:\null^H\text{Fin}_G\rightarrow\text{Fin}$ preserves pullbacks, finite coproducts, and epimorphisms. The canonical $H\times G^\text{op}$-action on $H^\text{op}\times G\in\null^H\text{Fin}_G$ endows $X=F(H^\text{op}\times G)$ with the structure of a $(G,H)$-biset.

Moreover, $H^\text{op}\times G\rightarrow H^\text{op}\times 1$ is a $G^\text{op}$-torsor in $^H\text{Fin}_G$ by Lemma \ref{EpiTors}, so the left $G$-action on $X$ is free by Example \ref{ExFinTors}. Hence $X\in\null^G\text{Fin}_H$.

We claim $X$ is separable. If not, there is some $x\in X$, $g\in G$, and $h\in H$ such that $gxh=x$. Moreover, there is an orbit $O\in\null^H\text{Fin}_G$ and a map $H^\text{op}\times G\rightarrow O$ which identifies $(h,g)$ with $(1,1)$. By Lemma \ref{EpiTors}, this means that $(g,h)$ actually acts freely on $X$, contradicting the assumption that $gxh=x$. Therefore $X$ is separable.

Now we claim that $F\cong-\times_{(H,G)}X$ are naturally isomorphic. First we show they agree on objects. In fact, since $F$ preserves finite coproducts, it suffices to prove they agree on orbits $T$; that is, when there is an epimorphism $H^\text{op}\times G\rightarrow T$. By Lemma \ref{EpiTors}, any such epimorphism is a $K$-torsor for some $K$, so $F(T)\cong X/K\cong T\times_{(H,G)}X$.

Finally, we claim that for any $f:S\rightarrow T$, $F(f)=f\times_{(H,G)}X$. This will complete the proof by showing that $F$ is in the image of $\Phi$ (therefore $\Phi$ is essentially surjective). Again, it suffices to prove this when $S$ and $T$ are orbits, since $F$ preserves coproducts. But then there are maps $H^\text{op}\times G\xrightarrow{g}S\xrightarrow{f}T$, and each of $g$ and $fg$ are torsors. Hence $F$ agrees with $-\times_{(H,G)}X$ on each of $g$ and $fg$ (Lemma \ref{EpiTors}). Since $F(g)$ is an epimorphism, $F(g)$ and $F(fg)$ uniquely determine $F(f)$, which must also therefore agree with $-\times_{(H,G)}X$. Therefore $F$ is in the image of $\Phi$, and we are done.
\end{proof}

\begin{proof}[Proof of Theorem \ref{PerfPair}]
We just need to show that each of the functors induced by $-\times_{(H,G)}-$ is an equivalence: $$\Phi:\null^G\text{Fin}_H^\text{sep}\rightarrow\text{Fun}_\text{pce}(\null^H\text{Fin}_G,\text{Fin})$$ $$\Psi:\null^H\text{Fin}_G\rightarrow\text{Fun}_\text{pce}(\null^G\text{Fin}_H^\text{sep},\text{Fin}).$$ In the case of $\Phi$, this is precisely Lemma \ref{LemmaPhi}. Therefore, we need only show that $\Psi$ is an equivalence. In fact, there is a commutative square $$\xymatrix{
^H\text{Fin}_G\ar[r]^-{\Psi}\ar[d]_{\sim} &\text{Fun}_\text{pce}(\null^G\text{Fin}_H^\text{sep},\text{Fin})\ar[d]^{\Theta} \\
\text{Fun}(BG,\null^H\text{Fin})\ar[r]_-{\Phi} &\text{Fun}(BG,\text{Fun}_\text{pce}(\text{Fin}_H,\text{Fin})).
}$$ The functor $\Theta$ is an equivalence by Lemma \ref{GFinHsep}, and $\Phi$ is an equivalence by Lemma \ref{LemmaPhi}. Therefore $\Psi$ is an equivalence.
\end{proof}

\begin{remark}
Glasman's \emph{epiorbital categories} \cite{Glasman} axiomatize many of the properties of $^H\text{Fin}_G$ we have used in this proof. It is possible that some version of Theorem \ref{PerfPair} can be generalized to that setting.
\end{remark}

\section{Global equivariant stable homotopy theory}\label{5.2}
\noindent By Example \ref{GFinBG}, $\text{Fun}_\text{pce}(\null^H\text{Fin},\null^G\text{Fin})\cong\null^G\text{Fin}_H$. Therefore, the full subcategory of $\text{Cat}_\text{pce}$ spanned by $^G\text{Fin}$ ($G$ ranges over finite groups) admits the following alternate description:
\begin{itemize}
\item an object is a finite group;
\item a morphism from $H$ to $G$ is a finite $(G,H)$-biset whose left $G$-action is free;
\item a 2-morphism is a map of bisets.
\end{itemize}

\noindent This is the global effective Burnside category of Schwede \cite{Schwede}, which we regard as a full subcategory of the larger (semiadditive) global effective Burnside category of Miller \cite{MillerBurnglo}. In this section, we will expand upon this observation to show that any product-preserving functor $F:\text{Cat}_\text{pce}\rightarrow\text{Sp}$ gives rise to two families of equivariant spectra: one by evaluating $F$ at $^G\text{Fin}\in\text{Cat}_\text{pce}$, and the other by evaluating $F$ at $\text{Fin}_G\in\text{Cat}_\text{pce}$.

We encourage the reader to begin by reading about global Burnside categories in \S\ref{A.3}. For more details on global equivariant homotopy theory, see Schwede's book \cite{Schwede}. He treats global equivariant homotopy theory for compact Lie groups, but the version for finite groups is very similar.

In \S\ref{5.2.1}, we prove that the global Burnside category (not just the smaller orbital version) is a full subcategory of $\text{Cat}_\text{pce}$, concluding that a finite-product-preserving functor $\text{Cat}_\text{pce}\rightarrow\text{Sp}$ gives rise to a family of $G$-spectra, one for each $G$. These assemble into a coglobal equivariant spectrum.

In \S\ref{5.2.2}, we prove that the \emph{opposite of the} global Burnside category is also a full subcategory of $\text{Cat}_\text{pce}$, utilizing the duality of Theorem \ref{PerfPair}. We conclude that such a functor gives rise to a second family of $G$-spectra, which assemble into a global equivariant spectrum.

Along the way, we will see that many basic examples in equivariant homotopy theory arise as functors out of $\text{Cat}_\text{pce}$.

\subsection{Coglobal spectra}\label{5.2.1}
\noindent Suppose that $\mathcal{G}$ is a finite groupoid (a groupoid with finitely many objects and finitely many morphisms). Let $^\mathcal{G}\text{Fin}$ denote the category obtained from $\mathcal{G}$ by freely adjoining finite coproducts.

\begin{remark}
Any finite groupoid is equivalent to a disjoint union of finite groups, via $\mathcal{G}\cong BG_1\amalg\cdots\amalg BG_n$. In this case, $$^\mathcal{G}\text{Fin}\cong\null^{G_1}\text{Fin}\times\cdots\times\null^{G_n}\text{Fin}.$$ That is, an object of $^\mathcal{G}\text{Fin}$ can be regarded as an $n$-tuple of finite sets, with $G_i$ acting freely on the left of the $i^\text{th}$ finite set.
\end{remark}

\begin{definition}
If $\mathcal{G}$ and $\mathcal{H}$ are finite groupoids, write $$\text{Fin}_\mathcal{G}=\text{Fun}(\mathcal{G},\text{Fin})$$ $$^\mathcal{H}\text{Fin}_\mathcal{G}=\text{Fun}(\mathcal{G},\null^\mathcal{H}\text{Fin}).$$
\end{definition}

\begin{remark}
\label{RmkBkGpoid}
If $\mathcal{H}\cong BH_1\amalg\cdots\amalg BH_m$ and $\mathcal{G}\cong BG_1\amalg\cdots\amalg BG_n$, then $$\text{Fin}_\mathcal{G}\cong\text{Fin}_{G_1}\times\cdots\times\text{Fin}_{G_n}$$ $$^\mathcal{H}\text{Fin}_\mathcal{G}\cong\bigtimes_{(i,j)\in[1,m]\times[1,n]}\null^{H_i}\text{Fin}_{G_j}.$$ That is, we may think of an object of $\text{Fin}_\mathcal{G}$ as an $n$-tuple of finite sets with right group actions, and an object of $^\mathcal{H}\text{Fin}_\mathcal{G}$ as an $m\times n$-matrix of finite sets with both (free) left and (arbitrary) right actions.
\end{remark}

\begin{proposition}
\label{PropCoglo}
For finite groupoids $\mathcal{G}$ and $\mathcal{H}$,  $$^\mathcal{H}\text{Fin}_\mathcal{G}\cong\text{Fun}_\text{pce}(\null^\mathcal{G}\text{Fin},\null^\mathcal{H}\text{Fin}).$$
\end{proposition}

\begin{proof}
Combine Remark \ref{RmkBkGpoid} with Example \ref{GFinBG}.
\end{proof}

\noindent Recall from the appendix (\S\ref{A.3}) that the global effective Burnside 2-category $\text{Burn}_\text{glo}^\text{eff}$ is the 2-category of spans of finite groupoids, with all leftward maps required to be discrete fibrations.

A \emph{coglobal Mackey functor} is a finite-product-preserving functor $\text{Burn}_\text{glo}^\text{eff}\rightarrow\mathcal{C}$.

For each finite group $G$, there is a product-preserving functor $\text{Burn}_G^\text{eff}\rightarrow\text{Burn}_\text{glo}^\text{eff}$ given by the \emph{action groupoid} construction. Given a coglobal Mackey functor, precomposition by this functor produces a $G$-Mackey functor for each $G$.

For example, if $\mathcal{C}=\text{Sp}$, then any coglobal equivariant spectrum gives rise to genuine equivariant $G$-spectra for each finite group $G$.

\begin{theorem}
\label{CogloCat}
The full subcategory of $\text{Cat}_\text{pce}$ spanned by $^\mathcal{G}\text{Fin}$ (as $\mathcal{G}$ ranges over finite groupoids) is equivalent to $\text{Burn}_\text{glo}^\text{eff}$.
\end{theorem}

\begin{proof}
By Example \ref{GFinBG}, the full subcategory of $\text{Cat}_\text{pce}$ spanned by $^G\text{Fin}$ ($G$ ranges over finite groups) is Schwede's global Burnside category: that is, the objects are finite groups, and the morphisms from $G$ to $H$ are objects of $^H\text{Fin}_G$. This is also equivalent to the full subcategory $\text{Burn}_\text{glo}^\text{orb}$ of $\text{Burn}_\text{glo}^\text{eff}$ spanned by finite groups (single object groupoids) \cite{MillerBurnglo}.

By the explicit description of Remark \ref{RmkBkGpoid} and Proposition \ref{PropCoglo}, the full subcategory of $\text{Cat}_\text{pce}$ spanned by $^\mathcal{G}\text{Fin}$ ($\mathcal{G}$ ranges over finite groupoids) is semiadditive, and therefore the fully faithful functor $\text{Burn}_\text{glo}^\text{orb}\rightarrow\text{Cat}_\text{pce}$ extends to a semiadditive fully faithful functor $\text{Burn}_\text{glo}^\text{eff}\rightarrow\text{Cat}_\text{pce}$, as desired.
\end{proof}

\subsubsection{Constructing coglobal spectra}
\begin{corollary}
If $F:\text{Cat}_\text{pce}\rightarrow\mathcal{C}$ is a finite-product-preserving functor, then $F(\text{Fin})$ has the structure of a $G$-Mackey functor (valued in $\mathcal{C}$) for each $G$, all of which assemble into a coglobal Mackey functor.

For example, if $\mathcal{C}=\text{Sp}$, then we obtain genuine equivariant $G$-spectra for each $G$, which assemble into a coglobal equivariant spectrum. The $G$-fixed point spectrum is $F(\null^G\text{Fin})$.
\end{corollary}

\begin{example}
\label{KCoglo}
For $\mathcal{C}\in\text{Cat}_\text{pce}$, let $K(\mathcal{C})$ denote the algebraic K-theory of the symmetric monoidal category $\mathcal{C}^\amalg$.

Then $K:\text{Cat}_\text{pce}\rightarrow\text{Sp}$ is a product-preserving functor, which induces a coglobal equivariant spectrum $E\mathcal{G}$. Since $^G\text{Fin}$ is freely generated by $BG$ under coproducts, we find $K(\null^G\text{Fin})\cong\Sigma^\infty_{+}BG$. Therefore, the associated $G$-spectrum is $\Sigma^\infty_{+}EG$.
\end{example}

\begin{example}
\label{SpCoglo}
For $\mathcal{C}\in\text{Cat}_\text{pce}$, let $\mathcal{P}_\amalg(\mathcal{C})=\text{Fun}^\times(\mathcal{C}^\text{op},\text{Top})$ denote algebraic presheaves. The resulting functor $$\mathcal{P}_\amalg:\text{Cat}_\text{pce}\rightarrow\text{Pr}^\text{L}$$ gives rise to a coglobal Mackey functor $^\mathcal{G}\text{Top}$ valued in presentable $\infty$-categories.

These in turn restrict to symmetric monoidal $G$-equivariant $\infty$-categories for each $G$ (Definition \ref{SymMonEq}).

Since $^G\text{Fin}$ is the free cocartesian monoidal $\infty$-category on $BG$, $^G\text{Top}\cong\text{Fun}(BG^\text{op},\text{Top})$, the $\infty$-category of \emph{naive equivariant left $G$-spectra}.

Thus, $^\mathcal{G}\text{Top}$ is the coglobal Mackey functor encoding `naive equivariant homotopy theory'.
\end{example}

\subsection{Global spectra}\label{5.2.2}
\noindent By Theorem \ref{CogloCat}, $\text{Burn}_\text{glo}^\text{eff}$ is equivalent to the full subcategory of $\text{Cat}_\text{pce}$ spanned by $^\mathcal{G}\text{Fin}$. Theorem \ref{PerfPair} gives a kind of duality between $^\mathcal{G}\text{Fin}$ and $\text{Fin}_\mathcal{G}$. Combining these:

\begin{theorem}
\label{GloCat}
The full subcategory of $\text{Cat}_\text{pce}$ spanned by $\text{Fin}_\mathcal{G}$ (as $\mathcal{G}$ ranges over finite groupoids) is equivalent to $(\text{Burn}_\text{glo}^\text{eff})^\text{op}$.
\end{theorem}

\noindent The key is the following lemma, which is dual to Proposition \ref{PropCoglo}.

\begin{lemma}
\label{PropGlo}
For any finite groupoids $\mathcal{G}$ and $\mathcal{H}$, there is an equivalence of categories $$^\mathcal{G}\text{Fin}_\mathcal{H}\cong\text{Fun}_\text{pce}(\text{Fin}_\mathcal{G},\text{Fin}_\mathcal{H}).$$
\end{lemma}

\begin{proof}
By Proposition \ref{PropCoglo}, $$\text{Fun}_\text{pce}(\text{Fin}_\mathcal{G},\text{Fin}_\mathcal{H})\cong\text{Fun}_\text{pce}(\text{Fin}_\mathcal{G},\text{Fun}_\text{pce}(\null^\mathcal{H}\text{Fin},\text{Fin})).$$ This in turn is equivalent to $$\text{Fun}_\text{pce}(\null^\mathcal{H}\text{Fin},\text{Fun}_\text{pce}(\text{Fin}_\mathcal{G},\text{Fin}))\cong\text{Fun}_\text{pce}(\null^\mathcal{H}\text{Fin},\null^\mathcal{G}\text{Fin})\cong\null^\mathcal{G}\text{Fin}_\mathcal{H},$$ with the last two steps again by Proposition \ref{PropCoglo}.
\end{proof}

\begin{proof}[Proof of Theorem \ref{GloCat}]
The proof is just like Theorem \ref{CogloCat}, substituting the lemma for Proposition \ref{PropCoglo}.
\end{proof}

\subsubsection{Constructing global spectra}
\begin{corollary}
If $F:\text{Cat}_\text{pce}\rightarrow\mathcal{C}$ is a finite-product-preserving functor, then $F(\text{Fin})$ has the structure of a $G$-Mackey functor (valued in $\mathcal{C}$) in a \emph{second} way for each $G$, all of which assemble into a \emph{global} Mackey functor.

For example, if $\mathcal{C}=\text{Sp}$, then we obtain genuine equivariant $G$-spectra for each $G$, which assemble into a global equivariant spectrum. The $G$-fixed point spectrum is $F(\text{Fin}_G)$.
\end{corollary}

\begin{example}
\label{KGlo}
Consider the algebraic K-theory functor $K:\text{Cat}_\text{pce}\rightarrow\text{Sp}$. In Example \ref{KCoglo}, we saw that $K$ gives rise to a coglobal equivariant spectrum $E\mathcal{G}$.

It also gives rise to a global equivariant spectrum $\mathbb{S}_\mathcal{G}$. By the equivariant Barratt-Priddy-Quillen theorem \cite{EqBPQ}, the associated $G$-spectrum is the $G$-equivariant sphere spectrum $\mathbb{S}_G$ (the algebraic K-theory of $\text{Fin}_G$).
\end{example}

\begin{example}
\label{SpGlo}
Consider the algebraic presheaves functor $\mathcal{P}_\amalg:\text{Cat}_\text{pce}\rightarrow\text{Pr}^\text{L}$. In Example \ref{SpCoglo}, we saw that $\mathcal{P}_\amalg$ gives rise to a coglobal categorical Mackey functor $^\mathcal{G}\text{Top}$ which encodes naive equivariant homotopy theory.

It also gives rise to a global categorical Mackey functor $\text{Top}_\mathcal{G}$. The associated $G$-equivariant categorical Mackey functor is $\text{Top}_G$, the equivariant symmetric monoidal $\infty$-category of genuine equivariant right $G$-spaces (Definition \ref{ElmenDef}).

In this way, $\text{Top}_G$ is the global Mackey functor encoding `genuine equivariant homotopy theory'.
\end{example}

\begin{remark}
The last example demonstrates that $\text{Cat}_\text{pce}$ includes simultaneously information about genuine equivariant homotopy theory (determined by finite $G$-sets $\text{Fin}_G$) and naive equivariant homotopy theory (determined by finite \emph{free} $G$-sets $^G\text{Fin}$).

It also includes information about \emph{biequivariant} homotopy theory. If $G$ and $H$ are two finite groups, then we may regard $\mathcal{P}_\amalg(\null^G\text{Fin}_H)$ as the $\infty$-category of spaces endowed (simultaneously and compatibly) with a (naive) right $G$-action and a (genuine) left $H$-action. See Definition \ref{BieqDef}.
\end{remark}

%\chapter{Lawvere Theories (combinatorially)}
%\section{The span construction}
%\section{The bispan construction}
%\subsection{Twisted functors}
%\subsection{Indexed spans}
%\section{Lawvere theories via categorified group cohomology}
%\subsection{Lawvere machines}
%\subsection{Big equivariant categories}
%\section{Applications}
%\subsection{Tambara functors}
%\subsection{Equivariant operads}

\appendix
\chapter{Bestiary of equivariant homotopy theory}\label{A}
\noindent This appendix reviews the basic objects of equivariant homotopy theory and equivariant category theory, including spaces, spectra, $\infty$-categories, and symmetric monoidal $\infty$-categories. Each of these objects comes in a variety of flavors, including naive equivariant, genuine equivariant, biequivariant, and global equivariant.

The difference between some of these objects (particularly the naive and genuine equivariant) has been a source of confusion in homotopy theory for many years. We hope to address this confusion by collecting many of the existing constructions from equivariant homotopy theory, global equivariant homotopy theory, and equivariant category theory in one place.

A second goal is to demonstrate that the various notions of equivariant homotopy theory arise naturally out of the idea of an \emph{equivariant Lawvere theory}, via the following two general principles:

\begin{principle}
\label{PrEqLawv}
Each flavor of equivariant homotopy theory (naive $G$-equivariant, genuine $G$-equivariant, $(G,H)$-biequivariant, etc.) is built systematically from a category F of finite sets. $\text{F}^\text{op}$ acts as a \emph{universal Lawvere theory} for that flavor of equivariant homotopy theory.
\end{principle}

\begin{example}
Ordinary homotopy theory is built out of the category Fin of finite sets, and $\text{Fin}^\text{op}$ is the universal Lawvere theory. This is the subject of Chapter \ref{4}. For a finite group $G$:
\begin{itemize}
\item Genuine $G$-equivariant homotopy theory is built out of $\text{Fin}_G$ (finite $G$-sets).
\item Naive $G$-equivariant homotopy theory is built out of $\text{Fin}^G$ (finite \emph{free} $G$-sets).
\item Global equivariant homotopy theory is built out of the 2-category of finite groupoids.
\end{itemize}
\end{example}

\begin{principle*}[Principle \ref{LawvMachPr}]
Suppose $\mathcal{C}$ is an $\infty$-category with associated Lawvere theory $\mathcal{L}$, and $\mathcal{L}$ can be built by applying a combinatorial machine $\mathcal{M}$ to the category Fin of finite sets. Then the equivariant analogue of $\mathcal{C}$ is $$\text{Fun}^\times(\mathcal{M}(\text{F}),\text{Top}),$$ where F is the category of finite sets associated to the particular flavor of equivariant homotopy theory (as in Principle \ref{PrEqLawv}).
\end{principle*}

\begin{example*}[Example \ref{GMDef1}]
The Lawvere theory for connective spectra is the Burnside $\infty$-category, which is obtained from Fin via the virtual span construction $\text{Burn}=\text{Span}_\text{vir}(\text{Fin})$.

Let $\text{Burn}_G=\text{Span}_\text{vir}(\text{Fin}_G)$ denote the $\infty$-category of virtual spans of finite $G$-sets, and $\text{Burn}^G=\text{Span}_\text{vir}(\text{Fin}^G)$ the $\infty$-category of virtual spans of finite \emph{free} $G$-sets.

Then a genuine equivariant connective (right) $G$-spectrum is a product-preserving functor $\text{Burn}_G\rightarrow\text{Top}$, and a naive equivariant connective (right) $G$-spectrum is a product-preserving functor $\text{Burn}^G\rightarrow\text{Top}$.
\end{example*}

\noindent This appendix is \emph{not} intended to be a first introduction to equivariant homotopy theory, particularly because we do not include any theorems, proofs, or applications of the subject.

\subsubsection{Origins of equivariant homotopy theory}
\noindent Equivariant homotopy theory arises out of the following devastating observation: group actions on topological spaces do not interact well with homotopies.

Fix a finite group $G$ and a category or $\infty$-category $\mathcal{C}$. (The example that interests us is $\mathcal{C}=\text{Top}$.) Ordinarily, we would think of an object with a right (left) $G$-action as a functor $$BG^{(\text{op})}\rightarrow\mathcal{C},$$ where $BG$ is the category with a single object corresponding to $G$. By basic properties of the functor category, a morphism of $\text{Fun}(BG,\mathcal{C})$ is invertible (up to equivalence) if and only if the underlying morphism in $\mathcal{C}$ is invertible.

Now consider the following example: Let $EG$ be any contractible space with a free $G$-action. There is a continuous function between $G$-spaces $f:EG\rightarrow\ast$, where $\ast$ is a point with the trivial $G$-action. The underlying function is a homotopy equivalence (because $EG$ is contractible). The previous discussion suggests that there should be an equivariant map $f^{-1}:\ast\rightarrow BG$ inverse to $f$. However, any such function would correspond to a fixed point of $EG$, of which there are none! \\

\noindent The result is that, if we wish to study equivariant spaces which arise geometrically (for example, manifolds with group actions), we should not work in the $\infty$-category $\text{Fun}(BG,\text{Top})$ of \emph{naive} equivariant spaces, but in a more subtle $\infty$-category of \emph{genuine} equivariant spaces.

Due to Elmendorf's Theorem \cite{Elmendorf}, we know that this $\infty$-category takes the form $$\text{Top}_G\cong\text{Fun}(\text{Orb}_G^\text{op},\text{Top})\cong\text{Fun}^\times(\text{Fin}_G^\text{op},\text{Top}),$$ where $\text{Orb}_G$ is the category of $G$-orbits $G/H$ (as $H$ ranges over subgroups of $G$).

In effect, a genuine equivariant $G$-space $X$ includes (as specified data) an $H$-fixed point space $X^H$ for each subgroup $H<G$, along with restriction maps $X^H\rightarrow X^K$ for each subgroup inclusion $K<H<G$. We have to specify this data explicitly, because fixed point spaces are a 1-categorical notion which are not recorded by the $\infty$-category $\text{Fun}(BG,\text{Top})$ --- and it is precisely the data which distinguishes between (for example) $EG$ and $\ast$.

\subsubsection{Organization}
\noindent In \S\ref{A.1}, we discuss the various flavors of equivariant homotopy theory, guided by Principle \ref{PrEqLawv}. These include the naive equivariant, genuine equivariant, and biequivariant.

In \S\ref{A.2}, we introduce equivariant spaces and spectra, guided by Principle \ref{LawvMachPr}.

In \S\ref{A.3}, we discuss global equivariant homotopy theory. A standard reference for this material is Schwede's book \cite{Schwede}.

In \S\ref{A.4}, we discuss equivariant category theory. This is a subject introduced recently by Hill and Hopkins \cite{HillHopkins}, and Barwick et al. are preparing a book \cite{Barwick}.

\section{Flavors of equivariance}\label{A.1}
\noindent The basic building blocks of equivariant homotopy theory are categories of finite sets, endowed with various types of group actions. To each such category F of finite sets, there is an orbit category O, which is the full subcategory of objects which are indecomposable under disjoint union. In each case, F is freely obtained from O by formally adjoining finite coproducts. That is, if $\mathcal{C}$ is an $\infty$-category with finite products, then $$\text{Fun}(\text{O}^\text{op},\mathcal{C})\cong\text{Fun}^\times(\text{F}^\text{op},\mathcal{C}).$$ When $\mathcal{C}=\text{Top}$, this construction produces an $\infty$-category of \emph{equivariant spaces}.

\subsubsection{Naive equivariant homotopy theory}
\noindent If $G$ is a group, let $BG$ denote the corresponding category with a single object. 

\begin{definition}
A right $G$-object in an $\infty$-category $\mathcal{C}$ is a functor $BG\rightarrow\mathcal{C}$. A left $G$-object is a functor $BG^\text{op}\rightarrow\mathcal{C}$.
\end{definition}

\noindent Let $^G\text{Fin}$ denote the category of finite \emph{free} left $G$-sets. Then $BG$ is the orbit category for $^G\text{Fin}$.

\begin{definition}
A \emph{naive\footnote{Some authors use `Borel' instead of `naive'.} equivariant left $G$-space} is a functor $BG^\text{op}\rightarrow\text{Top}$, or equivalently a product-preserving functor $^G\text{Fin}^\text{op}\rightarrow\text{Top}$.
\end{definition}

\subsubsection{Genuine equivariant homotopy theory}
\noindent Let $\text{Fin}_G$ denote the category of finite right $G$-sets (not necessarily free). The corresponding orbit category $\text{Orb}_G$ is the full subcategory of $G$-orbits $G/H=\{Hg|g\in G\}$, as $H$ varies over all subgroups of $G$.

\begin{definition}
\label{ElmenDef}
A \emph{genuine equivariant right $G$-space} is a functor $\text{Orb}_G^\text{op}\rightarrow\text{Top}$, or equivalently a product-preserving functor $\text{Fin}_G^\text{op}\rightarrow\text{Top}$.
\end{definition}

\begin{remark}
\label{ElmenRmk}
We think of the functor $X:\text{Orb}_G^\text{op}\rightarrow\text{Top}$ as sending each orbit $G/H$ to a space of $H$-fixed points $X^H$. If $H<K$, there is a $G$-map $G/H\rightarrow G/K$ which corresponds to the inclusion of fixed point spaces $X^K\rightarrow X^H$, which we call a \emph{restriction}.

Traditionally, a genuine equivariant $G$-space is a topological space in the point-set sense (for example, a manifold) with a $G$-action. We put a model structure on the resulting category by declaring a continuous function compatible with the $G$-action is a weak equivalence precisely when it has an inverse up to homotopy.

Elmendorf's Theorem \cite{Elmendorf} asserts that this model category models the presentable $\infty$-category $\text{Fun}(\text{Orb}_G^\text{op},\text{Top})$.
\end{remark}

\subsubsection{Biequivariant homotopy theory}
\noindent The last two examples can be combined: If $G$ and $H$ are finite groups, let $$^G\text{Fin}_H=\text{Fun}(BH,\null^G\text{Fin})$$ denote the category of $(G,H)$-bisets such that the left $G$-action is free.

\begin{definition}
Let $^G\text{Orb}_H$ denote the following category: an object consists of a pair $(H/K,\phi)$ where $K$ is a subgroup of $H$ and $\phi:K\rightarrow G$ a group homomorphism. A morphism $(H/K,\phi)\rightarrow(H/K^\prime,\phi^\prime)$ consists of a map of $H$-orbits $H/K\rightarrow H/K^\prime$ (which means in particular that $K<K^\prime$) and a natural isomorphism $\eta:\phi\rightarrow\psi_K$ between the functors $\phi,\psi_K:BK\rightarrow BG$.
\end{definition}

\noindent There is a functor $^G\text{Orb}_H\rightarrow\null^G\text{Fin}_H$ which sends $(H/K,\phi)$ to the orbit $G\times H/K$, with action given by $$g(a,b)kh=(ga\phi(k),bh),$$ for any $g\in G$, $k\in K$, and $h\in H$. This exhibits $^G\text{Orb}_H$ as the orbit category of $\null^G\text{Fin}_H$ (a full subcategory).

\begin{definition}
\label{BieqDef}
A \emph{$(G,H)$-biequivariant space} is a functor $^G\text{Orb}_H^\text{op}\rightarrow\text{Top}$, or equivalently a product-preserving functor $^G\text{Fin}_H^\text{op}\rightarrow\text{Top}$.
\end{definition}

\begin{remark}
Biequivariant homotopy theory features, for example, in work of Bonventre and Pereira on genuine equivariant operads \cite{BonPer}. For them, a $G$-equivariant operad $\mathcal{O}$ includes (for each $n\geq 0$) the data of a space $\mathcal{O}(n)$ with a genuine $G$-action and a naive $\Sigma_n$-action.
\end{remark}

\section{Equivariance}\label{A.2}
\noindent For each flavor of equivariant homotopy theory, there are notions of equivariant space, equivariant spectrum, etc.

These are built following an informal principle.

\begin{principle}
\label{LawvMachPr}
Suppose $\mathcal{C}$ is an $\infty$-category with associated Lawvere theory $\mathcal{L}$, and $\mathcal{L}$ can be built by applying a combinatorial machine $\mathcal{M}$ to the category Fin of finite sets.

Then the equivariant analogue of $\mathcal{C}$ is $$\text{Fun}^\times(\mathcal{M}(\text{F}),\text{Top}),$$ where F is a category of finite sets from \S\ref{A.1} (such as $^G\text{Fin}$ for naive equivariant homotopy theory, or $\text{Fin}_G$ for genuine equivariant homotopy theory).
\end{principle}

\noindent In each case, $\mathcal{M}(\text{F})$ is an \emph{equivariant Lawvere theory}; that is, a cyclic $\text{F}^\text{op}$-module, or Lawvere theory relative to F.

\begin{example}
If $\mathcal{C}=\text{Top}$, the associated Lawvere theory is $\text{Fin}^\text{op}$. The combinatorial machine is the opposite category construction $\mathcal{M}(-)=(-)^\text{op}$.

If F is one of the categories of \S\ref{A.1} ($^G\text{Fin}$, $\text{Fin}_G$, etc.), then an equivariant space is a product-preserving functor $$\text{F}^\text{op}\rightarrow\text{Top}.$$ In general, if $\mathcal{C}$ is an $\infty$-category with finite products, a product-preserving functor $\text{F}^\text{op}\rightarrow\mathcal{C}$ is called a \emph{coefficient system}.
\end{example}

\begin{example}
If $\mathcal{C}=\text{CMon}_\infty$, the associated Lawvere theory is $\text{Burn}^\text{eff}$. The combinatorial machine is the (effective) span construction $\mathcal{M}(-)=\text{Span}(-)$.

If F is one of the categories of \S\ref{A.1}, an equivariant $\mathbb{E}_\infty$-space is a product-preserving functor $$\text{Span}(\text{F})\rightarrow\text{Top}.$$

In general, a product-preserving functor out of $\text{Span}(\text{F})$ is called a \emph{semi-Mackey functor}.
\end{example}

\noindent We usually write $\text{Burn}_G^\text{eff}=\text{Span}(\text{Fin}_G)$, $^G\text{Burn}^\text{eff}=\text{Span}(\null^G\text{Fin})$, etc. for these \emph{effective Burnside 2-categories}.

\begin{example}
\label{GMDef1}
If $\mathcal{C}=\text{Ab}_\infty$, the associated Lawvere theory is Burn. The combinatorial machine is the virtual span construction $\mathcal{M}(-)=\text{Span}(-)_\text{vir}$, obtained by group-completing Hom objects of $\text{Span}(-)$.

If F is one of the categories of \S\ref{A.1}, an equivariant connective spectrum is a product-preserving functor $$\text{Span}(\text{F})_\text{vir}\rightarrow\text{Top}.$$

In general, a product-preserving functor out of $\text{Span}(\text{F})_\text{vir}$ is called a \emph{Mackey functor}.
\end{example}

\noindent We write $\text{Burn}_G=\text{Span}(\text{Fin}_G)_\text{vir}$, etc. for these \emph{virtual Burnside $\infty$-categories} (or just Burnside $\infty$-categories).

\begin{definition}
\label{GMDef}
If F is one of the categories of \S\ref{A.1}, an equivariant spectrum (not necessarily connective) is a product-preserving functor $\text{Span}(\text{F})_\text{vir}\rightarrow\text{Sp}$ (a \emph{spectral Mackey functor}).
\end{definition}

\noindent Since Sp is an additive $\infty$-category and $\text{Span}(\text{F})_\text{vir}$ is the free additive $\infty$-category on the semiadditive 2-category $\text{Span}(\text{F})$, an equivariant spectrum is equivalently a product-preserving functor $$\text{Span}(\text{F})\rightarrow\text{Sp}.$$ This is often a more useful definition, since effective Burnside 2-categories can be described explicitly, while virtual Burnside $\infty$-categories are much more complicated.

These definitions are summarized in the following table:
\begin{flushleft}
\begin{tabular}{|c|c|c|c|}\hline
&naive left $G$- &genuine right $G$- &$(G,H)$-biequivariant \\ \hline
space &$\text{Fun}^\times(\null^G\text{Fin}^\text{op},\text{Top})$ &$\text{Fun}^\times(\text{Fin}_G^\text{op},\text{Top})$ &$\text{Fun}^\times(\null^G\text{Fin}_H^\text{op},\text{Top})$ \\ \hline
$\mathbb{E}_\infty$-space &$\text{Fun}^\times((\null^G\text{Burn}^\text{eff})^\text{op},\text{Top})$ &$\text{Fun}^\times((\text{Burn}^\text{eff}_G)^\text{op},\text{Top})$ &$\text{Fun}^\times((\null^G\text{Burn}_H^\text{eff})^\text{op},\text{Top})$ \\ \hline
connective &$\text{Fun}^\times(\null^G\text{Burn}^\text{op},\text{Top})$ &$\text{Fun}^\times(\text{Burn}_G^\text{op},\text{Top})$ &$\text{Fun}^\times(\null^G\text{Burn}_H^\text{op},\text{Top})$ \\ 
spectrum & & & \\ \hline
spectrum &$\text{Fun}^\times((\null^G\text{Burn}^{(\text{eff})})^\text{op},\text{Sp})$ &$\text{Fun}^\times((\null^G\text{Burn}_H^{(\text{eff})})^\text{op},\text{Sp})$ &$\text{Fun}^\times((\null^G\text{Burn}_H^{(\text{eff})})^\text{op},\text{Sp})$ \\ \hline
\end{tabular}
\end{flushleft}

\subsubsection{Genuine equivariant homotopy theory}
\noindent We have just asserted that genuine $G$-spaces are \emph{topological coefficient systems} (functors $\text{Fin}_G^\text{op}\rightarrow\text{Top}$) and genuine $G$-spectra are \emph{spectral Mackey functors} ($\text{Burn}_G\rightarrow\text{Sp}$). Actually, each of these statements is a deep theorem. The first, as we discussed in Remark \ref{ElmenRmk}, is due to Elmendorf \cite{Elmendorf}. The second is due to Guillou and May \cite{GMay2}.

In identifying a genuine $G$-spectrum with a spectral Mackey functor $E:\text{Burn}_G^\text{eff}\rightarrow\text{Sp}$, we think of $E(G/H)$ as the $H$-fixed points of $E$. If $H<K$, then the span $$\xymatrix{
 &G/H\ar[ld]\ar@{=}[rd] & \\
G/K &&G/H
}$$ is sent to the restriction $E^K\rightarrow E^H$, while the span $$\xymatrix{
 &G/H\ar[rd]\ar@{=}[ld] & \\
G/H &&G/K
}$$ is sent to the transfer $E^H\rightarrow E^K$.

In this way, a genuine $G$-spectrum encodes a family of fixed point spectra, along with restrictions and transfers. A genuine $G$-space encodes a family of fixed point spaces, along with (only) restrictions.

\subsubsection{Naive equivariant homotopy theory}
\noindent The case of naive equivariant homotopy theory is especially simple. In this case, $^G\text{Fin}$ is the free cocartesian monoidal category on $BG$, $^G\text{Burn}^\text{eff}$ is the free semiadditive $\infty$-category on $BG$ (\cite{Glasman} A.1), and $^G\text{Burn}$ is the free additive $\infty$-category on $BG$. Combining this information with the orbit categories of \S\ref{A.1}, some entries of the table above can be simplified: \\

\begin{tabular}{|c|c|c|c|}\hline
&naive left $G$- &genuine right $G$- &$(G,H)$-biequivariant \\ \hline
space &$\text{Fun}(BG^\text{op},\text{Top})$ &$\text{Fun}(\text{Orb}_G^\text{op},\text{Top})$ &$\text{Fun}(\null^G\text{Orb}_H^\text{op},\text{Top})$ \\ \hline
$\mathbb{E}_\infty$-space &$\text{Fun}(BG^\text{op},\text{Top})$ & & \\ \hline
connective spectrum &$\text{Fun}(BG^\text{op},\text{Top})$ & & \\ \hline
spectrum &$\text{Fun}(BG^\text{op},\text{Sp})$ & & \\ \hline
\end{tabular}

\section{Global equivariance}\label{A.3}
\subsection{Groupoids and global Burnside categories}\label{A.3.1}
\subsubsection{Unstable homotopy theory}
\noindent Let $\text{Gpoid}_\text{f}$ denote the 2-category of finite groupoids (groupoids with finitely many objects and finitely many morphisms), and $\text{Gp}_\text{f}$ the full subcategory of groupoids with a single object.

Note that $\text{Gp}_\text{f}$ is a 2-category: its objects are groups, and its morphisms are group homomorphisms, but its 2-morphisms are natural transformations of functors $BG\rightarrow BH$. These 2-morphisms encode information about conjugacy classes.

Just as $\text{Orb}_G$ is the orbit category for $\text{Fin}_G$, $\text{Gp}_\text{f}$ is the orbit category for $\text{Gpoid}_\text{f}$. That is, if $\mathcal{C}$ is an $\infty$-category with finite products, then $$\text{Fun}(\text{Gp}_\text{f}^\text{op},\mathcal{C})\cong\text{Fun}^\times(\text{Gpoid}_\text{f}^\text{op},\mathcal{C}).$$

\subsubsection{Stable homotopy theory}
\noindent Like the effective Burnside 2-category for a finite group $G$ can be described as the 2-category of spans of finite $G$-sets, the global effective Burnside 2-category is built out of spans of finite groupoids.

Unlike in `local' equivariant homotopy theory, this construction is asymmetric! That is, $\text{Burn}_\text{glo}\not\cong\text{Burn}_\text{glo}^\text{op}$. The reason for the asymmetry is that, for each span in $\text{Burn}_\text{glo}$, the \emph{ingressive} morphisms are required to be discrete fibrations.

\begin{definition}
A functor $F:\mathcal{H}\rightarrow\mathcal{G}$ between two groupoids is a \emph{discrete fibration} if for each morphism $f:X\rightarrow Y$ in $\mathcal{G}$ and lift of $X$ to $X^\prime\in\mathcal{H}$, there is a unique lift of $f$ to $f^\prime:X^\prime\rightarrow Y^\prime$ in $\mathcal{H}$.
\end{definition}

\noindent A discrete fibration is precisely a functor which is obtained via the Grothendieck construction applied to a functor $\mathcal{G}\rightarrow\text{Set}$. Thus, discrete fibrations of \emph{finite} groupoids over $BG$ correspond to objects of $\text{Fin}_G\cong\text{Fun}(BG,\text{Fin})$. The correspondence is via the action groupoid:

\begin{example}
\label{ActGpoid}
If $G$ is a finite group and $X$ is a $G$-set, then there is an \emph{action groupoid} $B_GX$ defined as follows: an object of $B_GX$ is an object of $X$, and for each relation $gx=y$ in $X$, there is a morphism $x\xrightarrow{g}y$ in $B_GX$.

There is a discrete fibration $B_GX\rightarrow BG$ which sends $x\xrightarrow{g}y$ to $g$. Moreover, $B_G$ produces an equivalence of categories between $\text{Fin}_G$ and finite discrete fibrations over $BG$.
\end{example}

\begin{definition}[\cite{MillerBurnglo}]
The global effective Burnside category $\text{Burn}_\text{glo}^\text{eff}$ is the $\infty$-category of spans of finite groupoids, where left-pointing (ingressive) morphisms in each span are required to be discrete fibrations. In other words, a morphism in $\text{Burn}_\text{glo}^\text{eff}$ is a span of finite groupoids $\mathcal{G}\leftarrow\mathcal{K}\rightarrow\mathcal{H}$, where $\mathcal{K}\rightarrow\mathcal{G}$ is a discrete fibration.

Although $\text{Burn}_\text{glo}^\text{eff}$ is a priori an $\infty$-category, in fact it is a 2-category.
\end{definition}

\noindent Miller published this construction in 2016, but he claims it has been known to experts for decades. See his paper \cite{MillerBurnglo} for details.

\subsubsection{The orbital category}%Add justification
\begin{remark}
Most treatments of global equivariant homotopy theory consider a smaller version of the global Burnside category, which is equivalent to the full subcategory of $\text{Burn}_\text{glo}^\text{eff}$ spanned by groups (groupoids with a single object). See for example Schwede's book \cite{Schwede}. This full subcategory admits a surprising description as follows:
\begin{itemize}
\item objects are finite groups;
\item a morphism from $H$ to $G$ is a $(G,H)$-biset such that the left $G$-action is free;
\item a 2-morphism is a map of bisets.
\end{itemize}
\end{remark}

\subsection{Global equivariant homotopy theory}\label{A.3.2}
\begin{definition}
A \emph{global equivariant space} is (equivalently) a finite-product-preserving functor $\text{Gpoid}_\text{f}^\text{op}\rightarrow\text{Top}$, or an arbitrary functor $\text{Gp}_\text{f}^\text{op}\rightarrow\text{Top}$.
\end{definition}

\noindent Unlike in `local' homotopy theory, $\text{Burn}_\text{glo}$ is not equivalent to its opposite. As a result, there are now two \emph{dual} theories of equivariant spectra:

\begin{definition}
A \emph{global equivariant spectrum} is a finite-product-preserving functor $(\text{Burn}_\text{glo}^\text{eff})^\text{op}\rightarrow\text{Sp}$.

A \emph{coglobal equivariant spectrum} is a finite-product-preserving functor $\text{Burn}_\text{glo}^\text{eff}\rightarrow\text{Sp}$.
\end{definition}

\begin{remark}
Both global and coglobal spectra consist of spectra $E_G$ for each group $G$, along with restrictions $E_G\rightarrow E_H$ and transfers $E_H\rightarrow E_G$ for any subgroup inclusion $H<G$.

However, global spectra also include restrictions along \emph{any} homomorphism $H\rightarrow G$, while coglobal spectra include transfers along any group homomorphism.

The author has found the following intuition useful to distinguish between global and coglobal spectra:
\begin{itemize}
\item A global spectrum can be thought of as a compatible family of equivariant spectra, one for each group. The universal example is the equivariant sphere spectrum $\mathbb{S}_\mathcal{G}$. (See Example \ref{KGlo}.)
\item A coglobal spectrum can be thought of as a single (hypothetical) spectrum on which every group acts simultaneously; the value at $G\in\text{Burn}_\text{glo}^\text{orb}$ records the fixed points of the $G$-action. The universal example $E\mathcal{G}$ is the hypothetical `spectrum on which every group acts freely', for which the value at $G$ ($G$-fixed points) is $\Sigma^\infty_{+}BG$. The corresponding $G$-spectrum is $\Sigma^\infty_{+}EG$. (See Example \ref{KCoglo}.)
\end{itemize}
\end{remark}

\subsubsection{From global equivariance to local equivariance}
\noindent Fix a finite group $G$. The action groupoid construction (Example \ref{ActGpoid}) produces a semiadditive functor $$B_G:\text{Burn}_G^\text{eff}\rightarrow\text{Burn}_\text{glo}^\text{eff},$$ which sends an object $X$ to $B_GX$, and a span $X\leftarrow T\rightarrow Y$ to $B_GX\leftarrow B_GT\rightarrow B_GY$. Precomposition along this functor induces $$\text{Sp}_\text{coglo}\cong\text{Fun}^\times(\text{Burn}_G^\text{eff},\text{Sp})\rightarrow\text{Fun}^\times(\text{Burn}_G^\text{eff},\text{Sp})\cong\text{Sp}_G,$$ which assigns to any coglobal spectrum a genuine $G$-spectrum for each $G$.

Similarly, since $\text{Burn}_G^\text{eff}$ is equivalent to its opposite, there are functors $\text{Sp}_\text{glo}\rightarrow\text{Sp}_G$ which assign to any global spectrum a genuine $G$-spectrum for each $G$.

\section{Equivariant category theory}\label{A.4}
\noindent The idea to study equivariant categories is due to Hill and Hopkins \cite{HillHopkins}, and there is an upcoming book on equivariant higher category theory by Barwick et al. \cite{Barwick}.

If $H$ is a subgroup of $G$, there is a \emph{restriction} functor $\text{Sp}_G\rightarrow\text{Sp}_H$ and two wrong-way functors $\text{Sp}_H\rightarrow\text{Sp}_G$, the \emph{transfer} and the \emph{norm}. The external norm features prominently in Hill, Hopkins, and Ravenel's celebrated solution to the Kervaire invariant one problem \cite{HHR}.

Ignoring the norm for now, the restriction and transfer combine to give $\text{Sp}_{-}$ the structure of a Mackey functor valued in $\infty$-categories!

\begin{definition}
\label{SymMonEq}
A \emph{genuine equivariant $\infty$-category} is a categorical coefficient system. A \emph{genuine equivariant symmetric monoidal $\infty$-category} is a categorical semiMackey functor. That is, $$\text{Cat}_G=\text{Fun}^\times(\text{Fin}_G^\text{op},\text{Cat})$$ $$\text{SymMon}_G=\text{Fun}^\times(\text{Burn}_G^\text{eff},\text{Cat}).$$
\end{definition}

\noindent In particular, a genuine equivariant $G$-$\infty$-category consists of an $\infty$-category $\mathcal{C}_H$ for each subgroup $H<G$, along with restrictions. A genuine equivariant symmetric monoidal $\infty$-category also includes transfers.

\begin{remark}
These definitions can be reproduced for any of the flavors of equivariant homotopy theory (\S\ref{A.1}) or global equivariant homotopy theory (\S\ref{A.3}). That is, a \emph{global equivariant $\infty$-category} is a product-preserving functor $\text{Burn}_\text{glo}\rightarrow\text{Cat}$.

Essentially all examples that we are aware of are really global equivariant $\infty$-categories.
\end{remark}

\begin{example}
The universal example of an equivariant symmetric monoidal $\infty$-category is given by the corepresentable functor $$\text{Map}(\ast,-):\text{Burn}_G^\text{eff}\rightarrow\text{Top}\subseteq\text{Cat}.$$ The corresponding $\infty$-category of $H$-objects is $\text{Fin}_H^\text{iso}$.
\end{example}

\begin{example}
There is a global analogue of the last example: the representable functor $$\text{Map}(-,\ast):(\text{Burn}_\text{glo}^\text{eff})^\text{op}\rightarrow\text{Top}\subseteq\text{Cat}$$ is the universal global equivariant symmetric monoidal $\infty$-category. The corresponding $\infty$-category of $G$-objects is $\text{Fin}_G^\text{iso}$.

On the other hand, the universal \emph{coglobal} equivariant symmetric monoidal $\infty$-category is $$\text{Map}(\ast,-):\text{Burn}_\text{glo}^\text{eff}\rightarrow\text{Top}\subseteq\text{Cat},$$ and the corresponding $\infty$-category of $G$-objects is $^G\text{Fin}^\text{iso}$.
\end{example}

\subsection{\texorpdfstring{$(\infty,2)$-}{(infinity,2)-}Lawvere theories}\label{A.4.1}
\noindent The material of this section is conjectural, pending a better understanding of enriched higher category theory (and in particular $(\infty,2)$-Lawvere theories).

We have proven that the following properties of symmetric monoidal $\infty$-categories correspond to the structure of modules over a semiring $\infty$-category $\mathcal{R}$:
\begin{itemize}
\item ($\mathcal{R}=\text{Fin}$) cocartesian monoidal $\infty$-category (Theorem \ref{2T3});
\item ($\mathcal{R}=\text{Fin}^\text{op}$) cartesian monoidal $\infty$-category (Theorem \ref{2T3b});
\item ($\mathcal{R}=\text{Burn}^\text{eff}$) semiadditive $\infty$-category (Theorem \ref{2T4});
\item ($\mathcal{R}=\text{Burn}$) additive $\infty$-category (Corollary \ref{2T7b}).
\end{itemize}

\noindent By Theorem \ref{2T7}, for any semiring space $R$, there is a Lawvere theory $\text{Burn}_R$ for $R$-modules, which is uniquely determined by the properties that $\text{Burn}_R$ is semiadditive and the endomorphism semiring of the generating object is $R$.

If $\mathcal{R}$ is instead a semiring $\infty$-category, we anticipate that there is an \emph{$(\infty,2)$-Lawvere theory} $\text{Burn}_\mathcal{R}$ such that $$\text{Mod}_\mathcal{R}\cong\text{Fun}^\times(\text{Burn}_\mathcal{R},\text{Cat}).$$ Again, $\text{Burn}_\mathcal{R}$ should be semiadditive (therefore enriched in symmetric monoidal $\infty$-categories), and the endomorphism semiring of the generating object should be $\mathcal{R}$.

\begin{example}
If $\mathcal{R}=\text{Fin}$, $\text{Burn}_\text{Fin}$ is the \emph{effective Burnside $(\infty,2)$-category} $\text{Burn}^{(\infty,2)}$ (which is really a $(2,2)$-category):
\begin{itemize}
\item an object is a finite set;
\item a morphism from $X$ to $Y$ is a span $X\leftarrow T\rightarrow Y$;
\item a 2-morphism from $X\leftarrow T\rightarrow Y$ to $X\leftarrow T^\prime\rightarrow Y$ is a function $T\rightarrow T^\prime$ which is compatible with the functions down to $X$ and $Y$.
\end{itemize}
$\text{opBurn}^{(\infty,2)}=\text{Burn}_{\text{Fin}^\text{op}}$ is similar but with the direction of the 2-morphisms reversed.
\end{example}

\begin{example}
If $\mathcal{R}=\text{Burn}^\text{eff}$, $\text{Burn}_\mathcal{R}$ is the \emph{effective iterated Burnside $(\infty,2)$-category} $\text{itBurn}^\text{eff}$ (which is really a $(2,2)$-category):
\begin{itemize}
\item an object is a finite set;
\item a morphism from $X$ to $Y$ is a span $X\leftarrow T\rightarrow Y$;
\item a 2-morphism from $X\leftarrow T\rightarrow Y$ to $X\leftarrow T^\prime\rightarrow Y$ is a span $T\leftarrow S\rightarrow T^\prime$ which is compatible with the functions down to $X$ and $Y$.
\end{itemize}
If $\mathcal{R}=\text{Burn}$, $\text{Burn}_\mathcal{R}$ is the \emph{virtual iterated Burnside $(\infty,2)$-category} itBurn, obtained from $\text{itBurn}^\text{eff}$ by additivizing Hom-categories (which are already semiadditive $\infty$-categories).
\end{example}

\noindent In this way, each of $\text{Burn}^{(\infty,2)}$, $\text{opBurn}^{(\infty,2)}$, $\text{itBurn}^\text{eff}$, and itBurn can be described via a combinatorial machine applied to Fin. Applying Principle \ref{LawvMachPr}:

\begin{definition}
A genuine $G$-equivariant cocartesian monoidal $\infty$-category is a product-preserving functor $$\text{Burn}_G^{(\infty,2)}\rightarrow\text{Cat}.$$ A genuine equivariant cartesian monoidal $\infty$-category is a product-preserving functor $$\text{opBurn}_G^{(\infty,2)}\rightarrow\text{Cat}.$$ A genuine equivariant semiadditive $\infty$-category is a product-preserving functor $$\text{itBurn}^\text{eff}_G\rightarrow\text{Cat}.$$ A genuine equivariant additive $\infty$-category is a product-preserving functor $$\text{itBurn}_G\rightarrow\text{Cat}.$$
\end{definition}

\begin{example}
The universal example of a $G$-cocartesian monoidal $\infty$-category is the corepresentable functor $$\underline{\text{Fin}}_G:\text{Map}(\ast,-):\text{Burn}_G^{(\infty,2)}\rightarrow\text{Cat},$$ which sends the orbit $G/H$ to $\text{Fin}_H$. Similarly, the universal $G$-additive $\infty$-category is $\underline{\text{Burn}}_G$, etc.
\end{example}

\begin{definition}
If $\underline{\mathcal{C}}$ is a genuine equivariant symmetric monoidal $\infty$-category, then a commutative algebra in $\underline{\mathcal{C}}$ is a map $\underline{\text{Fin}}_G\rightarrow\underline{\mathcal{C}}$ in $\text{SymMon}_G$.
\end{definition}

\noindent A central theme of this thesis is the use of algebraic techniques in category theory. In a sense, these techniques are possible due to constructions involving the $\infty$-category $\text{Pr}^\text{L}$ of presentable $\infty$-categories (\S\ref{1.1.4}).

Therefore, if we want to study equivariant category theory, we might want to begin with a good notion of equivariant presentable $\infty$-category. In this spirit, we end with a conjecture suggested to the author by Clark Barwick.

\begin{conjecture}
A genuine $G$-equivariant presentable $\infty$-category is a product-preserving functor $$\text{Burn}_G^{(\infty,2)}\rightarrow\text{Pr}^\text{L};$$ that is, a genuine $G$-equivariant cocartesian monoidal $\infty$-category $\mathcal{C}$ such that each $\infty$-category $\mathcal{C}_H$ ($H<G$) is presentable, and each transfer and restriction has a right adjoint.
\end{conjecture}

\noindent This kind of structure may be related to the Grothendieck six-operations formalism in stable motivic homotopy theory \cite{Hoyois}.

\end{document}